\documentclass[11pt,reqno]{amsart}

\usepackage{fourier}
\usepackage{fullpage}

\usepackage{amsmath,amsfonts,amssymb,mathrsfs}
\usepackage{graphicx,extpfeil}

\usepackage{indentfirst, latexsym,  enumerate}
\usepackage{float}
\usepackage{epsfig,subfigure}

\usepackage{graphicx,colortbl}

\usepackage{caption}
\usepackage{bm}
\usepackage{bbm}
\usepackage{esint}

\title[Quantitative nonembeddability of groups of polynomial growth into uniformly convex spaces]{Quantitative nonembeddability of groups of polynomial growth into \\uniformly convex spaces}

\author[Seung-Yeon Ryoo]{Seung-Yeon Ryoo}
\address[Seung-Yeon Ryoo]{\newline
Department of Mathematics, California Institute of Technology
\newline
1200 East California Boulevard
\newline
Pasadena California 91125 United States
} \email{sryoo@caltech.edu}

\newtheorem{theorem}{Theorem}
\newtheorem{lemma}[theorem]{Lemma}
\newtheorem{corollary}[theorem]{Corollary}
\newtheorem{proposition}[theorem]{Proposition}
\newtheorem{example}[theorem]{Example}
\newtheorem{remark}[theorem]{Remark}
\newtheorem{conjecture}[theorem]{Conjecture}
\newtheorem{question}[theorem]{Question}
\newtheorem{definition}[theorem]{Definition}

\newcommand{\Z}{\mathbb{Z}}
\newcommand{\1}{\mathbf 1}
\newcommand{\R}{\mathbb{R}}
\newcommand{\Norm}[2]{\|#1\|_{#2}}

\newcommand{\bbr}{\mathbb R}

\newcommand{\bbh}{\mathbb H}

\begin{document}
\allowdisplaybreaks

\date{\today}


\thanks{
Acknowledgments: Part of this work appeared in my doctoral dissertation under the supervision of Professor Assaf Naor at Princeton University. I thank him, Professors Tuomas Orponen and Robert Young, and Ian Fleschler and Petr Kosenko for helpful discussions and suggestions. I am indebted to the anonymous referee for an extremely thorough reading of the manuscript, many useful corrections, and suggested improvements. This work was partially supported by the Korea Foundation for Advanced Studies and an AMS--Simons Travel Grant.
}

\begin{abstract}
Nonabelian simply connected nilpotent Lie groups and not virtually abelian finitely generated groups of polynomial growth do not quasi-isometrically embed into uniformly convex Banach spaces. We quantify this fact by showing that a ball of radius $r\ge 2$ in the aforementioned groups must incur bilipschitz distortion at least a constant multiple of $(\log r)^{1/q}$ into a $q(\ge 2)$-uniformly convex Banach space. This bound is sharp for the $L^p$ ($1<p<\infty$) spaces. We prove this by establishing ``vertical versus horizontal inequalities'' for functions from the aforementioned groups into uniformly convex spaces, using the vector-valued Littlewood--Paley--Stein theory approach of Lafforgue and Naor (2012). These inequalities are quantitative nonembeddability statements that any Lipschitz mapping from the aforementioned groups into a uniformly convex space quantitatively collapses along certain central subgroups.

In the special case of mappings of Carnot groups into the $L^p$ ($1<p<\infty$) spaces, we prove that the quantitative collapse occurs on the commutator subgroup; this is in line with the qualitative Pansu--Semmes nonembeddability argument given by Cheeger and Kleiner (2006) and Lee and Naor (2006). We prove this by establishing a version of the classical Dorronsoro theorem on Carnot groups. Previously, in the setting of Heisenberg groups, F\"assler and Orponen (2019) established a one-sided Dorronsoro theorem with a restriction $0<\alpha<2$ on the range of exponents $\alpha$ of the Laplacian; this restriction does not appear in the commutative setting and is caused by their use of horizontal polynomials as approximants. We identify the correct class of approximant polynomials and prove the two-sided Dorronsoro theorem with the full range $0<\alpha<\infty$ of exponents in the general setting of Carnot groups, thus strengthening and extending the work of F\"assler and Orponen.
\end{abstract}
\maketitle \centerline{\date}

\tableofcontents

\section{Introduction}

Many groups of polynomial growth, such as nonabelian simply connected nilpotent Lie groups and not virtually abelian finitely generated groups of polynomial growth, fail to embed bilipschitzly (or even quasi-isometrically) into uniformly convex Banach spaces.\footnote{The previous version of this paper incorrectly claimed that not virtually abelian finitely generated groups of polynomial growth fail to bilipschitzly embed into Banach spaces with the Radon--Nikod\'ym property. The anonymous referee pointed out that this statement is false since uniformly discrete spaces, in particular finitely generated groups, isometrically embed into some Banach spaces with the Radon--Nikod\'ym property \cite[Proposition 4.4]{kalton2004spaces}.} This is because a not virtually abelian finitely generated group of polynomial growth is quasi-isometric to some nonabelian simply connected nilpotent Lie group \cite{gromov1981groups,malcev1949class}, a nonabelian simply connected nilpotent Lie group has a nonabelian Carnot group as asymptotic cone, and a nonabelian Carnot group does not bilipschitzly embed into a uniformly convex Banach space \cite{cheeger2006differentiability,lee2006lp}. The aim of this paper is to provide quantitative counterparts to these qualitative nonembeddability statements.

When considering embeddability of groups of polynomial growth into some other space, it is often enough to consider simply connected nilpotent Lie groups and finitely generated groups of polynomial growth. Indeed, given a locally compact group $\mathcal{G}$ with left Haar measure $\mu$ that is generated by a compact symmetric neighborhood $U$ of the identity, suppose $\mathcal{G}$ is of polynomial growth, i.e., suppose that $\mu\left(U^n\right)$ grows at most polynomially in $n$. Then $\mathcal{G}$, with the left-invariant word distance induced by $U$, is quasi-isometric to a connected and simply connected solvable Lie group $\mathcal{S}$ of polynomial growth, called the Lie shadow of $\mathcal{G}$
 \cite[Theorem 1.2, Proposition 1.3]{breuillard2014geometry}.\footnote{I thank Petr Kosenko for pointing out this reference.} In turn, a connected solvable Lie group of polynomial growth may be made isometric to a connected nilpotent Lie group, namely its nilshadow \cite[Corollary 98]{cowling2024homogeneous}. An alternative point of view is that such a $\mathcal{G}$ has a normal series $C\trianglelefteq R\trianglelefteq N\trianglelefteq \mathcal{G}$, such that $C$ and $\mathcal{G}/N$ are compact, $R/C$ is a connected solvable Lie group of polynomial growth, and $N/R$ is a discrete nilpotent group \cite[Corollary]{losert1987structure}.

We now begin by defining the concepts appearing in the first sentence of the first paragraph.
Let $G$ denote a nonabelian connected simply connected nilpotent Lie group with Lie algebra $\mathfrak{g}$. Endow $G$ with a left-invariant Carnot--Carath\'eodory distance. This means the following. There are linearly independent left-invariant vector fields $X_1,\cdots,X_k$ that satisfy the H\"ormander condition, i.e., $X_1,\cdots,X_k$ and their Lie brackets span the Lie algebra $\mathfrak{g}$. The pointwise span of $X_1,\cdots,X_k$ forms a left-invariant vector bundle $B$ over $G$, which is a subbundle of the tangent bundle of $G$, and on each fibre of the vector  bundle $B$ we define a left-invariant Euclidean norm\footnote{We could also take any general left-invariant norm on $B$, in which case the distance is Finsler or sub-Finsler, according to whether $\operatorname{span}\left\{X_1,\cdots,X_k\right\}=\mathfrak{g}$ or $\subsetneq \mathfrak{g}$, respectively. Compared to the Riemannian or sub-Riemannian case, the resulting distance $d_G$ is then distorted by a factor of at most $\sqrt{k}$ by the John ellipsoid theorem \cite{john1948extremum}, and the results of this paper follow up to multiplicative factors of  $\sqrt{k}$.} $\left|\cdot \right|$ that has $X_1,\cdots,X_k$ as an orthonormal basis. We define the \emph{Carnot--Carath\'eodory distance associated to $B$ and $\left|\cdot \right|$,} i.e.,
\[
d_G\left(p,q\right)\coloneqq \inf\left\{\int_0^1 \left|\dot{\gamma}\left(t\right) \right|dt : \gamma\in C_{\mathrm{pw}}^\infty \left(\left[0,1\right];G\right),\gamma\left(0\right)=p,\gamma\left(1\right)=q,\dot{\gamma}\in B \right\},\quad p,q\in G,
\]
where $C_{\mathrm{pw}}^\infty \left(\left[0,1\right];G\right)$ consists of the piecewise smooth functions from $\left[0,1\right]$ to $G$. The distance is called a \emph{Riemannian} distance if $\operatorname{span}\left\{X_1,\cdots,X_k\right\}=\mathfrak{g}$, and a \emph{sub-Riemannian} distance if $\operatorname{span}\left\{X_1,\cdots,X_k\right\}\subsetneq \mathfrak{g}$. By the Chow--Rashevskii theorem \cite{chow1940systeme,rashevskii1938connecting}, $d_G\left(x,y\right)$ is finite for every $x,y\in G$. We write
\[
B_r\left(x\right)\coloneqq \left\{y\in G:d_G\left(x,y\right)<r\right\},~B_r\coloneqq B_r\left(e_G\right),\quad x\in G,~r>0,
\]
where $e_G$ is the identity element of $G$. Note that $B_r\left(x\right)=xB_r$ by left-invariance.

Likewise, given a connected and simply connected solvable Lie group $\mathcal{S}$ of polynomial growth, one may endow it with a left-invariant Carnot--Carath\'eodory distance.

Denote by $\exp:\mathfrak{g}\to G$ the Lie group exponential map. Let $\mu$ denote the bi-invariant Haar measure of $G$, which coincides with the push-forward of the Lebesgue measure on $\mathfrak{g}$ by the exponential map.

Let $\Gamma$ denote a finitely generated group of polynomial growth. This means that $\Gamma$ has a finite generating set $S$ that is symmetric ($S=S^{-1}$), and that if we denote by $d_W\left(\cdot,\cdot\right)$ the left-invariant word metric on $\Gamma$ induced by $S$, i.e.,
\[
d_W\left(p,q\right)=\inf\left\{n\ge 0:p=qa_1\cdots a_n,~a_1,\cdots,a_n\in S\right\},
\]
and denote by $B_n^\Gamma=\left\{x\in \Gamma: d_W\left(x,e_\Gamma\right)\le n\right\}$ the corresponding closed ball of radius $n\in \mathbb{N}$, where $e_\Gamma$ is the identity element of $\Gamma$, then the cardinality $\left|B^\Gamma_n\right|$ grows at most polynomially in $n$. We assume in addition that $\Gamma$ is not virtually abelian, i.e., that it has no finite index subgroup isomorphic to $\mathbb{Z}^n$; by \cite[Corollary 1.5]{de2007isometric} $\Gamma$ has no finite index subgroups such that some quotient by a finite normal subgroup is abelian.

Let the Banach space $\left(X,\left\|\cdot\right\|_X\right)$ be a uniformly convex Banach space, that is, for every $\varepsilon\in \left(0,1\right)$ there exists $\delta\in \left(0,1\right)$ such that every $x,y\in X$ with $\left\|x\right\|_X=\left\|y\right\|_X=1$ and $\left\|x-y\right\|_X\ge \varepsilon$ satisfy $\left\|x+y\right\|_X\le 2\left(1-\delta\right)$. By \cite{ball1994sharp,figiel1976moduli,pisier1975martingales}, we can renorm $X$ with an equivalent norm, so that there is some exponent $q\in \left[2,\infty\right)$ for which the Banach space $\left(X,\left\|\cdot\right\|_X\right)$ is $q$-uniformly convex, which means that the $q$-uniform convexity constant of $X$ defined by
\begin{equation}\label{eq:unif-cvx-def}
K_q\left(X\right)\coloneqq \inf\left\{K>0:\forall x,y\in X~\left\|x\right\|_X^q+\frac{1}{K^q}\left\|y\right\|_X^q\le \frac{\left\|x+y\right\|_X^q+\left\|x-y\right\|_X^q}{2}\right\}
\end{equation}
is finite. Henceforth, unless noted otherwise, $X$ will denote a $q$-uniformly convex Banach space for some $q\ge 2$, because the bilipschitz distortion caused by renorming by an equivalent norm does not qualitatively affect the conclusion of our theorems. By \cite{hanner1956uniform} and \cite{ball1994sharp}, the $L^p$ spaces for $1<p<\infty$ are $\max\left\{p,2\right\}$-uniformly convex with
\begin{equation}\label{eq:unifcvxlp}
K_2\left(L^p\right)\le \frac{1}{\sqrt{p-1}},~1<p\le 2,\quad K_p\left(L^p\right)= 1,~p\ge 2.
\end{equation}

A separable metric space $\left({\mathcal{M}},d_{\mathcal{M}}\right)$ is said to embed \emph{($D$-)bilipschitzly} into a normed space $\left(X,\left\|\cdot\right\|_X\right)$ if there is a finite $D\in \left[1,\infty\right)$ and a mapping $f:{\mathcal{M}}\to X$ such that
\[
d_{\mathcal{M}}\left(x,y\right)\le \left\|f\left(x\right)-f\left(y\right)\right\|_X\le D d_{\mathcal{M}}\left(x,y\right) \quad \forall x,y\in {\mathcal{M}}.
\]
The bilipschitz distortion $c_X\left({\mathcal{M}},d_{\mathcal{M}}\right)$ is defined to be the infimum over those $D$ for which such a mapping exists. We write $c_p\left({\mathcal{M}},d_{\mathcal{M}}\right)\coloneqq c_{L^p\left[0,1\right]}\left({\mathcal{M}},d_{\mathcal{M}}\right)$. This satisfies $c_p\left({\mathcal{M}},d_{\mathcal{M}}\right)\ge c_{L^p\left(N,\sigma\right)}\left({\mathcal{M}},d_{\mathcal{M}}\right)$ for any measure space $\left(N,\sigma\right)$, since any separable subspace of $L^p\left(N,\sigma\right)$ is isometric to a subspace of $L^p\left[0,1\right]$ \cite[Fact 1.20]{ostrovskii2013metric}. We call $c_2\left({\mathcal{M}},d_{\mathcal{M}}\right)$ the Euclidean distortion of $\left({\mathcal{M}},d_{\mathcal{M}}\right)$. The space $\left({\mathcal{M}},d_{\mathcal{M}}\right)$ is said to \emph{quasi-isometrically embed} into $\left(X,\left\|\cdot\right\|_X\right)$ if there is a mapping $f:{\mathcal{M}}\to X$, $D\in \left[1,\infty\right)$ and $C>0$ such that
\[
d_{\mathcal{M}}\left(x,y\right)-C \le \left\|f\left(x\right)-f\left(y\right)\right\|_X\le D d_{\mathcal{M}}\left(x,y\right)+C\quad \forall x,y\in {\mathcal{M}}.
\]

Now we have defined all the concepts used in the first sentence of the first paragraph, which, in other words, can now be written concisely as
\[
c_X\left(G\right)=c_X\left(\Gamma\right)=\infty,
\]
for $G$, $\Gamma$, and $X$ given as above. In this paper, we quantify this fact by providing growth rates of the bilipschitz distortion of balls in such groups. First, we have the following theorem for nilpotent Lie groups:

\begin{theorem}\label{mainthm:nilpotentdistortion}
Let $G$ be a nonabelian simply connected nilpotent Lie group endowed with a left-invariant Riemannian distance, and let $X$ be a $q$-uniformly convex Banach space with $q\ge 2$. Then\footnote{We will use the following (standard) asymptotic notation. For $P, Q>0$, the notations $P\lesssim Q$, $Q\gtrsim P$, $P=O\left(Q\right)$, and $Q=\Omega\left(P\right)$ mean that $P\le KQ$ for a universal constant $K\in \left(0,\infty\right)$, and the notation $P\asymp Q$ means $P\lesssim Q$ and $Q\lesssim P$. If we need to allow for dependence on parameters, we indicate this by subscripts. For example, in the presence of auxiliary parameters $\psi$ and $\xi$, the notations $P\lesssim_{\psi,\xi}Q$, $Q\gtrsim_{\psi,\xi}P$, $P=O_{\psi,\xi}\left(Q\right)$, $Q=\Omega_{\psi,\xi}\left(P\right)$ mean that $P\le K\left(\psi,\xi\right)Q$, where $K\left(\psi,\xi\right)\in \left(0,\infty\right)$ may depend only on $\psi$ and $\xi$. Similarly, $P\asymp_{\psi,\xi}Q$ means that $P\lesssim_{\psi,\xi}Q$ and $Q\lesssim_{\psi,\xi}P$.}
\[
c_X\left(B_r\right)\gtrsim_G \frac{\left(\log r\right)^{1/q}}{K_q\left(X\right)},\quad r\ge 2.
\]
\end{theorem}
Likewise, we obtain the following corollary for solvable Lie groups of polynomial growth, since such a group is isometric to its nilshadow.
\begin{corollary}
Let $\mathcal{S}$ be a simply connected solvable Lie group of polynomial growth endowed with a Riemannian distance and whose nilshadow is nonabelian. Then
\[
c_X\left(B^{\mathcal{S}}_r\right)\gtrsim_\mathcal{S} \frac{\left(\log r\right)^{1/q}}{K_q\left(X\right)},\quad r\ge 2,
\]
where $B^{\mathcal{S}}_r$ denotes the ball of radius $r$ centered at the identity in $\mathcal{S}$.
\end{corollary}
The precise dependence on $G$ of the constant in Theorem \ref{mainthm:nilpotentdistortion} is given in Theorem \ref{thm:precise-dist-nilp} in subsubsection \ref{subsubsec:quant} below. Also, the reason we are only considering Riemannian distances in Theorem \ref{mainthm:nilpotentdistortion} is that if $G$ were given a sub-Riemannian distance, then $c_X\left(B_r\right)=\infty$ for $r>0$ \cite[Theorem 1.4]{huang2020non}. For the sub-Riemannian case, non-coarse-embeddability results are given later in subsection \ref{subsec:coarse-nonembed} such as in Theorem \ref{thm:continuous-snowflake}, and refined quantitative nonembeddability statements are given in subsubsection \ref{subsubsec:net} in terms of nets, see for example Theorem \ref{thm:net} there.

Next, we have the following theorem for finitely generated groups of polynomial growth:
\begin{theorem}\label{mainthm:polygrowthdistortion}
Let $\Gamma$ be a not virtually abelian finitely generated group of polynomial growth endowed with a left-invariant word distance, and let $X$ be a $q$-uniformly convex Banach space with $q\ge 2$. Then
\[
c_X\left(B^\Gamma_n\right)\gtrsim_\Gamma \frac{\left(\log n\right)^{1/q}}{K_q\left(X\right)},\quad n\ge 2.
\]
\end{theorem}
Since a vertex transitive graph of polynomial growth is quasi-isometric to a Cayley graph of a nilpotent group \cite{godsil1989note,trofimov1985graphs}, we obtain the following.
\begin{corollary}
	Let $H$ be a vertex transitive graph of polynomial growth not quasi-isometric to $\mathbb{Z}^n$ for any $n\in \mathbb{N}$, and let $X$ be a $q$-uniformly convex Banach space with $q\ge 2$. Then
	\[
	c_X\left(B^H_n\right)\gtrsim_H \frac{\left(\log n\right)^{1/q}}{K_q\left(X\right)},\quad n\ge 2.
	\]
\end{corollary}

As a special case, we obtain sharp bounds on the $L^p$-distortions, the upper bounds following from a version of the Assouad embedding theorem \cite{assouad1983plongements} given as Theorem \ref{thm:lp-assouad} in Appendix \ref{app:assouad}.

\begin{corollary}\label{cor:nilpotentlp}
Let $G$ be a nonabelian simply connected nilpotent Lie group endowed with a left-invariant Riemannian distance. Then for $r\ge 2$,
\[
\begin{cases}
\sqrt{p-1}\sqrt{\log r}\lesssim_G c_p\left(B_r\right)\lesssim_{G} \sqrt{\log r},& 1<p<2,\\
 \phantom{\sqrt{p-1}\sqrt{\log r}\lesssim_G~ } c_p\left(B_r\right)\asymp_{G} \left(\log r\right)^{1/p},& 2\le p<\infty.
\end{cases}
\]
\end{corollary}
Again, we remark that if $G$ were given a sub-Riemannian distance, then $c_p\left(B_r\right)=\infty$ for $r>0$.

The same result is true for simply connected solvable Lie groups of polynomial growth equipped with a Riemannian distance and whose nilshadow is nonabelian.

\begin{corollary}\label{cor:polygrowthlp}
Let $\Gamma$ be a not virtually abelian finitely generated group of polynomial growth endowed with a left-invariant word distance. Then for $n\ge 2$,
\[
\begin{cases}
\sqrt{p-1}\sqrt{\log n}\lesssim_\Gamma c_p\left(B^\Gamma_n\right)\lesssim_{\Gamma} \sqrt{\log n},& 1<p<2,\\
 \phantom{\sqrt{p-1}\sqrt{\log n}\lesssim_G~ } c_p\left(B^\Gamma_n\right)\asymp_{\Gamma} \left(\log n\right)^{1/p},& 2\le p<\infty.
\end{cases}
\]
\end{corollary}
Likewise, we obtain the same result for vertex transitive graphs of polynomial growth not quasi-isometric to any $\mathbb{Z}^n$.

The bounds of Corollaries \ref{cor:nilpotentlp} and \ref{cor:polygrowthlp} give us the precise asymptotic behavior, namely $(\log r)^{1/\max\{p,2\}}$ and $(\log n)^{1/\max\{p,2\}}$, with respect to $r$ and $n$, up to constant factors for $X=L^p$. Therefore, the lower bounds of Theorems \ref{mainthm:nilpotentdistortion} and \ref{mainthm:polygrowthdistortion} are the best asymptotic result we can attain in the class of $q$-uniformly convex spaces (with possible room for improvement in the constant factors). We remark that the conclusions of Corollaries \ref{cor:nilpotentlp} and \ref{cor:polygrowthlp} hold with $L^p[0,1]$ replaced by the Schatten class $S_p$, whose modulus of uniform convexity was computed in \cite{tomczak1974moduli}.

We will prove Theorems \ref{mainthm:nilpotentdistortion} and \ref{mainthm:polygrowthdistortion} and Corollaries \ref{cor:nilpotentlp} and \ref{cor:polygrowthlp} in Section \ref{sec:net}, after developing the necessary machinery.

We remark that the special case of Theorem \ref{mainthm:polygrowthdistortion} when $\Gamma=\mathbb{H}_\mathbb{Z}^3$, the 3-dimensional discrete Heisenberg group, was proven by Lafforgue and Naor \cite{lafforgue2014vertical}. Here, the discrete Heisenberg groups $\mathbb{H}^{2k+1}_\mathbb{Z}$, $k\in \mathbb{Z}_{>0}$, are defined as the groups with word relations as follows:
\begin{align*}
\mathbb{H}^{2k+1}_\mathbb{Z}= \Big\langle a_1,\cdots,a_k,b_1,\cdots,b_k,c\,\Big|\,& \forall i\left(\left[a_i,b_i\right]=c,~\left[a_i,c\right]=\left[b_i,c\right]=e_{\mathbb{H}_\mathbb{Z}^{2k+1}}\right) \\
& \wedge \forall i,j\left(i\neq j\rightarrow \left[a_i,a_j\right]=\left[b_i,b_j\right]=\left[a_i,b_j\right]=e_{\mathbb{H}_\mathbb{Z}^{2k+1}}\right)\Big\rangle.
\end{align*}

Previously, a weaker version of Theorem \ref{mainthm:polygrowthdistortion} was given by Li \cite[Theorem 1.4]{li2014coarse}, who proved that if $\Gamma$ is a finitely generated nonabelian torsion-free nilpotent group, and $X$ is a uniformly convex Banach space, then there exists $c>0$ depending on $\Gamma$ and $X$ such that
\[
c_X\left(B^\Gamma_n\right)\gtrsim_{\Gamma,X}\left(\log n\right)^{c}, \quad n \ge 2.
\]
A stronger bound for Euclidean targets was given for cocompact lattices $\Gamma$ of Carnot groups by Gartland \cite[Corollary 1.6]{gartland2020estimates}: if $G$ is a Carnot group of step $s$ (to be defined in subsubsection \ref{subsubsec:carnot-qual}) and $\Gamma$ is a cocompact lattice of $G$, then
\[
c_{\mathbb{R}^d}\left(B_n^\Gamma\right)\gtrsim_{\Gamma,d} \frac{\left(\log n\right)^{\frac 12-\frac 1{2s}}}{\left(\log\log n\right)^{\frac 12+\frac 1{2s}}},\quad n\ge 3,
\]
under the additional restriction that $G$ contains a copy of the model filiform group $J^{s-1}\left(\mathbb{R}\right)$ of step $s$ (this requirement is automatically fulfilled if $s\le 3$; see right after Definition \ref{def:carnot} in subsubsection \ref{subsubsec:carnot-qual} for the definition of $J^{s-1}\left(\mathbb{R}\right)$). Compared to this result, Corollary \ref{cor:polygrowthlp} removes the exponent gap of $\frac 1{2s}$, has no lower order factor, and does not assume an embedding of $\Gamma$ into a Carnot group.

\bigskip
\noindent{\bf Roadmap.} The rest of the introduction is organized as follows. In subsection \ref{subsec:VvsH}, we describe the ``vertical versus horizontal inequalities'' which lead to the quantitative nonembeddability statements of Theorems \ref{mainthm:nilpotentdistortion} and \ref{mainthm:polygrowthdistortion}. In subsection \ref{subsec:coarse-nonembed}, we give nonembeddability results with respect to other coarse embeddings, by considering snowflake embeddings and compression rates of Lipschitz functions. In subsection \ref{subsec:net}, we provide refined distortion estimates by elucidating the dependence of the constants on $G$, and by discussing nonembeddability of nets.

In subsection \ref{subsec:carnot}, we present our results in the setting of Carnot groups. In subsection \ref{subsec:dorronsoro}, we discuss in detail Dorronsoro's theorem on Carnot groups (Theorem \ref{lpgenthm}) and describe the resulting more refined fractional vertical versus horizontal inequalities (Theorem \ref{VvsH}). In subsection \ref{subsec:l1}, we provide conjectural quantitative nonembeddability statements of nonabelian simply connected nilpotent Lie groups into $L^1$ (Conjecture \ref{conj:finite}) to suggest a candidate behavior for the $L^1$-distortion of balls in Riemannian nilpotent Lie groups and finitely generated groups of polynomial growth (Question \ref{ques:polyL1}).

After the introduction, the rest of this paper is organized as follows. We begin by proving the ``continuous vertical versus horizontal inequalities'' of Theorems \ref{thm:nilpotentVvsH} and \ref{thm:carnotVvsH} in Section \ref{sec:cvx}, and we derive the ``discrete vertical versus horizontal inequality'' of Theorem \ref{thm:discrete} in Section \ref{sec:proof main}. We then prove the distortion bounds of Theorems \ref{mainthm:nilpotentdistortion}, \ref{mainthm:polygrowthdistortion}, \ref{thm:continuous-snowflake}, \ref{thm:discrete-snowflake}, \ref{thm:precise-dist-nilp}, \ref{thm:precise-holder-nilp}, \ref{thm:net}, \ref{thm:net-snowflake}, \ref{thm:netdistortion}, and \ref{thm:netdistortion-snowflake} and Corollaries \ref{cor:nilpotentlp}, \ref{cor:polygrowthlp}, and \ref{cor:growth-function-char} in Section \ref{sec:net}. In Section \ref{sec:amenable}, we prove Theorem \ref{thm:amenablesublinear} by analyzing cocycles. In Section \ref{sec:dorronsoro}, we then prove Theorem \ref{lpgenthm}, Dorronsoro's theorem on Carnot groups. Finally, in Section \ref{sec:VvsH}, we derive the vertical versus horizontal inequality of Theorem \ref{thm:lp} and its fractional variant Theorem \ref{VvsH} from Theorem \ref{lpgenthm}.


\subsection{Vertical versus horizontal inequalities}\label{subsec:VvsH}

We obtain Theorems \ref{mainthm:nilpotentdistortion} and \ref{mainthm:polygrowthdistortion} by proving ``vertical versus horizontal inequalities'' on the groups $G$ and $\Gamma$, namely Theorems \ref{thm:discrete} and \ref{thm:nilpotentVvsH}. We begin by stating the vertical versus horizontal inequality on finitely generated groups of polynomial growth $\Gamma$, namely Theorem \ref{thm:discrete}, since it is easier to state than Theorem \ref{thm:nilpotentVvsH} as it requires less terminology.

\begin{theorem}\label{thm:discrete}
Let $\Gamma$ be a not virtually abelian finitely generated group of polynomial growth. There exist $v_\Gamma\in \Gamma$, $\rho\in \mathbb{N}$ with $\rho\ge 2$, and $c=c\left(\Gamma\right)\in \mathbb{N}$ such that the following is true. First, $d_W\left(v_\Gamma^n,e_\Gamma\right)\asymp_\Gamma n^{1/\rho}$ for $n\in \mathbb{N}$. Second, let $q\in \left[2,\infty\right)$ and $p\in (1, q]$. Suppose that $\left(X,\left\|\cdot\right\|_X\right)$ is a Banach space satisfying $K_q\left(X\right)<\infty$.  Then for every $n\in \mathbb{N}$ and every $f:\Gamma\to X$ we have
\begin{align}\label{eq:desired local pq}
\begin{aligned}
&\left(\sum_{k=1}^{n^\rho}\frac{1}{k^{1+q/\rho}}\left(\sum_{x\in B^\Gamma_n} \left\|f\left(xv_\Gamma^k\right)-f\left(x\right)\right\|_X^p\right)^{q/p}\right)^{1/q}\\
&\lesssim_{\Gamma,v_\Gamma} \max\left\{\left(p-1\right)^{(1/q)-1},K_q\left(X\right)\right\} 
\left(\sum_{x\in B^\Gamma_{cn}} \sum_{a\in S}\left\|f\left(xa\right)-f\left(x\right)\right\|^p_X\right)^{1/p}.
\end{aligned}
\end{align}
In particular, when $p=q$,
\begin{equation}\label{eq:main}
\left(\sum_{k=1}^{n^\rho}\sum_{x\in B^\Gamma_n}\frac{ \left\|f\left(xv_\Gamma^k\right)-f\left(x\right)\right\|_X^q}{k^{1+q/\rho}}\right)^{1/q}\lesssim_{\Gamma,v_\Gamma}
K_q\left(X\right)\left( \sum_{x\in B^\Gamma_{cn}} \sum_{a\in S}\left\|f\left(xa\right)-f\left(x\right)\right\|^q_X\right)^{1/q}.
\end{equation}
\end{theorem}

\begin{remark}
Our choice of $\rho$ and $v_\Gamma$ in the proof of Theorem \ref{thm:discrete} is as follows. By \cite{gromov1981groups}, $\Gamma$ admits a subgroup $\Gamma'$ of finite index that is nilpotent. Let $T$ be the torsion subgroup of $\Gamma'$ and consider the quotient subgroup $\Gamma''=\Gamma'/T$. We may take any $v''\in Z\left(\Gamma''\right)\setminus\left\{e_{\Gamma''}\right\}$ with the property that $d_W\left(\left(v''\right)^n,e_{\Gamma''}\right)\asymp_{\Gamma''}n^{1/\rho}$, $n\ge 2$, for some integer $\rho\ge 2$: for example, we may take $\rho=s$ to be the nilpotency step of $\Gamma''$ and $v''\in \underbrace{\left[\Gamma'',\left[\Gamma'',\cdots,\Gamma''\right]\right]}_{s\mathrm{~times}}\setminus\left\{e_{\Gamma''}\right\}$. Take $v_\Gamma$ to be any representative of $v''$ in $\Gamma$. For a detailed account of these choices, see the proof of Theorem \ref{thm:discrete} in Section \ref{sec:proof main}.
\end{remark}

\begin{remark}
    Changing to another finite symmetric generating set $S$ of $\Gamma$ or changing the vertical element $v_\Gamma$ affects only $\rho,c\in \mathbb{N}$ and the implied constant in the inequalities \eqref{eq:desired local pq} and \eqref{eq:main}.
\end{remark}

Theorem \ref{thm:discrete}, the vertical versus horizontal inequality on finitely generated groups of polynomial growth, will follow from Theorem \ref{thm:nilpotentVvsH}, the vertical versus horizontal inequality on nilpotent Lie groups, via a discretization argument.

Before we state Theorem \ref{thm:nilpotentVvsH}, we first establish some terminology. Recalling the linearly independent left-invariant vector fields $X_1,\cdots, X_k$ that, along with their brackets, span the Lie algebra $\mathfrak{g}$ of $G$ at $e_G$, we write the horizontal gradient of a Fr\'echet differentiable function $f:G\to X$ as the function $\nabla f:G\to X^k$ given by
\[
\nabla f\coloneqq \left(X_1f,\cdots,X_kf\right).
\]
Let $C_c^\infty\left(G;X\right)$ denote the space of compactly supported, infinitely Fr\'echet differentiable functions $f:G\to X$. For $p\ge 1$, we define the Sobolev norm $\|\cdot\|_{W_0^{1,p}\left(G;X\right)}$ on $C_c^\infty\left(G;X\right)$ as
\[
\left\|f\right\|_{W_0^{1,p}\left(G;X\right)}\coloneqq \left\|f\right\|_{L^{p}\left(G;X\right)}+\left\|\nabla f\right\|_{L^{p}\left(G;\ell_2^k\left(X\right)\right)},
\]
where $L^p\left(G;X\right)$ denotes the Lebesgue--Bochner space $L^p\left(G,\mu;X\right)$, and we define the Sobolev space
\[
W_0^{1,p}\left(G;X\right)\coloneqq \mathrm{the~completion~of~}\left(C_c^\infty\left(G;X\right),\|\cdot\|_{W_0^{1,p}\left(G;X\right)}\right).
\]
By definition, each $g\in W_0^{1,p}\left(G;X\right)$ signifies a sequence $\left\{f_n\right\}\subset C_c^\infty\left(G;X\right)$ that is Cauchy in the norm $\|\cdot\|_{W_0^{1,p}\left(G;X\right)}$, or equivalently Cauchy in the seminorms $\|\cdot\|_{L^{p}\left(G;X\right)}$ and $\|\nabla\left( \cdot\right)\|_{L^{p}\left(G;\ell_2^k\left(X\right)\right)}$. By completeness of $L^p$, $f_n$ converges in $L^{p}\left(G;X\right)$ to a limit, which we also denote by $g\in L^{p}\left(G;X\right)$, and $\nabla f_n$ converges in $L^{p}\left(G;\ell_2^k\left(X\right)\right)$ to a limit, which we denote by $\nabla g\in L^{p}\left(G;\ell_2^k\left(X\right)\right)$.\footnote{The more usual approach would be to define, for $f\in L^1_{\mathrm{loc}}\left(G;X\right)$, the distributional derivatives $X_if$ as linear functionals from $C_c^\infty\left(G;\mathbb{R}\right)$ to $X$, and to define the Sobolev space $W^{1,p}\left(G;X\right)$ as the space of functions $f\in L^{p}\left(G;X\right)$ for which each $X_if$ arises as a function in $L^{p}\left(G;X\right)$. However, we suspect that a Meyers--Serrin type theorem holds, namely $W_0^{1,p}\left(G;X\right)=W^{1,p}\left(G;X\right)$, since it is true in specific instances: for an open set $\Omega\subset \mathbb{R}^n$ equipped with a sub-Riemannian structure, we have $W^{1,p}_0\left(\Omega;\mathbb{R}\right)=W^{1,p}\left(\Omega;\mathbb{R}\right)$  \cite[Theorem 11.9]{hajlasz2000sobolev}, and for an open set $\Omega\subset \mathbb{R}^n$ equipped with the usual Euclidean structure and $X$ a Banach space, we have $W_0^{1,p}\left(\Omega;X\right)=W^{1,p}\left(\Omega;X\right)$ \cite[Theorem 4.11]{kreuter2015sobolev}. It seems likely that combining these two approaches (and of course using a partition of unity argument to deal with the manifold structure) would prove a Meyers--Serrin-type theorem stating that $W_0^{1,p}\left(M;X\right)=W^{1,p}\left(M;X\right)$ for $M$ a sub-Riemannian manifold and $X$ a Banach space. We leave this investigation to future work.

We remark that there are further generalizations of the notion of a Sobolev space, for example, the Sobolev--Reshetnyak space $R^{1,p}\left(G;X\right)$ \cite{reshetnyak1997sobolev} defined using post-compositions with Lipschitz functions $X\to \mathbb{R}$, and the Newton--Sobolev space $N^{1,p}\left(G;X\right)$ \cite[Chapter 7.1]{heinonen2015sobolev} defined using upper gradients. By \cite[Theorem 7.1.20]{heinonen2015sobolev}, since $X$ is a Banach space, we have $R^{1,p}\left(G;X\right)=N^{1,p}\left(G;X\right)$. By \cite[Theorem 4.6]{caamano2021sobolev}, for $\Omega\subset \mathbb{R}^n$ open with Euclidean structure, $W^{1,p}\left(\Omega;X\right)=R^{1,p}\left(\Omega;X\right)$ if and only if $X$ has the Radon--Nikod\'ym property. It seems likely that extending their proof method will yield that for $M$ a sub-Riemannian manifold, $W^{1,p}\left(M;X\right)=R^{1,p}\left(M;X\right)$ if and only if $X$ has the Radon--Nikod\'ym property. We leave this investigation to future work.\label{foot:meyers-serrin}}

For $f\in L^{p}\left(G;X\right)$ define its \emph{Cheeger $p$-energy} as
\[
\left\|f\right\|_{\dot{Ch}_0^{1,p}\left(G;X\right)}\coloneqq \inf_{\left\{f_n\right\}_{n=1}^\infty} \liminf_{n\to\infty}\left\|\nabla f_n\right\|_{L^p\left(G;\ell_2^k\left(X\right)\right)},
\]
where the infimum is taken over all sequences $\left\{f_n\right\}_{n=1}^\infty\subset C_c^\infty\left(G;X\right)$ such that $f_n\to f$ in $L^p\left(G;X\right)$. Define the Cheeger $\left(1,p\right)$-Sobolev space by\footnote{By definition of $W_0^{1,p}\left(G;X\right)$, it is equivalent to define $\left\|\cdot\right\|_{\dot{Ch}_0^{1,p}\left(G;X\right)}$ by taking the sequence $\left\{f_n\right\}_{n=1}^\infty$ in $W_0^{1,p}\left(G;X\right)$. In contrast, the original definition in \cite{cheeger1999differentiability} took a sequence $\left\{\left(f_n,g_n\right)\right\}_{n=1}^\infty\subset N^{1,p}\left(G;X\right)\times L^p\left(G\right)$ of pairs, where $f_n$ is a sequence in the Newton--Sobolev space $N^{1,p}\left(G;X\right)$ converging to $f$ in $L^p\left(G;X\right)$ and $g_n\in L^p\left(G\right)$ is an upper gradient of $f_n$, and defined
\[
\left\| f\right\|_{\dot{Ch}^{1,p}\left(G;X\right)}\coloneqq \inf_{\left\{\left(f_n,g_n\right)\right\}_{n=1}^\infty} \liminf_{n\to\infty}\left\|g_n\right\|_{L^p\left(G\right)}.
\]
Since $X$ is a Banach space, $Ch^{1,p}\left(G;X\right)=N^{1,p}\left(G;X\right)$ \cite[Theorem 10.1.1]{heinonen2015sobolev}. Given a sub-Riemannian manifold $M$ and a Banach space $X$ with the Radon--Nikod\'ym property, if it were the case that $W_0^{1,p}\left(M;X\right)=W^{1,p}\left(M;X\right)=N^{1,p}\left(M;X\right)$ as suggested in footnote \ref{foot:meyers-serrin}, then it would follow from the above equivalences that $W_0^{1,p}\left(M;X\right)=W^{1,p}\left(M;X\right)=R^{1,p}\left(M;X\right)=N^{1,p}\left(M;X\right)=Ch_0^{1,p}\left(M;X\right)=Ch^{1,p}\left(M;X\right)$.
\label{foot:Cheeger}}
\[
Ch_0^{1,p}\left(G;X\right)\coloneqq \left\{f\in L^p\left(G;X\right):\left\|f\right\|_{\dot{Ch}_0^{1,p}\left(G;X\right)}<\infty\right\}.
\]

Compactly supported Lipschitz functions $f:G\to X$, to which we will apply Theorem \ref{thm:nilpotentVvsH}, belong to $Ch_0^{1,p}\left(G;X\right)$. Denote the local Lipschitz constant of $f$ at $x\in G$ by
\[
\operatorname{lip}_x\left(f\right)\coloneqq \limsup_{r\to 0} \sup_{y,z\in B_r\left(x\right),y\neq z}\frac{\left\|f(y)-f(z)\right\|_X}{d_G\left(y,z\right)},
\]
and define the function $\operatorname{lip}\left(f\right):G\to [0,\infty]$ as
\[
\operatorname{lip}\left(f\right)\left(x\right)\coloneqq\operatorname{lip}_x\left(f\right),\quad x\in G.
\]
Since $G$ is given a Carnot--Carath\'eodory distance, the Lipschitz constant of $f$ is given by
\[
\operatorname{Lip}\left(f\right)\coloneqq \sup_{x,y\in G,x\neq y}\frac{\left\|f(x)-f(y)\right\|_X}{d_G\left(x,y\right)}=\sup_{x\in G}\operatorname{lip}_x\left(f\right).
\]
If $f:G\to X$ is continuously differentiable, then
\begin{equation}\label{ineq:lip-nabla}
    \frac{1}{\sqrt{k}}\left\|\nabla f(x)\right\|_{\ell_2^k\left(X\right)}\le \operatorname{lip}_x\left(f\right)\le \left\|\nabla f(x)\right\|_{\ell_2^k\left(X\right)},\quad x\in G.
\end{equation}
This can be used to prove the claim made in the first sentence of this paragraph: more precisely, if $f:G\to X$ is compactly supported Lipschitz, then $f\in Ch^{1,p}_0\left(G;X\right)$ with
\begin{equation}\label{eq:cheeger-lipschitz}
    \left\|f\right\|_{\dot{Ch}_0^{1,p}\left(G;X\right)}\le \sqrt{k}\operatorname{Lip}\left(f\right)\mu\left(\operatorname{supp}f\right)^{1/p}<\infty,
\end{equation}
where $\operatorname{supp}f\coloneqq \overline{\left\{x\in G:f\left(x\right)\neq 0\right\}}$ is the support of $f$. See Appendix \ref{app:lipnabla} for a proof of these statements.

Denote by $s$ the nilpotency step of $G$, i.e., the largest integer $s$ such that $\underbrace{\left[\mathfrak{g},\left[\mathfrak{g},\cdots,\mathfrak{g}\right]\right]}_{s\mathrm{~times}}\neq 0$. Given $0\neq v\in \mathfrak{g}$, let $s'\in \left[1,s\right]\cap \mathbb{N}$ be the integer such that
\begin{equation}\label{eq:global-degree}
v\in \underbrace{\left[\mathfrak{g},\left[\mathfrak{g},\cdots,\mathfrak{g}\right]\right]}_{s'\mathrm{~times}}\setminus \underbrace{\left[\mathfrak{g},\left[\mathfrak{g},\cdots,\mathfrak{g}\right]\right]}_{s'+1\mathrm{~times}},
\end{equation}
and let $r'$ be the smallest positive integer such that
\begin{equation}\label{eq:local-degree}
v\in V+\cdots+\underbrace{\left[V,\left[V,\cdots,V\right]\right]}_{r'\mathrm{~times}},
\end{equation}
where $V=\operatorname{span}\left\{X_1,\cdots,X_k\right\}$. By \cite{pansu1983croissance}, \cite[Proposition 2.13]{breuillard2012nilpotent} and \cite[Theorem 2.1]{jean2014control}, we have
\begin{equation}\label{distAsymp}
d_G\left(\exp\left(tv\right),e_G\right)\asymp_G \begin{cases}
t^{1/s'},&t\ge 1,\\
t^{1/r'},&0\le t<1.
\end{cases}
\end{equation}
(See Appendix \ref{app:doubling} for a detailed account of this formula.) In particular, by possibly replacing $v$ by a scalar multiple of $v$, we may assume in addition
\begin{equation}\label{eq:sharpDistIneq}
d_G\left(\exp\left(tv\right),e_G\right)\le \begin{cases}
t^{1/s'},&t\ge 1,\\
t^{1/r'},&0\le t<1.
\end{cases}
\end{equation}

We may now state our vertical versus horizontal inequality on nilpotent Lie groups as follows.

\begin{theorem}\label{thm:nilpotentVvsH}
Let $G$ be a nonabelian simply connected nilpotent Lie group. Let $v\in Z\left(G\right)\cap\left[\mathfrak{g},\mathfrak{g}\right]\setminus \left\{0\right\}$ such that if we let $s'\in \left[1,s\right]\cap \mathbb{N}$ be the integer such that \eqref{eq:global-degree} holds and let $r'$ be the smallest positive integer such that \eqref{eq:local-degree} holds, then $r'\le s'$. Assume $v$ is normalized so that \eqref{eq:sharpDistIneq} holds.

Suppose that $q\in \left[2,\infty\right)$ and $p\in (1,q]$. Let $\left(X,\left\|\cdot\right\|_X\right)$ be a Banach space with $K_q\left(X\right)<\infty$, and let $f\in Ch_0^{1,p}\left(G;X\right)$. If $r'=1$, then
\begin{align}\label{eq:nilpotentVvsH-r'=1}
\begin{aligned}
&\left(\int_0^\infty \left(\int_{G}\left(\frac{\left\|f\left(h\exp\left(tv\right)\right)-f\left(h\right)\right\|_X}{t^{1/s'}}\right)^pd\mu\left(h\right)\right)^{q/p}\frac{dt}{t}\right)^{1/q}\\
&\lesssim s'\max\left\{\left(p-1\right)^{(1/q)-1},K_q\left(X\right)\right\} \left\|f\right\|_{\dot{Ch}_0^{1,p}\left(G;X\right)}.
\end{aligned}
\end{align}
On the other hand, if $r'\ge 2$, then
\begin{align}\label{eq:nilpotentVvsH-r'>1}
\begin{aligned}
&\left(\int_0^\infty \left(\int_{G}\left(\frac{\left\|f\left(h\exp\left(tv\right)\right)-f\left(h\right)\right\|_X}{\min\left\{t^{1/r'},t^{1/s'}\right\}}\right)^pd\mu\left(h\right)\right)^{q/p}\frac{dt}{t}\right)^{1/q}\\
&\lesssim s'\max\left\{\left(p-1\right)^{(1/q)-1},K_q\left(X\right)\right\} \left\|f\right\|_{\dot{Ch}_0^{1,p}\left(G;X\right)}.
\end{aligned}
\end{align}
\end{theorem}

\begin{remark}
Note that the $v$ described in this theorem exists, for example, if we choose $v\in \underbrace{\left[\mathfrak{g},\left[\mathfrak{g},\cdots,\mathfrak{g}\right]\right]}_{s\mathrm{~times}}\setminus\left\{0\right\}$, then automatically $r'\le s=s'$ and $v\in Z\left(\mathfrak{g}\right)$. We remark that $\exp\left(\underbrace{\left[\mathfrak{g},\left[\mathfrak{g},\cdots,\mathfrak{g}\right]\right]}_{j\mathrm{~times}}\right)=\underbrace{\left[G,\left[G,\cdots,G\right]\right]}_{j\mathrm{~times}}$ for $j=1,\cdots,s$, where $\left[\mathfrak{g},\mathfrak{g}\right]$ denotes the Lie algebra bracket and $\left[G,G\right]$ denotes the commutator subgroup.
\end{remark}

Theorems \ref{thm:discrete} and \ref{thm:nilpotentVvsH} are extensions from the real Heisenberg group $\bbh^{3}$ to general simply connected nilpotent Lie groups, and from the discrete Heisenberg group $\mathbb{H}^{3}_\mathbb{Z}$ to general finitely generated groups of polynomial growth, respectively, of the ``vertical versus horizontal inequalities'' established by Austin, Naor, and Tessera \cite{austin2013sharp} and Lafforgue and Naor \cite{lafforgue2014vertical}. See also Naor and Young \cite{naor2018vertical,naor2022foliated} for the endpoint case $p=1$ corresponding to the Heisenberg groups $\mathbb{H}^{2k+1}$. Here, the real Heisenberg groups $\bbh^{2k+1}$, $k\in \mathbb{Z}_{>0}$, are defined to be the simply connected nilpotent Lie groups whose Lie algebras $\mathfrak{h}^{2k+1}$ are spanned by the $2k+1$ elements $x_1,\cdots,x_k,y_1,\cdots,y_k,$ and $z$, the only nontrivial bracket relations among which are
\[
\left[x_i,y_i\right]=z,\quad i=1,\cdots,k.
\]
The discrete Heisenberg groups $\mathbb{H}^{2k+1}_\mathbb{Z}$, $k\ge 1$, embed naturally into the real Heisenberg groups $\mathbb{H}^{2k+1}$.

The inequalities \eqref{eq:nilpotentVvsH-r'=1} and \eqref{eq:nilpotentVvsH-r'>1} of Theorem \ref{thm:nilpotentVvsH} are proven following the argument of Lafforgue and Naor \cite{lafforgue2014vertical} by comparing the left-hand sides of \eqref{eq:nilpotentVvsH-r'=1} and \eqref{eq:nilpotentVvsH-r'>1} against convolutions of $\nabla f$ with derivatives of the heat kernel along the span of $v$, and then upper bounding this quantity using the vector-valued Littlewood--Paley--Stein theory of \cite{hytonen2019heat,martinez2006vector,xu2020vector}. The inequality \eqref{eq:desired local pq} of Theorem \ref{thm:discrete} follows from inequalities \eqref{eq:nilpotentVvsH-r'=1} and \eqref{eq:nilpotentVvsH-r'>1} of Theorem \ref{thm:nilpotentVvsH} by a discretization argument.

The utility of Theorems \ref{thm:discrete} and \ref{thm:nilpotentVvsH} is that they prove quantitatively that not virtually abelian finitely generated groups of polynomial growth and nonabelian simply connected nilpotent Lie groups, respectively, do not bilipschitzly embed into uniformly convex spaces. More precisely, they give Theorems \ref{mainthm:nilpotentdistortion} and \ref{mainthm:polygrowthdistortion} as Corollaries; see Section \ref{sec:net} for the deductions.

We will prove Theorem \ref{thm:nilpotentVvsH} in Section \ref{sec:cvx}. There, Theorem \ref{thm:nilpotentVvsH} will be a special case of the more general Theorem \ref{thm:cvx}, which will be stated in Section \ref{sec:cvx}. Theorem \ref{thm:discrete} will be derived from Theorem \ref{thm:nilpotentVvsH} via a discretization argument in Section \ref{sec:proof main}. In Section \ref{sec:net}, Theorem \ref{mainthm:nilpotentdistortion} and Corollary \ref{cor:nilpotentlp} will be special cases of the more general Theorem \ref{thm:gendistortion} derived from Theorem \ref{thm:cvx}, while Theorem \ref{mainthm:polygrowthdistortion} and Corollary \ref{cor:polygrowthlp} will be derived from Theorem \ref{thm:discrete}.

\subsection{Coarse nonembeddability}\label{subsec:coarse-nonembed}
We may also derive similar constraints for other coarse embeddings (see \cite[Definition 1.36]{ostrovskii2013metric} for the definition of a coarse embedding). For example, we have the following theorem, where for a metric space $\left(\mathcal{M},d_\mathcal{M}\right)$ and $0<\alpha<1$, $\left(\mathcal{M},d_\mathcal{M}^\alpha\right)$ is a metric space, called the \emph{$\alpha$-snowflake} of $\left(\mathcal{M},d_\mathcal{M}\right)$.

\begin{theorem}\label{thm:continuous-snowflake}
Let $G$ be a nonabelian simply connected nilpotent Lie group, let $X$ be a $q(\ge 2)$-uniformly convex Banach space, and let $\varepsilon \in \left(0,1\right)$. Then
\[
c_X\left(G, d_G^{1-\varepsilon}\right)\gtrsim_{G} \frac{1}{K_q\left(X\right)\varepsilon^{1/q}}.
\]
In particular,
\[
c_p\left(G, d_G^{1-\varepsilon}\right)\asymp_{p,G} \varepsilon^{-1/\max\left\{p,2\right\}},\quad 1<p<\infty.
\]
\end{theorem}
Theorem \ref{thm:continuous-snowflake} follows from Theorem \ref{thm:precise-holder-nilp}. See Theorem \ref{thm:precise-holder-nilp} for the explicit dependence on $p$ and $G$ of the constants.

\begin{theorem}\label{thm:discrete-snowflake}
Let $\Gamma$ be a not virtually abelian finitely generated group of polynomial growth, let $X$ be a $q(\ge 2)$-uniformly convex Banach space, and let $\varepsilon \in \left(0,1\right)$. Then
\[
c_X\left(\Gamma, d_W^{1-\varepsilon}\right)\gtrsim_{\Gamma} \frac{1}{K_q\left(X\right)\varepsilon^{1/q}}.
\]
In particular,
\[
c_p\left(\Gamma, d_W^{1-\varepsilon}\right)\asymp_{p,\Gamma} \varepsilon^{-1/\max\left\{p,2\right\}},\quad 1<p<\infty.
\]
\end{theorem}

Using Theorem \ref{thm:discrete}, we may also characterize the compression rate of not virtually abelian finitely generated groups of polynomial growth into $L^p$ spaces. Here, the \emph{compression rate} of a Lipschitz function $f:\left(\mathcal{M},d_\mathcal{M}\right)\to \left(X,\left\|\cdot\right\|_X\right)$ from a metric space $\left(\mathcal{M},d_\mathcal{M}\right)$ into a Banach space $\left(X,\left\|\cdot\right\|_X\right)$ is the largest nondecreasing function $\omega_f:\left(0,\infty\right)\to \left[0,\infty\right)$ such that for all $x,y\in \mathcal{M}$ we have $\left\|f\left(x\right)-f\left(y\right)\right\|_X\ge \omega_f\left(d_\mathcal{M}\left(x,y\right)\right)$.

\begin{corollary}\label{cor:growth-function-char}
Let $\Gamma$ be a not virtually abelian finitely generated group of polynomial growth. There exist constants $c_{\Gamma,p}$, $c_{\Gamma,p}'$ such that the following holds.
\begin{enumerate}
\item
Let $n\in \mathbb{N}$, $n\ge 2$, let $p>1$, and let $\theta:\left(0,\infty\right)\to \left[0,\infty\right)$ be a nondecreasing function. If $\theta$ satisfies $\theta\left(t\right)\lesssim \omega_f\left(t\right)$ for some $1$-Lipschitz function $f:B^\Gamma_n\to L^p$, then
\[
\int_1^{2n} \left(\frac{\theta\left(t\right)}{t}\right)^{\max\left\{p,2\right\}}\frac{dt}{t}\le c_{\Gamma,p}.
\]
Conversely, if
\[
\int_1^{2n} \left(\frac{\theta\left(t\right)}{t}\right)^{\max\left\{p,2\right\}}\frac{dt}{t}\le c_{\Gamma,p}',
\]
then $\theta$ satisfies $\theta\left(t\right)\lesssim \omega_f\left(t\right)$ for some $1$-Lipschitz function $f:B^\Gamma_n\to L^p$.

\item
Let $p>1$, and let $\theta:\left(0,\infty\right)\to \left[0,\infty\right)$ be a nondecreasing function. Then $\theta$ satisfies $\theta\left(t\right)\lesssim \omega_f\left(t\right)$ for some $1$-Lipschitz function $f:\Gamma\to L^p$, if
\[
\int_1^{\infty} \left(\frac{\theta\left(t\right)}{t}\right)^{\max\left\{p,2\right\}}\frac{dt}{t}\le c_{\Gamma,p}.
\]
Conversely, if
\[
\int_1^{\infty} \left(\frac{\theta\left(t\right)}{t}\right)^{\max\left\{p,2\right\}}\frac{dt}{t}\le c_{\Gamma,p}',
\]
then $\theta\left(t\right)\lesssim \omega_f\left(t\right)$ for some $1$-Lipschitz function $f:\Gamma\to L^p$.
\end{enumerate}
\end{corollary}

Under a slightly more abstract setting, one can obtain the following statement about the compression rate.
\begin{theorem}\label{thm:amenablesublinear}
Let $\Gamma$ be an amenable group with finite symmetric generating set $S$, with $v\in Z\left(\Gamma\right)$ and $\rho>1$ such that $d_W\left(v^k,e_\Gamma\right)\asymp_\Gamma k^{1/\rho}$, $k\in \mathbb{N}$. Let $\left(X,\left\|\cdot\right\|_X\right)$ be a $q$-uniformly convex space, and let $f:\Gamma\to X$ be a $1$-Lipschitz function. Then for every $\varepsilon>0$ there is an integer $N_\varepsilon$ such that for every $t\ge N_\varepsilon$ there exists an integer $t\le n\le t^{1+\varepsilon}$ such that
\[
\frac{\omega_f\left(n\right)}{n}\lesssim_\Gamma K_q\left(X\right)\left(\frac{\log \log n}{\varepsilon\log n}\right)^{1/q}.
\]
(Here we consider $\rho$ to be dependent on $\Gamma$, so that $\lesssim_\Gamma$ includes dependence on $\rho$.)
\end{theorem}
\begin{remark}\label{rem:qualitativesublinear}\,
\begin{enumerate}
    \item A qualitative version of Theorem \ref{thm:amenablesublinear}, i.e., that, for some positive integer $a>0$,
    \[
    \frac{\omega_f\left(d_W\left(v^{ak},e_\Gamma\right)\right)}{d_W\left(v^{ak},e_\Gamma\right)}\to 0\quad \mathrm{as}~k\to\infty,
    \]
    is implied by \cite[Proposition 3.5]{das2016integrable} together with \cite[Theorem 9.1]{naor2011lp}.\footnote{I thank the anonymous referee for pointing out this fact and providing the reference \cite{das2016integrable}.} We provide a proof of this argument at the beginning of Section \ref{sec:amenable}.
    \item The range $t\le n\le t^{1+\varepsilon}$ appearing in Theorem \ref{thm:amenablesublinear} might not be optimal. When $\Gamma$ is a not virtually abelian finitely generated group of polynomial growth, for every sufficiently large integer $t\ge 3$ there exists a $1$-Lipschitz function $f:\Gamma\to L^q$, $q\ge 2$, such that
    \[
\frac{\omega_f\left(n\right)}{n}\gtrsim \left(\frac{\log \log n}{\log n}\right)^{1/q},\quad t\le n\le t\exp\left(c''_\Gamma \frac{\log t}{\log\log t}\right)
\]
for some constant $c''_\Gamma\in (0,1)$ to be determined by $\Gamma$. Indeed, defining the nondecreasing function $\theta:(0,\infty)\to [0,\infty)$ as
\[
\theta\left(x\right)=
\begin{cases}
    0,&x<t,\\
    x\left(\frac{\log \log x}{\log x}\right)^{1/q},&t\le x\le t\exp\left(c''_\Gamma \frac{\log t}{\log\log t}\right),\\
    t\exp\left(c''_\Gamma \frac{\log t}{\log\log t}\right)\left(\frac{\log \left(\log t+c''_\Gamma \frac{\log t}{\log\log t}\right)}{\log t+c''_\Gamma \frac{\log t}{\log\log t}}\right)^{1/q},&x> t\exp\left(c''_\Gamma \frac{\log t}{\log\log t}\right),\\
\end{cases}
\]
we have
\begin{align*}
\int_1^{\infty} \left(\frac{\theta\left(x\right)}{x}\right)^{q}\frac{dx}{x}&\le\int_t^{t\exp\left(c''_\Gamma \frac{\log t}{\log\log t}\right)}\frac{\log\log t}{\log t}\frac{dx}{x}+\frac{\log\log t}{\log t}\int_{t\exp\left(c''_\Gamma \frac{\log t}{\log\log t}\right)}^\infty\left(\frac{t\exp\left(c''_\Gamma \frac{\log t}{\log\log t}\right)}{x}\right)^q\frac{dx}{x}\\
&=c''_\Gamma+\frac{\log\log t}{q\log t}.
\end{align*}
Thus, as long as $c''_\Gamma$ is small enough and $t$ is large enough so that $c''_\Gamma+\frac{\log\log t}{q\log t}\le c'_{\Gamma,p}$, the constant from Corollary \ref{cor:growth-function-char}, there exists a function $f:\Gamma\to L^q$ such that $\omega_f(n)\gtrsim \theta(n)$. Perhaps the range $t\le n\le t^{1+\varepsilon}$ of Theorem \ref{thm:amenablesublinear} might be improved to a range of the form $t\le n\le t\exp\left(c'\frac{\log t}{\log\log t}\right)$.

A similar estimate gives that as long as $c''_\Gamma$ is small enough and $t$ is large enough so that $c_\Gamma+\frac{1}{q\log t}\le c'_{\Gamma,p}$, again the constant from \ref{cor:growth-function-char}, there exists a $1$-Lipschitz function $f:\Gamma\to L^q$, $q\ge 2$, such that
    \[
\frac{\omega_f\left(n\right)}{n}\gtrsim \left(\frac{1}{\log n}\right)^{1/q},\quad t\le n\le t^{1+c''_\Gamma}.
\]
\end{enumerate}
\end{remark}
We prove Theorem \ref{thm:amenablesublinear} in Section \ref{sec:amenable} following the argument of \cite{austin2013sharp}. We remark that under the setting of groups of polynomial growth, one may obtain, following the argument of \cite[Section 6]{austin2013sharp}, the distortion lower bound $c_X\left(B^\Gamma_n\right)\gtrsim_\Gamma \frac{1}{K_q\left(X\right)} \left(\frac{\log n}{\log\log n}\right)^{1/q}$. As this is weaker than that of Theorem \ref{mainthm:polygrowthdistortion}, we will not outline the deduction here.

\subsection{Refined quantitative nonembeddability estimates of nilpotent Lie groups}\label{subsec:net}
In this subsection, we give refined quantitative nonembeddability estimates of a nilpotent Lie group $G$ into a uniformly convex Banach space $X$.
\subsubsection{Quantitative dependence of the nonembeddability constants on the nilpotent group}\label{subsubsec:quant}
What is the precise dependence on $G$ of the constants in Theorem \ref{mainthm:nilpotentdistortion}, Corollary \ref{cor:nilpotentlp}, and Theorem \ref{thm:continuous-snowflake}? Given the control on the constants in Theorem \ref{thm:nilpotentVvsH}, we are able to derive a more precise version of Theorem \ref{mainthm:nilpotentdistortion} and Corollary \ref{cor:nilpotentlp} as follows.

\begin{theorem}\label{thm:precise-dist-nilp}
Let $G$ be a nonabelian simply connected nilpotent Lie group of nilpotency step $s$ endowed with a left-invariant Riemannian distance, and let $X$ be a $q(\ge 2)$-uniformly convex Banach space. Let $K\ge 1$ denote the measure-doubling constant of $G$:
\[
\mu\left(B_{2r}\right)\le K\mu\left(B_r\right),\quad r>0.
\]
Let $c_1\in (0,1]$ be a constant such that there exists $v\in \underbrace{\left[\mathfrak{g},\left[\mathfrak{g},\cdots,\mathfrak{g}\right]\right]}_{s~\mathrm{times}}$ such that
\[
c_1t^{1/s}\le d_G\left(\exp\left(tv\right),e_G\right)\le  t^{1/s}, \quad \forall t\ge 1.
\]
\begin{enumerate}
    \item Then
\[
c_X\left(B_r\right)\gtrsim \frac{c_1}{\sqrt{\dim \mathfrak{g}}s^{1-1/q}K^{2/q}} \cdot\frac{\left(\log r\right)^{1/q}}{K_q\left(X\right)}\quad r\ge 2.
\]
\item We have, for $r\ge 2$ and $1<p<\infty$,
\[
\frac{c_1\sqrt{\min\left\{p,2\right\}-1}}{\sqrt{\dim \mathfrak{g}}s^{1-1/\max\left\{p,2\right\}}K^{2/\max\left\{p,2\right\}}} \cdot \left(\log r\right)^{1/\max\left\{p,2\right\}}\lesssim c_p\left(B_r\right)\lesssim \max\left\{\frac{\log K}{p},1\right\}\left(\log r\right)^{1/\max\left\{p,2\right\}}+ K^{3/p}c_p\left(B_1\right).
\]
\end{enumerate}

\end{theorem}
Given the control on the constants in Theorem \ref{thm:nilpotentVvsH}, we are able to derive a more precise version of Theorem \ref{thm:continuous-snowflake} as follows.

\begin{theorem}\label{thm:precise-holder-nilp}
Let $G$ be a nonabelian simply connected nilpotent Lie group of nilpotency step $s$ endowed with a left-invariant Carnot--Carath\'eodory distance, and let $X$ be a $q(\ge 2)$-uniformly convex Banach space. Let $K\ge 1$ denote the measure-doubling constant of $G$:
\[
\mu\left(B_{2r}\right)\le K\mu\left(B_r\right),\quad r>0.
\]
Let $c_1\in (0,1]$ be a constant such that there exists $v\in \underbrace{\left[\mathfrak{g},\left[\mathfrak{g},\cdots,\mathfrak{g}\right]\right]}_{s~\mathrm{times}}$ such that
\[
c_1t^{1/s}\le d_G\left(\exp\left(tv\right),e_G\right)\le  t^{1/s}, \quad \forall t\ge 1.
\]
\begin{enumerate}
    \item Then for $\varepsilon\in (0,1)$,
\[
c_X\left(G,d_G^{1-\varepsilon}\right)\gtrsim \frac{c_1^{1-\varepsilon}}{k^{(1-\varepsilon)/2}s^{(1-\varepsilon)(1-1/q)}K^{2(1-\varepsilon)/q}\log K} \cdot\frac{1}{K_q\left(X\right)^{1-\varepsilon}\varepsilon^{1/q}}.
\]
\item We have, for $1<p<\infty$ and $\varepsilon\in (0,1)$,
\begin{align*}
&\frac{c_1^{1-\varepsilon}}{k^{(1-\varepsilon)/2}s^{(1-\varepsilon)(1-1/\max\{p,2\})}K^{2(1-\varepsilon)/\max\{p,2\}}\log K} \cdot \frac{\left(\min\{p,2\}-1\right)^{(1-\varepsilon)/2}}{\varepsilon^{1/\max\{p,2\}}}\\
&\lesssim c_p\left(G,d_G^{1-\varepsilon}\right)\lesssim \max\left\{\frac{\log K}{p},1\right\}\frac{(1-\varepsilon)^{1/\min\{p,2\}}}{\varepsilon^{1/\max\{p,2\}}}+1.
\end{align*}
\end{enumerate}

\end{theorem}
The proof of Theorems \ref{thm:precise-dist-nilp} and \ref{thm:precise-holder-nilp} are presented in Section \ref{sec:net}. They are corollaries of Theorems \ref{thm:gendistortion} and \ref{thm:gen-net-snowflake}, respectively.

However, we are not able to obtain precise dependence of the constants in Theorem \ref{mainthm:polygrowthdistortion}, Corollary \ref{cor:polygrowthlp}, and Theorem \ref{thm:discrete-snowflake}. This is because we do not have control on the dependence on $\Gamma$ of the constants in Theorem \ref{thm:discrete}, due to the many factors arising from the discretization proof in Section \ref{sec:proof main} when deriving Theorem \ref{thm:discrete} from Theorem \ref{thm:nilpotentVvsH}.

\subsubsection{Nonembeddability of nets}\label{subsubsec:net}

On a different note, we can investigate the distortion of nets in nonabelian simply connected nilpotent Lie groups. In the following, with $r_1,r_2>0$, an \emph{$r_1$-covering} of a metric space $\left(\mathcal{M},d_\mathcal{M}\right)$ is a subset $N\subset \mathcal{M}$ such that for any $x\in \mathcal{M}$ there exists $n_x\in N$ such that $d_\mathcal{M}\left(x,n_x\right)<r_1$, an \emph{$r_2$-packing} of $\left(\mathcal{M},d_\mathcal{M}\right)$ is a subset $N\subset \mathcal{M}$ such that for all distinct $n_1,n_2\in N$ we have $d_\mathcal{M}\left(n_1,n_2\right)\ge r_2$, and an \emph{$\left(r_1,r_2\right)$-net} of $\left(\mathcal{M},d_\mathcal{M}\right)$ is a subset which is both an $r_1$-covering and an $r_2$-packing.

We remark that the group $G$ with its bi-invariant Haar measure is measure doubling, i.e., there is a constant $K\ge 1$ such that $\mu\left(B_{2r}\right)\le K \mu\left(B_r\right)$ for all $r\ge 0$; see Appendix \ref{app:doubling}. 
It follows from a volume-measuring argument that $G$ is metrically doubling, i.e., there is a constant $K'\in \left[2,K^4\right]$ such that for all $r>0$ there exist $y_1,\cdots,y_{K'}\in G$ such that $B_{2r}\subset \cup_{i=1}^{K'} B_r\left(y_i\right)$.

\begin{theorem}\label{thm:net}
Let $G$ be a nonabelian simply connected nilpotent Lie group endowed with a left-invariant Carnot--Carath\'eodory distance.
Let $X$ be a $q(\ge 2)$-uniformly convex Banach space.
Let $k$ be the number of vectors $X_1,\cdots,X_k$ which satisfy the H\"ormander condition and with which we measure distances on $G$, $s$ is the nilpotency step of $G$, let $K$ be the doubling constant of the measure $\mu$, and let $c_1>0$ denote a constant such that $c_1t^{1/s}\le d\left(\exp\left(tv\right),e_G\right)\le t^{1/s}$ for $t\ge r_2^s$ for some $v\in \underbrace{\left[\mathfrak{g},\left[\mathfrak{g},\cdots,\mathfrak{g}\right]\right]}_{s\mathrm{~times}}$.
\begin{enumerate}
    \item If $N_{r_1,r_2}$ is an $r_2$-covering of $B_{r_1}$, then
\[
c_X\left(N_{r_1,r_2}\right)\gtrsim \left(\frac{c_1}{\sqrt{k}s^{1-1/q}K^{2/q}\log K}\right)\frac{\left(\log \left(r_1/r_2\right)\right)^{1/q}}{K_q\left(X\right)},
\]
where $k$ is the number of vectors $X_1,\cdots,X_k$ which satisfy the H\"ormander condition and with which we measure distances on $G$.
\item For an auxiliary parameter $r_3\in \left(0,\frac{r_1}{2}\right]$, if  $N_{r_1,r_2,r_3}$ is an $\left(r_2,r_3\right)$-net of $B_{r_1}$, then for $1<p<\infty$,
\[
\frac{c_1\sqrt{\min\left\{p,2\right\}-1}}{\sqrt{k}s^{1-1/\max\left\{p,2\right\}}K^{2/\max\left\{p,2\right\}}\log K} \cdot \left(\log \left(r_1/r_2\right)\right)^{1/\max\left\{p,2\right\}}\lesssim c_p\left(N_{r_1,r_2,r_3}\right)\lesssim \max\left\{\frac{\log K}{p},1\right\}\left(\log \left(r_1/r_3\right)\right)^{1/\max\left\{p,2\right\}}.
\]
\end{enumerate}
\end{theorem}


We may also derive the following variant of Theorem \ref{thm:net} for snowflakes.
\begin{theorem}\label{thm:net-snowflake}
Let $N'_{r_1,r_2}$ be a $\left(r_1,r_2\right)$-net of a nonabelian simply connected nilpotent Lie group $G$ endowed with a left-invariant Carnot--Carath\'eodory distance, where $r_1,r_2>0$ with $r_1\ge r_2/2$. When
\[
Z\left(\mathfrak{g}\right)\subset \operatorname{span}\left\{X_1,\cdots,X_k\right\},
\]
we require in addition $r_1,r_2\ge 1$. Let $s$, $K$, $c_1$, and $k$ be as in Theorem \ref{thm:net}.
\begin{enumerate}
    \item  If $X$ is a $q(\ge 2)$-uniformly convex Banach space, then
\[
c_X\left(N'_{r_1,r_2},d_G^{1-\varepsilon}\right)\gtrsim \left(\frac{c_1^{1-\varepsilon}\left(r_2/r_1\right)^{\varepsilon}}{k^{\left(1-\varepsilon\right)/2}s^{(1-\varepsilon)(1-1/q)} K^{2\left(1-\varepsilon\right)/q}\log K}\right)\frac{1}{K_q\left(X\right)^{1-\varepsilon}\varepsilon^{1/q}},\quad 0<\varepsilon<1,
\]
\item For $1<p<\infty$,
\begin{align*}
&\left(\frac{c_1^{1-\varepsilon}\left(r_2/r_1\right)^{\varepsilon}}{k^{\left(1-\varepsilon\right)/2}s^{(1-\varepsilon)(1-1/\max\left\{p,2\right\})} K^{2\left(1-\varepsilon\right)/\max\left\{p,2\right\}}\log K}\right)\frac{\left(\min\left\{p,2\right\}-1\right)^{(1-\varepsilon)/2}}{\varepsilon^{1/\max\left\{p,2\right\}}}\\
& \lesssim c_p\left(N'_{r_1,r_2},d_G^{1-\varepsilon}\right)\\
&\lesssim 1+\frac{(1-\varepsilon)^{1/\min\left\{p,2\right\}}}{\varepsilon^{1/\max\left\{p,2\right\}}}\max\left\{\frac{\log K}{p},1\right\}.
\end{align*}
\end{enumerate}
\end{theorem}

The proof of Theorems \ref{thm:net} and \ref{thm:net-snowflake} are presented in Section \ref{sec:net}. They are corollaries of Theorems \ref{thm:gen-net} and \ref{thm:gen-net-snowflake}, respectively.

\subsection{Nonembeddability of Carnot groups}\label{subsec:carnot}
\subsubsection{Qualitative nonembeddability}\label{subsubsec:carnot-qual}
The nonembeddability of nonabelian simply connected nilpotent Lie groups and not virtually abelian finitely generated groups of polynomial growth into uniformly convex Banach spaces is known from the fact that the asymptotic cone of these groups are nonabelian Carnot groups, and that nonabelian Carnot groups do not admit bilipschitz embeddings into uniformly convex Banach spaces. The definition of Carnot groups is as follows.

\begin{definition}[{\cite[Section 3.3]{le2017primer}}]\label{def:carnot}
A \emph{Carnot group} is a 5-tuple $\left(G,\delta_\lambda,B,\left|\cdot \right|,d_G\right)$, where:
\begin{itemize}
    \item The simply connected nilpotent Lie group $G$ is such that its Lie algebra $\mathfrak{g}$ admits an \emph{$s(\in \mathbb{Z}_{>0})$-step stratification}, i.e., a direct sum decomposition
    \[
    \mathfrak{g}=V_1\oplus V_2\oplus\cdots\oplus V_s,
    \]
    where $V_s\neq 0$, $V_{s+1}=0$, and $V_{r+1}=\left[V_1,V_r\right]$ for $r=1,\cdots,s$. 
    \item For each $\lambda\in \bbr^+$, the linear map $\delta_\lambda:\mathfrak{g}\to \mathfrak{g}$ is defined by
    \[
    \left.\delta_\lambda\right|_{V_r}=\lambda^r \mathrm{id}_{V_r},\quad r=1,\cdots,s.
    \]
    This map induces a group automorphism of $G$ whose derivative at $e_G$ is $\delta_\lambda:\mathfrak{g}\to\mathfrak{g}$, which, by abuse of notation, we also denote by $\delta_\lambda:G\to G$.
    \item The bundle $B$ over $G$ is the extension of $V_1$ to a left-invariant subbundle:
    \[
    B_p\coloneqq \left(dL_p\right)_e V_1,\quad p\in G.
    \]
    \item The inner product norm\footnote{Again, for Finsler or sub-Finsler Carnot groups we allow $\left|\cdot \right|$ to be more generally a norm. Then, compared to the Riemannian or sub-Riemannian case, the resulting distance $d_G$ is then distorted by a factor of at most $\sqrt{\operatorname{dim}V_1}$ by the John ellipsoid theorem \cite{john1948extremum}, and the results of this paper follow up to multiplicative factors of  $\sqrt{\operatorname{dim}V_1}$.} $\left|\cdot \right|$ is initially defined on $V_1$, and is then extended to $B$ as a left-invariant norm:
    \[
    \left|\left(dL_p\right)_e\left(v\right)\right|\coloneqq \left|v\right|,\quad p\in G,~v\in V_1.
    \]
    \item The metric $d_G$ on $G$ is the (Riemannian or sub-Riemannian) \textit{Carnot--Carath\'eodory distance} associated to $B$ and $\left|\cdot \right|$, i.e.,
    \[
    d_G\left(p,q\right)\coloneqq \inf\left\{\int_0^1 \left|\dot{\gamma}\left(t\right)\right|dt : \gamma\in C_{\mathrm{pw}}^\infty \left(\left[0,1\right];G\right),\gamma\left(0\right)=p,\gamma\left(1\right)=q,\dot{\gamma}\in B \right\},\quad p,q\in G,
    \]
    where $C_{\mathrm{pw}}^\infty \left(\left[0,1\right];G\right)$ consists of the piecewise smooth functions from $\left[0,1\right]$ to $G$.
\end{itemize}
For simplicity, we will call $G$ the Carnot group. When we wish to emphasize the step size, we will call $G$ an $s$-step Carnot group or a Carnot group of step $s$.
\end{definition}

It is easy to see that abelian Carnot groups are precisely those of step $1$ and are the Euclidean spaces $\mathbb{R}^d$. The nonabelian ones are those of step $s\ge 2$. Examples of these are the Heisenberg groups $\mathbb{H}^{2k+1}$ defined above, which are Carnot groups of step $2$ whose Lie algebras $\mathfrak{h}^{2k+1}$ admit the $2$-step stratification
\[
V_1\left(\mathfrak{h}^{2k+1}\right)=\operatorname{span}\left\{x_1,\cdots,x_k,y_1,\cdots,y_k\right\}, \quad V_2\left(\mathfrak{h}^{2k+1}\right)=\operatorname{span}\left\{z\right\}.
\]
More examples of Carnot groups are the aforementioned model filiform groups $J^{s-1}\left(\mathbb{R}\right)$, $s\ge 1$,  which is the Lie group whose Lie algebra $\mathfrak{j}^{s-1}$ is spanned by elements $x,y_0,\cdots,y_{s-1}$ with the only nontrivial bracket relations being
\[
\left[x,y_i\right]=y_{i+1},\quad i=0,\cdots,s-2.
\]
For example, $J^0\left(\mathbb{R}\right)=\mathbb{R}^2$, $J^1\left(\mathbb{R}\right)= \mathbb{H}^3$ is the real Heisenberg group of dimension $3$, and $J^2\left(\mathbb{R}\right)$ is called the Engel group. The Lie group $J^{s-1}\left(\mathbb{R}\right)$ is a Carnot group of step $s$ whose Lie algebra $\mathfrak{j}^{s-1}$ admits the $s$-step stratification
\[
V_1\left(\mathfrak{j}^{s-1}\right)=\operatorname{span}\left\{x,y_0\right\},~ V_2\left(\mathfrak{j}^{s-1}\right)=\operatorname{span}\left\{y_1\right\}, ~\cdots,~ V_s\left(\mathfrak{j}^{s-1}\right)=\operatorname{span}\left\{y_{s-1}\right\}.
\]

We remark that the commutator subgroup of a Carnot group is given by $\left[G,G\right]=\exp\left(V_2\oplus\cdots\oplus V_s\right)$, and more generally the lower central series is given by  $\underbrace{\left[G,\left[G,\cdots,G\right]\right]}_{r~\mathrm{times}}=\exp\left(V_r\oplus\cdots\oplus V_s\right)$, for $r=1,\cdots,s$, ending in $\underbrace{\left[G,\left[G,\cdots,G\right]\right]}_{s~\mathrm{times}}=\exp\left(V_s\right)\subseteq Z\left(G\right)$ contained in the center. 

Let our choice of left-invariant vector fields $X_1,\cdots,X_k$, $k\coloneqq \operatorname{dim}V_1$, be such that $\left(X_1\right)_{e_G},\cdots, \left(X_k\right)_{e_G}$ form an orthonormal basis of $\left(V_1,\left|\cdot \right|\right)$. More generally, choose a basis $X_{r,1},\cdots, X_{r,k_r}$, where $k_r=\dim V_r$, of each stratum $V_r$; we set $X_i=X_{1,i}$ when $r=1$. As $G$ is nilpotent and simply connected, the exponential map $\exp:\mathfrak{g}\to G$ is a diffeomorphism. Thus, each point $p\in G$ can be expressed in the coordinates $p=\exp\left(\sum_{r=1}^s\sum_{i=1}^{k_r}x_{r,i}X_{r,i}\right)$, $x_{r,i}\in \mathbb{R}$; let this be the single coordinate chart on $G$.

The maps $\delta_\lambda$ act as scalings in the Carnot--Carath\'eodory distance, in the sense that
\[
d_G\left(\delta_\lambda\left(p\right),\delta_\lambda\left(q\right)\right)=\lambda d_G\left(p,q\right), \quad \lambda>0,~p,q\in G.
\]
Thus $B_r=\delta_r\left(B_1\right)$, $r>0$, and
\begin{equation}\label{eq:carnot-volume}
\mu\left(B_r\right)=r^{n_h}\mu\left(B_1\right), \quad r>0,
\end{equation}
where $\mu$ is the bi-invariant Haar measure on $G$, also given as the push-forward of the Lebesgue measure on $\mathfrak{g}$ by $\exp:\mathfrak{g}\to G$, and $n_h\coloneqq \sum_{r=1}^s rk_r=\sum_{r=1}^s r\operatorname{dim}V_r$ is the Hausdorff dimension of $G$. Also, by compactness, we have
\begin{equation}\label{eq:carnot-distance}
    d_G\left(y,e_G\right)\asymp_G \sum_{r=1}^s\sum_{i=1}^{k_r}\left|y_{r,i}\right|^{1/r}.
\end{equation}
A more general distance formula on nilpotent Lie groups is given in Appendix \ref{app:doubling}.

That a nonabelian Carnot group $G$ fails to embed bilipschitzly into $\mathbb{R}^n$ was first proven by Semmes \cite{semmes1996nonexistence} using the differentiation theorem of Pansu \cite{pansu1989metriques}. Pansu's theorem states that any Lipschitz mapping $f:G\to \mathbb{R}^n$ is differentiable a.e.~with the derivative being a Lie group homomorphism that commutes with the dilations $\delta_\lambda$ of $G$ and $\mathbb{R}^n$. In other words, one can `zoom-in' on $f$ almost everywhere and the resulting `zoomed-in' function (namely, the derivative of $f$) is a Carnot group homomorphism. Semmes' observation was that if $f$ were a \emph{bilipschitz} mapping, then the derivative of $f$ at a point of differentiability would also be bilipschitz, resulting in a Carnot group \emph{monomorphism} $G\to \mathbb{R}^n$. This is impossible because a group homomorphism $G\to \mathbb{R}^n$ must send the nontrivial commutator subgroup $\left[G,G\right]$ to $0$. Later this argument was generalized, independently by Cheeger and Kleiner \cite{cheeger2006differentiability} and Lee and Naor \cite{lee2006lp}, to Banach space targets with the Radon--Nikod\'ym property, which includes uniformly convex spaces.

We remark that the nonembeddability statements, both qualitative and quantitative, for nonabelian Carnot groups into uniformly convex spaces, do not formally follow from those of the Heisenberg group $\mathbb{H}^3$, because it is not true that any nonabelian Carnot group contains a bilipschitz copy of $\mathbb{H}^3$.\footnote{I thank Professors Assaf Naor and Robert Young for this comment.} For example, the aforementioned model filiform spaces $J^{s-1}\left(\mathbb{R}\right)$ with $s\ge 3$ do not contain a bilipschitz copy of $\mathbb{H}^3$. Recalling the definition of the model filiform group, $J^{s-1}\left(\mathbb{R}\right)$ admits the $s$-step grading $V_1\left(\mathfrak{j}^{s-1}\right)=\operatorname{span}\left\{x,y_0\right\}$, $V_2\left(\mathfrak{j}^{s-1}\right)=\operatorname{span}\left\{y_1\right\}$, $\cdots$, $V_s\left(\mathfrak{j}^{s-1}\right)=\operatorname{span}\left\{y_{s-1}\right\}$. If $J^{s-1}\left(\mathbb{R}\right)$ were to contain a bilipschitz copy of $\mathbb{H}^3$, then by the Pansu--Semmes differentiation theorem\footnote{This applies also when the codomain itself is a Carnot group.} there would be a Carnot group monomorphism $\mathbb{H}^3\to J^{s-1}\left(\mathbb{R}\right)$, namely a Lie group monomorphism that commutes with the dilations. This would induce a Lie algebra monomorphism $\phi:\mathfrak{h}^3\to \mathfrak{j}^{s-1}$ that sends $\left.\phi\right|_{V_i\left(\mathfrak{h}^3\right)}:V_i\left(\mathfrak{h}^3\right)\to V_i\left(\mathfrak{j}^{s-1}\right)$ for all $i$; by matching dimensions, we have vector space isomorphisms $\left.\phi\right|_{V_1\left(\mathfrak{h}^3\right)}:V_1\left(\mathfrak{h}^3\right)\stackrel{\approx}{\to} V_1\left(\mathfrak{j}^{s-1}\right)$ and $\left.\phi\right|_{V_2\left(\mathfrak{h}^3\right)}:V_2\left(\mathfrak{h}^3\right)\stackrel{\approx}{\to} V_2\left(\mathfrak{j}^{s-1}\right)$, but now we have a contradiction since this, along with the fact that $\phi$ is a Lie algebra homomorphism, implies we have a surjection
\[
\left.\phi\right|_{V_3\left(\mathfrak{h}^3\right)}:0=V_3\left(\mathfrak{h}^3\right)=\left[V_1\left(\mathfrak{h}^3\right),V_2\left(\mathfrak{h}^3\right)\right]\to \left[V_1\left(\mathfrak{j}^{s-1}\right),V_2\left(\mathfrak{j}^{s-1}\right)\right]=V_3\left(\mathfrak{j}^{s-1}\right)\neq 0.
\]

\subsubsection{Quantitative nonembeddability of Carnot groups along central directions}
Nonabelian Carnot groups are instances of nonabelian simply connected nilpotent Lie groups endowed with a sub-Riemannian distance. Thus, we may state Theorem \ref{thm:nilpotentVvsH} for nonabelian Carnot groups $G$ as follows.

\begin{theorem}\label{thm:carnotVvsH}
Let $G$ be a nonabelian Carnot group of step $s\ge 2$. Let $r\in \left[2,s\right]\cap \mathbb{N}$ be such that $V_r\cap Z\left(\mathfrak{g}\right)\neq \left\{0\right\}$. Let $v\in V_r\cap Z(\mathfrak{g})\setminus \left\{0\right\}$ be normalized so that $d_G\left(\exp\left(v\right),e_G\right)=1$. Suppose that $q\in \left[2,\infty\right)$ and $p\in (1,q]$. Let $\left(X,\left\|\cdot\right\|_X\right)$ be a Banach space with $K_q\left(X\right)<\infty$, and let $f\in Ch_0^{1,p}\left(G;X\right)$. 
Then
\begin{align}\label{eq:carnotVvsH}
\begin{aligned}
&\left(\int_0^\infty \left(\int_{G}\left(\frac{\left\|f\left(h\exp\left(tv\right)\right)-f\left(h\right)\right\|_X}{t^{1/r}}\right)^pd\mu\left(h\right)\right)^{q/p}\frac{dt}{t}\right)^{1/q}\\
&\lesssim r\max\left\{\left(p-1\right)^{(1/q)-1},K_q\left(X\right)\right\} \left\|f\right\|_{\dot{Ch}^{1,p}_0\left(G;X\right)}.
\end{aligned}
\end{align}
\end{theorem}

For nilpotent Lie groups, when we choose $v\in \underbrace{\left[\mathfrak{g},\left[\mathfrak{g},\cdots,\mathfrak{g}\right]\right]}_{s\mathrm{~times}}$, we have $d_G\left(\exp\left(tv\right),e_G\right)\asymp_G t^{1/s}$ for $t\ge 1$ but $d_G\left(\exp\left(tv\right),e_G\right)\asymp_G t^{1/r'}$ for $0<t<1$, where $r'$ is the minimum number of Lie brackets of the H\"ormander vector fields $X_i$'s needed to span $v$,
while for Carnot groups we have $d_G\left(\exp\left(tv\right),e_G\right)= t^{1/s}$ for all $t>0$. Thus, when $r'=1$ (which holds e.g.~when the distance is Riemannian), Theorem \ref{thm:nilpotentVvsH} only gives large-scale nonembeddability of nonabelian simply connected nilpotent Lie groups into uniformly convex spaces, i.e., that $c_X\left(B_r\right)\gtrsim_{G,X} \left(\log r\right)^{1/q}$, $r\ge 2$, while Theorem \ref{thm:carnotVvsH} gives local nonembeddability of Carnot groups into uniformly convex spaces, i.e., we know also that any nonempty open subset of the group fails to embed into $X$, in line with the Pansu--Semmes nonembeddability result.

Theorem \ref{thm:discrete} applies to cocompact lattices $\Gamma$ of a nonabelian Carnot group $G$.

\begin{theorem}\label{thm:discrete-Carnot}
Let $G$ be an $s\left(\ge 2\right)$-step Carnot group that has a cocompact lattice $\Gamma$. Choose $v_\Gamma\in \exp\left(V_s\right)\cap \Gamma\setminus\left\{e_G\right\}$. Then there exists $c=c\left(\Gamma,v_\Gamma\right)\in \mathbb{N}$ such that the following is true. Let $q\in \left[2,\infty\right)$ and $p\in (1,q]$. Suppose that $\left(X,\left\|\cdot\right\|_X\right)$ is a Banach space satisfying $K_q\left(X\right)<\infty$.  Then for every $n\in \mathbb{N}$ and every $f:\Gamma\to X$ we have
\begin{align*}
\begin{aligned}
&\left(\sum_{k=1}^{n^s}\frac{1}{k^{1+q/s}}\left(\sum_{x\in B^\Gamma_n} \left\|f\left(xv_\Gamma^k\right)-f\left(x\right)\right\|_X^p\right)^{q/p}\right)^{1/q}\\
&\lesssim_{\Gamma,v_\Gamma} \max\left\{\left(p-1\right)^{(1/q)-1},K_q\left(X\right)\right\} 
\left(\sum_{x\in B^\Gamma_{cn}} \sum_{a\in S}\left\|f\left(xa\right)-f\left(x\right)\right\|^p_X\right)^{1/p}.
\end{aligned}
\end{align*}
\end{theorem}

We may state simplified versions of Theorems \ref{thm:net} and \ref{thm:net-snowflake} in the setting of Carnot groups, thanks to the scale-invariance. In the following, $n_h\coloneqq \sum_{r=1}^s r\dim V_r$ is the Hausdorff dimension of $G$.
\begin{theorem}\label{thm:netdistortion}
Let $G$ be a nonabelian Carnot group endowed with a left-invariant Carnot--Carath\'eodory distance.
Let $X$ be a $q(\ge 2)$-uniformly convex Banach space.
Let $k$ be the number of vectors $X_1,\cdots,X_k$ which satisfy the H\"ormander condition and with which we measure distances on $G$, $s$ is the nilpotency step of $G$.
Let $r_1,r_2>0$ with $r_1\ge 2r_2$.
\begin{enumerate}
    \item If $N_{r_1,r_2}$ is an $r_2$-covering of $B_{r_1}$, then
\[
c_X\left(N_{r_1,r_2}\right)\gtrsim \left(\frac{1}{\sqrt{k}s^{1-1/q}4^{n_h/q}n_h}\right)\frac{\left(\log \left(r_1/r_2\right)\right)^{1/q}}{K_q\left(X\right)},
\]
where $k$ is the number of vectors $X_1,\cdots,X_k$ which satisfy the H\"ormander condition and with which we measure distances on $G$.
\item For an auxiliary parameter $r_3\in \left(0,\frac{r_1}{2}\right]$, if  $N_{r_1,r_2,r_3}$ is an $\left(r_2,r_3\right)$-net of $B_{r_1}$, then for $1<p<\infty$,
\[
\frac{c_1\sqrt{\min\left\{p,2\right\}-1}}{\sqrt{k}s^{1-1/\max\left\{p,2\right\}}K^{2/\max\left\{p,2\right\}}\log K} \cdot \left(\log \left(r_1/r_2\right)\right)^{1/\max\left\{p,2\right\}}\lesssim c_p\left(N_{r_1,r_2,r_3}\right)\lesssim \max\left\{\frac{\log K}{p},1\right\}\left(\log \left(r_1/r_3\right)\right)^{1/\max\left\{p,2\right\}}.
\]
\end{enumerate}
\end{theorem}
\begin{theorem}\label{thm:netdistortion-snowflake}
Let $N'_{r_1,r_2}$ be a $\left(r_1,r_2\right)$-net of a nonabelian Carnot group $G$, where $r_1,r_2>0$ with $r_1\ge r_2/2$. Let $X$ be a $q(\ge 2)$-uniformly convex Banach space. Then
\[
c_X\left(N'_{r_1,r_2},d_G^{1-\varepsilon}\right)\gtrsim \left(\frac{\left(r_2/r_1\right)^{\varepsilon}}{k^{\left(1-\varepsilon\right)/2}s^{\left(1-\varepsilon\right)(1-1/q)} 4^{n_h\left(1-\varepsilon\right)/q}n_h}\right)\frac{1}{K_q\left(X\right)^{1-\varepsilon}\varepsilon^{1/q}},\quad 0<\varepsilon<1.
\]
In particular,
\[
c_p\left(N'_{r_1,r_2},d_G^{1-\varepsilon}\right)\asymp_{G,p,r_1,r_2} \varepsilon^{-1/\max\left\{p,2\right\}},\quad 1<p<\infty.
\]
\begin{align*}
	&\left(\frac{\left(r_2/r_1\right)^{\varepsilon}}{k^{\left(1-\varepsilon\right)/2}s^{(1-\varepsilon)(1-1/\max\left\{p,2\right\})} 4^{n_h\left(1-\varepsilon\right)/\max\left\{p,2\right\}}n_h}\right)\frac{\left(\min\left\{p,2\right\}-1\right)^{(1-\varepsilon)/2}}{\varepsilon^{1/\max\left\{p,2\right\}}}\\
	& \lesssim c_p\left(N'_{r_1,r_2},d_G^{1-\varepsilon}\right)\\
	&\lesssim 1+\frac{(1-\varepsilon)^{1/\min\left\{p,2\right\}}}{\varepsilon^{1/\max\left\{p,2\right\}}}\max\left\{\frac{n_h}{p},1\right\},\quad 1<p<\infty.
\end{align*}
\end{theorem}

Focusing on Euclidean distortion for concreteness, Theorem \ref{thm:netdistortion} tells us that
\[
c_2\left(N_R\right)\asymp_{n_h} \sqrt{\log R},
\]
for any nonabelian Carnot group $G$ and $\left(1,\Omega\left(1\right)\right)$-net $N_R$ of $B_R$. Note that the nonabelian Carnot groups of lowest nonabelian step $s=2$, such as the Heisenberg group $\mathbb{H}^3$, already give the dependence $\sqrt{\log R}$. One may have expected that Carnot groups with higher steps may have a ``hierarchy of algebraic nonembeddability,'' with distortions being propagated and amplified as one passes from $\left[G,G\right]$, $\left[G,\left[G,G\right]\right]$, all the way down to $\underbrace{\left[G,\left[G,\cdots,G\right]\right]}_{s\mathrm{~times}}$, with the center ending up ``super-collapsed''; certainly if $G$ is high-dimensional and $V_1$ is low dimensional, the subgroup $\left[G,G\right]$ of small codimension is highly compressed, and one may have expected that there is enough room to have many more interesting phenomena happen. This might have led one to suspect that the dependency on $s$ of the asymptotics of $c_2\left(N_R\right)$ would appear in the exponent of $\log R$, but in reality, the asymptotics is exactly that of $\sqrt{\log R}$. Thus, it seems that the collapsing happens ``only once'' and happens ``uniformly across all of $\left[G,G\right]$''. The algebraic structure of the nonabelian Carnot group seems to only affect the constants involved while having no effect on the exponent. We then ask what is the correct dependence on the Lie group structure:\footnote{I thank Professor Assaf Naor for discussion on this matter and for asking a rudimentary version of this question.}
\begin{question}
In the case of Euclidean distortion, what is the precise asymptotics? More precisely, we ask the following.
\begin{enumerate}
\item
Let $G$ be a nonabelian simply connected nilpotent Lie group. 
\begin{enumerate}
	\item If $G$ has a Riemannian metric, does a constant $c_G$ exist such that we have the following?\footnote{Note that Corollary \ref{cor:nilpotentlp} and Theorem \ref{thm:precise-dist-nilp} only state that $c_2\left(B_r\right)\asymp_G \sqrt{\log r}$, i.e., $c_{1,G}\sqrt{\log r}\le c_2\left(B_r\right)\le c_{2,G}\sqrt{\log r}$ for constants $c_{1,G}$ and $c_{2,G}$ with $c_{2,G}\not\lesssim c_{1,G}$.}
	\[
	c_2\left(B_r\right)\asymp c_G\sqrt{\log r},\quad r\ge 2.
	\]
	Or better yet, do we have the following?
	\[
	\lim_{r\to\infty}\frac{c_2\left(B_r\right)}{\sqrt{\log r}}=c_G.
	\]
	\item If $G$ has a sub-Riemannian metric, for $r_1,r_2,r_3>0$ with $r_1\ge 2r_2$ and $r_1\ge 2r_3$, let $N_{r_1,r_2,r_3}$ be an $\left(r_2,r_3\right)$-net of $B_{r_1}$. Does there exist a constant $c'_G$ such that we have the following?
	\[
	c'_G \sqrt{\log \left(r_1/r_2\right)} \lesssim c_2\left(N_{r_1,r_2,r_3}\right)\lesssim c''_G \sqrt{\log \left(r_1/r_3\right)}\quad \mathrm{as}~r_1\to\infty,r_2,r_3\to 0.
	\]
	If so, how are the values of $c'_G$ and $c''_G$ determined by the algebraic structure of $G$ and the choice of $X_1,\cdots,X_k$?
\end{enumerate}

\item
Let $\Gamma$ be a not virtually abelian finitely generated group of polynomial growth. Does there exist a constant $c_\Gamma$ such that we have the following?
\[
c_2\left(B_n^\Gamma\right)\asymp c_\Gamma \sqrt{\log n}\quad \mathrm{as}~n\to\infty.
\]
Do we even have the following?
\[
\lim_{n\to\infty}\frac{c_2\left(B_n^\Gamma\right)}{\sqrt{\log n}}=c_\Gamma.
\]
If so, how is the value of $c_\Gamma$ determined by the algebraic structure of $\Gamma$ and the choice of generating set?
\end{enumerate}
\end{question}

We know that, by Theorem \ref{thm:netdistortion}(2), for nonabelian Carnot groups $G$, $c'_G$ and $c''_G$ are bounded above and below by functions of $n_h$, since $\log K\asymp n_h$ and $k,s\le n_h$.

\subsubsection{Quantitative $L^p$-nonembeddability along commutator directions}

From the standpoint of the Pansu--Semmes argument which simply outputs the qualitative nonexistence of a bilipschitz mapping $G\to \mathbb{R}^n$, the significance and advantage of Theorems \ref{thm:discrete}, \ref{thm:nilpotentVvsH}, \ref{thm:carnotVvsH}, and \ref{thm:discrete-Carnot} is that they give quantitative nonembeddability statements, such as  Theorems \ref{mainthm:nilpotentdistortion}, \ref{mainthm:polygrowthdistortion}, \ref{thm:continuous-snowflake}, \ref{thm:discrete-snowflake}, \ref{thm:precise-dist-nilp}, \ref{thm:precise-holder-nilp}, \ref{thm:net}, \ref{thm:net-snowflake}, \ref{thm:netdistortion}, \ref{thm:netdistortion-snowflake} and Corollaries \ref{cor:nilpotentlp}, \ref{cor:polygrowthlp}, \ref{cor:growth-function-char}. However, the Pansu--Semmes argument suggests that any Lipschitz function $f:G\to X$ should `collapse' along directions of the commutator group $\left[G,G\right]$, and from this point of view, Theorems \ref{thm:nilpotentVvsH} and \ref{thm:carnotVvsH} possess a curious limitation, namely that it requires us to measure the collapse along \emph{central} commutator directions.

This limitation is caused by a technical requirement in the proof of Theorems \ref{thm:nilpotentVvsH} and \ref{thm:carnotVvsH}, which is for the horizontal derivative and the convolution along the direction of $v$ to commute (more precisely, equation \eqref{eq:nablaConvolutionCommute} below). Since there is no such limitation in the original Pansu--Semmes proof, we pose the following question.

\begin{question}\label{ques:full commutator}
Let $G$ be a nonabelian Carnot group, and let $q\in \left[2,\infty\right)$ and $p\in (1,q]$. Let $\left(X,\left\|\cdot\right\|_X\right)$ be a Banach space with $K_q\left(X\right)<\infty$, and let $f\in Ch_0^{1,p}\left(G;X\right)$. Is a variant of Theorem \ref{thm:carnotVvsH} true for all $v\in \left[G,G\right]$ with $d_G\left(v,e_G\right)=1$? More specifically, do we have the following inequality?\footnote{Note that, compared to Theorem \ref{thm:carnotVvsH}, we are measuring the vertical variations along the displacements of $\delta_t(v)$ and not $\exp\left(tv\right)$. This is because the Pansu--Semmes argument uses Carnot group differentiation, and that directly involves the quantity $\frac{f\left(h\delta_t\left(v\right)\right)-f\left(h\right)}{t}$. There is the added benefit that we may capitalize on the simple distance relation $d_G\left(\delta_t(v\right),e_G)=t$.}
\begin{align*}
&\left(\int_0^\infty \left(\int_{G}\left(\frac{\left\|f\left(h\delta_t\left(v\right)\right)-f\left(h\right)\right\|_X}{t}\right)^pd\mu\left(h\right)\right)^{q/p}\frac{dt}{t}\right)^{1/q}\lesssim_{G,p,q,X}  \left\|f\right\|_{\dot{Ch}^{1,p}_0\left(G;X\right)}.
\end{align*}

More generally, let $G$ be a nonabelian simply connected nilpotent Lie group endowed with a left-invariant Carnot--Carath\'eodory distance. Is a variant of Theorem \ref{thm:nilpotentVvsH} true for all $v\in \left[G,G\right]$ with $d_G\left(v,e_G\right)=1$? More precisely, if $1<r'\le s'$ are such that $d_G\left(v^t,e_G\right)\lesssim \min\left\{t^{1/s'},t^{1/r'}\right\}$ for all $t>0$, then does the following hold?
\begin{align*}
\begin{aligned}
&\left(\int_0^\infty \left(\int_{G}\left(\frac{\left\|f\left(h\exp\left(tv\right)\right)-f\left(h\right)\right\|_X}{\min\left\{t^{1/s'},t^{1/r'}\right\}}\right)^pd\mu\left(h\right)\right)^{q/p}\frac{dt}{t}\right)^{1/q}\lesssim_{G,p,q,X} \left\|f\right\|_{\dot{Ch}^{1,p}_0\left(G;X\right)}.
\end{aligned}
\end{align*}
\end{question}

We answer this question in the affirmative for nonabelian Carnot groups and $L^p$ $\left(1<p\le 2\right)$ targets, albeit with some possible loss of control on the constants for the inequality. Here, for $f\in L^1_{\mathrm{loc}}\left(G\right)$, if the distributional derivatives $X_if$ belong to $L^1_{\mathrm{loc}}\left(G\right)$, then denote $\nabla f\coloneqq\left(X_1f,\cdots,X_kf\right)$.

\begin{theorem}\label{thm:lp}
Let $G$ be a nonabelian Carnot group.
For $p\in (1,2]$ and $q\in [2,\infty)$, $f\in L^p\left(G\right)$ with distributional derivatives $X_i f \in L^p\left(G\right)$ for $i=1,\cdots,k$, and $v\in \left[G,G\right]$ with $d_G\left(v,e_G\right)= 1$, we have
\begin{equation}\label{eq:lp main}
\left(\int_0^\infty \left[ \int_G \left(\frac{\left|f\left(h\right) - f\left(h\delta_r\left(v\right)\right)\right|}{r} \right)^p d\mu\left(h\right)\right]^{q/p}  \frac{dr}{r} \right)^{1/q} \lesssim_{G,p} \left\|\nabla f\right\|_{L^p\left(G;\ell_2^k\right)}.
\end{equation}
\end{theorem}

Although Theorem \ref{thm:lp} is stated for $\mathbb{R}$-targets, it is straightforward to tensorize inequality \eqref{eq:lp main} using the triangle inequality when $f\in L^p\left(G,L^p\left(\sigma\right)\right)$ with $\nabla f\in L^p\left(G,L^p\left(\sigma,\ell_2^k\right)\right)$.

\begin{remark}\label{qge2}
The case $q=\infty$ of Theorem \ref{thm:lp} formally holds. More specifically, under the assumptions of Theorem \ref{thm:lp},
\begin{equation}\label{eq:qinfty-carnot}
\left[ \int_G \left(\frac{\left|f\left(h\right) - f\left(h\delta_r\left(v\right)\right)\right|}{r} \right)^p d\mu\left(h\right)\right]^{1/p}\le \left\|\nabla f\right\|_{L^p\left(G;\ell_2^k\right)},\quad r>0.
\end{equation}
See equation \eqref{eq:deterministic-path} on page \pageref{proofrmkqge2} of Section \ref{sec:cvx}, in conjunction with \eqref{ineq:lip-nabla}, for a proof of this simple fact.
Thus, by H\"older's inequality, the inequality \eqref{eq:lp main} for $q=2$ implies the inequalities \eqref{eq:lp main} for $q\ge 2$, and this is why we don't see dependence on $q$ of the implied constant of inequality \eqref{eq:lp main}.  We cannot make this simplification in Theorem \ref{thm:carnotVvsH} or in Theorem \ref{thm:nilpotentVvsH} due to the stated dependence of the constants on $q$.
\end{remark}

However, we have not managed to prove a strengthening of Theorem \ref{thm:discrete} to directions in the commutator subgroup for lattices in Carnot groups, formulated as follows.
\begin{question}\label{ques:discrete-Carnot}
Let $\Gamma$ be a cocompact lattice of a nonabelian Carnot group $G$. For any $v_\Gamma\in \left[\Gamma,\Gamma\right]\setminus\left\{e_G\right\}$, does there exist $c\in \mathbb{N}$ such that the following is true? If we let $s'\ge 2$ be the largest integer such that $v_\Gamma\in \underbrace{\left[G,\left[G,\cdots,G\right]\right]}_{s'\mathrm{~times}}$, let $q\in \left[2,\infty\right)$ and $p\in (1,q]$, and let $\left(X,\left\|\cdot\right\|_X\right)$ be a Banach space satisfying $K_q\left(X\right)<\infty$, then for every finitely supported $f:\Gamma\to X$ we have
\begin{align*}
\begin{aligned}
&\left(\sum_{k=1}^{n}\frac{1}{k}\left(\sum_{x\in B^\Gamma_n} \left(\frac{\left\|f\left(x\delta_k\left(v_\Gamma\right)\right)-f\left(x\right)\right\|_X}{k}\right)^p\right)^{q/p}\right)^{1/q}\lesssim_{\Gamma,p,q,X}
\left(\sum_{x\in B^\Gamma_{cn}} \sum_{a\in S}\left\|f\left(xa\right)-f\left(x\right)\right\|^p_X\right)^{1/p}.
\end{aligned}
\end{align*}
\end{question}

We will deduce Theorem \ref{thm:lp} from a generalized version of Dorronsoro's theorem \cite{dorronsoro1985characterization} for Carnot groups, Theorem \ref{lpgenthm}, following the argument of F\"assler and Orponen \cite{fassler2020dorronsoro}.

\begin{remark}
So far, we have mentioned two proof methods of the vertical versus horizontal inequalities, namely the Littlewood--Paley--Stein theory for Theorems \ref{thm:nilpotentVvsH} and \ref{thm:carnotVvsH} and the Dorronsoro theorem (Theorem \ref{lpgenthm}) for Theorem \ref{thm:lp}. There is a third possible proof method for $L^2$-targets which follows the representation theoretic proof of \cite{austin2013sharp} for the case $G=\mathbb{H}^3$, $X=L^2$. Specifically, one can reduce proving the vertical versus horizontal inequality to proving a certain inequality regarding $1$-cocycles $G\to L^2$ of irreducible unitary representations of $G$. The irreducible unitary representations of $\mathbb{H}^3$ are given by the Stone--von Neumann theorem, while the irreducible unitary representations of simply connected nilpotent Lie groups have been described by Dixmier \cite{dixmier1959representations,dixmier1957representations} and Kirillov \cite{kirillov1962unitary}. We will not pursue this proof method since the anticipated results do not seem to provide an improvement or advantage over neither Theorem \ref{thm:nilpotentVvsH} nor Theorem \ref{thm:lp}.
\end{remark}


\subsection{Dorronsoro's theorem on Carnot groups}\label{subsec:dorronsoro}

As mentioned earlier, a byproduct of this paper is a formulation and proof of a Dorronsoro theorem for Carnot groups (Theorem \ref{lpgenthm}), which is used to prove Theorem \ref{thm:lp}. In this subsection, we state Theorem \ref{lpgenthm} and obtain more refined vertical versus horizontal inequalities.

Let $G$ be a (possibly abelian) Carnot group in this subsection. Following Folland \cite{folland1975subelliptic} and denoting by $\Delta=\sum_{i=1}^k X_i^2$ the sub-Laplacian and $H_t$ the corresponding heat kernel, we define the operator $\left(-\Delta_p\right)^\alpha$ for $1<p<\infty$ and $\operatorname{Re}\alpha>0$ by
\[
\left(-\Delta_p\right)^\alpha f=\lim_{\varepsilon\to 0}\frac{1}{\Gamma\left(\lfloor\operatorname{Re}\alpha\rfloor+1-\alpha\right)}\int_\varepsilon^\infty t^{\lfloor \operatorname{Re}\alpha\rfloor-\alpha}\left(-\Delta\right)^{\lfloor\operatorname{Re}\alpha\rfloor+1}H_tfdt
\]
for all $f\in L^p\left(G\right)$ such that the limit exists in the $L^p$ norm. Then, for $1<p<\infty$ and $\alpha>0$, the Sobolev space $S^p_\alpha$ is the Banach space $\operatorname{Dom}\left(\left(-\Delta_p\right)^{\alpha/2}\right)$ with norm
\[
\left\|\cdot\right\|_{p,\alpha}\coloneqq \left\|\cdot\right\|_{L^p\left(G\right)}+\left\|\left(-\Delta_p\right)^{\alpha/2}\left(\cdot\right)\right\|_{L^p\left(G\right)}.
\]

When $G=\mathbb{R}^n$, Dorronsoro's Theorem \cite{dorronsoro1985characterization} states that the $L^p\left(\mathbb{R}^n\right)$ norm of the $L^p$-fractional Laplacian $\left\|\left(-\Delta_p\right)^\alpha f\right\|_{L^p\left(\mathbb{R}^n\right)}$ ($1<p<\infty$, $\alpha>0$) of a function $f\in L^p\left(\mathbb{R}^n\right)$ is equivalent up to constant factors to a singular integral that, roughly speaking, measures, at all points of $\mathbb{R}^n$ and at all scales, the deviation of $f$ from being a polynomial of weighted degree $\le \lfloor\alpha\rfloor$.

We first state Dorronsoro's original theorem \cite[Theorem 2, Theorem 6]{dorronsoro1985characterization}. Considering the abelian case $G=\mathbb{R}^n$ and denoting coordinates by $x=\left(x_1,\cdots,x_n\right)\in \mathbb{R}^n$, we have $\Delta=\sum_{i=1}^n \frac{\partial^2}{\partial x_i^2}$, and for $1<p<\infty$, the fractional Laplacian $\left(-\Delta_p\right)^{\alpha/2}$ is defined as above. For polynomials of $x_1,\cdots,x_n$, we may define their (total) degree in the usual multiplicative way by assigning degree $1$ to each $x_i$, and we may consider, for $d\in \mathbb{Z}_{\ge 0}$, the family $\mathcal{A}_d$ of polynomials in $x_1,\cdots,x_n$ of degree at most $d$. For a locally integrable function $f:\mathbb{R}^n\to\mathbb{R}$, $x\in \mathbb{R}^n$ and $r>0$, let $A^d_{x,r}f$ denote the unique element of $\mathcal{A}_d$ such that
\[
\int_{B_r\left(x\right)}\left(f\left(y\right)-A^d_{x,r}f\left(y\right)\right)A\left(y\right)d\mu\left(y\right)=0,\quad \forall A\in \mathcal{A}_d.
\]
For example, $A^0_{x,r}f=\fint_{B_r\left(x\right)} f\left(z\right)d\mu\left(z\right)$, where $\fint$ denotes the average integral, and
\begin{equation*}
A^1_{x,r}f\left(x+y\right) = \fint_{B_r\left(x\right)}f\left(z\right)d\mu\left(z\right) + \frac{r^2}{n+2}\sum_{i=1}^{n}y_i\fint_{B_r\left(x\right)} f\left(z\right)\left(z_{i}-x_{i}\right)d\mu\left(z\right), \quad y \in \mathbb{R}^n.
\end{equation*}
We measure how well $A^d_{x,r}f$ approximates $f$ in the ball $B_r\left(x\right)$ in the $L^q$ norm and denote the average $L^q$-normed difference by the following quantity:
\[
\beta_{f,d,q}\left(B_r\left(x\right)\right)=
\begin{cases}
\left(\fint_{B_r\left(x\right)}\left|f\left(y\right)-A^d_{x,r}f\left(y\right)\right|^q d\mu\left(y\right)\right)^{1/q},& 1\le q<\infty,\\
\left\|f-A^d_{x,r}f\right\|_{L^\infty\left(B_r\left(x\right)\right)},& q=\infty,
\end{cases}
\]
whenever the norm exists; otherwise define $\beta_{f,d,q}\left(B_r\left(x\right)\right)=\infty$. 

Now we are ready to state Dorronsoro's theorem.
\begin{theorem}[{\cite[Theorem 2, Theorem 6]{dorronsoro1985characterization}}]\label{Euclidean-lpgenthm}
Let $n\in \mathbb{Z}_{>0}$, $1<p<\infty$, $\alpha>0$, and $1\le q\le\infty$ satisfy
\begin{equation*}
\frac{\alpha}{n}+\frac 1q>\frac{1}{\min\{p,2\}}.\end{equation*}
Then for all $f\in L^p\left(\mathbb{R}^n\right)$,
\begin{equation}\label{eq:rn-dorronsoro-singint}
\left(\int_{\mathbb{R}^n} \left(\int_0^\infty \left[\frac 1{r^\alpha} \beta_{f,\lfloor \alpha \rfloor,q}\left(B_r\left(x\right)\right) \right]^2 \frac{dr}{r} \right)^{p/2} dx \right)^{1/p} \asymp_{n,\alpha,p,q} \left\|\left(-\Delta_p\right)^{\alpha/2} f\right\|_{L^p\left(\mathbb{R}^n\right)},
\end{equation}
in the sense that $f\in S^p_\alpha\left(\mathbb{R}^n\right)$ if and only if the left-hand side of \eqref{eq:rn-dorronsoro-singint} is finite, in which case the above relation \eqref{eq:rn-dorronsoro-singint} holds.
\end{theorem}

We will generalize Dorronsoro's theorem to Carnot groups. Following the above formulation of the case $G=\mathbb{R}^n$, we need to suitably define the (weighted) degree of polynomials on Carnot groups.

In the case of the Heisenberg groups $\mathbb{H}^{2k+1}$ one side of the Dorronsoro statement (namely that the aforementioned singular integral is bounded above by $\left\|\left(-\Delta_p\right)^\alpha f\right\|_{L^p\left(\mathbb{H}^{2k+1}\right)}$) was proven for the restricted exponent range $0<\alpha<2$ in \cite{fassler2020dorronsoro}. The reason for this restriction $0<\alpha<2$ in \cite{fassler2020dorronsoro} is that they considered deviations from the smaller class of ``horizontal polynomials", which is inadequate for $\alpha\ge 2$ (see Remark \ref{remark:noHori} for a proof of the inadequacy of horizontal polynomials as approximants). To remove this restriction, we need to reconsider our class of weighted polynomials.

Let $G$ be a nonabelian Carnot group. Recall that $n_h= \sum_{r=1}^s rk_r$ is the Hausdorff dimension of $G$. We must have $n_h\ge 4$, since $s\ge 2$, $k_1\ge 2$ and $k_2\ge 1$.

Fix a basis $X_{r,1},\cdots, X_{r,k_r}$, where $k_r=\dim V_r$, of each stratum $V_r$. As $G$ is nilpotent and simply connected, the exponential map $\exp:\mathfrak{g}\to G$ is a diffeomorphism. Thus, each point $p\in G$ can be expressed in the coordinates $p=\exp\left(\sum_{r=1}^s\sum_{i=1}^{k_r}x_{r,i}X_{r,i}\right)$, $x_{r,i}\in \mathbb{R}$; let this be the single coordinate chart on $G$. For polynomials of the coordinates $x_{r,i}$, we assign weight $r$ to the variable $x_{r,i}$, $r=1,\cdots,s$, $i=1,\cdots,k_r$, i.e., for any linear combination $\sum_{\lambda=1}^\Lambda a_\lambda \prod_{\sigma=1}^{S^\lambda} x_{r^\lambda_\sigma,i^\lambda_\sigma}$, $\Lambda\in \mathbb{Z}_{>0}$, $a_\lambda\in \mathbb{R}\setminus\{0\}$, $S^\lambda\in \mathbb{Z}_{\ge 0}$, $r^\lambda_\sigma\in \left\{1,\cdots,s\right\}$, $i^\lambda_\sigma\in \left\{1,\cdots,k_{r^\lambda_\sigma}\right\}$, of distinct monomials with nonzero coefficients, we define
\[
\deg\left(\sum_{\lambda=1}^\Lambda a_\lambda \prod_{\sigma=1}^{S^\lambda} x_{r^\lambda_\sigma,i^\lambda_\sigma}\right)\coloneqq\max_{\lambda=1,\cdots,\Lambda}\sum_{\sigma=1}^{S^\lambda}r^\lambda_\sigma
\]
and we define $\deg 0\coloneqq -\infty$. Let $d\in \mathbb{Z}_{\ge 0}$ and let $\mathcal{A}_d$ denote the family of polynomials $G\to\mathbb{R}$ of weighted degree at most $d$.
\begin{remark}\,
\begin{enumerate}
\item Note that the family $\mathcal{A}_d$ is left-invariant. Indeed, one can express the group law in this coordinate system using the Baker--Campbell--Hausdorff formula
\[
\exp\left(g\right)\exp\left(h\right) =\exp\left(
\sum_{m = 1}^s\frac {\left(-1\right)^{m-1}}{m}
\sum_{\begin{smallmatrix} r_1 + s_1 > 0  \\\cdots \\ r_m + s_m > 0 \end{smallmatrix}}
\frac{\left[ g^{r_1} h^{s_1} g^{r_2} h^{s_2} \dotsm g^{r_m} h^{s_m} \right]}{\left(\sum_{j = 1}^m \left(r_j + s_j\right)\right) \cdot \prod_{j = 1}^m r_j! s_j!}\right),\quad g,h\in \mathfrak{g},
\]
where the sum is finite since $G$ is nilpotent of step $s$, and we have used the following notation:
\[
\left[ g^{r_1} h^{s_1} \dotsm g^{r_m} h^{s_m} \right] \coloneqq [ \underbrace{g,[g,\dotsm[g}_{r_1} ,[ \underbrace{h,[h,\dotsm[h}_{s_1} ,\dotsm [ \underbrace{g,[g,\dotsm[g}_{r_m} ,[ \underbrace{h,[h,\dotsm h}_{s_m} ]]\dotsm]].
\]
Thus, we can see that
\begin{equation}\label{eq:grouplaw1}
\exp\left(\sum_{r=1}^s\sum_{i=1}^{k_r}x^0_{r,i} X_{r,i}\right)\exp\left(\sum_{r=1}^s\sum_{i=1}^{k_r}x^1_{r,i} X_{r,i}\right)=\exp\left(\sum_{r=1}^s\sum_{i=1}^{k_r}x^2_{r,i} X_{r,i}\right)
\end{equation}
where
\begin{equation}\label{eq:grouplaw2}
x^2_{r,i}=x^0_{r,i}+x^1_{r,i}+\left(\mathrm{homogeneous~ polynomial ~of~ }\left\{x^0_{r',i'}\right\}_{r'<r},\left\{x^1_{r',i'}\right\}_{r'<r}\mathrm{~of~ weighted~ degree~ }r\right).
\end{equation}
Therefore, a polynomial of weighted degree $d$ precomposed with a left translation is still of weighted degree $d$.

\item This definition of $\mathcal{A}_d$ is in contrast from \cite{fassler2020dorronsoro}, where they only considered horizontal polynomials, i.e., polynomials that depend only on the `horizontal coordinates' $x_{1,1},\cdots,x_{1,k_1}$; here we are allowing for non-horizontal coordinates, provided they satisfy the weighted degree condition. Our definition of $\mathcal{A}_d$ in the case $G=\mathbb{H}^3$ agrees with that of \cite{fassler2020dorronsoro} only for $d=0$ and $d=1$.
\end{enumerate}
\end{remark}

As in the $G=\mathbb{R}^n$ case, for a locally integrable function $f:G\to\mathbb{R}$, $x\in G$ and $r>0$, let $A^d_{x,r}f$ denote the unique element of $\mathcal{A}_d$ such that
\[
\int_{B_r\left(x\right)}\left(f\left(y\right)-A^d_{x,r}f\left(y\right)\right)A\left(y\right)d\mu\left(y\right)=0,\quad \forall A\in \mathcal{A}_d.
\]
For example, $A^0_{x,r}f=\fint_{B_r\left(x\right)} f\left(z\right)d\mu\left(z\right)$, the average of $f$ on $B_r\left(x\right)$, and a formula for $A^1_{x,r}f$ is given later in \eqref{A1}. We measure how well $A^d_{x,r}f$ approximates $f$ in the ball $B_r\left(x\right)$ in the $L^q$ norm and denote the average $L^q$-normed difference by the following quantity:
\begin{equation}\label{eq:beta-defn}
\beta_{f,d,q}\left(B_r\left(x\right)\right)=
\begin{cases}
\left(\fint_{B_r\left(x\right)}\left|f\left(y\right)-A^d_{x,r}f\left(y\right)\right|^q d\mu\left(y\right)\right)^{1/q},& 1\le q<\infty,\\
\left\|f-A^d_{x,r}f\right\|_{L^\infty\left(B_r\left(x\right)\right)},& q=\infty,
\end{cases}
\end{equation}
whenever the norm exists; otherwise define $\beta_{f,d,q}\left(B_r\left(x\right)\right)=\infty$. 

Now we are ready to state our version of Dorronsoro's theorem for Carnot groups.
\begin{theorem}[Dorronsoro's theorem for Carnot groups]\label{lpgenthm}
Let $G$ be a Carnot group with Hausdorff dimension $n_h$. Let $1<p<\infty$, $\alpha>0$, and $1\le q\le\infty$ satisfy
\begin{equation}\label{eq:dorronsoro-condition}
    \frac{\alpha}{n_h}+\frac 1q>\frac{1}{\min\{p,2\}}.
\end{equation}
Then for all $f\in L^p\left(G\right)$, with the definition \eqref{eq:beta-defn},
\begin{equation}\label{form13}
\left(\int_G \left(\int_0^\infty \left[\frac 1{r^\alpha} \beta_{f,\lfloor \alpha \rfloor,q}\left(B_r\left(x\right)\right) \right]^2 \frac{dr}{r} \right)^{p/2} d\mu\left(x\right) \right)^{1/p} \asymp_{G,\alpha,p,q} \left\|\left(-\Delta_p\right)^{\alpha/2} f\right\|_{L^p\left(G\right)},
\end{equation}
in the sense that $f\in S^p_\alpha\left(G\right)$ if and only if the left-hand side of \eqref{form13} is finite, in which case the above relation holds.
\end{theorem}

Our proof of Theorem \ref{lpgenthm} is based on Dorronsoro's original proof \cite{dorronsoro1985characterization} and adds on ingredients from F\"assler and Orponen's proof \cite{fassler2020dorronsoro}. Actually, in the case of $\alpha=1$ and $G=\mathbb{H}^{2k+1}$ there is a simpler proof of Dorronsoro's theorem involving the Fourier transform \cite[Subsection 7.3]{azzam2016bi}.\footnote{I thank Ian Fleschler for pointing out this reference.} It seems likely that there would be a similar simpler proof of Theorem \ref{lpgenthm} for $\alpha=1$, but we have not pursued this direction since we can obtain the full range $\alpha>0$ with our proof.

The novelty of Theorem \ref{lpgenthm} is threefold. First, we recognize the correct class of polynomials to approximate by, namely polynomials of weighted degree that depend on the ``full set of coordinates'' as opposed to just the ``horizontal coordinates.'' Second, we recover the full Dorronsoro statement, i.e., we prove both directions of the equivalence. Third, we verify that generalizing to higher step Carnot groups introduces no serious problems.

\begin{remark}\label{remark:noHori}
F\"assler and Orponen's version \cite{fassler2020dorronsoro} of the Dorronsoro theorem \cite{dorronsoro1985characterization} for Heisenberg groups states the $\lesssim$ portion of the above inequality \eqref{form13} for $f\in S^p_\alpha\left(G\right)$ in the case $G=\mathbb{H}^{2k+1}$ and $0<\alpha<2$. (They did not prove the $\gtrsim$ portion because it is not needed when proving the vertical versus horizontal inequalities.) The restriction $0<\alpha<2$ was necessary because they were approximating by horizontal polynomials, i.e., polynomials that depend only on the horizontal coordinates; in the range $0<\alpha<2$ our formulation of Theorem \ref{lpgenthm} agrees with that of \cite{fassler2020dorronsoro} because then $\mathcal{A}_1$ is precisely the family of linear horizontal polynomials. To see why the restriction $0<\alpha<2$ is necessary when we consider only horizontal polynomials, let $\alpha\ge 2$, and suppose $\mathcal{A}_d$ were defined as the set of horizontal polynomials of degree at most $d$. In the case of the Heisenberg group $\mathbb{H}^{2k+1}$ the group structure is given by
\begin{align*}
&\exp\left(\sum_{i=1}^k\left(x^0_i\partial_{x_i}+y^0_i\partial_{y_i}\right)+z^0\partial_z\right)\cdot \exp\left(\sum_{i=1}^k\left(x^1_i\partial_{x_i}+y^1_i\partial_{y_i}\right)+z^1\partial_z\right) \\
&= \exp\left(\sum_{i=1}^k\left(\left(x^0_i+x^1_i\right)\partial_{x_i}+\left(y^0_i+y^1_i\right)\partial_{y_i}\right)+\left(z^0+z^1+\frac 12\sum_{i=1}^k\left(x_i^0y_i^1-x_i^1y_i^0\right)\right)\partial_z\right)
\end{align*}
and thus the left-invariant horizontal vector fields are given by
\[
X_i=\frac{\partial}{\partial x_i}+\frac{y_i}{2} \frac{\partial}{\partial z},\quad Y_i=\frac{\partial}{\partial y_i}-\frac{x_i}{2} \frac{\partial}{\partial z},\qquad i=1,\cdots,k,
\]
so that the Laplacian is given by
\[
\Delta = \sum_{i=1}^k\left(X_i^2+Y_i^2\right)= \sum_{i=1}^k \left(\frac{\partial^2}{\partial x_i^2}+\frac{\partial^2}{\partial y_i^2}+y_i\frac{\partial^2}{\partial x_i\partial z}-x_i\frac{\partial^2}{\partial y_i\partial z}\right)+\frac 14\sum_{i=1}^k\left(x_i^2+y_i^2\right) \frac{\partial^2}{\partial z^2}.
\]
Let $f:\mathbb{H}^{2k+1}\to \mathbb{R}$ be a function which is smooth, supported on $B_2$ and agrees with the function $z$ on $B_1$. Clearly $\left\|\left(-\Delta\right)^{\alpha/2} f\right\|_{L^p\left(G\right)}$ is finite. However, for $r\in \left(0,\frac 12\right)$, $A^d_{0,r}f=A^d_{0,r}z=0$ because $B_r$ is symmetric about reflection with respect to the plane $z=0$; thus
\[
\beta_{f,d,q}\left(B_r\right)=\left(\fint_{B_r}\left|z\right|^qd\mu\left(z\right)\right)^{1/q}\asymp_{G,q} r^2.
\]
Since $\mathcal{A}_d$ is invariant under left-translation and left-translation is measure-preserving, at
\[
p=\exp\left(\sum_{i=1}^k\left(x^0_i\partial_{x_i}+y^0_i\partial_{y_i}\right)+z^0\partial_z\right)\in B_{1/2}
\]
we have
\[
A^d_{p,r}f=z^0+\frac 12 \sum_{i=1}^k\left(x_i^0y_i-x_iy_i^0\right),\quad \beta_{f,d,q}\left(B_r\left(x\right)\right)=\beta_{f,d,q}\left(B_r\right)\asymp_{G,q}r^2.
\]
Therefore, in this hypothetical case, i.e., when we defined $\mathcal{A}_d$ as the set of horizontal polynomials of degree at most $d$, the left-hand side of \eqref{form13} would be bounded below by
\[
\left(\int_{B_{1/2}} \left(\int_0^{1/2} \left[\frac 1{r^\alpha} \beta_{f,\lfloor \alpha \rfloor,q}\left(B_r\left(x\right)\right) \right]^2 \frac{dr}{r} \right)^{p/2} d\mu\left(x\right) \right)^{1/p}\asymp_{G,q} \mu\left(B_{1/2}\right)^{1/p}\left(\int_0^{1/2}  \frac{dr}{r^{1+2\left(\alpha-2\right)}} \right)^{1/2}=\infty.
\]
Thus \eqref{form13} fails to hold if $\alpha\ge 2$ when $\mathcal{A}_d$ is restricted to horizontal polynomials.
\end{remark}

We will use the Dorronsoro Theorem for Carnot groups, Theorem \ref{lpgenthm}, to prove the vertical versus horizontal inequality of Theorem \ref{thm:lp}. The proof of Theorem \ref{thm:lp} from Theorem \ref{lpgenthm} is given in Section \ref{sec:VvsH}, and uses the special case of $\alpha=1$ of Theorem \ref{lpgenthm} along with the fact that $\left\|\left(-\Delta_p\right)^{1/2}f\right\|_{L^p\left(G\right)}\asymp \left\|\nabla f\right\|_{L^p\left(G;\ell_2^k\right)}$ \cite[(52)]{coulhon2001sobolev} to see that it is enough to upper bound the left-hand side of \eqref{eq:lp main} by the left-hand side of \eqref{form13}.

However, one may ask what vertical versus horizontal inequalities emerge when we don't specialize to $\alpha=1$. In this case, we obtain fractional order generalizations of Theorem \ref{thm:nilpotentVvsH}.


\begin{theorem}\label{VvsH}
Let $G$ be a nonabelian Carnot group. Let $1<p\le 2$, $\alpha>0$, $n\in \mathbb{N}$, and let $f\in S^p_\alpha\left(G\right)$. Let $v\in\exp\left( V_{\lfloor\alpha/n\rfloor+1}\oplus \cdots \oplus V_s\right)$ with $d_G\left(v,e_G\right)=1$.
Then
\begin{equation}\label{eq:VvsH}
\left(\int_0^\infty \left[ \int_G \left(\frac{1}{r^\alpha}\left|\Delta_{\delta_{r}\left(v\right)}^n f\left(h\right)\right| \right)^p d\mu\left(h\right) \right]^{2/p}  \frac{dr}{r} \right)^{1/2} \lesssim_{G,p,\alpha,n} \left\|\left(-\Delta_p\right)^{\alpha/2} f\right\|_{L^p\left(G\right)},
\end{equation}
where for $g\in G$ and $F:G\to\mathbb{R}$, $\Delta_gF\left(x\right)\coloneqq F\left(xg\right)-F\left(x\right)$ denotes the finite difference.
\end{theorem}

The simple example $n=1$ is given as follows.

\begin{example}\label{easyVvsH}
Let $G$ be a nonabelian Carnot group. Let $1<p\le 2$, $\alpha>0$, and let $f\in S^p_\alpha\left(G\right)$. Let $v\in \exp\left(V_{\lfloor\alpha\rfloor+1}\oplus \cdots \oplus V_s\right)$ with $d_G\left(v,e_G\right)=1$.
Then
\[
\left(\int_0^\infty \left[ \int_G \left(\frac{\left|f\left(h\right) - f\left(h\delta_r\left(v\right)\right)\right|}{r^\alpha} \right)^p d\mu\left(h\right) \right]^{2/p}  \frac{dr}{r} \right)^{1/2} \lesssim_{G,p,\alpha} \left\|\left(-\Delta_p\right)^{\alpha/2} f\right\|_{L^p\left(G\right)}.
\]
\end{example}

The case $G=\mathbb{H}^{2k+1}$, $n=1$, $0<\alpha<2$ of Example \ref{easyVvsH} was obtained in \cite{fassler2020vertical}.

In the above, Theorem \ref{lpgenthm} was used for proving nonembeddability, but there are also other applications. For example, the one-sided, $0<\alpha<2$ case of Theorem \ref{lpgenthm} for the Heisenberg groups $\mathbb{H}^{2k+1}$, $k\ge 1$, due to \cite{fassler2020dorronsoro}, is used in the work \cite{chousionis2022strong};\footnote{I thank Professor Tuomas Orponen for pointing out this fact.} it stands to reason that our more general result will have similar applications, but we defer this to future investigations.


\subsection{$L^1$-distortion: vertical perimeter versus horizontal perimeter}\label{subsec:l1}
Given the discussion so far, one may ask whether the results of this paper follow for $L^1$-targets. Since $\ell_1$ is not uniformly convex, our results in this paper do not apply; in fact, that general nonabelian simply connected nilpotent Lie groups $G$ do not embed bilipschitzly into $L^1$ spaces was proven only recently by Eriksson-Bique, Gartland, Le Donne, Naples, and Nicolussi-Golo \cite{eriksson2023nilpotent}. For the Heisenberg group, this was proven previously by Cheeger and Kleiner \cite{cheeger2010differentiating}.

The results of this paper give quantitative nonembeddability of nonabelian simply connected nilpotent Lie groups into uniformly convex spaces. In this subsection, which was inspired by Section 4 of \cite{lafforgue2014vertical}, we will present hypothetical analogs of the results of this paper for quantitative nonembeddability into $L^1$.

Theorem \ref{thm:nilpotentVvsH} states in the case $X=\mathbb{R}$ and $p\in \left(1,q\right]$ that for every smooth and compactly supported $f:G\to \mathbb{R}$, we have, recalling that $s$ is the nilpotency step of $G$ and $v\in \underbrace{\left[\mathfrak{g},\left[\mathfrak{g},\cdots,\mathfrak{g}\right]\right]}_{s~\mathrm{times}}$ is normalized so that $d_G\left(\exp\left(v\right),e_G\right)=1$,
\begin{align}\label{eq:toR}
\begin{aligned}
&\left(\int_0^\infty \left(\int_{G}\left|f\left(h\exp\left(tv\right)\right)-f\left(h\right)\right|^pd\mu\left(h\right)\right)^{q/p}\frac{dt}{t^{1+q/s}}\right)^{1/q}\\
&\lesssim_G \max\left\{1,\left(p-1\right)^{(1/q)-1}\right\} \left(\int_{G} \left\|\nabla f\left(h\right)\right\|_{\ell_2^k}^p d\mu\left(h\right)\right)^{1/p}.
\end{aligned}
\end{align}
The constant $\left(p-1\right)^{(1/q)-1}$ is unbounded as $p\to 1$; nevertheless, we ask whether the endpoint case $p=1$ of \eqref{eq:toR} does hold true.\footnote{We ask Question \ref{Q:strongest} for smooth compactly supported functions since the generalization to the class $W_0^{1,1}\left(G\right)$ is straightforward; see the approximation argument of subsubsection \ref{subsec:step0}. We remark that since the target is $\mathbb{R}$, a Meyers--Serrin theorem stating that $W_0^{1,p}\left(G\right)=W^{1,p}\left(G\right)$ holds by \cite[Theorem 11.9]{hajlasz2000sobolev} (see footnote \ref{foot:meyers-serrin}).}
\begin{question}\label{Q:strongest}
Let $G$ be a nonabelian simply connected nilpotent Lie group of nilpotency step $s$, and let $v\in \underbrace{\left[G,\left[G,\cdots,G\right]\right]}_{s~\mathrm{times}}$ be normalized so that $d_G\left(v,e_G\right)=1$.
For which exponents $q\ge 1$ does every smooth and compactly supported $f:G\to \mathbb{R}$  satisfy the following?
\begin{equation}\label{eq:p=1}
\left(\int_0^\infty \left(\int_{G}\left|f\left(h\exp\left(tv\right)\right)-f\left(h\right)\right|d\mu\left(h\right)\right)^{q}\frac{dt}{t^{1+q/s}}\right)^{1/q}\lesssim_{G,q} \int_{G} \left\|\nabla f\left(h\right)\right\|_{\ell_2^k} d\mu\left(h\right).
\end{equation}
\end{question}
This tensorizes easily into an inequality for functions $f:G\to L^1$. Similarly, we may ask the following question.
\begin{question}\label{ques:poly}
Let $\Gamma$ be a not virtually abelian finitely generated group of polynomial growth. Choose $v_\Gamma$ and $s>0$ as in Theorem \ref{thm:discrete}. For which exponents $q\ge 1$ do there exist $c=c\left(\Gamma\right)\in \mathbb{N}$ such that for any $n\in \mathbb{N}$ and $f:\Gamma\to L^1$ we have the following?
\[
\left(\sum_{k=1}^{n^s}\frac{1}{k^{1+q/s}}\sum_{x\in B^\Gamma_n} \left\|f\left(xv_\Gamma^k\right)-f\left(x\right)\right\|_{L^1}^q\right)^{1/q}\lesssim_{\Gamma,q} \sum_{x\in B^\Gamma_{cn}} \sum_{a\in S}\left\|f\left(xa\right)-f\left(x\right)\right\|_{L^1}.
\]
\end{question}

By the discretization argument of Section \ref{sec:proof main}, an exponent $q$ that answers Question \ref{Q:strongest} positively also answers Question \ref{ques:poly} positively, when $G$ is chosen as the Malcev completion of the torsion subgroup quotient of a finite index nilpotent subgroup of $\Gamma$.

Because the case $q=\infty$ formally holds for Question \ref{Q:strongest} (see Remark \ref{qge2}), by H\"older's inequality the set of $q$ that satisfies Question \ref{Q:strongest} is either of the form $\left(q_{G,v},\infty\right)$ or $\left[q_{G,v},\infty\right)$ for some $q_{G,v}\in \left[1,\infty\right]$ depending on $G$. We make the following formal definition.
\begin{definition}
    Given a nonabelian nilpotent Lie group $G$, let $v\in \underbrace{\left[\mathfrak{g},\left[\mathfrak{g},\cdots,\mathfrak{g}\right]\right]}_{s~\mathrm{times}}$with $d_G\left(\exp\left(v\right),e_G\right)=1$. Let $q_{G,v}$ denote the infimum of those $q\ge 1$ such that \eqref{eq:p=1} holds for every smooth and compactly supported function $f:G\to \mathbb{R}$. Let $q_G$ be the infimum of $q_{G,v}$ over all choices of $v$ as above.
\end{definition}
For $G=\mathbb{H}^3$, there is only one choice of $v$ up to a sign, and we have $q_{\mathbb{H}^3}=4$ with the exponent range satisfying \eqref{eq:p=1} being $\left[4,\infty\right)$ \cite{naor2022foliated}. For  $G=\mathbb{H}^{2k+1}$ with $k\ge 2$, there is again only one choice of $v$ up to a sign, and $q_{\mathbb{H}^{2k+1}}=2$  with the exponent range satisfying \eqref{eq:p=1} being $\left[2,\infty\right)$ \cite{naor2018vertical}.

For a general nilpotent Lie group $G$, we wish to know whether there are finite exponents $q$ satisfying \eqref{eq:p=1}, since then not only would we reprove \cite{eriksson2023nilpotent}, but we would also have analogous quantitative nonembeddability statements: whenever $q$ satisfies \eqref{eq:p=1}, then, similarly to Theorem \ref{mainthm:nilpotentdistortion}, if $G$ is Riemannian, we would have
\[
c_1\left(B_r\right)\gtrsim_{G,q} (\log r)^{1/q},\quad r\ge 2,
\]
and similarly to Theorem \ref{thm:net}, if $G$ is Riemannian or sub-Riemannian and if $N_{r_1,r_2}$ is an $r_2$-covering of $B_{r_1}$, where $r_1,r_2>0$ with $r_1\ge 2r_2$, and with the extra condition that $r_1,r_2\ge 1$ if $v\in \operatorname{span}\left\{X_1,\cdots,X_k\right\}$, then
\[
c_1\left(N_{r_1,r_2}\right)\gtrsim_{G,q} \left(\log \left(r_1/r_2\right)\right)^{1/q}.
\]
Also, if $\Gamma$ is a not virtually abelian finitely generated group of polynomial growth that is quasi-isometric to $G$, we would have
\[
c_1\left(B_n^\Gamma\right)\gtrsim_{\Gamma,q} \left(\log n\right)^{1/q},\quad n\ge 2.
\]

By the co-area formula it suffices to prove \eqref{eq:p=1} when $f$ is an indicator of a measurable set $A\subseteq G$, in which case the right-hand side of \eqref{eq:p=1} is interpreted as the horizontal perimeter $\mathrm{PER}\left(A\right)$ of $A$ (see \cite{ambrosio2001some} and \cite[Section 2]{cheeger2010differentiating} for a precise definition).

\begin{definition}[Vertical perimeter at scale $t$ \cite{lafforgue2014vertical,naor2018vertical}]\label{def vert}
Let $v\in \underbrace{\left[\mathfrak{g},\left[\mathfrak{g},\cdots,\mathfrak{g}\right]\right]}_{s~\mathrm{times}}$ be as in Theorem \ref{thm:nilpotentVvsH}. Let $A\subseteq G$ be measurable and $t\in \left(0,\infty\right)$. The vertical perimeter $v_t\left(A\right)$ of $A$ at scale $t$ is defined as the quantity
\begin{equation}\label{eq:def Vt}
v_t\left(A\right)\coloneqq \mu\left(\left\{h\in A:\ h\exp\left(tv\right)\notin A\quad\mathrm{or}\quad h\exp\left(-tv\right)\notin A\right\}\right).
\end{equation}
\end{definition}
By this definition, we may reformulate Question \ref{Q:strongest} as an isoperimetric inequality.
\begin{question}\label{Q:isoperimetric}
For which exponents $q$ is it true that for every measurable $A\subseteq G$ one has
\begin{equation}\label{eq:con isoperimetric}
\left(\int_0^\infty \frac{v_t\left(A\right)^q}{t^{1+q/s}}dt\right)^{1/q}\lesssim_{G,q} \mathrm{PER}\left(A\right).
\end{equation}
\end{question}

Of course, the set of $q$ that satisfies Question \ref{Q:isoperimetric} is the same as that of Question \ref{Q:strongest}. In light of the nonembeddability result \cite{eriksson2023nilpotent} and the computations of $q_{\mathbb{H}^{2k+1}}$ \cite{naor2018vertical,naor2022foliated}, we make the following conjecture.

\begin{conjecture}\label{conj:finite}
There exist finite exponents $q$ that answer Questions \ref{Q:strongest} and \ref{Q:isoperimetric} positively, and the infimum $q_G$ among such $q$ is attained.
\end{conjecture}

Conditioned on a positive resolution of Conjecture \ref{conj:finite}, we would have:
\begin{itemize}
\item $c_1\left(B_r\right)\gtrsim_G \left(\log r\right)^{1/q_G}$ for $r\ge 2$ when $G$ is Riemannian,
\item $c_1\left(N_{r_1,r_2}\right)\gtrsim_G \left(\log \left(r_1/r_2\right)\right)^{1/q_G}$ for $r_1,r_2>0$ with $r_1\ge 2r_2$ and an $r_2$-covering $N_{r_1,r_2}$ of $B_{r_1}$, with $G$ Riemannian or sub-Riemannian, such that $r_1,r_2\ge 1$ if $v\in \operatorname{span}\left\{X_1,\cdots,X_k\right\}$,
\item $c_1\left(B_n^\Gamma\right)\gtrsim_\Gamma \left(\log n\right)^{1/q_G}$, $n\ge 2$, for a not virtually abelian finitely generated group $\Gamma$ of polynomial growth, where $G$ is quasi-isometric to $\Gamma$ (an example of $G$ being the Malcev completion of the torsion subgroup quotient of a finite index nilpotent subgroup of $\Gamma$).
\end{itemize}
However, for the $5$ or higher dimensional Heisenberg groups $\mathbb{H}^{2k+1}$, $k\ge 2$, with $q_{\mathbb{H}^{2k+1}}=2$ we have the matching upper bounds $c_1\left(N_{r_1,r_2}\right)\lesssim \sqrt{\log \left(r_1/r_2\right)}$ and $c_1\left(B_n^{\mathbb{H}^{2k+1}_\mathbb{Z}}\right)\lesssim \sqrt{\log n}$  by the Assouad embedding theorem (Theorem \ref{thm:lp-assouad}(1)), while for the $3$-dimensional Heisenberg group $\mathbb{H}^3$ we have the matching upper bound $c_1\left(N_{r_1,r_2}\right)\lesssim \left(\log \left(r_1/r_2\right)\right)^{1/4}$ and $c_1\left(B_n^{\mathbb{H}^{3}_\mathbb{Z}}\right)\lesssim \left(\log n\right)^{1/4}$ by \cite[Theorem 3.1]{naor2022foliated}. We thus pose the following conjecture on the exact asymptotics of the $L^1$ distortion.

\begin{conjecture}\label{ques:polyL1}
Let $G$ be a nonabelian simply connected nilpotent Lie group.
\begin{itemize}
\item When $G$ is given a Riemannian metric, $c_1\left(B_r\right)\asymp_G\left(\log r\right)^{1/q_G}$ for $r\ge 2$.
\item When $G$ is given a Riemannian or sub-Riemannian metric, for an $\left(r_2,\Omega\left(r_2\right)\right)$-net $N_{r_1,r_2}$ of $B_{r_1}$ where $r_1,r_2>1$ with $r_1\ge 2r_2$, we have $c_1\left(N_{r_1,r_2}\right)\asymp_G \left(\log \left(r_1/r_2\right)\right)^{1/q_G}$.
\item Given $\Gamma$, a not virtually abelian finitely generated group of polynomial growth such that $G$ is the Malcev completion of the torsion subgroup quotient of a finite index nilpotent subgroup of $\Gamma$, we have $c_1\left(B^\Gamma_n\right)\asymp_G \left(\log n\right)^{1/q_G}$, $n\ge 2$.
\end{itemize}
\end{conjecture}
We remark that if this conjecture had a positive answer, then the exponent $q_G$ will be a quasi-isometric invariant of nonabelian simply connected nilpotent Lie groups. It would then be of interest to see if $q_G$ is a function of the known invariants of \cite{shalom2004harmonic,sauer2006homological,cornulier2012quasi}, or if it is a genuinely new invariant, contributing to the open problem \cite[Conjecture 19.114]{cornulier2012quasi} asking whether two simply connected nilpotent Lie groups are quasi-isometrically equivalent only if they are isomorphic as Lie groups.


\section{Proof of Theorem \ref{thm:nilpotentVvsH}: heat flow along the center, pencil of curves, and martingales in Banach spaces}\label{sec:cvx}
We will prove a theorem more general than Theorems \ref{thm:nilpotentVvsH} and \ref{thm:carnotVvsH}, namely Theorem \ref{thm:cvx} below. We first explain the terminology behind the statements.

A Lie group $G$ is said to be unimodular if its left-invariant Haar measure $\mu$ is also right-invariant. Given left-invariant vector fields $X_1,\cdots,X_k$, the pointwise span of $X_1,\cdots,X_k$ forms a left-invariant vector bundle $B$ over $G$, a subbundle of the tangent bundle of $G$, and on each fibre of the vector  bundle $B$ we may define a left-invariant Euclidean norm\footnote{Again, taking a general left-invariant norm on $B$, in which case the distance is Finsler or sub-Finsler, would only distort distances by a factor of $\sqrt{k}$.} $\left|\cdot \right|$ that has $X_1,\cdots,X_k$ as an orthonormal basis. We may define the Riemannian or sub-Riemannian Carnot--Carath\'eodory distance associated to $B$ and $\left|\cdot \right|$, i.e.,
\[
d_G\left(p,q\right)\coloneqq \inf\left\{\int_0^1 \left|\dot{\gamma}\left(t\right) \right|dt : \gamma\in C_{\mathrm{pw}}^\infty \left(\left[0,1\right];G\right),\gamma\left(0\right)=p,\gamma\left(1\right)=q,\dot{\gamma}\in B \right\},\quad p,q\in G.
\]
If $X_1,\cdots,X_k$ and their nested brackets generate $\mathfrak{g}$, and if $G$ is connected, then $d_G\left(\cdot,\cdot\right)$ is finite everywhere by the Chow--Rashevskii theorem \cite{chow1940systeme,rashevskii1938connecting}. The center $Z\left(\mathfrak{g}\right)$ of the Lie algebra $\left(\mathfrak{g},\left[,\right]\right)$ of $G$ consists of elements $v\in \mathfrak{g}$ such that $\left[v,h\right]=0$ for all $h\in \mathfrak{g}$, and has the property that if $v\in Z\left(\mathfrak{g}\right)$ then $\exp\left(v\right)\in Z\left(G\right)$, the group theoretic center of $G$.

This generalizes our earlier choice of $G$ as a simply connected nilpotent Lie group, and our choice of $X_1,\cdots,X_k$. As such, the following theorem encompasses Theorems \ref{thm:nilpotentVvsH} and \ref{thm:carnotVvsH}.

\begin{theorem}\label{thm:cvx}
Let $G$ be a unimodular Lie group with Haar measure $\mu$, with left-invariant vector fields $X_1,\cdots,X_k$, which span $\mathfrak{g}$ by themselves and their nested brackets, and with the associated Riemannian or sub-Riemannian distance $d_G\left(\cdot,\cdot\right)$. Suppose there is an element $v\in Z\left(\mathfrak{g}\right)$ such that
\begin{equation}\label{eq:rho-ineq}
d_G\left(\exp\left(tv\right),e_G\right)\le t^{1/\rho}, \quad \forall t>0,
\end{equation}
for some real number $\rho>1$, and such that there is a subset $S$ of $G$ with measure $\nu$ such that the push-forward of the product measure of $\nu$ and the Lebesgue measure of $\mathbb{R}$ under the map $S\times \mathbb{R}\to G$, $\left(s,t\right)\mapsto s\exp\left(tv\right)$ is the Haar measure $\mu$ of $G$.

Suppose that $q\in \left[2,\infty\right)$ and $p\in (1,q]$. Let $\left(X,\left\|\cdot\right\|_X\right)$ be a Banach space with $K_q\left(X\right)<\infty$. Let $f\in Ch_0^{1,p}\left(G;X\right)$. Then
\begin{align}\label{eq:continuous main p<q}
\begin{aligned}
&\left(\int_0^\infty \left(\int_{G}\left(\frac{\left\|f\left(h\exp\left(tv\right)\right)-f\left(h\right)\right\|_X}{t^{1/\rho}}\right)^pd\mu\left(h\right)\right)^{q/p}\frac{dt}{t}\right)^{1/q}\\
&\lesssim \left(\rho+\frac{1}{\rho-1}\right)\max\left\{\left(p-1\right)^{(1/q)-1},K_q\left(X\right)\right\} \left\|f\right\|_{\dot{Ch}_0^{1,p}\left(G;X\right)}.
\end{aligned}
\end{align}
In particular, if $p= q$, then
\begin{equation}\label{eq:continuous main}
\left(\int_0^\infty \int_{G}\left(\frac{\left\|f\left(h\exp\left(tv\right)\right)-f\left(h\right)\right\|_X}{t^{1/\rho}}\right)^pd\mu\left(h\right)\frac{dt}{t}\right)^{1/p}
\lesssim \left(\rho+\frac{1}{\rho-1}\right)K_p\left(X\right) \left\|f\right\|_{\dot{Ch}_0^{1,p}\left(G;X\right)}.
\end{equation}
\end{theorem}

We check that Theorem \ref{thm:cvx} implies Theorem \ref{thm:nilpotentVvsH}, which in turn implies Theorem \ref{thm:carnotVvsH}.

\begin{proof}[proof of Theorem \ref{thm:nilpotentVvsH} assuming Theorem \ref{thm:cvx}]
    We are to prove \eqref{eq:nilpotentVvsH-r'=1} and \eqref{eq:nilpotentVvsH-r'>1}.

    The hypothesis for Theorem \ref{thm:cvx} is satisfied when $G$ is a simply connected nilpotent Lie group. The Haar measure, being the Lebesgue measure on $\mathfrak{g}$, is unimodular. By the assumption \eqref{eq:sharpDistIneq} and $r'\le s'$, we have $\eqref{eq:rho-ineq}$ for $\rho=r'$ and $\rho=s'$; because $v\in \left[\mathfrak{g},\mathfrak{g}\right]$, we know for sure that $s'\ge 2$, but $r'$ may or may not be $1$. The product decomposition of the Haar measure easily follows from that of the Lebesgue measure on $\mathbb{R}^n$.

    If $r'=1$, then \eqref{eq:nilpotentVvsH-r'=1} follows from \eqref{eq:continuous main p<q} with $\rho=s'$. If $r'>1$, then from \eqref{eq:continuous main p<q} with $\rho=s'$, we have
\[
\left(\int_1^\infty \left(\int_{G}\left(\frac{\left\|f\left(h\exp\left(tv\right)\right)-f\left(h\right)\right\|_X}{t^{1/s'}}\right)^pd\mu\left(h\right)\right)^{q/p}\frac{dt}{t}\right)^{1/q}\lesssim s'\max\left\{\left(p-1\right)^{(1/q)-1},K_q\left(X\right)\right\} \left\|f\right\|_{\dot{Ch}_0^{1,p}\left(G;X\right)},
\]
while from \eqref{eq:continuous main p<q} with $\rho=r'$, we have
    \[
\left(\int_0^1 \left(\int_{G}\left(\frac{\left\|f\left(h\exp\left(tv\right)\right)-f\left(h\right)\right\|_X}{t^{1/r'}}\right)^pd\mu\left(h\right)\right)^{q/p}\frac{dt}{t}\right)^{1/q}\lesssim r'\max\left\{\left(p-1\right)^{(1/q)-1},K_q\left(X\right)\right\} \left\|f\right\|_{\dot{Ch}_0^{1,p}\left(G;X\right)},
\]
whence we obtain \eqref{eq:nilpotentVvsH-r'>1}.
\end{proof}

\begin{proof}[proof of Theorem \ref{thm:carnotVvsH} assuming Theorem \ref{thm:nilpotentVvsH}]
We are to prove \eqref{eq:carnotVvsH}.

    A Carnot group $G$ is a simply connected nilpotent Lie group. We have \eqref{eq:sharpDistIneq} for $\rho=r$, since
    \[
    d_G\left(\exp\left(tv\right),e_G\right)=d_G\left(\delta_{t^{1/r}}\left(\exp\left(v\right)\right),\delta_{t^{1/r}}\left(e_G\right)\right)=t^{1/r}d_G\left(\exp\left(v\right),e_G\right)=t^{1/r}.
    \]
    Now \eqref{eq:carnotVvsH} follows from \eqref{eq:nilpotentVvsH-r'>1} with $\rho=r$.
\end{proof}

\begin{remark}
    In the $r'=1$ case of Theorem \ref{thm:nilpotentVvsH}, the `short-range estimate' of \eqref{eq:nilpotentVvsH-r'=1}, which is obtained from restricting the $t$-integral of the left-hand side of the inequality to the interval $\left[0,1\right]$, namely
    \[
\left(\int_0^1 \left(\int_{G}\left(\frac{\left\|f\left(h\exp\left(tv\right)\right)-f\left(h\right)\right\|_X}{t^{1/s'}}\right)^pd\mu\left(h\right)\right)^{q/p}\frac{dt}{t}\right)^{1/q}\lesssim s'\max\left\{\left(p-1\right)^{(1/q)-1},K_q\left(X\right)\right\} \left\|f\right\|_{\dot{Ch}_0^{1,p}\left(G;X\right)},
\]
is not a very meaningful or useful estimate: it does not tell us local nonembeddability, following the arguments of Section \ref{sec:net}. One may imagine strengthening the `short-range estimate' of the $r'=1$ case of Theorem \ref{thm:nilpotentVvsH}, in the following manner: in the proof of Theorem \ref{thm:nilpotentVvsH}, instead of specializing \eqref{eq:continuous main} and \eqref{eq:continuous main p<q} to $\rho=s'$, specialize them to $\rho=1+\epsilon$, where $\epsilon\in \left(0,1\right)$ is very small. This will result in the estimate
\[
\left(\int_0^1\left(\int_{G}\left(\frac{\left\|f\left(h\exp\left(tv\right)\right)-f\left(h\right)\right\|_X}{t^{1/\left(1+\epsilon\right)}}\right)^pd\mu\left(h\right)\right)^{q/p}\frac{dt}{t}\right)^{1/q}\lesssim \epsilon^{-1}\max\left\{\left(p-1\right)^{(1/q)-1},K_q\left(X\right)\right\} \left\|f\right\|_{\dot{Ch}_0^{1,p}\left(G;X\right)}.
\]

However, by taking $\rho=1$ and $\varepsilon=\frac{\epsilon}{1+\epsilon}$ in part (1) of Proposition \ref{prop:trivial-vvsh} below, which is a `trivial vertical versus horizontal inequality', the following better estimate follows: for any $p,q\ge 1$,
\[
\left(\int_0^1\left(\int_{G}\left(\frac{\left\|f\left(h\exp\left(tv\right)\right)-f\left(h\right)\right\|_X}{t^{1/\left(1+\epsilon\right)}}\right)^pd\mu\left(h\right)\right)^{q/p}\frac{dt}{t}\right)^{1/q}\le\frac{2}{\epsilon^{1/q}}\left\|f\right\|_{\dot{Ch}_0^{1,p}\left(G;X\right)},
\]
without any assumption on $G$ and $X$, besides the fact that they are a connected Lie group and Banach space, respectively, and only using the distance condition
\[
d_G\left(\exp\left(tv\right),e_G\right)\le t, \quad \forall t\in \left[0,1\right].
\]
Thus, it is futile to make improvements to the $t$-integral on the interval $\left[0,1\right]$ in the $r'=1$ case of Theorem \ref{thm:nilpotentVvsH}, using the result of Theorem \ref{thm:cvx} in the above manner.
\end{remark}

Before we begin the proof of Theorem \ref{thm:cvx}, for motivation we begin by proving Remark \ref{qge2} which states that the case $q=\infty$ of Theorem \ref{thm:lp} formally holds. We start with the following.
\begin{lemma}\label{lem:deterministic-path}
    Let $G$ be a Lie group with left-invariant H\"ormander vector fields $X_1,\cdots,X_k$, a corresponding left-invariant Carnot--Carath\'eodory distance $d_G\left(\cdot,\cdot\right)$ and right-invariant Haar measure $\mu$, and let $g\in G$. Let $\left(X,\|\cdot\|_X\right)$ be a Banach space and let $p\ge 1$. Then for any continuously differentiable function $f:G\to X$,
    \begin{equation}\label{eq:deterministic-path}
        \left(\int_{G}\left\|f\left(hg\right)-f\left(h\right)\right\|_X^pd\mu\left(h\right)\right)^{1/p}
\le d_G\left(g,e_G\right)  \left\|\nabla f\right\|_{L^p\left(G;\ell2_2^k(X)\right)}.
    \end{equation}
\end{lemma}
\begin{proof}\label{proofrmkqge2}
For $\varepsilon> 0$, let $\gamma:\left[0,1\right]\to G$ be a piecewise smooth horizontal curve connecting $e_G$ to $g$, with
\[
\int_0^1\left|\dot\gamma\left(\tau\right)\right|d\tau<d_G\left(g,e_G\right)+\varepsilon.
\]
Writing
\[
\dot{\gamma}\left(t\right)=\sum_{i=1}^k \dot{\gamma}_i(t)\left(X_i\right)_{\gamma\left(t\right)},\quad \left|\dot\gamma\left(t\right)\right|=\sqrt{\sum_{i=1}^k \left(\dot \gamma_i\left(t\right)\right)^2},
\]
we have
\begin{align*}
	\left\|f\left(g_1\right)-f\left(g_0\right)\right\|_X&\le \int_0^1 \left\|\frac{df\left(\gamma\left(t\right)\right)}{dt}\right\|_Xdt=\int_0^1 \left\|\sum_{i=1}^k \dot\gamma_i\left(t\right)\left(X_if\left(\gamma\left(t\right)\right)\right)\right\|_Xdt\\
	&\le \int_0^1 \left\|\nabla f\left(\gamma\left(t\right)\right)\right\|_{B\left(\mathbb{R}^k; X\right)}\left|\dot\gamma\left(t\right)\right|dt\stackrel{\eqref{eq:elltwo-operator}}{\le} \int_0^1 \left\|\nabla f\left(\gamma\left(t\right)\right)\right\|_{\ell_2^k(X)}\left|\dot\gamma\left(t\right)\right|dt
\end{align*}
Thus, by H\"older's inequality and right-invariance of the Haar measure $\mu$,
\begin{align*}
\left(\int_{G}\left\|f\left(hg\right)-f\left(h\right)\right\|_X^pd\mu\left(h\right)\right)^{1/p}&\le \left( \left(\int_0^1\left|\dot\gamma\left(t\right)\right|dt\right)^{p-1}\int_G  \int_0^{1}\left\|\nabla f\left(\gamma\left(t\right)\right)\right\|_{\ell_2^k(X)}^p\left|\dot\gamma\left(s\right)\right|dsd\mu\left(h\right)\right)^{1/p}\\
&\le\left(\int_0^1\left|\dot\gamma\left(t\right)\right|dt\right)\left\|\nabla f\right\|_{L^p\left(G;\ell_2^k(X)\right)}\\
&\le\left(d_G\left(g,e_G\right)+\varepsilon\right)\left\|\nabla f\right\|_{L^p\left(G;\ell_2^k(X)\right)},
\end{align*}
and since $\varepsilon> 0$ was arbitrary, we are done.
\end{proof}

\begin{remark}
    Using \eqref{ineq:lip-nabla} and by a routine approximation argument (given in subsection \ref{subsec:step0} below), we may replace $\left\|\nabla f\right\|_{L^p\left(G;\ell_2^k(X)\right)}$ by $\left\| f\right\|_{\dot{Ch}_0^{1,p}\left(G;X\right)}$ in Lemma \ref{lem:deterministic-path} and Proposition \ref{prop:trivial-vvsh} below.
\end{remark}
Using Lemma \ref{lem:deterministic-path}, we obtain a weak version of the $q=\infty$ case of Theorem \ref{thm:cvx}, that places almost no assumptions on the geometry of the Lie group $G$ and Banach space $X$ involved.
\begin{proposition}\label{prop:trivial-vvsh}
    Let $G$ be a Lie group with left-invariant H\"ormander vector fields $X_1,\cdots,X_k$, a corresponding left-invariant Carnot--Carath\'eodory distance $d_G\left(\cdot,\cdot\right)$ and right-invariant Haar measure $\mu$, and let $v\in \mathfrak{g}$. Let $\left(X,\|\cdot\|_X\right)$ be a Banach space, let $p,q\ge 1$, and let $f:G\to X$ be any continuously differentiable function.
    \begin{enumerate}
        \item Suppose that there are real numbers $\rho\ge 1$ and $t_0>0$ such that
        \begin{equation}\label{eq:rho-ineq-short}
        d_G\left(\exp\left(tv\right),e_G\right)\le t^{1/\rho}, \quad \forall t\in \left[0,t_0\right].
        \end{equation}
        Then, for any $\varepsilon>0$,
        \begin{equation}\label{eq:trivial-vvsh-short}
        \left(\int_0^{t_0} \left(\int_{G}\left(\frac{\left\|f\left(h\exp\left(tv\right)\right)-f\left(h\right)\right\|_X}{t^{1/\rho-\varepsilon}}\right)^pd\mu\left(h\right)\right)^{q/p}\frac{dt}{t}\right)^{1/q}\le t_0^\varepsilon\varepsilon^{-1/q}\left\|\nabla f\right\|_{L^p\left(G;\ell_2^k(X)\right)}.
        \end{equation}
        \item Suppose that there are real numbers $\rho\ge 1$ and $t_0>0$ such that
        \begin{equation}\label{eq:rho-ineq-long}
        d_G\left(\exp\left(tv\right),e_G\right)\le t^{1/\rho}, \quad \forall t\ge t_0.
        \end{equation}
        Then, for any $\varepsilon>0$,
        \begin{equation}\label{eq:trivial-vvsh-long}
        \left(\int_{t_0}^\infty \left(\int_{G}\left(\frac{\left\|f\left(h\exp\left(tv\right)\right)-f\left(h\right)\right\|_X}{t^{1/\rho+\varepsilon}}\right)^pd\mu\left(h\right)\right)^{q/p}\frac{dt}{t}\right)^{1/q}\le t_0^{-\varepsilon}\varepsilon^{-1/q}\left\|\nabla f\right\|_{L^p\left(G;\ell_2^k(X)\right)}.
        \end{equation}
    \end{enumerate}
\end{proposition}
\begin{proof}
We simply use \eqref{eq:deterministic-path} of Lemma  \ref{lem:deterministic-path}.
\begin{enumerate}
    \item Under \eqref{eq:rho-ineq-short}, we have by \eqref{eq:deterministic-path},
    \begin{align*}
&\left(\int_0^{t_0} \left(\int_{G}\left(\frac{\left\|f\left(h\exp\left(tv\right)\right)-f\left(h\right)\right\|_X}{t^{1/\rho-\varepsilon}}\right)^pd\mu\left(h\right)\right)^{q/p}\frac{dt}{t}\right)^{1/q}\\
&\stackrel{\mathclap{\eqref{eq:deterministic-path}}}{\le} ~\left(\int_0^{t_0} t^{-q\left(1/\rho-\varepsilon\right)}\left(d_G\left(\exp\left(tv\right),e_G\right)  \left\|\nabla f\right\|_{L^p\left(G;\ell_2^k(X)\right)}\right)^{q}\frac{dt}{t}\right)^{1/q}\\
&\stackrel{\mathclap{\eqref{eq:rho-ineq-short}}}{\le} ~\left(\int_0^{t_0} t^{q\varepsilon-1}dt\right)^{1/q}\left\|\nabla f\right\|_{L^p\left(G;\ell_2^k(X)\right)}\le t_0^\varepsilon\varepsilon^{-1/q}\left\|\nabla f\right\|_{L^p\left(G;\ell_2^k(X)\right)},
\end{align*}
where in the last inequality we used $q^{1/q}\ge 1$.
    \item Under \eqref{eq:rho-ineq-long}, we have by \eqref{eq:deterministic-path},
    \begin{align*}
&\left(\int_{t_0}^\infty \left(\int_{G}\left(\frac{\left\|f\left(h\exp\left(tv\right)\right)-f\left(h\right)\right\|_X}{t^{1/\rho+\varepsilon}}\right)^pd\mu\left(h\right)\right)^{q/p}\frac{dt}{t}\right)^{1/q}\\
&\stackrel{\mathclap{\eqref{eq:deterministic-path}}}{\le} ~\left(\int_{t_0}^\infty t^{-q\left(1/\rho+\varepsilon\right)}\left(d_G\left(\exp\left(tv\right),e_G\right)  \left\|\nabla f\right\|_{L^p\left(G;\ell_2^k(X)\right)}\right)^{q}\frac{dt}{t}\right)^{1/q}\\
&\stackrel{\mathclap{\eqref{eq:rho-ineq-long}}}{\le} ~\left(\int_{t_0}^\infty t^{-q\varepsilon-1}dt\right)^{1/q}\left\|\nabla f\right\|_{L^p\left(G;\ell_2^k(X)\right)}\le t_0^{-\varepsilon}\varepsilon^{-1/q}\left\|\nabla f\right\|_{L^p\left(G;\ell_2^k(X)\right)}.
\end{align*}
\end{enumerate}
\end{proof}

What has transpired is that in Lemma \ref{lem:deterministic-path}, we took a single deterministic path from $h$ to $hg=h\exp\left(tv\right)$ and differentiated $f$ along that path. From this, we obtain weak versions of the vertical versus horizontal inequalities, namely \eqref{eq:trivial-vvsh-short} and \eqref{eq:trivial-vvsh-long}, in Proposition \ref{prop:trivial-vvsh}. However, the inequalities \eqref{eq:trivial-vvsh-short} and \eqref{eq:trivial-vvsh-long} are weak in the following sense: in the ``short range,'' namely under assumption \eqref{eq:rho-ineq-short}, we are not able to directly compare $\left\|f\left(h\exp\left(tv\right)\right)-f\left(h\right)\right\|_X$ with the quantity $t^{1/\rho}$ in the ``short range vertical versus horizontal inequality'' \eqref{eq:trivial-vvsh-short}; instead we compare it with the larger quantity $t^{1/\rho-\varepsilon}$. Note that we are not able to conclude the nonexistence of bilipschitz embeddings $B_{t_0}\to X$ as in the deductions of Section \ref{sec:net}, as the integral $\int_0^{t_0}t^{q\varepsilon-1}dt$ does not diverge. Likewise, in the ``long range,'' namely under assumption \eqref{eq:rho-ineq-long}, we are left to compare $\left\|f\left(h\exp\left(tv\right)\right)-f\left(h\right)\right\|_X$ with the larger quantity $t^{1/\rho+\varepsilon}$ in the ``long range vertical versus horizontal inequality'' \eqref{eq:trivial-vvsh-short}, and similarly we are not able to conclude the nonexistence of mappings $G\to X$ that fail to be bilipschitz at large scales.

This weakness is necessary: we have placed no restrictions on the geometry of the Lie group $G$ nor the Banach space $X$, and there are plenty of examples of bilipschitz embeddings $G\to X$, such as the identity map $\mathbb{R}^n\to \mathbb{R}^n$. The existence of a proper vertical versus horizontal inequality would have precluded such bilipschitz embeddings from existing in the first place!

Thus, in our proof of Theorem \ref{thm:cvx}, we will need to capitalize on the given geometry of $G$ and $X$. To avoid ending up with the weak inequalities of Lemma \ref{lem:deterministic-path} and Proposition \ref{prop:trivial-vvsh}, taking a deterministic path from $h$ to $h\exp\left(tv\right)$ will not be enough; in fact, we will choose over a pencil of curves, the randomness arising from a heat flow on the line spanned by $v$ (see Lemma \ref{lem:first step} of subsection \ref{subsec:step1}); this is the part where we use the geometry of $G$, namely that the distance function on $\exp\left(\operatorname{span}v\right)$ is sublinear. Then, along each random path, we will apply the differentiation technique of Lemma \ref{lem:deterministic-path} (in Lemma \ref{lem:increment} of subsection \ref{subsec:step2}). (Before this, we will have used an approximation argument in subsection \ref{subsec:step0} to assume $f$ is continuously differentiable, which will guarantee the existence of the vector-valued function $\nabla f:G\to X^k$.) At this stage, we will have bounded the left-hand side of the vertical versus horizontal inequality \eqref{eq:continuous main} in terms of a convolution of $\nabla f$ with a derivative of the heat kernel (in Lemma \ref{lem:integrated increment} of subsection \ref{subsec:step2}). We will bound this using Littlewood--Paley--Stein $g$-function estimates (in subsection \ref{subsec:step3}); this is the part where we use the uniform convexity of $X$. This proof method is inspired by \cite{lafforgue2014vertical,hytonen2019heat}.

It is technically challenging to combine all the ingredients of \cite{lafforgue2014vertical,hytonen2019heat} into our proof of Theorem \ref{thm:cvx}. In particular, we have to replicate several technical claims from \cite{lafforgue2014vertical,hytonen2019heat} in our general context.

We set some notation and terminology.

Recall there is $v\in Z\left(\mathfrak{g}\right)$ and a subset $S$ of $G$ with measure $\nu$ such that the push-forward of the product measure of $\nu$ and the Lebesgue measure of $\mathbb{R}$ under the map $S\times \mathbb{R}\to G$, $\left(s,t\right)\mapsto s\exp\left(tv\right)$ is the Haar measure $\mu$. For $p\in \left[1,\infty\right)$, the Lebesgue--Bochner spaces
\[
L^p\left(\mathbb{R},\mathrm{Lebesgue};X\right)=L^p\left(\mathbb{R};X\right), \quad L^p\left(S,\nu;X\right)=L^p\left(S;X\right), \quad L^p\left(G,\mu;X\right)=L^p\left(G;X\right)
\]
are defined. For $f\in L^p\left(G;X\right)$ we may define $f^v\in L^p\left(\mathbb{R}, L^p\left(S;X\right)\right)$ by\footnote{More precisely, by \cite[Theorem 3(1)]{bogdanowicz1965fubini} and Fubini's theorem, the function $f$ is an equivalence class of functions up to measure-zero differences. Given a representative $f$, we have $f^v\left(z\right)\in L^p\left(S;X\right)$ for Lebesgue-a.e.~$z\in \mathbb{R}$, and given a different representative $\tilde{f}$, we have $\tilde{f}^v=f^v$ $\nu$-a.e.~Lebesgue-a.e.}
\begin{equation}\label{eq:fvdef}
    f^v\left(z\right)\left(x\right)=f\left(x\exp\left(zv\right)\right),\quad x\in S,~z\in \mathbb{R},
\end{equation}
which satisfies
\begin{equation}\label{eq:Lebesgue-Bochner}
\left\|f\right\|_{L^p\left(G;X\right)}=\left\|f^v\right\|_{L^p\left(\mathbb{R}, L^p\left(S;X\right)\right)}.
\end{equation}
Given $\psi\in L^1\left(\mathbb{R}\right)$, define the convolution $\psi*f\coloneqq \psi*f^v\in L^p\left(G;X\right)$,
\begin{equation*}
    \psi*f\left(x \exp\left(zv\right)\right)\coloneqq \int_\mathbb{R} \psi\left(u\right) f\left(x\exp\left(\left(z-u\right)v\right)\right)du\in X.
\end{equation*}
So $\psi*f$ is the usual group convolution of $f$ with the measure supported on $\exp\left(\operatorname{span}\left(v\right)\right)$ whose density is $\psi$.

We will take $\psi$ to be the heat kernel or its derivatives. The heat kernel on $\mathbb{R}$ is defined for $t>0$, $x\in \mathbb{R}$ as
\[
h_t\left(x\right)\coloneqq \frac{1}{2\sqrt{\pi t}}e^{-\frac{x^2}{4t}},\quad x\in \mathbb{R},
\]
and has derivatives
\[
\dot{h}_t\left(x\right)=\partial_t h_t\left(x\right)=\frac{x^2-2t}{8\sqrt{\pi}t^{5/2}}e^{-\frac{x^2}{4t}},\quad \partial_x h_t\left(x\right)=-\frac{x}{4\sqrt{\pi }t^{3/2}}e^{-\frac{x^2}{4t}}.
\]
Of course, the heat kernel has unit mass:
\begin{equation}\label{eq:heat-unitmass}
\int_{-\infty}^\infty h_t\left(x\right)dx=1,
\end{equation}
and for a universal constant $C$,
\begin{equation}\label{hintegrals}
\int_0^\infty x^{1/\rho}h_t\left(x\right)dx=\Gamma\left(\frac{1}{2}+\frac{1}{2\rho}\right)\frac{\left(4t\right)^{1/{2\rho}}}{2\sqrt{\pi}},\quad \left\|\dot{h}_t\right\|_{L^1\left(\mathbb{R}\right)}=\frac{C}{t},\quad \left\|\partial_x h_t\right\|_{L^1\left(\mathbb{R}\right)}=\frac{1}{\sqrt{\pi t}}\qquad t>0.
\end{equation}

We now begin the proof of Theorem \ref{thm:cvx}.

\subsection{Step 0: Reduction to compactly supported continuously differentiable functions}\label{subsec:step0}
We begin by a reduction. It is enough to prove the following variant of \eqref{eq:continuous main p<q}: if $f$ is compactly supported and continuously differentiable, then
\begin{align}\label{eq:continuous main-C1}
\begin{aligned}
&\left(\int_0^\infty \left(\int_{G}\left(\frac{\left\|f\left(h\exp\left(tv\right)\right)-f\left(h\right)\right\|_X}{t^{1/\rho}}\right)^pd\mu\left(h\right)\right)^{q/p}\frac{dt}{t}\right)^{1/q}\\
&\lesssim \left(\rho+\frac{1}{\rho-1}\right)\max\left\{\left(p-1\right)^{(1/q)-1},K_q\left(X\right)\right\} \left\|\nabla f\right\|_{L^p\left(G;\ell_2^k\left(X\right)\right)}.
\end{aligned}
\end{align}
Indeed, assume \eqref{eq:continuous main-C1} holds for compactly supported continuously differentiable functions $f:G\to X$ with $f\in L^p\left(G;X\right)$. We prove \eqref{eq:continuous main p<q} for $f\in Ch_0^{1,p}\left(G;X\right)$. By definition, we may find a sequence $\left\{f_n\right\}_{n=1}^\infty\subset C_c^\infty\left(G;X\right)$ such that $f_n\to f$ in $L^p\left(G;X\right)$ and
\[
\left\|f\right\|_{\dot{Ch}^{1,p}_0\left(G;X\right)}=\lim_{n\to \infty}\left\|\nabla f_n\right\|_{L^p\left(G;\ell_2^k\left(X\right)\right)}.
\]
By assumption, \eqref{eq:continuous main-C1} is valid for $f_n$:
\begin{align*}
&\left(\int_0^\infty \left(\int_{G}\left(\frac{\left\|f_n\left(h\exp\left(tv\right)\right)-f_n\left(h\right)\right\|_X}{t^{1/\rho}}\right)^pd\mu\left(h\right)\right)^{q/p}\frac{dt}{t}\right)^{1/q}\\
&\lesssim \left(\rho+\frac{1}{\rho-1}\right)\max\left\{\left(p-1\right)^{(1/q)-1},K_q\left(X\right)\right\} \left\|\nabla f_n\right\|_{L^p\left(G;\ell_2^k\left(X\right)\right)}.
\end{align*}
On the other hand, by the right-invariance of the Haar measure $\mu$ and the triangle inequality, we have for each $0<T_1<T_2$
\begin{align*}
&\left(\int_{T_1}^{T_2} \left(\int_{G}\left(\frac{\left\|f\left(h\exp\left(tv\right)\right)-f\left(h\right)\right\|_X}{t^{1/\rho}}\right)^pd\mu\left(h\right)\right)^{q/p}\frac{dt}{t}\right)^{1/q}\\
&\le \left(\int_{T_1}^{T_2} \left(\int_{G}\left(\frac{\left\|f_n\left(h\exp\left(tv\right)\right)-f_n\left(h\right)\right\|_X}{t^{1/\rho}}\right)^pd\mu\left(h\right)\right)^{q/p}\frac{dt}{t}\right)^{1/q}+2\left(\frac \rho q\right)^{1/q}\left(T_1^{-q/\rho}-T_2^{-q/\rho}\right)^{1/q}\left\|f_n-f\right\|_{L^p\left(G\right)}
\end{align*}
and thus, since $f_n\to f$ in $L^p\left(G\right)$ as $n\to \infty$,
\begin{align*}
&\left(\int_{T_1}^{T_2} \left(\int_{G}\left(\frac{\left\|f\left(h\exp\left(tv\right)\right)-f\left(h\right)\right\|_X}{t^{1/\rho}}\right)^pd\mu\left(h\right)\right)^{q/p}\frac{dt}{t}\right)^{1/q}\\
&\le
\liminf_{n\to\infty}\left(\int_{T_1}^{T_2} \left(\int_{G}\left(\frac{\left\|f_n\left(h\exp\left(tv\right)\right)-f_n\left(h\right)\right\|_X}{t^{1/\rho}}\right)^pd\mu\left(h\right)\right)^{q/p}\frac{dt}{t}\right)^{1/q}\\
&\stackrel{\mathclap{\eqref{eq:continuous main-C1}}}{\lesssim}  \left(\rho+\frac{1}{\rho-1}\right)\max\left\{\left(p-1\right)^{(1/q)-1},K_q\left(X\right)\right\} \liminf_{n\to\infty}\left\|\nabla f_n\right\|_{L^p\left(G;X\right)}\\
&=\left(\rho+\frac{1}{\rho-1}\right)\max\left\{\left(p-1\right)^{(1/q)-1},K_q\left(X\right)\right\} \left\|f\right\|_{\dot{Ch}_0^{1,p}\left(G;X\right)}.
\end{align*}
Taking $T_1\to 0$ and $T_2\to\infty$, we have \eqref{eq:continuous main p<q} as desired.

So, it remains to prove \eqref{eq:continuous main-C1} when $f\in L^p\left(G;X\right)$ is compactly supported and continuously differentiable.

\subsection{Step 1: Heat flow along the vertical direction}\label{subsec:step1}
We now set out to prove \eqref{eq:continuous main-C1}. The first step is to bound the vertical variations $f\left(g\exp\left(tv\right)\right)-f\left(g\right)$ by the quantities $\dot{h}_t*f$. This generalizes Lemma 2.6 of \cite{lafforgue2014vertical} to nilpotent Lie groups.

\begin{lemma}\label{lem:first step}
Fix $p,q\in \left[1,\infty\right)$ with $p\le q$ and let $\left(X,\left\|\cdot\right\|_X\right)$ be a Banach space. Every $f\in L^p\left(G;X\right)$ satisfies
\begin{equation}\label{eq:g function besov}
\left(\int_0^\infty \left(\int_{G}\left\|f\left(g\exp\left(tv\right)\right)-f\left(g\right)\right\|_X^pd\mu\left(g\right)\right)^{q/p}\frac{dt}{t^{1+q/\rho}}\right)^{1/q}\lesssim \left(\rho+\frac{1}{\rho-1}\right)
\left(\int_0^\infty t^{q\frac{2\rho-1}{2\rho}-1}\left\|\dot{h}_t*f\right\|_{L^p\left(G;X\right)}^qdt\right)^{1/q}.
\end{equation}
\end{lemma}

This is the part of the argument that uses that $\rho>1$. In the nilpotent setting, we put $\rho=s$, so this means that our group $G$ has nilpotency step $s>1$ and hence is nonabelian.

\begin{proof}[Proof of Lemma \ref{lem:first step}]
For every $g\in G$ and $t>0$ we have
\begin{align*}
f\left(g\exp\left(tv\right)\right)-f\left(g\right)=&\left[f\left(g\exp\left(tv\right)\right)-h_{t^2}*f\left(g\exp\left(tv\right)\right)\right]\\
&+\left[h_{t^2}*f\left(g\exp\left(tv\right)\right)-h_{t^2}*f\left(g\right)\right]+\left[h_{t^2}*f\left(g\right)-f\left(g\right)\right],
\end{align*}
so by the triangle inequality,
\begin{align}\label{eq:three parts}
\begin{aligned}
&\left(\int_0^\infty \left(\int_{G}\left\|f\left(g\exp\left(tv\right)\right)-f\left(g\right)\right\|_X^pd\mu\left(g\right)\right)^{q/p}\frac{dt}{t^{1+q/\rho}}\right)^{1/q}\\
&\le  \left(\int_0^\infty \left(\int_{G}\left\|f\left(g\exp\left(tv\right)\right)-h_{t^2}*f\left(g\exp\left(tv\right)\right)\right\|_X^pd\mu\left(g\right)\right)^{q/p}\frac{dt}{t^{1+q/\rho}}\right)^{1/q}\\
&\quad+\left(\int_0^\infty \left(\int_{G}\left\|h_{t^2}*f\left(g\exp\left(tv\right)\right)-h_{t^2}*f\left(g\right)\right\|_X^pd\mu\left(g\right)\right)^{q/p}\frac{dt}{t^{1+q/\rho}}\right)^{1/q}\\
&\quad +\left(\int_0^\infty \left(\int_{G}\left\|h_{t^2}*f\left(g\right)-f\left(g\right)\right\|_X^pd\mu\left(g\right)\right)^{q/p}\frac{dt}{t^{1+q/\rho}}\right)^{1/q}.
\end{aligned}
\end{align}
We bound each term in the right-hand side of \eqref{eq:three parts} by the right-hand side of \eqref{eq:g function besov}. For the first and third terms of the right-hand side of \eqref{eq:three parts}, we have
\begin{align*}
&\left(\int_0^\infty \left(\int_{G}\left\|f\left(g\exp\left(tv\right)\right)-h_{t^2}*f\left(g\exp\left(tv\right)\right)\right\|_X^pd\mu\left(g\right)\right)^{q/p}\frac{dt}{t^{1+q/\rho}}\right)^{1/q}\\
&=\left(\int_0^\infty \left( \int_{G}\left\|f\left(g\right)-h_{t^2}*f\left(g\right)\right\|_X^pd\mu\left(g\right)\right)^{q/p}\frac{dt}{t^{1+q/\rho}}\right)^{1/q}\\
&= \left(\int_0^\infty \left(\int_{G}\left\|\int_0^{t^2} \dot{h}_\tau *f\left(g\right)d\tau\right\|_{X}^pd\mu\left(g\right)\right)^{q/p} \frac{dt}{t^{1+q/\rho}}\right)^{1/q}\\
&\le \left(\int_0^\infty \left(\int_0^{t^2} \left\|\dot{h}_\tau*f\right\|_{L^p\left(G;X\right)} d\tau\right)^q\frac{dt}{t^{1+q/\rho}}\right)^{1/q}
\\
&\le 2^{1-1/q}\rho\left(\int_0^\infty t^{q\frac{2\rho-1}{2\rho}-1}\left\|\dot{h}_t*f\right\|_{L^p\left(G;X\right)}^qdt\right)^{1/q},
\end{align*}
where in the first equality we used the right-invariance of the Haar measure on $G$, in the second equality we used the definition of $\dot{h}_\tau$, in the first inequality we used the triangle inequality in $L^p\left(G;X\right)$, and in the second inequality we used Hardy's inequality \cite[Theorem 330]{hardy1952inequalities}
\begin{equation}\label{Hardy}
\left(\int_0^\infty \left[x^{-\nu}\int_0^x h\left(t\right)dt\right]^q dx\right)^{1/q}\le \frac{1}{\nu-1/q}\left(\int_0^\infty x^{\left(1-\nu\right)q}h\left(x\right)^q dx\right)^{1/q},\quad \nu> \frac 1q,~h:\left[0,\infty\right)\to \left[0,\infty\right),
\end{equation}
with $\nu=\frac 1q+\frac{1}{2\rho}$.

We now bound the second term of the right-hand side of \eqref{eq:three parts}. By the semigroup property we have $h_t=h_{t/2}*h_{t/2}$ for $t>0$, hence $\dot{h}_t=h_{t/2}*\dot{h}_{t/2}$ and so
\[
\partial_t\partial_xh_t=\partial_x\dot{h}_t=\partial_xh_{t/2}*\dot{h}_{t/2}.
\]
Since by Young's inequality
\begin{equation}\label{Rtf0}
\left\|\partial_xh_t*f\right\|_{L^p\left(G;X\right)}\le \left\|\partial_xh_t\right\|_{L^1\left(\mathbb{R}\right)}\left\|f\right\|_{L^p\left(G;X\right)}\quad\stackrel{\mathclap{\eqref{hintegrals}\mathrm{-(iii)}}}{=}\quad\frac{1}{\sqrt{\pi t}}\left\|f\right\|_{L^p\left(G;X\right)},
\end{equation}
we have $\lim_{t\to \infty} \partial_xh_t*f=0$ in $L^p\left(G;X\right)$, and thus
\[
\partial_x h_{t^2}*f=-\int_{t^2}^\infty \partial_\tau\left(\partial_x h_\tau*f\right)d\tau=-\int_{t^2}^\infty \partial_xh_{\tau/2}*\dot{h}_{\tau/2}*fd\tau,
\]
and we may write
\[
h_{t^2}*f\left(g\exp\left(tv\right)\right)-h_{t^2}*f\left(g\right)=\int_0^t \partial_xh_{t^2}*f\left(g\exp\left(uv\right)\right)du=-\int_0^t \int_{t^2}^\infty \partial_xh_{\tau/2}*\dot{h}_{\tau/2}*f\left(g\exp\left(uv\right)\right)d\tau du.
\]
Observe that
\begin{align*}
&\left(\int_0^\infty \left(\int_{G}\left\|h_{t^2}*f\left(g\exp\left(tv\right)\right)-h_{t^2}*f\left(g\right)\right\|_X^pd\mu\left(g\right)\right)^{q/p}\frac{dt}{t^{1+q/\rho}}\right)^{1/q}\\
&=\left(\int_0^\infty\left(\int_{G}\left\|\int_0^t \int_{t^2}^\infty \partial_xh_{\tau/2}*\dot{h}_{\tau/2}*f\left(g\exp\left(uv\right)\right)d\tau du\right\|_X^pd\mu\left(g\right)\right)^{q/p}\frac{dt}{t^{1+q/\rho}}\right)^{1/q}\\
&\le \left(\int_0^\infty\left(\int_0^t \int_{t^2}^\infty \left(\int_{G}\left\|\partial_xh_{\tau/2}*\dot{h}_{\tau/2}*f\left(g\exp\left(uv\right)\right)\right\|_X^pd\mu\left(g\right)\right)^{1/p}d\tau du\right)^q \frac{dt}{t^{1+q/\rho}}\right)^{1/q}\\
&= \left(\int_0^\infty\left(t\int_{t^2}^\infty \left\|\partial_xh_{\tau/2}*\dot{h}_{\tau/2}*f\right\|_{L^p\left(G;X\right)}d\tau\right)^q \frac{dt}{t^{1+q/\rho}}\right)^{1/q}\\
&= \left(\int_0^\infty t^{q\frac{\rho-1}{\rho}-1} \left(\int_{t^2}^\infty \left\|\partial_xh_{\tau/2}*\dot{h}_{\tau/2}*f\right\|_{L^p\left(G;X\right)}d\tau\right)^qdt\right)^{1/q}\\
& \le \frac{2^{1-1/q}\rho}{\rho-1} \left(\int_0^\infty t^{q\frac{3\rho-1}{2\rho}-1} \left\|\partial_xh_{t/2}*\dot{h}_{t/2}*f\right\|_{L^p\left(G;X\right)}^qdt\right)^{1/q},\\
\end{align*}
where the first inequality uses the triangle inequality in $L^p\left(G;X\right)$, the second equality uses the right-invariance of the Haar measure $\mu$, and the second inequality uses the second form of Hardy's inequality \cite[Theorem 330]{hardy1952inequalities}
\begin{equation}\label{secondHardy}
\left(\int_0^\infty \left[t^\nu \int_t^\infty h\left(u\right)du\right]^qdt\right)^{1/q}\le \frac{1}{\nu+1/q}\left(\int_0^\infty x^{q\left(1+\nu\right)}h\left(x\right)^qdx\right)^{1/q},\quad \nu>-\frac 1q,~h:\left[0,\infty\right)\to \left[0,\infty\right),
\end{equation}
with $\nu=\frac{\rho-1}{2\rho}-\frac{1}{q}$. (This is the step where we use $\rho>1$.)
By Young's inequality,
\[
\left\| \partial_xh_{t/2}*\dot{h}_{t/2}*f\right\|_{L^p\left(G;X\right)}\le \left\|\partial_xh_{t/2}\right\|_{L^1\left(\mathbb{R}\right)}\left\|\dot{h}_{t/2}*f\right\|_{L^p\left(G;X\right)}~\stackrel{\mathclap{\eqref{hintegrals}-\mathrm{(iii)}}}{=} \quad\sqrt{\frac{2}{\pi t}} \left\|\dot{h}_{t/2}*f\right\|_{L^p\left(G;X\right)}.
\]
Therefore
\begin{align*}
&\left(\int_0^\infty\left( \int_{G}\left\|h_{t^2}*f\left(g\exp\left(tv\right)\right)-h_{t^2}*f\left(g\right)\right\|_X^pd\mu\left(g\right)\right)^{q/p}\frac{dt}{t^{1+q/\rho}}\right)^{1/q}\\
&\le \frac{2^{3/2-1/q+\left(2\rho-1\right)/2\rho}\rho}{\sqrt{\pi}\left(\rho-1\right)}\left( \int_0^\infty  t^{q\frac{2\rho-1}{2\rho}-1}  \left\|\dot{h}_{t}*f\right\|_{L^p\left(G;X\right)}^q dt\right)^{1/q}.
\end{align*}
This completes the proof.
\end{proof}

\subsection{Step 2: Differentiation along a pencil of curves}\label{subsec:step2} We will next bound $\dot{h}_t*f$ using $\dot{h}_t*\nabla f$. We first prove the following lemma. This is a generalization of Lemma 2.4 of \cite{lafforgue2014vertical}.

\begin{lemma}\label{lem:increment}
Suppose that $p\in \left[1,\infty\right)$ and $t\in \left(0,\infty\right)$. Then for every Banach space $\left(X,\left\|\cdot\right\|_X\right)$ and every continuously differentiable $f:G\to X$ we have
\begin{equation}\label{eq:increment}
\left\|\dot{h}_t*f-\dot{h}_{2t}*f \right\|_{L^p\left(G;X\right)}\lesssim t^{1/2\rho}\left\|\dot{h}_t*\nabla f\right\|_{L^p\left(G;\ell_2^k\left(X\right)\right)}.
\end{equation}
\end{lemma}
\begin{proof}
Recalling the semigroup property $\dot{h}_{2t}=h_t*\dot{h}_t$, we have
\begin{equation}\label{eq:use semigroup of root}
\dot{h}_t*f\left(g\right)-\dot{h}_{2t}*f\left(g\right)=\dot{h}_t*f\left(g\right)-h_t*\dot{h}_t*f\left(g\right)~\stackrel{\mathclap{\eqref{eq:heat-unitmass}}}{=}\int_\mathbb{R} h_t\left(u\right)\left(\dot{h}_t*f\left(g\right)-\dot{h}_t*f\left(g\exp\left(-uv\right)\right)\right)du.
\end{equation}

For $u\in \left[0,\infty\right)$, let $\gamma_u:\left[0,d_G\left(\exp\left(uv\right),e_G\right)\right]\to G$ be a measurable family of geodesics parametrized by arclength joining $e_G$ to $\exp\left(uv\right)$. (Such a measurable family exists by the Aumann measurable selection theorem \cite[Theorem 6.9.13]{bogachev2007measure}). For every $u\in \left[0,\infty\right)$ and $g\in G$,
\begin{align*}
\dot{h}_t*f\left(g\right)-\dot{h}_t*f\left(g\exp\left(-uv\right)\right)&=\dot{h}_t*f\left(g\exp\left(-uv\right)\gamma_u\left(d_G\left(\exp\left(uv\right),e_G\right)\right)\right)-\dot{h}_t*f\left(g\exp\left(-uv\right)\right)\\
&=\int_0^{d_G\left(\exp\left(uv\right),e_G\right)} \frac{d}{d\theta}\dot{h}_t*f\left(g\exp\left(-uv\right)\gamma_u\left(\theta\right)\right) d\theta.
\end{align*}
Because $\gamma_u$ is horizontal and parametrized by arclength, because convolution with $\dot{h_t}$ commutes with $\nabla$ (this is because $v$ is in the center of $G$), and because of the Cauchy--Schwarz inequality, we have\footnote{More precisely, borrowing notation from Appendix \ref{app:lipnabla}, we first bound by $\left\|\dot{h}_t*\nabla f\left(g\exp\left(-uv\right)\gamma_u\left(\theta\right)\right)\right\|_{B\left(\mathbb{R}^k;X\right)}$ and then apply \eqref{eq:elltwo-operator}.} for every $\theta\in \left[0,d_G\left(\exp\left(uv\right),e_G\right)\right]$
\begin{equation}\label{eq:nablaConvolutionCommute}
 \left\|\frac{d}{d\theta}\dot{h}_t*f\left(g\exp\left(-uv\right)\gamma_u\left(\theta\right)\right)\right\|_X\le \left\|\dot{h}_t*\nabla f\left(g\exp\left(-uv\right)\gamma_u\left(\theta\right)\right)\right\|_{\ell_2^k\left(X\right)}.
\end{equation}
We thus obtain
\begin{align*}
&\left(\int_{G}\left\|\int_0^\infty h_t\left(u\right)\left(\dot{h}_t*f\left(g\right)-\dot{h}_t*f\left(g\exp\left(-uv\right)\right)\right)du\right\|_X^pd\mu\left(g\right)\right)^{1/p}\\
&=\left(\int_{G}\left\|\int_0^\infty \int_0^{d_G\left(\exp\left(uv\right),e_G\right)} h_t\left(u\right) \frac{d}{d\theta}\left(\dot{h}_t*f\left(g\exp\left(-uv\right)\gamma_u\left(\theta\right)\right)\right) d\theta du\right\|_X^pd\mu\left(g\right)\right)^{1/p}\\
&\le \int_0^\infty \int_0^{d_G\left(\exp\left(uv\right),e_G\right)} h_t\left(u\right) \left(\int_{G}\left\| \frac{d}{d\theta}\left(\dot{h}_t*f\left(g\exp\left(-uv\right)\gamma_u\left(\theta\right)\right)\right) \right\|_X^pd\mu\left(g\right)\right)^{1/p}d\theta du\\
&\stackrel{\mathclap{\eqref{eq:nablaConvolutionCommute}}}{\le}  \int_0^\infty \int_0^{d_G\left(\exp\left(uv\right),e_G\right)} h_t\left(u\right)\left(\int_{G}\left\|\dot{h}_t*\nabla f\left(g\exp\left(-uv\right)\gamma_u\left(\theta\right)\right)\right\|_{\ell_2^k\left(X\right)}^pd\mu\left(g\right)\right)^{1/p}d\theta du\\ 
&=  \int_0^\infty \int_0^{d_G\left(\exp\left(uv\right),e_G\right)} h_t\left(u\right)d\theta du \left\| \dot{h}_t*\nabla f \right\|_{L^p\left(G;\ell_2^k\left(X\right)\right)}\\ 
&\stackrel{\mathclap{\eqref{eq:rho-ineq}}}{\le}~  \left(\int_0^\infty u^{1/\rho}h_t\left(u\right)du\right)\left\|\dot{h}_t*\nabla f\right\|_{L^p\left(G;\ell_2^k\left(X\right)\right)}\\
&\stackrel{\mathclap{\eqref{hintegrals}-\mathrm{(i)}}}{\lesssim}\quad  t^{1/2\rho} \left\|\dot{h}_t*\nabla f\right\|_{L^p\left(G;\ell_2^k\left(X\right)\right)},
\end{align*}
where the first inequality uses the triangle inequality in $L^p\left(G;X\right)$ and in the second equality we used the right-invariance of the Haar measure on $G$.

Similarly, for $u\in \left(-\infty,0\right]$,
\begin{equation*}
\dot{h}_t*f\left(g\right)-\dot{h}_t*f\left(g\exp\left(-uv\right)\right)=
-\int_0^{d_G\left(\exp\left(-uv\right),e_G\right)} \frac{d}{d\theta}\dot{h}_t*f\left(g\gamma_{-u}\left(\theta\right)\right) d\theta,
\end{equation*}
and by the same reasoning,
\[
\left(\int_{G}\left\|\int_{-\infty}^0 h_t\left(u\right)\left(\dot{h}_t*f\left(g\right)-\dot{h}_t*f\left(g\exp\left(-uv\right)\right)\right)du\right\|_X^pd\mu\left(g\right)\right)^{1/p}\lesssim  t^{1/2\rho} \left\|\dot{h}_t*\nabla f\right\|_{L^p\left(G;\ell_2^k\left(X\right)\right)}.
\]
These two estimates along with \eqref{eq:use semigroup of root} gives the stated inequality.
\end{proof}

Now we bound $\dot{h}_t*f$ using $\dot{h}_t*\nabla f$, as promised. This step is inspired by Lemma 2.5 of \cite{lafforgue2014vertical}.

\begin{lemma}\label{lem:integrated increment}
Fix $p,q\in \left[1,\infty\right)$. For every Banach space $\left(X,\left\|\cdot\right\|_X\right)$, every continuously differentiable $f:G\to X$ with $f\in L^p\left(G;X\right)$ satisfies
\begin{equation*}
\left(\int_0^\infty t^{q\frac{2\rho-1}{2\rho}-1}\left\|\dot{h}_t*f\right\|_{L^p\left(G;X\right)}^qdt\right)^{1/q}\lesssim  \left(\int_0^\infty t^{q-1}\left\|\dot{h}_t*\nabla f\right\|_{L^p\left(G;\ell_2^k\left(X\right)\right)}^qdt\right)^{1/q}.
\end{equation*}
\end{lemma}
\begin{proof}
By Young's inequality,
\[
\left\|\dot{h}_t*f\right\|_{L^p\left(G;X\right)}\le \left\|\dot{h}_t\right\|_{L^1\left(\mathbb{R}\right)} \left\|f\right\|_{L^p\left(G;X\right)}\stackrel{\mathclap{\eqref{hintegrals}-\mathrm{(ii)}}}{\asymp}\quad \frac{1}{ t}\left\|f\right\|_{L^p\left(G;X\right)},
\]
so $\lim_{t\to\infty} \dot{h}_t*f=0$ in $L^p\left(G;X\right)$. Therefore
\[
\dot{h}_t*f=\sum_{m=1}^\infty \left(\dot{h}_{2^{m-1}t}*f-\dot{h}_{2^mt}*f\right)  \mathrm{~in~} L^p\left(G;X\right),
\]
from which it follows that
\begin{align*}
&\left(\int_0^\infty t^{q\frac{2\rho-1}{2\rho}-1}\left\|\dot{h}_t*f\right\|_{L^p\left(G;X\right)}^qdt\right)^{1/q}\\
 &\le \sum_{m=1}^\infty \left(\int_0^\infty t^{q\frac{2\rho-1}{2\rho}-1}\left\|\dot{h}_{2^{m-1}t}*f-\dot{h}_{2^mt}*f\right\|_{L^p\left(G;X\right)}^qdt\right)^{1/q}\\
&\lesssim  \sum_{m=1}^\infty \left(\int_0^\infty t^{q\frac{2\rho-1}{2\rho}-1}\left(2^{m-1}t\right)^{q/2\rho}\left\|\dot{h}_{2^{m-1}t}*\nabla f\right\|_{L^p\left(G;\ell_2^k\left(X\right)\right)}^qdt\right)^{1/q}\\
&= \left(\sum_{m=1}^\infty \frac{1}{2^{\left(m-1\right)\left(2\rho-1\right)/2\rho}}\right)\left(\int_0^\infty t^{q-1}\left\|\dot{h}_t*\nabla f\right\|_{L^p\left(G;\ell_2^k\left(X\right)\right)}^qdt\right)^{1/q},
\end{align*}
where the second inequality uses Lemma \ref{lem:increment}.
\end{proof}

Combining Lemmas \ref{lem:first step}, \ref{lem:increment}, and \ref{lem:integrated increment}, we have, for compactly supported and continuously differentiable $f:G\to X$,
\begin{align*}
	&\left(\int_0^\infty \left(\int_{G}\left\|f\left(g\exp\left(tv\right)\right)-f\left(g\right)\right\|_X^pd\mu\left(g\right)\right)^{q/p}\frac{dt}{t^{1+q/\rho}}\right)^{1/q}\\
	&\lesssim \left(\rho+\frac{1}{\rho-1}\right)
	 \left(\int_0^\infty t^{q-1}\left\|\dot{h}_t*\nabla f\right\|_{L^p\left(G;\ell_2^k\left(X\right)\right)}^qdt\right)^{1/q}.
\end{align*}
Compared to Lemma \ref{lem:deterministic-path}, we have taken a pencil of curves from $g$ to $g \exp\left(tv\right)$ to obtain this inequality.

So far we haven't used the uniform convexity of $X$. This will come into play in the next step.

\subsection{Step 3: Littlewood--Paley--Stein theory and martingales in Banach spaces}\label{subsec:step3}
We begin by calculating some uniform convexity constants.
\begin{lemma}\label{lem:lp-unif-cvx}
	Let $q\in [2,\infty)$ and $p\in (1,q]$. For any measure space $\left(N,\sigma\right)$ and a $q$-uniformly convex Banach space $Y$, we have 
	\[
		K_q\left(L^p\left(N,\sigma;Y\right)\right)\le  \max\left\{\left(p-1\right)^{(1/q)-1},K_q\left(Y\right)\right\}
	\]
		In particular, we have for any $k\in \mathbb{N}$,
	\[
	K_q\left(\ell_2^k\left(Y\right)\right)\lesssim K_q\left(Y\right)
	\]
	and
	\begin{equation}\label{eq:lpl2-unifcvx}
	K_q\left(L^p\left(N,\sigma;\ell_2^k\left(Y\right)\right)\right)\lesssim \max\left\{\left(p-1\right)^{(1/q)-1},K_q\left(\ell_2^k\left(Y\right)\right)\right\}\lesssim \max\left\{\left(p-1\right)^{(1/q)-1},K_q\left(Y\right)\right\}.
	\end{equation}
\end{lemma}
\begin{proof}
	For $r\in \left(1,2\right]$, the $r$-uniform smoothness constant of a Banach space $Y$ is defined by
	\[
	S_r\left(Y\right)\coloneqq \inf\left\{S>0:\forall x,y\in Y~\left(\left\|x\right\|_Y^r+S^r\left\|y\right\|_Y^r\right)^{1/r}\ge \left(\frac{\left\|x+y\right\|_Y^r+\left\|x-y\right\|_Y^r}{2}\right)^{1/r}\right\}.
	\]
	By \cite[Lemma 5]{ball1994sharp}, $S_r(Y)=K_{r/(r-1)}(Y^*)$. Since uniformly convex Banach spaces are reflexive by the Milman--Pettis theorem \cite{milman1938some,pettis1939proof}, this also means that $S_r(Y^*)=K_{r/(r-1)}(Y)$. Also, by \cite[equation (4.4)]{naor2014comparison}, $S_r(L^s(N,\sigma;Y))\lesssim \max\left\{s^{1/r},S_r(Y)\right\}$ for $r\in [1,2]$ and $s\in [r,\infty)$.
	
	By these facts we have, for any measure space $\left(N,\sigma\right)$, Banach space $Y$, and $p\le q$,
	\begin{align*}
		K_q\left(L^p\left(N,\sigma;Y\right)\right)&=S_{q/\left(q-1\right)}\left(L^{p/\left(p-1\right)}\left(N,\sigma;Y^*\right)\right)\\
		&\lesssim \max\left\{\left(\frac{p}{p-1}\right)^{1-(1/q)},S_{q/\left(q-1\right)}\left(Y^*\right)\right\}= \max\left\{\left(\frac{p}{p-1}\right)^{1-(1/q)},K_q\left(Y\right)\right\}.
	\end{align*}
If $2\le p\le q$, then
\[
\left(\frac{p}{p-1}\right)^{1-(1/q)}\le 2^{1-(1/q)}\le 2,
\]
and if $1<p<2$, then
\[
\left(\frac{p}{p-1}\right)^{1-(1/q)}\le \left(\frac{2}{p-1}\right)^{1-(1/q)}\le 2\left(p-1\right)^{(1/q)-1},
\]
so coupled with $K_q(Y)\ge 1$, we obtain that
\[
\max\left\{\left(\frac{p}{p-1}\right)^{1-(1/q)},K_q\left(Y\right)\right\}\lesssim \max\left\{\left(p-1\right)^{(1/q)-1},K_q\left(Y\right)\right\}.
\]

\end{proof}

We now claim the following inequality:
\begin{lemma}\label{lem:gHL}
	For $q\in [2,\infty)$, $p\in (1,q]$, a 
	$q$-uniformly convex Banach space $\mathbb{B}$, and any $\phi\in L^p\left(\mathbb{R};\mathbb{B}\right)$, we have
	\begin{equation}\label{eq:MTX-1D}
		\left(\int_0^\infty t^{q-1}\left\|\dot{h}_t *\phi\right\|_{L^p\left(\mathbb{R};\mathbb{B}\right)}^qdt\right)^{1/q}
		\lesssim \max\left\{\left(p-1\right)^{(1/q)-1},K_q\left(\mathbb{B}\right)\right\} \left\|\phi\right\|_{L^p\left(\mathbb{R};\mathbb{B}\right)}.
	\end{equation}
\end{lemma}

Assuming Lemma \ref{lem:gHL}, we are ready to finish the proof of Theorem \ref{thm:cvx}.

\begin{proof}[Proof of Theorem \ref{thm:cvx} assuming Lemma \ref{lem:gHL}]
Recall that $S$ is a subset of $G$ with a measure $\nu$ such that the push-forward of the product measure of $\nu$ and the Lebesgue measure of $\mathbb{R}$ under the map $S\times \mathbb{R}\to G$, $\left(s,t\right)\mapsto s\exp\left(tv\right)$ is the Haar measure $\mu$ of $G$. Take $\mathbb{B}=L^p\left(S,\nu;\ell_2^k\left(X\right)\right)$ and  $\phi=\left(\nabla f\right)^v:\mathbb{R}\to \mathbb{B}$. Then, assuming \eqref{eq:MTX-1D},
\begin{align*}
	\left(\int_0^\infty t^{q-1}\left\|\dot{h}_t*\nabla f\right\|_{L^p\left(G;\ell_2^k\left(X\right)\right)}^qdt\right)^{1/q}
	&\stackrel{\mathclap{\eqref{eq:Lebesgue-Bochner}}}{=}~\left(\int_0^\infty t^{q-1}\left\|\dot{h}_t *\phi\right\|_{L^p\left(\mathbb{R};\mathbb{B}\right)}^qdt\right)^{1/q}\\
	&\stackrel{\mathclap{\eqref{eq:MTX-1D}}}{\lesssim} \max\left\{\left(p-1\right)^{(1/q)-1},K_q\left(\mathbb{B}\right)\right\} \left\|\phi\right\|_{L^p\left(\mathbb{R};\mathbb{B}\right)}\\
	&\stackrel{\mathclap{\eqref{eq:Lebesgue-Bochner},\eqref{eq:lpl2-unifcvx}}}{\lesssim} \max\left\{\left(p-1\right)^{(1/q)-1},K_q\left(X\right)\right\}\left\|\nabla f\right\|_{L^p\left(G;\ell_2^k\left(X\right)\right)},
\end{align*}
which, in conjunction with Lemmas \ref{lem:first step} and \ref{lem:integrated increment}, completes the proof of \eqref{eq:continuous main-C1} and thereby also the proof of Theorem \ref{thm:cvx}.
\end{proof}

Lemma \ref{lem:gHL} is inspired by a Littlewood--Paley--Stein inequality for the heat semigroup given in \cite[Theorem 17]{hytonen2019heat}, which is a quantitative version of\footnote{More precisely, \cite{martinez2006vector} proves the inequality for `subordinated' semigroups, such as the Poisson semigroup subordinated by the heat semigroup, and \cite{xu2020vector} proves the inequality for general symmetric diffusion semigroups. The work \cite{hytonen2019heat} has the advantage that it obtains, for the heat semigroup, the explicit constant $K_p\left(\mathbb{B}\right)$. We could have also started with the Poisson semigroup $P_t$, which gives the same constant in the inequality due to semigroup subordination. See \cite[Appendix A]{naor2022foliated} for a detailed account of this discussion.} \cite[Theorem 2.1]{martinez2006vector} and \cite[Theorem 2]{xu2020vector}. The inequality of \cite[Theorem 17]{hytonen2019heat} is stated as follows.

\begin{theorem}[{\cite[Theorem 17]{hytonen2019heat}}] Fix $q\in [2,\infty)$ and $n\in \mathbb{N}$. Suppose that $(\mathbb{B},\|\cdot\|_\mathbb{B})$ is a $q$-uniformly convex Banach space. Then for every $f\in L^q(\mathbb{R}^n;\mathbb{B})$ we have\footnote{In \cite{hytonen2019heat} the inequality obtained actually has the martingale cotype $q$ constant $m_q(\mathbb{B})$ instead of $K_q(\mathbb{B})$. We have $m_q(\mathbb{B})\le K_q(\mathbb{B})$ by \cite{pisier1975martingales}.}
	\begin{equation*}
		\left(\int_0^\infty\left\|t\partial_t H_t f\right\|_{L^q\left(\R^n;\mathbb{B}\right)}^q\frac{d t}{t}\right)^{1/q}
		\lesssim\sqrt{n}\cdot K_q(\mathbb{B})\Norm{f}{L^q(\R^n;\mathbb{B})}.
	\end{equation*}
	where $\{H_t\}_{t\ge 0}$ denotes the $\mathbb{B}$-valued heat semigroup on $\mathbb{R}^n$.
\end{theorem}


Here, we prove the following generalization.
\begin{theorem}\label{thm:LPtime2} Fix $q\in [2,\infty)$, $p\in (1,q]$ and $n\in \mathbb{N}$. Suppose that $(\mathbb{B},\|\cdot\|_\mathbb{B})$ is a $q$-uniformly convex Banach space. Then for every $f\in L^p(\mathbb{R}^n;\mathbb{B})$ we have
	\begin{equation}\label{eq:temporal heat intro}
		\left(\int_0^\infty\left\|t\partial_t H_t f\right\|_{L^p\left(\R^n;\mathbb{B}\right)}^q\frac{d t}{t}\right)^{1/q}
		\lesssim\sqrt{n}\cdot\max\left\{\left(p-1\right)^{(1/q)-1},K_q\left(\mathbb{B}\right)\right\}\Norm{f}{L^p(\R^n;\mathbb{B})}.
	\end{equation}
\end{theorem}
It is clear that Theorem \ref{thm:LPtime2} implies Lemma \ref{lem:gHL}. We now prove Theorem \ref{thm:LPtime2}.

A sequence of $\mathbb{B}$-valued random variables $\{M_k\}_{k=1}^\infty$ on a $\sigma$-finite measure space $(N,\mathscr{F},\sigma)$, with $M_k\in L^1_{\mathrm{loc}}(\sigma;\mathbb{B})$ for every $k\in \mathbb{N}$, is a \emph{martingale} if there exists an increasing sequence of $\sigma$-finite sub-$\sigma$-algebras $\mathscr{F}_1\subset \mathscr{F}_2\subset\ldots\subset \mathscr{F}$ such that $\mathbb{E}[M_{k+1}|\mathscr{F}_k]=M_k$ for every $k\in \mathbb{N}$. Here $\mathbb{E}[\,\cdot\, |\mathscr{F}_k]$ means for the conditional expectation relative to the $\sigma$-algebra $\mathscr{F}_k$.

The following is a modification of Pisier's martingale inequality \cite[Theorem 1.3]{pisier1975martingales}.
\begin{theorem}\label{thm:pisier inequality} Let $q\in [2,\infty)$ and $p\in (1,q]$. Suppose that $(\mathbb{B},\|\cdot\|_\mathbb{B})$ is a $q$-uniformly convex Banach space. Then every martingale $\{M_k\}_{k=1}^\infty\subset L^p(\sigma;\mathbb{B})$  satisfies
	\begin{equation}\label{eq:pisier's ineq}
		\left(\sum_{k=1}^\infty \|M_{k+1}-M_k\|_{L^p(\sigma;\mathbb{B})}^q\right)^{1/q}\lesssim \max\left\{\left(p-1\right)^{(1/q)-1},K_q\left(\mathbb{B}\right)\right\}\sup_{k\in \mathbb{N}} \|M_k\|_{L^p(\sigma;\mathbb{B})}.
	\end{equation}
\end{theorem}
\begin{proof}
Following \cite[Definition 10.8]{pisier2016martingales} a sequence $\{x_i\}_{i=0}^\infty$ of elements in a Banach space $X$ is called a monotone basic sequence if for any $N\in \mathbb{N}$ and $\lambda_0,\cdots,\lambda_n\in \mathbb{R}$ we have
\[
\sup_{0\le n\le N}\left\|\sum_{k=0}^n \lambda_kx_k\right\|\le\left\|\sum_{k=0}^N \lambda_kx_k\right\|.
\]
The sequence $x_0=M_0$, $x_i=M_{i+1}-M_i$, $i\ge 1$, is a monotone basic sequence in $X=L^p\left(\sigma;\mathbb{B}\right)$.

The modulus of uniform convexity $\delta_X(\cdot)$ is defined as
\[
\delta_X(\varepsilon)\coloneqq \inf\left\{1-\left\|\frac{a+b}{2}\right\|_{X}:a,b,\in X,\|a\|\le 1, \|b\|\le 1, \|a-b\|\ge \varepsilon\right\},\quad \varepsilon\in (0,2].
\]
Setting $x=\frac{a+b}{2}$ and $y=\frac{a-b}{2}$ in the definition \eqref{eq:unif-cvx-def} of uniform convexity, we have
\[
\delta_X(\varepsilon)\ge \frac{\varepsilon^q}{q2^qK_q(X)^q}.
\]

By \cite[Theorem 10.9, Corollary 10.10]{pisier2016martingales}, we have
\[
\sum_{k=1}^N \delta_{L^p(\sigma;\mathbb{B})}\left(\left\|M_{k+1}-M_k\right\|_{L^p(\sigma;\mathbb{B})}\right)\le \left\|M_{N+1}\right\|^q_{L^p(\sigma;\mathbb{B})}
\]
so that
\[
\left(\sum_{k=1}^N \left\|M_{k+1}-M_k\right\|_{L^p(\sigma;\mathbb{B})}^q\right)^{1/q}\le K_q(L^p(\sigma;\mathbb{B}))\left\|M_{N+1}\right\|_{L^p(\sigma;\mathbb{B})}\stackrel{\mathclap{\mathrm{Lemma~}\ref{lem:lp-unif-cvx}}}{\le} \max\left\{\left(p-1\right)^{(1/q)-1},K_q\left(\mathbb{B}\right)\right\}\left\|M_{N+1}\right\|_{L^p(\sigma;\mathbb{B})},
\]
whence \eqref{eq:pisier's ineq} follows.
\end{proof}

The rest of the proof of Theorem \ref{thm:LPtime2} follows the proof of \cite[Theorem 17]{hytonen2019heat}. We outline the proof here for completeness.

A symmetric diffusion semigroup \cite[page~65]{stein2016topics} on a measure space $(N,\sigma)$ is a one-parameter family of self-adjoint linear operators $\{T_t\}_{t\in [0,\infty)}$ that map real-valued measurable functions on $(N,\sigma)$ to measurable functions on $(N,\sigma)$, such that
\begin{itemize}
	\item $T_0$ is the identity operator and $T_{t+s}=T_tT_s$ for every $s,t \in [0,\infty)$,
	\item for every $t\in [0,\infty)$ and $p\in [1,\infty]$ we have  $\|T_t\|_{L^p(\sigma)\to L^p(\sigma)}\le 1$,
	\item for every $f\in L^2(\sigma)$ we have $\lim_{t\to 0} T_tf=f$ in $L^2(\sigma)$,
	\item for every nonnegative measurable $f:\mathscr{M}\to \R$ the function $T_tf$ is also nonnegative,
	\item and $T_t\1_{\mathscr{M}}=\1_{\mathscr{M}}$.
\end{itemize}
For every Banach space $\mathbb{B}$ the above semigroup $\{T_t\}_{t\in [0,\infty)}$ extends to a semigroup of contractions on $L^p(\sigma;\mathbb{B})$ for every $p\in [1,\infty]$ \cite[page~433]{martinez2006vector}. By a standard density argument, for every $p\in [1,\infty)$ and $f\in L^p(\sigma,\mathbb{B})$ the mapping $[0,\infty)\to L^p(\sigma,\mathbb{B}),$ $t\mapsto T_tf$, is continuous.

\begin{proposition}\label{prop:Rota} Let $q\in [2,\infty)$, let $(\mathbb{B},\|\cdot\|_\mathbb{B})$ be a $q$-uniformly convex Banach space, and let $p\in (1,q]$.  Let $\{T_t\}_{t\in [0,\infty)}$ be a symmetric diffusion semigroup on a measure space $(N,\sigma)$. Then for every $f\in L^q(\sigma;\mathbb{B})$ and an increasing sequence $\{t_j\}_{j\in\Z}\subset (0,\infty)$ we have
	\begin{equation}\label{eq:semigroup version UC}
		\bigg(\sum_{j\in\Z}\Norm{T_{t_j}f-T_{t_{j+1}}f}{L^p(\sigma;\mathbb{B})}^q\bigg)^{\frac{1}{q}}\le \max\left\{\left(p-1\right)^{(1/q)-1},K_q\left(\mathbb{B}\right)\right\}\Norm{f}{L^p(\sigma;\mathbb{B})}.
	\end{equation}
\end{proposition}

The following is a generalization of \cite[Proposition 22]{hytonen2019heat}.
\begin{proof}
	It suffices to prove for finite sequences  $0<t_0<t_1<\ldots<t_N$  that
	\begin{equation}\label{eq:finite semigroup version UC}
		\bigg(\sum_{j=0}^N\Norm{T_{t_j}f-T_{t_{j+1}}f}{L^p(\sigma;\mathbb{B})}^q\bigg)^{\frac{1}{q}}\le \max\left\{\left(p-1\right)^{(1/q)-1},K_q\left(\mathbb{B}\right)\right\}\Norm{f}{L^p(\sigma;\mathbb{B})}.
	\end{equation}
	Since $t\mapsto T_t f$ is a continuous mapping $[0,\infty) \to L^p(\sigma;\mathbb{B})$, we may assume by approximation that each $t_j$ is an integer multiple of some $\delta\in (0,\infty)$, i.e., that  $t_j= k_j\delta$ with $k_j\in \mathbb{N}$. Writing $Q\coloneqq T_{\delta/2}$, inequality \eqref{eq:finite semigroup version UC} can be rewritten as follows.
	\begin{equation}\label{eq:Q version}
		\left(\sum_{j=0}^N\left\|Q^{2k_j}f-Q^{2k_{j+1}}f\right\|_{L^p(\sigma;\mathbb{B})}^q\right)^{1/q}\le \max\left\{\left(p-1\right)^{(1/q)-1},K_q\left(\mathbb{B}\right)\right\}\Norm{f}{L^p(\sigma;\mathbb{B})}.
	\end{equation}
	The operator $Q$ satisfies the assumptions of Rota's representation theorem~\cite{rota1962alternierende} and so that its even powers admit the following representation:
	\begin{equation}\label{eq:Q representation}
 Q^{2k}=J^{-1}\circ E'\circ E_k\circ J	\quad \forall k\in \mathbb{N},
	\end{equation}
	where
	\begin{itemize}
		\item  $J:L^p(\sigma;\mathbb{B})\to L^p(\mathscr{S},\mathscr{F}',\nu;\mathbb{B})\subset L^p(\mathscr{S},\mathscr{F},\nu;\mathbb{B})$ is an isometric isomorphism for some $\sigma$-finite $\sigma$-algebras $\mathscr{F}'\subset\mathscr F$ of a measure space $(\mathscr{S},\mathscr{F},\nu)$,
		\item $E_k:L^p(\mathscr{S},\mathscr{F},\nu;\mathbb{B})\to L^p(\mathscr{S},\mathscr{F}_k,\nu;\mathbb{B})\subset L^p(\mathscr{S},\mathscr{F},\nu;\mathbb{B})$ is the conditional expectation (in the $\sigma$-finite setting), where $\mathscr{F}\supset\mathscr{F}_1\supset\mathscr{F}_2\supset\cdots$ is a decreasing sequence of $\sigma$-finite sub-$\sigma$-algebras, and
		\item $E':L^p(\mathscr{S},\mathscr{F},\nu;\mathbb{B})\to L^p(\mathscr{S},\mathscr{F}',\nu;\mathbb{B})\subset L^p(\mathscr{S},\mathscr{F},\nu;\mathbb{B})$ is the conditional expectation for the sub-$\sigma$-algebra $\mathscr{F}'\subset\mathscr F$.
	\end{itemize}
	(More precisely, we are using the form of Rota's representation presented by Stein~\cite[page~106]{stein2016topics}, which extends to the vector-valued setting by the explanation following \cite[Theorem 2.5]{martinez2006vector}, i.e., that the operators under consideration extend to contractions on $\mathbb{B}$-valued $L^p$-spaces.) 
	Consequently, the desired estimate~\eqref{eq:Q version} is proven as follows.
	\begin{align*}
		\left(\sum_{j=0}^{N}\left\|(Q^{2k_j}-Q^{2k_{j+1}})f\right\|_{L^p(\sigma;\mathbb{B})}^q\right)^{1/q}
		&=\left(\sum_{j=0}^{N}\left\|J^{-1}E'(E_{k_{j}}-E_{k_{j+1}}) Jf\right\|_{L^p(\sigma;\mathbb{B})}^q \right)^{1/q}\\
		&\le \left(\sum_{j=0}^{N}\left\|(E_{k_{j}}-E_{k_{j+1}}) Jf\right\|_{L^p(\mathscr{S},\mathscr{F},\nu;\mathbb{B})}^q\right)^{1/q}\\
		&\le\max\left\{\left(p-1\right)^{(1/q)-1},K_q\left(\mathbb{B}\right)\right\}\Norm{Jf}{L^p(\mathscr{S},\mathscr{F}',\nu;Y)}\\
		&=\max\left\{\left(p-1\right)^{(1/q)-1},K_q\left(\mathbb{B}\right)\right\}\Norm{f}{L^p(\sigma;\mathbb{B})},
	\end{align*}
	where the first step uses \eqref{eq:Q representation}, the second step uses the fact that $J^{-1}$ is an isometry and that $E'$ is a contraction, the third step uses Theorem \ref{thm:pisier inequality} applied to the reverse martingale $\{E_{k_j}Jf)\}_{j=0}^N$, and the final step uses the fact that $J$ is an isometry.
\end{proof}

Using Proposition \ref{prop:Rota}, we obtain the following integral estimate, which is a generalization of \cite[Lemma 24]{hytonen2019heat}.
\begin{lemma}\label{lem:Rota} Fix $q\in [2,\infty)$, $p\in (1,q]$, $\alpha\in (1,\infty)$ and a $q$-uniformly convex Banach space $(\mathbb{B},\|\cdot\|_\mathbb{B})$. Suppose that $\{T_t\}_{t\in [0,\infty)}$ is a symmetric diffusion semigroup on a measure space $(N,\sigma)$. Then
	\begin{equation*}
	 \left(\int_0^\infty\Norm{(T_t-T_{\alpha t})f}{L^p(\R^n;\mathbb{B})}^q\frac{d t}{t}\right)^{1/q}\leq(\log\alpha)^{1/q}\max\left\{\left(p-1\right)^{(1/q)-1},K_q\left(\mathbb{B}\right)\right\}\Norm{f}{L^p(\sigma;\mathbb{B})}\quad \forall f\in L^p(\sigma;\mathbb{B}).
	\end{equation*}
\end{lemma}

\begin{proof} We have
	\begin{align*}\label{eq:sum to integral}
		\left(\int_0^\infty\Norm{(T_t-T_{\alpha t})f}{L^p(\sigma;\mathbb{B})}^q\frac{d t}{t}\right)^{1/q}
		&=\left(\sum_{j\in\Z}\int_{\alpha^j}^{\alpha^{j+1}}\Norm{(T_t-T_{\alpha t})f}{L^p(\sigma;\mathbb{B})}^q\frac{d t}{t}\right)^{1/q} \\
		&=\left(\int_{1}^{\alpha}\sum_{j\in\Z}\left\|(T_{\alpha^j t}-T_{\alpha^{j+1} t})f\right\|_{L^p(\sigma;\mathbb{B})}^q\frac{d t}{t}\right)^{1/q}\\
		&\le \left(\sup_{t\in [1,\alpha]}\sum_{j\in\Z}\left\|(T_{\alpha^j t}-T_{\alpha^{j+1} t})f\right\|_{L^p(\sigma;\mathbb{B})}^q\right)^{1/q}\left(\int_{1}^{\alpha}\frac{d t}{t}\right)^{1/q}\\
		&
		\le \max\left\{\left(p-1\right)^{(1/q)-1},K_q\left(\mathbb{B}\right)\right\}\Norm{f}{L^p(\sigma;\mathbb{B})}\left(\log\alpha\right)^{1/q},
	\end{align*}
	where the last inequality uses Proposition \ref{prop:Rota} with $t_j=\alpha^j t$.
\end{proof}

The final ingredient is the following, which we state without proof.
\begin{lemma}[{\cite[Lemma 27]{hytonen2019heat}}]\label{lem:n-heat-time-deriv}
	Suppose that $p\in [1,\infty]$. Then for every $f\in L^p(\R^n;\mathbb{B})$ and $t\in [1,\infty)$ we have
	$$
	\left\|t\dot{H}_tf\right\|_{L^p(\R^n;\mathbb{B})}\lesssim \sqrt{n}\cdot \left\|f\right\|_{L^p(\R^n;\mathbb{B})}.
	$$
\end{lemma}

\begin{proof}[Proof of Theorem \ref{thm:LPtime2}] Differentiating the semigroup identity $H_{t+s}=H_tH_s$  gives $\dot{H}_{t+s}=\dot{H}_tH_s$ for all $s,t\in (0,\infty)$. Hence
	\begin{equation*}
	\dot H_t f=\sum_{k=-1}^\infty \left(\dot H_{2^{k+1} t}-\dot H_{2^{k+2}t}\right)f
		=\sum_{k=-1}^\infty \dot H_{2^k t}\left(H_{2^k t}-H_{3\cdot 2^{k}t}\right)f,\quad t\in (0,\infty).
	\end{equation*}
Thus
	\begin{align*}
		\left(\int_0^\infty\left\|t\dot H_t f\right\|_{L^p(\R^n;\mathbb{B})}^q\frac{d t}{t}\right)^{1/q}
		&\le \sum_{k=-1}^\infty \left(\int_0^\infty\left\|t\dot H_{2^k t}(H_{2^k t}-H_{3\cdot 2^{k}t}) f \right\|_{L^p(\R^n;\mathbb{B})}^q\frac{d t}{t}\right)^{1/q} \\
		&=\sum_{k=-1}^\infty \frac{1}{2^{k}}\left(\int_0^\infty\left\|s \dot H_{s}(H_{s}-H_{3s}) f\right\|_{L^p(\R^n;\mathbb{B})}^q\frac{d s}{s}\right)^{1/q}\\
		&\lesssim \left(\int_0^\infty\left\|t \dot H_{t}(H_{t}-H_{3t}) f\right\|_{L^p(\R^n;\mathbb{B})}^q\frac{d t}{t}\right)^{1/q}\\
		&\stackrel{\mathclap{\mathrm{Lemma~}\ref{lem:n-heat-time-deriv}}}{\lesssim} \sqrt{n}\left(\int_0^\infty\left\|(H_{t}-H_{3t}) f\right\|_{L^p(\R^n;\mathbb{B})}^q\frac{d t}{t}\right)^{1/q}\\
		&\stackrel{\mathclap{\mathrm{Lemma~}\ref{lem:Rota}}}{\lesssim} \sqrt{n}\cdot \max\left\{\left(p-1\right)^{(1/q)-1},K_q\left(\mathbb{B}\right)\right\}\|f\|_{L^p(\R^n;\mathbb{B})}. \qedhere
	\end{align*}
\end{proof}

\section{Proof of Theorem \ref{thm:discrete}: discretization}\label{sec:proof main}

In this section, we prove Theorem \ref{thm:discrete} by discretizing the inequality of Theorem \ref{thm:nilpotentVvsH}.

\subsection{Reducing to the nilpotent case}
For the purposes of proving Theorem \ref{thm:discrete}, we may assume that the given group is a not virtually abelian finitely generated \emph{nilpotent} group. Let $\Gamma$ be a not virtually abelian finitely generated group of polynomial growth in the statement of Theorem \ref{thm:discrete}. By \cite{gromov1981groups}, $\Gamma$ has a nilpotent subgroup of finite index, say $\Gamma'\le \Gamma$ is a finite index nilpotent subgroup. This $\Gamma'$ is not virtually abelian \cite[Corollary 1.5]{de2007isometric}. We claim that Theorem \ref{thm:discrete} for $\Gamma'$ implies Theorem \ref{thm:discrete} for $\Gamma$.

Indeed, assume Theorem \ref{thm:discrete} holds for $\Gamma'$. This gives a certain $v_{\Gamma'}\in \Gamma'$ with the prescribed properties of Theorem \ref{thm:discrete}. Set $v_\Gamma=v_{\Gamma'}$. The distance requirement $d_W\left(v^n_\Gamma,e_\Gamma\right)\asymp_G n^{1/s}$ carries over from $\Gamma'$ to $\Gamma$ because a finitely generated word metric on a group and its finite index subgroup are quasi-isometric (via the inclusion map $\Gamma'\xhookrightarrow{} \Gamma$) by the Milnor--Svarc Lemma. Also, let $S'$ and $S$ be finite symmetric generating sets for $\Gamma'$ and $\Gamma$, respectively. (Observe that changing the generating set only affects the right-hand side of \eqref{eq:desired local pq} up to constant factors depending only on the groups, so the particular choice of $S$ or $S'$ does not affect the validity of Theorem \ref{thm:discrete}.) By possibly replacing $S$ with $S\cup S'$, we may assume $S\supset S'$. Since the inclusion map $\Gamma'\xhookrightarrow{}\Gamma$ is a quasi-isometry by the Milnor-Svarc Lemma, there is an integer $C\ge 1$ such that
\[
B_n^\Gamma \cap \Gamma'\subset B^{\Gamma'}_{Cn},~ B_n^{\Gamma'}\subset B^\Gamma_{Cn},\quad \forall n\ge 1.
\]
Finally, choose left coset representatives $b_1,\cdots,b_m\in \Gamma$, giving a left coset decomposition
\[
\Gamma=\bigsqcup_{i=1}^m b_i\Gamma'.
\]
By possibly enlarging $C$, we may assume $b_1,\cdots,b_m\in B_C^\Gamma$ and
\begin{equation}\label{eq:generator-qi}
B^\Gamma_n\subset \bigsqcup_{i=1}^m b_iB^{\Gamma'}_{Cn},\quad \forall n\ge 1.
\end{equation}

Now, for any function $f:\Gamma \to X$ into a $q$-uniformly convex Banach space, we have
\begin{align*}
&\left(\sum_{k=1}^{n^\rho}\frac{1}{k^{1+q/\rho}}\left(\sum_{x\in B^\Gamma_n} \left\|f\left(xv_\Gamma^k\right)-f\left(x\right)\right\|_X^p\right)^{q/p}\right)^{1/q}\\
&\stackrel{\mathclap{\eqref{eq:generator-qi}}}{\le}~\left(\sum_{k=1}^{n^\rho}\frac{1}{k^{1+q/\rho}}\left(\sum_{i=1}^m\sum_{x'\in B^{\Gamma'}_{Cn}} \left\|f\left(b_ix'v_\Gamma^k\right)-f\left(b_ix'\right)\right\|_X^p\right)^{q/p}\right)^{1/q}\\
&\le \sum_{i=1}^m\left(\sum_{k=1}^{n^\rho}\frac{1}{k^{1+q/\rho}}\left(\sum_{x'\in B^{\Gamma'}_{Cn}} \left\|f\left(b_ix'v_\Gamma^k\right)-f\left(b_ix'\right)\right\|_X^p\right)^{q/p}\right)^{1/q}\\
&\lesssim_{\Gamma,v_\Gamma} \max\left\{\left(p-1\right)^{(1/q)-1},K_q\left(X\right)\right\} \sum_{i=1}^m
\left(\sum_{x'\in B^{\Gamma'}_{cCn}} \sum_{a'\in S'}\left\|f\left(b_ix'a'\right)-f\left(b_ix'\right)\right\|^p_X\right)^{1/p}.\\
&\lesssim_{\Gamma,v_\Gamma} \max\left\{\left(p-1\right)^{(1/q)-1},K_q\left(X\right)\right\} m
\left(\sum_{x\in B^\Gamma_{(c+1)C^2n}} \sum_{a'\in S'}\left\|f\left(xa'\right)-f\left(x\right)\right\|^p_X\right)^{1/p}.\\
&\lesssim_{\Gamma,v_\Gamma} \max\left\{\left(p-1\right)^{(1/q)-1},K_q\left(X\right)\right\} 
\left(\sum_{x\in B^\Gamma_{(c+1)C^2n}} \sum_{a\in S}\left\|f\left(xa\right)-f\left(x\right)\right\|^p_X\right)^{1/p},
\end{align*}
where the second inequality uses the triangle inequality, the third inequality uses \eqref{eq:desired local pq} for $\Gamma'$ and the functions $\Gamma'\to X$ mapping $g'\mapsto f\left(b_ig'\right)$, the fourth inequality uses $b_i\in B_C^\Gamma$ and $B_n^{\Gamma'}\subset B_{Cn}^\Gamma$, and the final inequality uses $S'\subset \Gamma'$.

Hence, we are reduced to proving Theorem \ref{thm:discrete} for the not virtually abelian finitely generated nilpotent group $\Gamma'$.
\subsection{Reducing to the torsion-free nilpotent case}

Now we will see that, for the purposes of proving Theorem \ref{thm:discrete}, we may assume we are given a not virtually abelian finitely generated \emph{torsion-free} nilpotent group. Given a not virtually abelian finitely generated nilpotent group $\Gamma'$, let
\[
T\coloneqq \left\{x\in \Gamma':\exists n\in \mathbb{Z}_{>0} \mathrm{~such~that~}x^n=e_{\Gamma'}\right\}
\]
be the torsion subgroup of $\Gamma'$. It is well-known that $T$ is a finite normal subgroup of $\Gamma'$, and that $\Gamma''\coloneqq \Gamma'/T$ is a torsion-free finitely generated nilpotent group. This is not virtually abelian \cite[Corollary 1.5]{de2007isometric}. We now claim that Theorem \ref{thm:discrete} for $\Gamma''$ implies Theorem \ref{thm:discrete} for $\Gamma'$.

Indeed, assume Theorem \ref{thm:discrete} holds for $\Gamma''$. This gives a certain $v_{\Gamma''}\in \Gamma''$ with the prescribed properties of Theorem \ref{thm:discrete}: letting $S''\subset \Gamma''$ be a finite symmetric generating set for $\Gamma''$, the inequality \eqref{eq:desired local pq} is valid for the choice of $\Gamma''$,$S''$, and $v_{\Gamma''}$. Let $v_{\Gamma'}\in \pi^{-1}\left(v_{\Gamma''}\right)$. Let $S'=T\cup \pi^{-1}\left(S''\right)$, which is a symmetric generating set of $\Gamma'$. Choose a map $s:S''\to S'$ with $\pi\circ s=\operatorname{id}_{S''}$ and $s\left(g''^{-1}\right)=s\left(g''\right)^{-1}$. Again by the Milnor--Svarc Lemma, the natural quotient map $\pi:\Gamma'\to\Gamma''$ is a quasi-isometry. Thus the distance condition $d_W\left(v^n_{\Gamma''},e_{\Gamma''}\right)\asymp_{\Gamma''} n^{1/\rho}$ carries over from $\Gamma''$ to $\Gamma'$ as the distance condition $d_W\left(v^n_{\Gamma'},e_{\Gamma'}\right)\asymp_{\Gamma'} n^{1/\rho}$. Also, we may find $C>0$ be such that $\pi\left(B_k^{\Gamma'}\right)\subset B_{Ck}^{\Gamma''}$ and $\pi^{-1}\left(B_k^{\Gamma''}\right)\subset B^{\Gamma'}_{Ck-1}$ for $k\in \mathbb{Z}_{>0}$. This entails
\begin{equation}\label{eq:proj-inclu-1}
\pi\left(B_k^{\Gamma'}\left(x'\right)\right)\subset B_{Ck}^{\Gamma''}\left(\pi\left(x'\right)\right),\quad k\in \mathbb{Z}_{>0}, ~x'\in \Gamma',
\end{equation}
and
\begin{equation}\label{eq:proj-inclu-2}
\pi^{-1}\left(B_k^{\Gamma''}\left(y''\right)\right)\subset B^{\Gamma'}_{Ck}\left(x'\right),\quad k\in \mathbb{Z}_{>0},~y''\in \Gamma'',~x'\in \pi^{-1}\left(y''\right).
\end{equation}

Now, given $f:\Gamma'\to X$, define $\left[f\right]:\Gamma''\to X$ as
\[
\left[f\right]\left(y''\right)\coloneqq \frac{1}{\left|T\right|}\sum_{x'\in \pi^{-1}\left(y''\right)}f\left(x'\right),\quad y''\in \Gamma''.
\]
Then, by the triangle inequality,
\begin{align}\label{A-1}
\begin{aligned}
&\left(\sum_{k=1}^{n^\rho}\frac{1}{k^{1+q/\rho}}\left(\sum_{x'\in B^{\Gamma'}_n} \left\|f\left(x'v_{\Gamma'}^k\right)-f\left(x'\right)\right\|_X^p\right)^{q/p}\right)^{1/q}\\
&\le \left(\sum_{k=1}^{n^\rho}\frac{1}{k^{1+q/\rho}}\left(\sum_{x'\in B^{\Gamma'}_n} \left\|\left(f\left(x'v_{\Gamma'}^k\right)-\left[f\right]\left(\pi\left(x'v_{\Gamma'}^k\right)\right)\right)-\left(f\left(x'\right)-\left[f\right]\left(\pi\left(x'\right)\right)\right)\right\|_X^p\right)^{q/p}\right)^{1/q}\\
&\quad + \left(\sum_{k=1}^{n^\rho}\frac{1}{k^{1+q/\rho}}\left(\sum_{x'\in B^{\Gamma'}_n} \left\|\left[f\right]\left(\pi\left(x'v_{\Gamma'}^k\right)\right)-\left[f\right]\left(\pi\left(x'\right)\right)\right\|_X^p\right)^{q/p}\right)^{1/q}.
\end{aligned}
\end{align}
The second term of the right-hand side is at most, by assumption that Theorem \ref{thm:discrete} holds for $\Gamma''$,
\begin{align*}
&\left(\sum_{k=1}^{n^\rho}\frac{1}{k^{1+q/\rho}}\left(\sum_{x'\in B^{\Gamma'}_n} \left\|\left[f\right]\left(\pi\left(x'v_{\Gamma'}^k\right)\right)-\left[f\right]\left(\pi\left(x'\right)\right)\right\|_X^p\right)^{q/p}\right)^{1/q}\\
&\stackrel{\mathclap{\eqref{eq:proj-inclu-1}}}{\le} \left(\sum_{k=1}^{n^\rho}\frac{1}{k^{1+q/\rho}}\left(\left|T\right|\sum_{y''\in B^{\Gamma''}_{Cn}} \left\|\left[f\right]\left(y''v_{\Gamma''}^k\right)-\left[f\right]\left(y''\right)\right\|_X^p\right)^{q/p}\right)^{1/q}\\
&\lesssim_{\Gamma} \max\left\{\left(p-1\right)^{(1/q)-1},K_q\left(X\right)\right\} 
\left(|T|\sum_{y''\in B^{\Gamma''}_{cCn}} \sum_{a''\in S''}\left\|\left[f\right]\left(y''a''\right)-\left[f\right]\left(y''\right)\right\|^p_X\right)^{1/p}\\
&=\max\left\{\left(p-1\right)^{(1/q)-1},K_q\left(X\right)\right\} 
\left(|T|\sum_{y''\in B^{\Gamma''}_{cCn}} \sum_{a''\in S''}\left\|\frac{1}{\left|T\right|}\sum_{x'\in \pi^{-1}\left(y''\right)}\left(f\left(x's\left(a''\right)\right)-f\left(x'\right)\right)\right\|^p_X\right)^{1/p}\\
&\le\max\left\{\left(p-1\right)^{(1/q)-1},K_q\left(X\right)\right\} 
\left(\sum_{y''\in B^{\Gamma''}_{cCn}} \sum_{a''\in S''}\sum_{x'\in \pi^{-1}\left(y''\right)}\left\|f\left(x's\left(a''\right)\right)-f\left(x'\right)\right\|^p_X\right)^{1/p}\\
&\stackrel{\mathclap{\eqref{eq:proj-inclu-2}}}{\le} \max\left\{\left(p-1\right)^{(1/q)-1},K_q\left(X\right)\right\} 
\left(\sum_{x'\in B^{\Gamma'}_{cC^2n}} \sum_{a'\in S'}\left\|f\left(x'a'\right)-f\left(x'\right)\right\|^p_X\right)^{1/p},
\end{align*}
where we used Theorem \ref{thm:discrete} for $\Gamma''$ in the second inequality, and Jensen's inequality in the third inequality. Thus, the second term of \eqref{A-1} is bounded by the right-hand side of \eqref{eq:desired local pq}.

It remains to show that the first term of the right-hand side of \eqref{A-1} is bounded by the right-hand side of \eqref{eq:desired local pq}. Indeed, it is bounded by
\begin{align*}
&\left(\sum_{k=1}^{n^\rho}\frac{1}{k^{1+q/\rho}}\left(\sum_{x'\in B^{\Gamma'}_n} \left\|\left(f\left(x'v_{\Gamma'}^k\right)-\left[f\right]\left(\pi\left(x'v_{\Gamma'}^k\right)\right)\right)-\left(f\left(x'\right)-\left[f\right]\left(\pi\left(x'\right)\right)\right)\right\|_X^p\right)^{q/p}\right)^{1/q}\\
&=\left(\sum_{k=1}^{n^\rho}\frac{1}{k^{1+q/\rho}}\left(\sum_{x'\in B^{\Gamma'}_n} \left\|\frac{1}{\left|T\right|}\sum_{b'\in T}\left(f\left(x'v_{\Gamma'}^k\right)-f\left(x'v_{\Gamma'}^kb'\right)-f\left(x'\right)+f\left(x'b'\right)\right)\right\|_X^p\right)^{q/p}\right)^{1/q}\\
&\le \left(\sum_{k=1}^{n^\rho}\frac{1}{k^{1+q/\rho}}\left(\frac{2^{p-1}}{\left|T\right|}\sum_{x'\in B^{\Gamma'}_n}\sum_{b'\in T} \left(\left\|f\left(x'v_{\Gamma'}^k\right)-f\left(xv_{\Gamma'}^kb'\right)\right\|_X^p+\left\|f\left(x'\right)-f\left(x'b'\right)\right\|_X^p\right)\right)^{q/p}\right)^{1/q}\\
&\le \left(\sum_{k=1}^{n^\rho}\frac{1}{k^{1+q/\rho}}\left(\frac{2^{p}}{\left|T\right|}\sum_{x'\in B^{\Gamma'}_{c'n}}\sum_{b'\in T} \left\|f\left(x'\right)-f\left(x'b'\right)\right\|_X^p\right)^{q/p}\right)^{1/q}\\
&\lesssim \rho^{1/q} \left(\sum_{x\in B^{\Gamma'}_{c'n}}\sum_{b\in T} \left\|f\left(x\right)-f\left(xb\right)\right\|_X^p\right)^{1/p}\lesssim_{\Gamma'}\left(\sum_{x\in B^{\Gamma'}_{c'n}}\sum_{b\in T} \left\|f\left(x\right)-f\left(xb\right)\right\|_X^p\right)^{1/p}.
\end{align*}
where the first inequality uses Jensen's inequality, and in the second inequality the constant $c'>0$ is a natural number depending on $\Gamma'$ such that $d_W\left(v^n_{\Gamma'},e_{\Gamma'}\right)\le \left(c'-1\right) n^{1/\rho}$ for $n\in \mathbb{Z}_{>0}$, so that $B^{\Gamma'}_nv_{\Gamma'}^k\subset B^{\Gamma'}_{c'n}$ for $n\in \mathbb{Z}_{>0}$ and $k=1,\cdots,n^\rho$.

Hence it is enough to prove Theorem \ref{thm:discrete} for a not virtually abelian finitely generated torsion-free nilpotent group $\Gamma''$.

\subsection{Reducing to the nilpotent Lie group case of Theorem \ref{thm:nilpotentVvsH}}
From now on assume $\Gamma$ is a not virtually abelian finitely generated torsion-free nilpotent group. By the Malcev embedding theorem \cite{malcev1949class}, there exists a simply connected nilpotent Lie group $G$, called the Malcev completion of $\Gamma$, such that $\Gamma$ embeds as a cocompact subgroup of $G$. Because $\Gamma$ is nonabelian, $G$ is nonabelian, so Theorem \ref{thm:nilpotentVvsH} applies to $G$.

There exist choices of $\rho$ and $v_\Gamma$ as stated. Indeed,\footnote{I thank the anonymous referee for helping to simplify this argument.} let $\rho=s\ge 2$ be the nilpotency step of $G$, i.e., the largest integer so that $\underbrace{\left[G,\left[G,\cdots,G\right]\right]}_{s \mathrm{~times}}$ is a nontrivial subgroup, and choose $v_\Gamma\in \underbrace{\left[\Gamma,\left[\Gamma,\cdots,\Gamma\right]\right]}_{s \mathrm{~times}}\setminus \left\{e_\Gamma\right\}\subset Z\left(\Gamma'\right)\setminus\left\{e_\Gamma\right\}$. Then $d_G\left(v^k,e_G\right)\asymp_G k^{1/s}$ by \cite[Proposition 2.13]{breuillard2012nilpotent}, and so by the Milnor--Svarc Lemma
\[
d_W\left(v_\Gamma^k,e_\Gamma\right)\asymp_{\Gamma,v_\Gamma} k^{1/s},\quad k\in \mathbb{N}.
\]

Again, $\left|B^\Gamma_m\right|\asymp_\Gamma m^{n_h}$ for every $m\in \mathbb{N}$ \cite[Theorem 1.1]{breuillard2012nilpotent}, where $n_h$ is given by the Bass--Guivarc'h formula
\[
n_h\coloneqq \sum_{k=1}^s k\operatorname{dim}\left(\underbrace{\left[G,\left[G,\cdots,G\right]\right]}_{k \mathrm{~times}}/\underbrace{\left[G,\left[G,\cdots,G\right]\right]}_{k+1 \mathrm{~times}}\right).
\]

The rest of this section is inspired by the argument of Section 3 of \cite{lafforgue2014vertical}.

Before we begin the proof of Theorem \ref{thm:discrete}, we prove two metric-space valued local Poincar\'e inequalities on $\Gamma$ for preparation. Lemma \ref{lem:metric bruce} is a variant of the local Poincar\'e inequality of Kleiner~\cite[Theorem~2.2]{kleiner2010new}, and the special case of Lemma \ref{lem:metric bruce} for the Heisenberg group $\mathbb{H}^3_\mathbb{Z}$ was proven in \cite[Lemma 3.2]{lafforgue2014vertical}. We formulate and extend the inequality to all finitely generated groups.

\begin{lemma}\label{lem:metric bruce} Let $\Gamma$ be a group with finite symmetric generating set $S$, the associated word distance $d_W$, and corresponding balls $B_n^\Gamma$ about the origin for $n\in \mathbb{Z}_{\ge 0}$. Let $p\in \left[1,\infty\right)$, $n\in \mathbb{Z}_{\ge 0}$, and let $\left(\mathcal{M},d_\mathcal{M}\right)$ be a metric space. Then, for every $f:\Gamma\to \mathcal{M}$,
\[
\sum_{x,y\in B^\Gamma_n} d_\mathcal{M}\left(f\left(x\right),f\left(y\right)\right)^p \le \left(2n\right)^{p}\left|B^\Gamma_{2n}\right| \sum_{x\in B^\Gamma_{3n}} \sum_{a\in S}d_\mathcal{M}\left(f\left(xa\right),f\left(x\right)\right)^p.
\]
\end{lemma}
\begin{proof}
For every $z\in B^\Gamma_{2n}$ choose $s_1\left(z\right),\ldots,s_{2n}\left(z\right)\in S\cup \left\{e_\Gamma\right\}$ such that $z=s_1\left(z\right)\cdots s_{2n}\left(z\right)$. For  $i\in \left\{1,\ldots,2n\right\}$ write  $w_i\left(z\right)=s_1\left(z\right)\cdots s_i\left(z\right)$ and set $w_0\left(z\right)=e_\Gamma$. By the triangle inequality and H\"older's inequality, for every $x,y\in B^\Gamma_n$ we have
\begin{equation*}
d_\mathcal{M}\left(f\left(x\right),f\left(y\right)\right)^p\le \left(2n\right)^{p-1}\sum_{i=0}^{2n-1} d_\mathcal{M}\left(f\left(xw_{i}\left(x^{-1}y\right)\right),f\left(xw_{i}\left(x^{-1}y\right)s_{i+1}\left(x^{-1}y\right)\right)\right)^p.
\end{equation*}
Consequently,
\begin{align*}
\sum_{x,y\in B^\Gamma_n} d_\mathcal{M}\left(f\left(x\right),f\left(y\right)\right)^p&\le \sum_{x\in B^\Gamma_n}\sum_{z\in B^\Gamma_{2n}}d_\mathcal{M}\left(f\left(x\right),f\left(xz\right)\right)^p\\
&\le \left(2n\right)^{p-1}\sum_{z\in B^\Gamma_{2n}}\sum_{i=1}^{2n-1}\sum_{x\in B^\Gamma_n}d_\mathcal{M}\left(f\left(xw_{i}\left(z\right)\right),f\left(xw_{i}\left(z\right)s_{i+1}\left(z\right)\right)\right)^p\\
&= \left(2n\right)^{p-1} \sum_{z\in B^\Gamma_{2n}}\sum_{i=0}^{2n-1}\sum_{g\in B^\Gamma_nw_i\left(z\right)} d_\mathcal{M}\left(f\left(g\right),f\left(gs_{i+1}\left(z\right)\right)\right)^p\\
&\le \left(2n\right)^{p-1} \cdot \left|B^\Gamma_{2n}\right|\cdot 2n\sum_{g\in B^\Gamma_{3n}} \sum_{a\in S}d_\mathcal{M}\left(f\left(ga\right),f\left(g\right)\right)^p.
\end{align*}
\end{proof}

The following lemma is a variant Poincar\'e inequality, now summed globally. This inequality was proven for the Heisenberg group $\mathbb{H}^3_\mathbb{Z}$ in \cite[Lemma 3.4]{lafforgue2014vertical}; we formulate and extend the inequality for all finitely generated groups. Here, a mapping $f:\Gamma\to \mathcal{M}$ is said to be finitely supported if there exists $m_0\in \mathcal{M}$ such that $\left|f^{-1}\left(\mathcal{M}\setminus\left\{m_0\right\}\right)\right|<\infty$.

\begin{lemma}\label{lem:global poincare-simple} Let $\Gamma$ be a group with finite symmetric generating set $S$, the associated word distance $d_W$, and corresponding balls $B_n^\Gamma$ about the origin for $n\in \mathbb{Z}_{\ge 0}$. Let $p\in \left[1,\infty\right)$, $n\in \mathbb{Z}_{\ge 0}$, and let $\left(\mathcal{M},d_\mathcal{M}\right)$ be a metric space. Then, for every finitely supported $f:\Gamma\to \mathcal{M}$,
\[
\sum_{x\in \Gamma}\sum_{z\in B^\Gamma_n} d_M\left(f\left(xz\right),f\left(x\right)\right)^p\le n^{p}\left|B^\Gamma_n\right| \sum_{x\in \Gamma} \sum_{a\in S}d_M\left(f\left(xa\right),f\left(x\right)\right)^p.
\]
\end{lemma}

\begin{proof}
Let $s_i$ and $w_i$ be as in the proof of Lemma \ref{lem:metric bruce}. By the triangle inequality and H\"older's inequality,
\begin{align*}
\sum_{x\in \Gamma}\sum_{z\in B^\Gamma_n}d_M\left(f\left(xz\right),f\left(x\right)\right)^p&\le n^{p-1}\sum_{x\in \Gamma}\sum_{z\in B^\Gamma_n}\sum_{i=0}^{n-1} d_M\left(f\left(xw_{i}\left(z\right)\right),f\left(xw_{i}\left(z\right)s_{i+1}\left(z\right)\right)\right)^p\\
&\le n^{p-1}\cdot \left|B^\Gamma_n\right|\cdot n\sum_{x\in \Gamma} \sum_{a\in S}d_M\left(f\left(xa\right),f\left(x\right)\right)^p.
\end{align*}
\end{proof}

With Lemmas \ref{lem:metric bruce} and \ref{lem:global poincare-simple}, we begin by discretizing the inequality of Theorem \ref{thm:nilpotentVvsH}, as in the following theorem. After this, we will localize the statement into that of Theorem \ref{thm:discrete}, which states a local version of Lemma \ref{lem:pq-global} on balls in $\Gamma$.

The following lemma is a generalization of Theorem 3.3 of \cite{lafforgue2014vertical}, whose proof is in turn a variant of the proof of Claim 7.3 in \cite{austin2013sharp}. We may state it in terms of discrete groups of polynomial growth.
\begin{lemma}\label{lem:pq-global}
Let $\Gamma$ be a not virtually abelian finitely generated torsion-free nilpotent group of polynomial growth. There exist $v_\Gamma\in \Gamma$ and $s\in \mathbb{N}$ with $s\ge 2$ such that the following is true. First, $d_W\left(v_\Gamma^n,e_\Gamma\right)\asymp_\Gamma n^{1/s}$ for $n\in \mathbb{N}$. Second, let $p\in \left(1,\infty\right)$ and $q\in \left[2,\infty\right)$ with $p\le q$. Suppose that $\left(X,\left\|\cdot\right\|_X\right)$ is a Banach space satisfying $K_q\left(X\right)<\infty$.  Then for every finitely supported $f:\Gamma\to X$ we have
\begin{align*}
\begin{aligned}
&\left(\sum_{k=1}^{\infty}\frac{1}{k^{1+q/s}}\left(\sum_{x\in \Gamma} \left\|f\left(xv_\Gamma^k\right)-f\left(x\right)\right\|_X^p\right)^{q/p}\right)^{1/q}\\
&\lesssim_{\Gamma,v_\Gamma} \max\left\{\left(p-1\right)^{(1/q)-1},K_q\left(X\right)\right\} 
\left(\sum_{x\in \Gamma} \sum_{a\in S}\left\|f\left(xa\right)-f\left(x\right)\right\|^p_X\right)^{1/p}.
\end{aligned}
\end{align*}
\end{lemma}
\begin{proof}
Take $\rho=s$ to be the nilpotency step of $\Gamma$ and take $v_\Gamma \in \underbrace{\left[\Gamma,\left[\Gamma,\cdots,\Gamma\right]\right]}_{s\mathrm{~times}}\setminus\left\{e_{\Gamma}\right\}$. 

The idea of the proof is that given a finitely supported function $f:\Gamma\to X$, we extend it to a global function $F:G\to X$ via a partition of unity, and then Theorem \ref{thm:nilpotentVvsH} for $F$ will give Lemma \ref{lem:pq-global} for $f$.

Let $G$ be the Malcev completion of $\Gamma$. Since $\Gamma$ is a co-compact lattice of $G$, there exists a compactly supported smooth function $\chi:G\to \left[0,1\right]$ with
\begin{equation}\label{eq:partition of unity}
\forall\, h\in G,\qquad \sum_{x\in \Gamma} \chi_x\left(h\right)=1,
\end{equation}
where for each $x\in \Gamma$, $\chi_x:G\to X$ is given by $\chi_x\left(h\right)=\chi\left(x^{-1}h\right)$, $h\in G$. Let $A=\operatorname{supp}\chi$; we may assume $A^{-1}=A$. Note that $\operatorname{supp}\chi_x=xA$ and
\[
\bigcup_{x\in \Gamma}xA=G.
\]
As $A$ is compact, we may fix $m\in \mathbb{N}$ for which $A^2\cap \Gamma\subseteq B^\Gamma_m$.

Let $f:\Gamma\to X$ be finitely supported. Define $F:G\to X$ by
\begin{equation}\label{eq:defF}
F\left(h\right)\coloneqq \sum_{x\in \Gamma} \chi_x\left(h\right)f\left(x\right).
\end{equation}
This is smooth and compactly supported, so Theorem \ref{thm:nilpotentVvsH} applies to $F$.
For a fixed $x\in \Gamma$ and $h\in xA$, we have
\[
\sum_{y\in \Gamma} \nabla\chi_y\left(h\right)=0,
\]
with $\nabla\chi_y\left(h\right)\neq 0$ implying $y^{-1}h\in \operatorname{supp}\nabla\chi\subset A$, which implies $y\in hA^{-1}\subset xA^2$, hence $y\in xB^\Gamma_{m}$. We thus have the bound, for $h\in xA$,
\begin{align*}
\left\|\nabla F\left(h\right)\right\|_{\ell_2^k\left(X\right)}&=\left\|\sum_{y\in xB^\Gamma_{m}}\nabla\chi_y\left(h\right)\left(f\left(y\right)-f\left(x\right)\right)\right\|_{\ell_2^k\left(X\right)}
\\&\le \left\|\nabla\chi\right\|_{L^\infty\left(G;\ell_2^k\right)}\sum_{z\in B^\Gamma_{m}}\left\|f\left(xz\right)-f\left(x\right)\right\|_X\\
&\lesssim_\Gamma \left|B^\Gamma_{m}\right|^{1-1/p}\left(\sum_{z\in B^\Gamma_{m}}\left\|f\left(xz\right)-f\left(x\right)\right\|_X^p\right)^{1/p}.
\end{align*}
By integrating over $xA$ and summing over $x\in \Gamma$,
\begin{align}\label{eq:nablaFbound}
\begin{aligned}
\left(\int_{G} \left\|\nabla F\left(h\right)\right\|_{\ell_2^k\left(X\right)}^pd\mu\left(h\right)\right)^{1/p}&\le \left(\sum_{x\in \Gamma}\int_{xA} \left\|\nabla F\left(h\right)\right\|_{\ell_2^k\left(X\right)}^pd\mu\left(h\right)\right)^{1/p}\lesssim_\Gamma \left(\sum_{x\in \Gamma} \sum_{z\in B^\Gamma_{m}}\left\|f\left(xz\right)-f\left(x\right)\right\|_X^p\right)^{1/p}\\
&\stackrel{\mathclap{\mathrm{Lemma~} \ref{lem:global poincare-simple}}}{\lesssim_\Gamma}\quad \left(\sum_{x\in \Gamma} \sum_{a\in S}\left\|f\left(xa\right)-f\left(x\right)\right\|_X^p\right)^{1/p}.
\end{aligned}
\end{align}

Since $\Gamma$ induces a covering space action on $G$, we may find a bounded open neighborhood $U\subseteq G$ of $e_G$ such that $U\cap \left(xU\right)=\emptyset$ for all $x\in \Gamma\setminus \left\{e_\Gamma\right\}$.  As $v_\Gamma^{\left[0,1\right]}\coloneqq\left\{v_\Gamma^\tau:\tau\in \left[0,1\right]\right\}$ and $U$ are bounded in $G$, we can choose $r\in \mathbb{N}$ such that $\left(Uv_\Gamma^{\left[0,1\right]}A\right)\cap \Gamma\subseteq B^\Gamma_r$. For $x\in \Gamma$, $h\in xU$, $k\in \mathbb{N}$, and $t\in \left[k,k+1\right]$, we have by \eqref{eq:partition of unity} and \eqref{eq:defF} 
\[
F\left(hv_\Gamma^t\right)-f\left(xv_\Gamma^k\right)=\sum_{w\in \Gamma} \chi_w\left(hv_\Gamma^t\right)\left(f\left(w\right)-f\left(xv_\Gamma^k\right)\right).
\]
Every $w\in \Gamma$ that satisfies $\chi_w\left(hv_\Gamma^t\right)\neq 0$ should have $hv_\Gamma^t\in wA$, and since $A=A^{-1}$ and $v_\Gamma\in Z\left(G\right)$, we have $w\in hv_\Gamma^tA\subseteq xv_\Gamma^k Uv_\Gamma^{\left[0,1\right]}A$, so that $w\in xv_\Gamma^kB^\Gamma_r$. Therefore
\[
\left\|F\left(hv_\Gamma^t\right)-f\left(xv_\Gamma^k\right)\right\|_X\lesssim_\Gamma \left(\sum_{z\in B^\Gamma_r} \left\|f\left(xv_\Gamma^kz\right)-f\left(xv_\Gamma^k\right)\right\|_X^p\right)^{1/p}.
\]
The case $k=t=0$ gives
\[
\left\|F\left(h\right)-f\left(x\right)\right\|_X\lesssim_\Gamma \left(\sum_{z\in B^\Gamma_r} \left\|f\left(xz\right)-f\left(x\right)\right\|_X^p\right)^{1/p}.
\]
Therefore by the triangle inequality,
\[
\left\|f\left(xv_\Gamma^k\right)-f\left(x\right)\right\|_X \lesssim_\Gamma \left\|F\left(hv_\Gamma^t\right)-F\left(h\right)\right\|_X+\left(\sum_{z\in B^\Gamma_r}\left(\left\|f\left(xv_\Gamma^kz\right)-f\left(xv_\Gamma^k\right)\right\|_X^p+\left\|f\left(xz\right)-f\left(x\right)\right\|_X^p\right)\right)^{1/p}.
\]
Integration over $h\in xU$ gives
\begin{align*}
\left\|f\left(xv_\Gamma^k\right)-f\left(x\right)\right\|_X\lesssim_\Gamma& \left(\int_{xU}\left\|F\left(hv_\Gamma^t\right)-F\left(h\right)\right\|_X^pd\mu\left(h\right)\right)^{1/p}\\
&+\left(\sum_{z\in B^\Gamma_r}\left(\left\|f\left(xv_\Gamma^kz\right)-f\left(xv_\Gamma^k\right)\right\|_X^p+\left\|f\left(xz\right)-f\left(x\right)\right\|_X^p\right)\right)^{1/p}.
\end{align*}
Summing over $x\in \Gamma$ and using that the sets $\left\{xU\right\}_{x\in \Gamma}$ are pairwise disjoint,
\begin{align*}
\left(\sum_{x\in \Gamma}\left\|f\left(xv_\Gamma^k\right)-f\left(x\right)\right\|_X^p\right)^{1/p}&\lesssim_\Gamma \left(\int_{G}\left\|F\left(hv_\Gamma^t\right)-F\left(h\right)\right\|_X^pd\mu\left(h\right)\right)^{1/p}\\
&\quad +\left(\sum_{x\in\Gamma}\sum_{z\in B^\Gamma_r}\left(\left\|f\left(xv_\Gamma^kz\right)-f\left(xv_\Gamma^k\right)\right\|_X^p+\left\|f\left(xz\right)-f\left(x\right)\right\|_X^p\right)\right)^{1/p}\\
&\stackrel{\mathclap{\mathrm{Lemma~}\ref{lem:global poincare-simple}}}{\lesssim_\Gamma} \quad\left(\int_{G}\left\|F\left(hv_\Gamma^t\right)-F\left(h\right)\right\|_X^pd\mu\left(h\right)\right)^{1/p}+\left(\sum_{x\in \Gamma} \sum_{a\in S}\left\|f\left(xa\right)-f\left(x\right)\right\|_X^p\right)^{1/p}.
\end{align*}
Integrating over $t\in \left[k,k+1\right]$ yields
\begin{align*}
&\nonumber\frac{1}{k^{1+q/s}}\left(\sum_{x\in \Gamma}\left\|f\left(xv_\Gamma^k\right)-f\left(x\right)\right\|_X^p\right)^{q/p}\\&\le  C^{q}\int_k^{k+1}\left(\int_{G}\left\|F\left(hv_\Gamma^t\right)-F\left(h\right)\right\|_X^pd\mu\left(h\right)\right)^{q/p}\frac{dt}{t^{1+q/s}} +\frac{C^{q}}{k^{1+q/s}}\left(\sum_{x\in \Gamma} \sum_{a\in S}\left\|f\left(xa\right)-f\left(x\right)\right\|_X^p\right)^{q/p},
\end{align*}
where $C=C_\Gamma\in \left(0,\infty\right)$ is a constant depending on $\Gamma$. Summing over $k\in \mathbb{N}$,
\begin{align}\label{eq:vertical control}
&\nonumber\left(\sum_{k=1}^\infty\frac{1}{k^{1+q/s}}\left(\sum_{x\in \Gamma}\left\|f\left(xv_\Gamma^k\right)-f\left(x\right)\right\|_X^p\right)^{q/p}\right)^{1/q}\\&\lesssim_\Gamma \left(\int_0^{\infty}\left(\int_{G}\left\|F\left(hv_\Gamma^t\right)-F\left(h\right)\right\|_X^pd\mu\left(h\right)\right)^{q/p}\frac{dt}{t^{1+q/s}}\right)^{1/q}+
\left(\sum_{x\in \Gamma} \sum_{a\in S}\left\|f\left(xa\right)-f\left(x\right)\right\|_X^p\right)^{1/p}.
\end{align}
The desired inequality follows from \eqref{eq:nablaFbound} and \eqref{eq:vertical control} along with Theorem \ref{thm:nilpotentVvsH} for the smooth and compactly supported function $F$ (although it may be the case that $d_G\left(v_\Gamma^t,e_G\right)\nleq t^{1/s}$, a simple rescaling argument gives the same inequality up to constant factors depending on $v_\Gamma$).
\end{proof}

Now we are ready to prove Theorem \ref{thm:discrete}. Of course, assume $\Gamma$ is torsion-free nonabelian finitely generated nilpotent, and choose $v_\Gamma$ and $s$ as in the proof of Lemma \ref{lem:pq-global}.

\begin{proof}[Proof of Theorem~\ref{thm:discrete}]
The argument is inspired by the proof of Theorem 3.1 of \cite{lafforgue2014vertical} and Claim 7.2 in \cite{austin2013sharp}. Choose $c\in \mathbb{Z}_{>0}$ such that
\[
d_W\left(e_\Gamma,v_\Gamma^k\right)\le c k^{1/s}.
\]

 Fix $n\in \mathbb{N}$ and an $f:\Gamma\to X$. By possibly replacing $f$ by $f-\left|B^\Gamma_{\left(c+3\right)n}\right|^{-1}\sum_{x\in B^\Gamma_{\left(c+3\right)n}}f\left(x\right)$, we may assume $\sum_{x\in B^\Gamma_{\left(c+3\right)n}}f\left(x\right)=0$. Then
\begin{align}\label{eq:use bruce}
\begin{aligned}
\left(\sum_{x\in B^\Gamma_{\left(c+3\right)n}}\left\|f\left(x\right)\right\|_X^p\right)^{1/p}&= \left(\sum_{x\in B^\Gamma_{\left(c+3\right)n}}\left\|\frac{1}{\left|B^\Gamma_{\left(c+3\right)n}\right|}\sum_{y\in B^\Gamma_{\left(c+3\right)n}}\left(f\left(x\right)-f\left(y\right)\right)\right\|_X^p\right)^{1/p}\\
&\le  \left(\frac{1}{\left|B^\Gamma_{\left(c+3\right)n}\right|}\sum_{x,y\in B^\Gamma_{\left(c+3\right)n}}\left\|f\left(x\right)-f\left(y\right)\right\|_X^p\right)^{1/p}\\
&\stackrel{\mathclap{\mathrm{Lemma~}\ref{lem:metric bruce}}}{\lesssim_\Gamma} \quad n\left(\sum_{x\in B^\Gamma_{3\left(c+3\right)n}}\sum_{a\in S}\left\|f\left(xa\right)-f\left(x\right)\right\|_X^p\right)^{1/p}.
\end{aligned}
\end{align}
Define the cutoff function $\xi:\Gamma\to \left[0,1\right]$ as
\[
\xi\left(x\right)\coloneqq \left\{\begin{array}{ll}1& x\in B^\Gamma_{\left(c+1\right)n},\\
c+2-\frac{d_W\left(x,e\right)}{n}& x\in B^\Gamma_{\left(c+2\right)n}\setminus B^\Gamma_{\left(c+1\right)n},\\
0& x\in \Gamma\setminus B^\Gamma_{\left(c+2\right)n}.
\end{array}\right.
\]
Then $\phi \coloneqq \xi f$ is finitely supported on $B^\Gamma_{\left(c+2\right)n}$ and we may apply Lemma \ref{lem:pq-global}. As $\xi$ is $\frac{1}{n}$-Lipschitz and takes values in $\left[0,1\right]$, we have for all $a\in S$ and $x\in \Gamma$ 
\begin{align}\label{eq:tildef}
\begin{aligned}
\left\|\phi\left(x\right)-\phi\left(xa\right)\right\|_X&\le \left|\xi\left(x\right)-\xi\left(xa\right)\right|\cdot\left\|f\left(x\right)\right\|_X+\left|\xi\left(xa\right)\right|\cdot \left\|f\left(x\right)-f\left(xa\right)\right\|_X\\&
\le
\frac{1}{n}\left\|f\left(x\right)\right\|_X+\left\|f\left(x\right)-f\left(xa\right)\right\|_X.
\end{aligned}
\end{align}

We have $d_W\left(e_\Gamma,v_\Gamma^k\right)\le cn$ for all $k\in \left\{1,\ldots, n^s\right\}$. Thus, for every $x\in B^\Gamma_n$ and $k\in \left\{1,\ldots, n^s\right\}$ we have $xv_\Gamma^k\in B^\Gamma_{\left(c+1\right)n}$ and thus $\phi\left(x\right)=f\left(x\right)$ and $\phi\left(xv_\Gamma^k\right)=f\left(xv_\Gamma^k\right)$, so
\begin{equation}\label{eq:pass to phi}
\left(\sum_{k=1}^{n^s}\frac{1}{k^{1+q/s}}\left(\sum_{x\in B^\Gamma_n}\left\|f\left(xv_\Gamma^k\right)-f\left(x\right)\right\|_X^p\right)^{q/p}\right)^{1/q}\le  \left(\sum_{k=1}^{\infty}\frac{1}{k^{1+q/s}}\left(\sum_{x\in \Gamma} \left\|\phi\left(xv_\Gamma^k\right)-\phi\left(x\right)\right\|_X^p\right)^{q/p}\right)^{1/q}.
\end{equation}
On the other hand,
\begin{align*}
\left(\sum_{x\in \Gamma} \sum_{a\in S}\left\|\phi\left(xa\right)-\phi\left(x\right)\right\|^p_X\right)^{1/p}&= \left(\sum_{x\in B^\Gamma_{\left(c+3\right)n}} \sum_{a\in S}\left\|\phi\left(xa\right)-\phi\left(x\right)\right\|^p_X\right)^{1/p}\\
&\stackrel{\mathclap{\eqref{eq:tildef}}}{\lesssim} \frac{\left|S\right|^{1/p}}{n}\left(\sum_{x\in B^\Gamma_{\left(c+3\right)n}} \left\|f\left(x\right)\right\|_X^p\right)^{1/p}+\left(\sum_{x\in B^\Gamma_{\left(c+3\right)n}} \sum_{a\in S}\left\|f\left(xa\right)-f\left(x\right)\right\|^p_X\right)^{1/p}\\
&\stackrel{\mathclap{\eqref{eq:use bruce}}}{\lesssim}_\Gamma \label{eq:use 21}\left(\sum_{x\in B^\Gamma_{3\left(c+3\right)n}} \sum_{a\in S}\left\|f\left(xa\right)-f\left(x\right)\right\|^p_X\right)^{1/p},
\end{align*}
where the first equality holds since $\phi$ is supported on $B^\Gamma_{\left(c+2\right)n}$. The desired inequality now follows from the above two inequalities combined with Lemma \ref{lem:pq-global} for $\phi$.
\end{proof}

\section{Proof of the nonembeddability statements Theorems \ref{mainthm:nilpotentdistortion}, \ref{mainthm:polygrowthdistortion}, \ref{thm:discrete-snowflake}, \ref{thm:precise-dist-nilp}, \ref{thm:precise-holder-nilp}, \ref{thm:net}, \ref{thm:net-snowflake} 
 and Corollaries \ref{cor:nilpotentlp}, \ref{cor:polygrowthlp}, \ref{cor:growth-function-char}}\label{sec:net}
 
In this section we prove Theorems \ref{mainthm:nilpotentdistortion}, \ref{mainthm:polygrowthdistortion}, \ref{thm:discrete-snowflake}, \ref{thm:precise-dist-nilp}, \ref{thm:precise-holder-nilp}, \ref{thm:net}, and \ref{thm:net-snowflake}, and Corollaries \ref{cor:nilpotentlp}, \ref{cor:polygrowthlp}, and \ref{cor:growth-function-char}. We will first derive a nonembeddability statement, Theorem \ref{thm:gendistortion}, in the general setting of Theorem \ref{thm:cvx}, and then derive Theorem \ref{mainthm:nilpotentdistortion} from Theorem \ref{thm:gendistortion}. 
Theorems \ref{mainthm:polygrowthdistortion}, \ref{thm:discrete-snowflake} and Corollary \ref{cor:growth-function-char} will follow from Theorem \ref{thm:discrete}, and Corollaries \ref{cor:nilpotentlp} and \ref{cor:polygrowthlp} will follow from Theorems \ref{mainthm:nilpotentdistortion} and \ref{mainthm:polygrowthdistortion}, respectively. 
We will derive Theorems \ref{thm:gen-net} and \ref{thm:gen-net-snowflake} in the generality of \ref{thm:cvx}, and derive Theorems \ref{thm:precise-dist-nilp}, \ref{thm:precise-holder-nilp}, \ref{thm:net} and \ref{thm:net-snowflake} from them.

We begin, as promised, by deriving a nonembeddability statement in the generality of Theorem \ref{thm:cvx}.

\begin{theorem}\label{thm:gendistortion}
Let the Riemannian Lie group $G$, $v\in Z\left(\mathfrak{g}\right)$, and Banach space $X$ satisfy the hypotheses of Theorem \ref{thm:cvx}. In addition, assume, recalling \eqref{eq:rho-ineq}, that
\begin{equation}\label{eq:rho-ineq-double}
c_1t^{1/\rho}\le d_G\left(\exp\left(tv\right),e_G\right)\stackrel{\eqref{eq:rho-ineq}}{\le} t^{1/\rho}, \quad \forall t\ge 1,
\end{equation}
and that the measure $\mu$ is $K$-doubling:
\[
\mu(B_{2r})\le K\mu(B_r),\quad r>0.
\]
\begin{enumerate}
    \item Then
\[
c_X\left(B_r\right)\gtrsim \frac{c_1\rho^{1/q}}{\sqrt{\dim \mathfrak{g}}\left(\rho+\frac{1}{\rho-1}\right)K^{2/q}} \cdot\frac{\left(\log r\right)^{1/q}}{K_q\left(X\right)}\quad r\ge 2.
\]
In particular, $G$ fails to bilipschitzly embed into $X$.
\item We have, for $r\ge 2$ and $1<p<\infty$,
\begin{align*}
&\frac{c_1\rho^{1/\max\left\{p,2\right\}}\sqrt{\min\left\{p,2\right\}-1}}{\sqrt{\dim \mathfrak{g}}\left(\rho+\frac{1}{\rho-1}\right)K^{2/\max\{p,2\}}} \left(\log r\right)^{1/\max\left\{p,2\right\}}\\
&\lesssim c_p\left(B_r\right)\\
&\lesssim \max\left\{\frac{\log K}{p},1\right\}\left(\log r\right)^{1/\max\left\{p,2\right\}}+ K^{4/p}c_p\left(B_1\right).
\end{align*}
\end{enumerate}

\end{theorem}

\begin{remark}\,
    \begin{enumerate}
        \item Again, the reason we assume $G$ has a Riemannian distance in Theorem \ref{thm:gendistortion} is that if $G$ were given a strictly sub-Riemannian distance, then any nonempty open subset $U$ of $G$ would fail to embed bilipschitzly into $X$ \cite{huang2020non}, since the tangent space of $G$ is a nonabelian Carnot group. That is why we require $t\ge 1$ in \eqref{eq:rho-ineq-double}: $d_G\left(\exp\left(tv\right),e_G\right)$ is linear in $t$ for small $t$ when $G$ is Riemannian. If the distance requirement \eqref{eq:rho-ineq-double} held for $0<t<\epsilon$, with $\epsilon$ very small, then running the above proof would give a lower bound of $D$ in terms of $\left(\int_0^\epsilon \frac{dt}t\right)^{1/q}=\infty$, proving local nonembeddability. However, this proof would be unnecessary, since the fact that \eqref{eq:rho-ineq-double} holds for small $t$ means that $G$ is given a strictly sub-Riemannian distance.
        \item The right-hand side of the inequality of statement (2) of Theorem \ref{thm:gendistortion} is a finite quantity, since $c_p\left(B_1\right)\lesssim_G 1$. More generally $c_Y\left(B_R\right)\lesssim_{G,r}1$ for any infinite-dimensional Banach space $Y$, since $c_{\mathbb{R}^d}\left(B_R\right)\lesssim_G 1$ for some $d\in \mathbb{Z}_{>0}$ depending only on $G$ and $R$ \cite{eriksson2018quantitative}. If $G$ were a simply connected nilpotent Lie group, the inverse of the exponential map provides such an embedding $B_R\to \mathbb{R}^d$. 
    \end{enumerate}
\end{remark}
\begin{proof}[Proof of Theorem \ref{thm:gendistortion}]\,
\begin{enumerate}
\item
When $r\le 8$, we have, by $\dim \mathfrak{g}, \rho>1$ and $K_q(X),K\ge 1$, that 
\[
c_X\left(B_r\right)\ge 1\gtrsim \left(\log 8\right)^{1/q}\gtrsim \frac{c_1\rho^{1/q}}{\sqrt{\dim \mathfrak{g}}\left(\rho+\frac{1}{\rho-1}\right)K^{2/q}}\cdot\frac{\left(\log r\right)^{1/q}}{K_q\left(X\right)}.
\]
So we may assume $r>8$.

Suppose $f:B_r\to X$ satisfied $d_G\left(x,y\right)\le \left\|f\left(x\right)-f\left(y\right)\right\|_X\le D d_G\left(x,y\right)$ for $x,y\in B_r$. Let $\xi:G\to [0,1]$ be the bump function
\[
\xi\left(g\right)\coloneqq
\begin{cases}
    1,&d_G\left(g,e_G\right)\le r/2,\\
    2-2d_G\left(g,e_G\right)/r,& r/2<d_G\left(g,e_G\right)\le r,\\
    0,&d_G\left(g,e_G\right)>r.
\end{cases}
\]
This equals $1$ on $B_{r/2}$, is supported on $B_{r}$, takes values in $\left[0,1\right]$, and is $2/r$-Lipschitz. Define $F:G\to X$ by
\[
F\coloneqq
\begin{cases}
\left(f-f\left(e_G\right)\right)\xi&\mathrm{on~}B_r,\\
0&\mathrm{on~}G\setminus B_r.
\end{cases}
\]
This is $3D$-Lipschitz, agrees with $f-f\left(e_G\right)$ on $B_{r/2}$, and is supported on $B_r$.

For $h\in B_{r/4}$ and $1\le t\le r^\rho/4^\rho$, we have $h,h\exp\left(tv\right)\in B_{r/2}$ by \eqref{eq:rho-ineq-double}, and so
\[
\left\|F\left(h\exp\left(tv\right)\right)-F\left(h\right)\right\|_X=\left\|f\left(h\exp\left(tv\right)\right)-f\left(h\right)\right\|_X\ge d_G\left(h\exp\left(tv\right),h\right)\stackrel{\eqref{eq:rho-ineq-double}}{\ge }c_1 t^{1/\rho}.
\]
Therefore
\begin{align*}
    \left(\int_1^{r^\rho/4^\rho} \int_{B_{r/4}}\left(\frac{\left\|F\left(h\exp\left(tv\right)\right)-F\left(h\right)\right\|_X}{t^{1/\rho}}\right)^qd\mu\left(h\right)\frac{dt}{t}\right)^{1/q}&\ge c_1\left(\int_1^{r^\rho/4^\rho} \int_{B_{r/4}}d\mu\left(h\right)\frac{dt}{t}\right)^{1/q}\\
    &=c_1\mu\left(B_{r/4}\right)^{1/q}\rho^{1/q}\left(\log\left(r/4\right)\right)^{1/q}.
\end{align*}
On the other hand, we have, denoting $k=\dim \mathfrak{g}$,
\[
\left\|F\right\|_{\dot{Ch}_0^{1,q}\left(G;X\right)}\stackrel{\eqref{eq:cheeger-lipschitz}}{\le} \sqrt{k}\operatorname{Lip}\left(F\right)\mu\left(\overline{\operatorname{supp}F}\right)^{1/q}\lesssim \sqrt{k}D\mu\left(B_r\right)^{1/q}.
\]

We therefore have
\begin{align*}
c_1\mu\left(B_{r/4}\right)^{1/q}\rho^{1/q}\left(\log\left(r/4\right)\right)^{1/q}&\le \left(\int_1^{r^\rho/4^\rho} \int_{B_{r/4}}\left(\frac{\left\|F\left(h\exp\left(tv\right)\right)-F\left(h\right)\right\|_X}{t^{1/\rho}}\right)^qd\mu\left(h\right)\frac{dt}{t}\right)^{1/q}\\
&\stackrel{\mathclap{\eqref{eq:continuous main}}}{\lesssim} \left(\rho+\frac{1}{\rho-1}\right) K_q\left(X\right) \left\|F\right\|_{\dot{Ch}_0^{1,q}\left(G;X\right)}\lesssim_G \sqrt{k}\left(\rho+\frac{1}{\rho-1}\right) K_q\left(X\right)\mu\left(B_r\right)^{1/q}D.
\end{align*}
This gives the stated estimate.

\item
Let $r\ge 2$. The lower bound follows from part (1). The upper bound follows from Corollary \ref{cor:doubling-loc-emb}(2) with $R=1$, noting that $G$ must be $K^4$-doubling.

\end{enumerate}
\end{proof}

Now we begin the proofs of the remaining statements. First we observe that Theorem \ref{mainthm:nilpotentdistortion} follows from Theorem \ref{thm:precise-dist-nilp}, which in turn is a special case of Theorem \ref{thm:gendistortion}.

The proof of Theorem \ref{mainthm:polygrowthdistortion} from Theorem \ref{thm:discrete} is similar to the proof of Theorem \ref{thm:gendistortion}.
\begin{proof}[Proof of Theorem \ref{mainthm:polygrowthdistortion} from Theorem \ref{thm:discrete}]
Let $c$ be as in Theorem \ref{thm:discrete}. It is enough to show that $c_X\left(B_{\left(c+1\right)n}\right)\gtrsim_\Gamma \frac{\left(\log n\right)^{1/q}}{K_q\left(X\right)}$.
Suppose that $f:\Gamma\to X$ satisfies $d_{W}\left(x,y\right)\le \left\|f\left(x\right)-f\left(y\right)\right\|_X\le D d_W\left(x,y\right)$ for all $x,y\in B^\Gamma_{\left(c+1\right)n}$. From the discussion after Theorem \ref{thm:discrete}, $v_\Gamma$ can be chosen so that
\[
d_W\left(v_\Gamma^k,e_\Gamma\right)\asymp_\Gamma k^{1/s} ,\quad k\in \mathbb{N},
\]
while we have also that $\left|B_m\right|\asymp_\Gamma m^{n_h}$ for every $m\in \mathbb{N}$ (see \cite[Theorem 1.1]{breuillard2012nilpotent}). Thus, Theorem \ref{thm:discrete}, in particular \eqref{eq:main}, applied to $f$ yields the following estimate.
\begin{align*}
&n^{n_h/q}\left(\log n\right)^{1/q}\lesssim_\Gamma
\left(\sum_{k=1}^{n^\rho} n^{n_h}\frac{k^{q/\rho}}{k^{1+q/\rho}}\right)^{1/q}\lesssim_\Gamma 
\left(\sum_{k=1}^{n^\rho}\sum_{x\in B^\Gamma_n}\frac{ \left\|f\left(xv_\Gamma^k\right)-f\left(x\right)\right\|_X^q}{k^{1+q/\rho}}\right)^{1/q}\\
&\stackrel{\mathclap{\eqref{eq:main}}}{\lesssim}_{\Gamma}
K_q\left(X\right)\left( \sum_{x\in B^\Gamma_{cn}} \sum_{a\in S}\left\|f\left(xa\right)-f\left(x\right)\right\|^q_X\right)^{1/q}
\lesssim_\Gamma K_q\left(X\right) n^{n_h/q}D,
\end{align*}

which gives Theorem \ref{mainthm:polygrowthdistortion}.
\end{proof}
Note that we don't have precise control on the dependence of the constants on $\Gamma$ in Theorem \ref{mainthm:polygrowthdistortion}, since we don't have such control in Theorem \ref{thm:discrete}.

\begin{proof}[Proof of Corollary \ref{cor:polygrowthlp}]
The lower bound of Corollary \ref{cor:polygrowthlp} follows from Theorem \ref{mainthm:polygrowthdistortion}. The upper bound follows from a version of the Assouad embedding theorem, given as Theorem \ref{thm:lp-assouad}(1) below:
\begin{equation*}
c_p\left(\Gamma,d_\Gamma^{1-\varepsilon}\right)\lesssim_{\Gamma,p} 1/\varepsilon^{1/\max\left\{p,2\right\}},\quad 0<\varepsilon<1,
\end{equation*}
Setting $\varepsilon=\frac{1}{2\log n}$, the space $\left(B^\Gamma_n,d_W^{1-1/2\log n}\right)$ is $O\left(1\right)$-bilipschitz equivalent to $\left(B^\Gamma_n,d_W\right)$, so we have the upper bound.
\end{proof}

\begin{proof}[Proof of Corollary \ref{cor:nilpotentlp}]
Let $r\ge 2$. The lower bound of Corollary \ref{cor:nilpotentlp} follows from Theorem \ref{mainthm:nilpotentdistortion}. The upper bound follows from Theorem \ref{thm:gendistortion}(2).
\end{proof}

\begin{proof}[Proof of Theorem \ref{thm:discrete-snowflake}]
By \eqref{eq:main}, any $D$-bilipschitz embedding $\left(\Gamma,d_W^{1-\varepsilon}\right) \to X$ must satisfy for all $n\in \mathbb{N}$ the following bound:
\begin{align*}
\left|B_n^\Gamma\right|^{1/q}\left(\sum_{k=1}^{n^\rho}\frac 1{k^{1+\varepsilon q/\rho}}\right)^{1/q}&\lesssim_\Gamma
\left(\sum_{k=1}^{n^\rho}\sum_{x\in B^\Gamma_n}\frac{ \left\|f\left(xv_\Gamma^k\right)-f\left(x\right)\right\|_X^q}{k^{1+q/\rho}}\right)^{1/q}\stackrel{\mathclap{\eqref{eq:main}}}{\lesssim}_{\Gamma}
K_q\left(X\right)\left( \sum_{x\in B^\Gamma_{cn}} \sum_{a\in S}\left\|f\left(xa\right)-f\left(x\right)\right\|^q_X\right)^{1/q}\\
&\lesssim_\Gamma
 K_q\left(X\right) D\left|S\right|^{1/q}\left|B_{cn}^\Gamma\right|^{1/q},
\end{align*}
which gives the stated lower bound. The second assertion, namely the estimate for $c_p\left(\Gamma, d_W^{1-\varepsilon}\right)$, follows from a version of the Assouad embedding theorem \cite{assouad1983plongements} given as Theorem \ref{thm:lp-assouad}(1) in Appendix \ref{app:assouad}:
\begin{equation*}
c_p\left(\Gamma,d_W^{1-\varepsilon}\right)\lesssim_{\Gamma,p} 1/\varepsilon^{1/\max\left\{p,2\right\}},\quad 0<\varepsilon<1.
\end{equation*}
\end{proof}

\begin{proof}[Proof of Corollary \ref{cor:growth-function-char}]
    The if direction is due to \cite[Theorem 1]{tessera2011asymptotic}. The only if direction is given by Theorem \ref{thm:discrete}, from which it follows that
\[
s \int_1^{2n} \left(\frac{\theta\left(t\right)}{t}\right)^{\max\left\{p,2\right\}}\frac{dt}{t}=\int_1^{\left(2n\right)^s} \frac{\theta\left(\tau^{1/s}\right)^{\max\left\{p,2\right\}}}{\tau^{1+\max\left\{p,2\right\}/s}}d\tau\lesssim \sum_{k=1}^{\left(2n\right)^s} \frac{\theta\left(k^{1/s}\right)^{\max\left\{p,2\right\}}}{k^{1+\max\left\{p,2\right\}/s}}\lesssim_{\Gamma,p} 1. 
\]
\end{proof}

Now we compute the distortion of nets.
\begin{theorem}\label{thm:gen-net}
Let the Lie group $G$, $v\in Z\left(\mathfrak{g}\right)$, and Banach space $X$ satisfy the hypotheses of Theorem \ref{thm:cvx}. 
Let $r_1,r_2>0$ with $r_1\ge 2r_2$.
Assume that
\begin{equation}\label{eq:rho-ineq-double-second}
c_1t^{1/\rho}\le d_G\left(\exp\left(tv\right),e_G\right)\le t^{1/\rho}, \quad \forall t\ge r_2^\rho,
\end{equation}
and again assume the Haar measure $\mu$ of $G$ is measure-doubling:
\[
\mu\left(B_{2r}\right)\le K\mu\left(B_r\right),\quad r>0.
\]

\begin{enumerate}
    \item  If $N_{r_1,r_2}$ is an $r_2$-covering of $B_{r_1}$, then
\[
c_X\left(N_{r_1,r_2}\right)\gtrsim \left(\frac{c_1\rho^{1/q}}{\sqrt{k}\left(\rho+\frac{1}{\rho-1}\right)K^{2/q}\log K}\right)\frac{\left(\log \left(r_1/r_2\right)\right)^{1/q}}{K_q\left(X\right)},
\]
where $k$ is the number of vectors $X_1,\cdots,X_k$ which satisfy the H\"ormander condition and with which we measure distances on $G$.
\item For an auxiliary parameter $r_3\in \left(0,\frac{r_1}{2}\right]$, if  $N_{r_1,r_2,r_3}$ is an $\left(r_2,r_3\right)$-net of $B_{r_1}$, then for $1<p<\infty$,
\[
\frac{c_1\rho^{1/\max\left\{p,2\right\}}\sqrt{\min\left\{p,2\right\}-1}}{\sqrt{k}\left(\rho+\frac{1}{\rho-1}\right)K^{2/\max\left\{p,2\right\}}\log K} \cdot \left(\log \left(r_1/r_2\right)\right)^{1/\max\left\{p,2\right\}}\lesssim c_p\left(N_{r_1,r_2,r_3}\right)\lesssim \max\left\{\frac{\log K}{p},1\right\}\left(\log \left(r_1/r_3\right)\right)^{1/\max\left\{p,2\right\}}.
\]
\end{enumerate}
\end{theorem}
Note that if $Z\left(\mathfrak{g}\right)\subset \operatorname{span}\left\{X_1,\cdots,X_k\right\}$, in order to satisfy \eqref{eq:rho-ineq-double-second} with $c_1$ constant we require in addition $r_1,r_2=\Omega(1)$.
\begin{proof}\,

\begin{enumerate}
\item
Let $N_{r_1,r_2}$ be an $r_2$-covering of $B_{r_1}$ and let $f:N_{r_1,r_2}\to X$ be a mapping with
\[
d_G\left(x,y\right)\le \left\|f\left(x\right)-f\left(y\right)\right\|\le D d_G\left(x,y\right),\quad x,y\in N_{r_1,r_2}.
\]
By \cite[Theorem 1.6]{lee2005extending} there is an extension $F:G\to X$ of $f$ that is $CD\log K$-Lipschitz for a universal constant $C>0$. Let $\xi:G\to [0,1]$ be the bump function
\[
\xi\left(g\right)\coloneqq
\begin{cases}
    1,&d_G\left(g,e_G\right)\le r_1,\\
    2-d_G\left(g,e_G\right)/r_1,& r_1<d_G\left(g,e_G\right)\le 2r_1,\\
    0,&d_G\left(g,e_G\right)>2r_1.
\end{cases}
\]
This equals $1$ on $B_{r_1}$, is supported on $B_{2r_1}$, takes values in $\left[0,1\right]$, and is $1/r_1$-Lipschitz. Then $\phi\coloneqq \xi (F-F\left(e_G\right))$ is $O\left(D\log K\right)$-Lipschitz, equals $F$ on $B_{r_1}$ and is supported on $B_{2r_1}$. In particular, $\phi\in Ch_0^{1,p}\left(G;X\right)$. We wish to apply Theorem \ref{thm:cvx} with $p=q$. We first estimate
\[
    \left\|\phi\right\|_{\dot{Ch}_0^{1,p}\left(G;X\right)}\stackrel{\mathclap{\eqref{eq:cheeger-lipschitz}}}{\le} \sqrt{k}\operatorname{Lip}\left(\phi\right)\mu\left(\operatorname{supp}\phi\right)^{1/q}\lesssim \sqrt{k} D\left(\log K\right) \mu\left(B_{2r_1}\right)^{1/q}.
\]
On the other hand, for $x\in B_{r_1/2}$ and $t\in \left(\left(\frac{\left(4+4CD\log K\right)r_2}{c_1}\right)^{\rho},\left(r_1/2\right)^\rho\right]$ (we may assume $\left(\frac{\left(4+4CD\log K\right)r_2}{c_1}\right)^{\rho}\le \left(r_1/2\right)^\rho$ since otherwise $D>\frac{c_1\left(r_1/r_2\right)}{16C\log K}$ and the proof is complete), we have $\phi\left(x\right)=F\left(x\right)-F\left(e_G\right)$ and $\phi\left(x\exp\left(tv\right)\right)=F\left(x\exp\left(tv\right)\right)-F\left(e_G\right)$. Choosing $n_1,n_2\in N$ such that $d_G\left(x,n_1\right)<r_2$, $d_G\left(x\exp\left(tv\right),n_2\right)<r_2$, we have
\begin{align*}
\left\|F\left(x\exp\left(tv\right)\right)-F\left(x\right)\right\|&\ge \left\|f\left(n_1\right)-f\left(n_2\right)\right\|-\left\|f\left(n_1\right)-F\left(x\right)\right\|-\left\|f\left(n_2\right)-F\left(x\exp\left(tv\right)\right)\right\|\\
&\ge\left\|f\left(n_1\right)-f\left(n_2\right)\right\|-2CDr_2\log K\ge d\left(n_1,n_2\right)-2C Dr_2\log K\\
&\ge d\left(x,x\exp\left(tv\right)\right)-d\left(x,n_1\right)-d\left(x\exp\left(tv\right),n_2\right)-2C Dr_2\log K\\
&\stackrel{\mathclap{\eqref{eq:rho-ineq-double-second}}}{\ge} \left(c_1t^{1/\rho}-2r_2\right)-2CDr_2\log K>\frac {c_1}2 t^{1/\rho},
\end{align*}
where in the last inequality we used that $t>\left(\frac{\left(4+4CD\log K\right)r_2}{c_1}\right)^{\rho}$.
Thus
\begin{align*}
&\left(\int_0^\infty \int_{G}\left(\frac{\left\|\phi\left(h\exp\left(tv\right)\right)-\phi\left(h\right)\right\|_X}{t^{1/\rho}}\right)^qd\mu\left(h\right)\frac{dt}{t}\right)^{1/q} \ge \left(\int_{\left(\frac{\left(4+4CD\log K\right)r_2}{c_1}\right)^{\rho}}^{\left(r_1/2\right)^\rho}\int_{B_{r_1/2}} \frac{c_1^q}{2^q}d\mu\left(h\right) \frac{dt}{t}\right)^{1/q}\\
&\ge \frac {c_1\rho^{1/q}}2 \left(\log\left(r_1/r_2\right)-\log\left(\frac{8+8CD\log K}{c_1}\right)\right)^{1/q}\mu\left(B_{r_1/2}\right)^{1/q}.
\end{align*}
Thus by Theorem \ref{thm:cvx} with $p=q$,
\[
\frac {c_1\rho^{1/q}}2 \left(\log\left(r_1/r_2\right)-\log\left(\frac{8+8CD\log K}{c_1}\right)\right)^{1/q}\mu\left(B_{r_1/2}\right)^{1/q}\lesssim \sqrt{k}\left(\rho+\frac{1}{\rho-1}\right)K_q\left(X\right)D\left(\log K\right) \mu\left(B_{2r_1}\right)^{1/q}.
\]
But we may also assume $\log\left(r_1/r_2\right)-\log\left(\frac{8+8CD\log K}{c_1}\right)\ge \frac12 \log \left(r_1/r_2\right)$, for otherwise $D>\frac{c_1\sqrt{r_1/r_2}}{16C\log K}$ and the proof is complete. Thus
\[
c_1\rho^{1/q}\left(\log \left(r_1/r_2\right)\right)^{1/q}\lesssim \sqrt{k}\left(\rho+\frac{1}{\rho-1}\right)K_q\left(X\right)D K^{2/q}\log K,
\]
from which it follows that
\[
D\gtrsim \left(\frac{c_1\rho^{1/q}}{\sqrt{k}\left(\rho+\frac{1}{\rho-1}\right)K^{2/q}\log K}\right)\frac{\left(\log \left(r_1/r_2\right)\right)^{1/q}}{K_q\left(X\right)}.
\]

\item 
Note that the lower bound follows from Theorem \ref{thm:gen-net}(1) and inequality \eqref{eq:unifcvxlp}. To see the upper bound, observe that by Theorem \ref{thm:lp-assouad}(1) in Appendix \ref{app:assouad}, we have
\begin{equation*}
c_p\left(G,d_G^{1-\varepsilon}\right)\lesssim_{p} 1+\frac{(1-\varepsilon)^{1/\min\left\{p,2\right\}}}{\varepsilon^{1/\max\left\{p,2\right\}}}\max\left\{\frac{\log K}{p},1\right\},\quad 0<\varepsilon<1.
\end{equation*}
Since $N_{r_1,r_2,r_3}$ is an $r_3$-packing of $B_{r_1}$, we have $r_3\le d_G\left(n_1,n_2\right)\le 2r_1$ for distinct $n_1,n_2\in N_{r_1,r_2,r_3}$, so $\left(N_{r_1,r_2,r_3},d_G^{1-1/2\log \left(r_1/r_3\right)}\right)$ is $O\left(1\right)$-bilipschitz equivalent to $\left(N_{r_1,r_2,r_3},d_G\right)$. This completes the proof of (2).
\end{enumerate}
\end{proof}

We observe that Theorem \ref{thm:net} follows from Theorem \ref{thm:gen-net}.

\begin{theorem}\label{thm:gen-net-snowflake}
Let the Lie group $G$ and $v\in Z\left(\mathfrak{g}\right)$ satisfy the hypotheses of Theorem \ref{thm:cvx}. Again assume \eqref{eq:rho-ineq-double-second}:
\[\tag{\ref{eq:rho-ineq-double-second}}
	c_1t^{1/\rho}\le d_G\left(\exp\left(tv\right),e_G\right)\le t^{1/\rho}, \quad \forall t\ge r_2^\rho,
\]
and again assume the Haar measure $\mu$ of $G$ is measure-doubling:
\[
\mu\left(B_{2r}\right)\le K\mu\left(B_r\right),\quad r>0.
\]
Recall that $k$ is the number of vectors $X_1,\cdots,X_k$ which satisfy the H\"ormander condition and with which we measure distances on $G$.

Let $r_1,r_2>0$ with $2r_1\ge r_2$. If $Z\left(\mathfrak{g}\right)\subset \operatorname{span}\left\{X_1,\cdots,X_k\right\}$, we require in addition $r_1,r_2>1$. Let $N'_{r_1,r_2}$ be an $\left(r_1,r_2\right)$-net of $G$.

\begin{enumerate}
    \item  If $X$ is a $q(\ge 2)$-uniformly convex Banach space, then
\[
c_X\left(G,d_G^{1-\varepsilon}\right) \ge c_X\left(N'_{r_1,r_2},d_G^{1-\varepsilon}\right)\gtrsim \left(\frac{c_1^{1-\varepsilon}\left(r_2/r_1\right)^{\varepsilon}\rho^{\left(1-\varepsilon\right)/q}}{k^{\left(1-\varepsilon\right)/2}\left(\rho+\frac{1}{\rho-1}\right)^{1-\varepsilon} K^{2\left(1-\varepsilon\right)/q}\log K}\right)\frac{1}{K_q\left(X\right)^{1-\varepsilon}\varepsilon^{1/q}},\quad 0<\varepsilon<1,
\]
\item For $1<p<\infty$,
\begin{align*}
&\left(\frac{c_1^{1-\varepsilon}\left(r_2/r_1\right)^{\varepsilon}\rho^{\left(1-\varepsilon\right)/\max\left\{p,2\right\}}}{k^{\left(1-\varepsilon\right)/2}\left(\rho+\frac{1}{\rho-1}\right)^{1-\varepsilon} K^{2\left(1-\varepsilon\right)/\max\left\{p,2\right\}}\log K}\right)\frac{\left(\min\left\{p,2\right\}-1\right)^{(1-\varepsilon)/2}}{\varepsilon^{1/\max\left\{p,2\right\}}}\\
& \lesssim c_p\left(N'_{r_1,r_2},d_G^{1-\varepsilon}\right)\le c_p\left(G,d_G^{1-\varepsilon}\right)\\
&\lesssim 1+\frac{(1-\varepsilon)^{1/\min\left\{p,2\right\}}}{\varepsilon^{1/\max\left\{p,2\right\}}}\max\left\{\frac{\log K}{p},1\right\}.
\end{align*}
\end{enumerate}
\end{theorem}
\begin{proof}\,
\begin{enumerate}
\item
Let $N\coloneqq N'_{r_1,r_2}$ be a $\left(r_1,r_2\right)$-net of $G$ and let $f:N\to X$ be a mapping with
\[
d_G\left(x,y\right)^{1-\varepsilon}\le \left\|f\left(x\right)-f\left(y\right)\right\|\le D d_G\left(x,y\right)^{1-\varepsilon},\quad x,y\in N.
\]
Then $f$ is $r_2^{-\varepsilon}D$-Lipschitz, so by \cite[Theorem 1.6]{lee2005extending} there is an extension $F:G\to X$ of $f$ that is $Cr_2^{-\varepsilon}D\log K$-Lipschitz for a universal constant $C>0$. Fix $R>0$ sufficiently large, let $\xi:G\to [0,1]$ be the bump function
\[
\xi\left(g\right)\coloneqq
\begin{cases}
    1,&d_G\left(g,e_G\right)\le R,\\
    2-d_G\left(g,e_G\right)/R,& R<d_G\left(g,e_G\right)\le 2R,\\
    0,&d_G\left(g,e_G\right)>2R.
\end{cases}
\]
This equals $1$ on $B_R$, is supported on $B_{2R}$, takes values in $\left[0,1\right]$, and is $1/R$-Lipschitz. Then $\phi\coloneqq \xi \left(F-F\left(e_G\right)\right)$ is $O\left(r_2^{-\varepsilon}D\log K\right)$-Lipschitz, equals $F-F\left(e_G\right)$ on $B_R$ and is supported on $B_{2R}$. In particular, $\phi\in Ch_0^{1,p}\left(G;X\right)$. In order to apply Theorem \ref{thm:cvx} with $p=q$, we note that
\begin{equation}\label{eq:phi-cheeger-ub}
    \left\|\phi\right\|_{\dot{Ch}_0^{1,p}\left(G;X\right)}\stackrel{\mathclap{\eqref{eq:cheeger-lipschitz}}}{\le} \sqrt{k}\operatorname{Lip}\left(\phi\right)\mu\left(\operatorname{supp}\phi\right)^{1/q}\lesssim \sqrt{k}r_2^{-\varepsilon}D\left(\log K\right) \mu\left(B_{2R}\right)^{1/q}.
\end{equation}

On the other hand, for $x\in B_{R/2}$ and $t\in \left(\left(8r_1r_2^{-\varepsilon}D\log K\right)^{\rho/\left(1-\varepsilon\right)}/c_1^\rho,\left(R/2\right)^\rho\right]$ (we can make this interval nonempty by taking $R>0$ sufficiently large) we have $\phi\left(x\right)=F\left(x\right)-F\left(e_G\right)$ and $\phi\left(x\exp\left(tv\right)\right)=F\left(x\exp\left(tv\right)\right)-F\left(e_G\right)$. Note also that
\begin{equation}\label{eq:t-big}
t> \left(8r_1r_2^{-\varepsilon}D\log K\right)^{\rho/\left(1-\varepsilon\right)}/c_1^\rho\ge 4^\rho r_1^\rho / c_1^\rho,
\end{equation}
the last inequality following from
\[
2^{1+2\varepsilon}\left(r_1/r_2\right)^{\varepsilon}D\log K\ge 2^{1+2\varepsilon}\left(1/2\right)^{\varepsilon}\log 2\ge  1.
\]
 Choosing $n_1,n_2\in N$ such that $d_G\left(x,n_1\right)<r_1$, $d_G\left(x\exp\left(tv\right),n_2\right)<r_1$, we have
\begin{equation}
d\left(n_1,n_2\right)\ge d\left(x,x\exp\left(tv\right)\right)-d\left(n_1,x\right)-d\left(n_2,x\exp\left(tv\right)\right)\stackrel{\mathclap{\eqref{eq:rho-ineq-double-second}}}{\ge} c_1t^{1/\rho}-2r_1\stackrel{\mathclap{\eqref{eq:t-big}}}{\ge} \frac{c_1}{2}t^{1/\rho}
\end{equation}
and
\begin{align*}
\left\|F\left(x\exp\left(tv\right)\right)-F\left(x\right)\right\|&\ge \left\|f\left(n_1\right)-f\left(n_2\right)\right\|-\left\|f\left(n_1\right)-F\left(x\right)\right\|-\left\|f\left(n_2\right)-F\left(x\exp\left(tv\right)\right)\right\|\\
&\ge \left\|f\left(n_1\right)-f\left(n_2\right)\right\|-2r_1r_2^{-\varepsilon}CD\log K\ge d\left(n_1,n_2\right)^{1-\varepsilon}-2r_1r_2^{-\varepsilon}CD\log K\\
&\ge \left(\frac{c_1}{2}t^{1/\rho}\right)^{1-\varepsilon}-2r_1r_2^{-\varepsilon}CD\log K>\frac {c_1^{1-\varepsilon}}4 t^{\left(1-\varepsilon\right)/\rho},
\end{align*}
the last inequality following from \eqref{eq:t-big}.
Thus
\begin{align*}
&\left(\int_0^\infty \int_{G}\left(\frac{\left\|\phi\left(h\exp\left(tv\right)\right)-\phi\left(h\right)\right\|_X}{t^{1/\rho}}\right)^qd\mu\left(h\right)\frac{dt}{t}\right)^{1/q} \ge \left(\int_{\left(8r_1r_2^{-\varepsilon}D\log K\right)^{\rho/\left(1-\varepsilon\right)}/c_1^\rho}^{\left(R/2\right)^\rho}\int_{B_{R/2}} \frac{c_1^{q\left(1-\varepsilon\right)}}{4^q}d\mu\left(h\right) \frac{dt}{t^{1+\varepsilon q/\rho}}\right)^{1/q}\\
&\ge \frac {c_1^{1-\varepsilon}}4 \left(\frac{\rho}{\varepsilon q}\left[\left(8r_1r_2^{-\varepsilon}D\log K\right)^{-\varepsilon q/\left(1-\varepsilon\right)}/c_1^{-\varepsilon q}-\left(R/2\right)^{-\varepsilon q}\right]\right)^{1/q}\mu\left(B_{R/2}\right)^{1/q}.
\end{align*}
Thus by Theorem \ref{thm:cvx} with $p=q$, and \eqref{eq:phi-cheeger-ub},
\[
\frac {c_1^{1-\varepsilon}}4 \left(\frac{\rho}{\varepsilon q}\left[\left(8r_1r_2^{-\varepsilon}D\log K\right)^{-\varepsilon q/\left(1-\varepsilon\right)}/c_1^{-\varepsilon q}-\left(R/2\right)^{-\varepsilon q}\right]\right)^{1/q}\mu\left(B_{R/2}\right)^{1/q}\lesssim K_q\left(X\right)\left(\rho+\frac{1}{\rho-1}\right)\sqrt{k}r_2^{-\varepsilon}D\left(\log K\right) \mu\left(B_{2R}\right)^{1/q},
\]
which by the doubling condition implies
\[
c_1^{1-\varepsilon} \left(\frac{\rho}{\varepsilon q}\left[\left(8r_1r_2^{-\varepsilon}D\log K\right)^{-\varepsilon q/\left(1-\varepsilon\right)}/c_1^{-\varepsilon q}-\left(R/2\right)^{-\varepsilon q}\right]\right)^{1/q}\lesssim K_q\left(X\right)\left(\rho+\frac{1}{\rho-1}\right)\sqrt{k}r_2^{-\varepsilon}DK^{2/q}\log K,
\]
Since $R>0$ was arbitrarily large,
\[
\frac{c_1\rho^{1/q}}{\varepsilon^{1/q} }\left(8r_1r_2^{-\varepsilon}D\log K\right)^{-\varepsilon /\left(1-\varepsilon\right)}\lesssim K_q\left(X\right)\left(\rho+\frac{1}{\rho-1}\right)\sqrt{k}r_2^{-\varepsilon}D K^{2/q}\log K.
\]
It follows that
\[
D\gtrsim \left(\frac{c_1^{1-\varepsilon}\left(r_2/r_1\right)^{\varepsilon}\rho^{\left(1-\varepsilon\right)/q}}{k^{\left(1-\varepsilon\right)/2}\left(\rho+\frac{1}{\rho-1}\right)^{1-\varepsilon} K^{2\left(1-\varepsilon\right)/q}\log K}\right)\frac{1}{K_q\left(X\right)^{1-\varepsilon}\varepsilon^{1/q}}.
\]
\item The lower bound follows from Theorem \ref{thm:gen-net-snowflake}(1) and inequality \eqref{eq:unifcvxlp}. The upper bound follows from Theorem \ref{thm:lp-assouad}(1) in Appendix \ref{app:assouad}.
\end{enumerate}
\end{proof}
Theorem \ref{thm:precise-holder-nilp} and Theorem \ref{thm:net-snowflake} follow from Theorem \ref{thm:gen-net-snowflake}.

\section{Proof of Theorem \ref{thm:amenablesublinear}: sublinear growth of cocycles}\label{sec:amenable}
Our goal in this section is to prove Theorem \ref{thm:amenablesublinear} and Remark \ref{rem:qualitativesublinear}. The proof of Theorem \ref{thm:amenablesublinear} generalizes that of \cite{austin2013sharp}.

As in the setting of Theorem \ref{thm:amenablesublinear}, let $\left(X,\left\|\cdot\right\|_X\right)$ be a $q$-uniformly convex space, and let $\Gamma$ be a finitely generated amenable group with word distance induced by a finite symmetric generating set $S$, with $v\in Z\left(\Gamma\right)$ and $\rho>1$ such that $d_W\left(v^k,e_\Gamma\right)\asymp_\Gamma k^{1/\rho}$, $k\in \mathbb{N}$. Let $f:\Gamma\to X$ be a $1$-Lipschitz function.

Theorem \ref{thm:amenablesublinear} claims that for each $\varepsilon>0$, for every sufficiently large $t\ge 3$ there is an integer $n\in \left[t,t^{1+\varepsilon}\right]$ such that
\[
\frac{\omega_f\left(n\right)}{n}\lesssim_\Gamma K_q\left(X\right)\left(\frac{\log \log n}{\log n}\right)^{1/q},
\]
while Remark \ref{rem:qualitativesublinear} claims that, for some positive integer $a>0$,
\[
\frac{\omega_f\left(d_W\left(v^{ak},e_\Gamma\right)\right)}{d_W\left(v^{ak},e_\Gamma\right)}\to 0\quad \mathrm{as}~k\to\infty.
\]

By \cite[Theorem 9.1]{naor2011lp} and by amenability of $\Gamma$, if $X$ is $q$-uniformly convex and $f:\Gamma\to X$ is $1$-Lipschitz, then there exists a Banach space $\left(Y,\left\|\cdot\right\|_Y\right)$ that is also $q$-uniformly convex with $K_q\left(Y\right)=K_q\left(X\right)$, an action $\pi$ of $\Gamma$ on $Y$ by linear isometric automorphisms, and a $1$-cocycle $F:\Gamma\to Y$ such that $\omega_F=\omega_f$, where $F:\Gamma\to Y$ is a $1$-cocycle, denoted by $F\in Z^1\left(\pi\right)$, if $F\left(xy\right)=\pi\left(x\right)F\left(y\right)+F\left(x\right)$ for all $x,y\in \Gamma$. Note that the $1$-Lipschitz requirement for $F$ is equivalent to $\left\|F\left(a\right)\right\|_Y\le 1$ for all $a\in S$. Thus, for the purposes of proving Theorem \ref{thm:amenablesublinear} and Remark \ref{rem:qualitativesublinear}, we may assume without loss of generality that $f\in Z^1\left(\pi\right)$ for some action $\pi$ of $\Gamma$ on $X$ by linear isometric automorphisms.

Remark \ref{rem:qualitativesublinear} follows from a combination of \cite[Proposition 3.5]{das2016integrable} and \cite{naor2011lp}, as follows.\footnote{I thank the anonymous referee for suggesting Remark \ref{rem:qualitativesublinear} and pointing out the reference \cite[Proposition 3.5]{das2016integrable} and suggesting to use it in conjunction with \cite{naor2011lp} to prove Remark \ref{rem:qualitativesublinear}.}

\begin{proof}[Proof of Remark \ref{rem:qualitativesublinear}]
    The reference \cite[Proposition 3.5]{das2016integrable} together with \cite[Example 3.3]{das2016integrable} states that if $\Gamma$ is a finitely generated group with a word distance, $w\in \Gamma$ is a torsion-free central element in the commutator subgroup of $\Gamma$, and $\pi$ is a action of $\Gamma$ on a reflexive Banach space $X$ by linear isometric automorphisms, then for every cocycle $f\in Z^1\left(\pi\right)$ we have
    \[
    \lim_{k\to\infty}\frac{\left\|f\left(w^k\right)\right\|_X}{d_W\left(w^k,e_\Gamma\right)}=0.
    \]
    
    Thus, it is enough to verify that some positive power of $v$ is in the commutator subgroup $\left[\Gamma,\Gamma\right] $ of $\Gamma$. Suppose not. Then the quotient map $p:\Gamma\to\Gamma/\left[\Gamma,\Gamma\right]$ sends $v$ and its powers to nontrivial elements, i.e., $p\left(v\right)$ is a torsion-free element of $\Gamma/\left[\Gamma,\Gamma\right]$, and $p$ is a 1-Lipschitz map when we give $\Gamma/\left[\Gamma,\Gamma\right]$ the left-invariant word distance $d_{W'}$ induced by $p\left(S\right)$. But because $\Gamma/\left[\Gamma,\Gamma\right]$ is finitely generated abelian, we must have $d_{W'}\left(p\left(v\right)^k,e_{\Gamma/\left[\Gamma,\Gamma\right]}\right)\asymp_\Gamma k$, which contradicts the fact that $d_W\left(v^k,e_\Gamma\right)\asymp_\Gamma k^{1/\rho}$ and that $p$ is 1-Lipschitz.
\end{proof}

We now begin the proof of the quantitative compression rate, i.e., Theorem \ref{thm:amenablesublinear}. We first recall the following property of $q$-uniformly convex spaces.

\begin{lemma}[{\cite[Lemma 3.1]{austin2013sharp}}]\label{lem:unifcvxseq}
Let $\left(X,\left\|\cdot\right\|_X\right)$ be $q$-uniformly convex. For a fixed $z\in X$ and linear operator $T:X\to X$ with $\left\|T\right\|\le 1$, define
\[
s_n\coloneqq \frac {1}{2^n} \sum_{j=0}^{2^n-1} T^jz,\quad n\ge 0.
\]
Then, for all $l\in \mathbb{N}$, we have
\begin{equation}\label{unifcvxseq}
\sum_{i=0}^\infty \frac{1}{2^l}\sum_{j=0}^{2^l-1}\left\|s_{\left(i+1\right)l}-T^{j2^{il}}s_{il}\right\|_X^q\le \left(2K_q\left(X\right)\right)^q\left\|z\right\|_X^q.
\end{equation}
\end{lemma}

On the other hand, as $X$ is reflexive, $X$ is ergodic \cite[p.662]{dunford1988linear}, i.e., for every linear isometry $T:X\to X$ and $x\in X$ the sequence $\left\{\frac 1n \sum_{j=0}^{n-1} T^jx\right\}_{n=1}^\infty$ converges in norm. Thus, the operator $P:X\to X$ defined by
\[
Px\coloneqq \lim_{N\to\infty} \frac 1N \sum_{n=0}^{N-1}\pi\left(v\right)^n x,\quad x\in X,
\]
is well-defined, has norm $\le 1$, is a contraction onto the subspace $X_0$ of $\pi\left(v\right)$-invariant vectors, and is idempotent. As $v$ is a central element of $\Gamma$, $P$ commutes with $\pi\left(g\right)$ for all $g\in \Gamma$, and so the maps $Pf,\left(I-P\right)f:\Gamma\to X$ are both Lipschitz and are $1$-cocycles.

Now define linear operators $P_n:X\to X$, $n\in \mathbb{N}$, by
\[
P_n\coloneqq \frac{1}{2^n}\sum_{j=0}^{2^n-1}\pi\left(v\right)^j.
\]
Of course, $\left\|P_n\right\|\le 1$ and $P_n$ commutes with $\pi\left(g\right)$ for all $g\in \Gamma$.

The following Lemma was inspired by Lemma 4.1 of \cite{austin2013sharp}.
\begin{lemma}\label{L1}
Let $\left(X,\left\|\cdot\right\|_X\right)$ be a $q$-uniformly convex space. Then for every $l,k,m\in \mathbb{N}$ there exist integers $i\in \left[k+1,k+m\right]$ and $j\in \left[0,2^l-1\right]$ such that, for all $n\in \mathbb{N}$,
\begin{equation}
\left\|\pi\left(v^{-j2^{il}}\right)P_{\left(i+1\right)l}f\left(v^{\lfloor n^\rho\rfloor}\right)-P_{il}f\left(v^{\lfloor n^\rho\rfloor}\right)\right\|_X\lesssim_\Gamma \frac{K_q\left(X\right)n}{m^{1/q}}.
\end{equation}
\end{lemma}
\begin{proof}
By Lemma \ref{lem:unifcvxseq} and the fact that $\left\|f\left(a\right)\right\|_X\le 1$, we have for each $a\in S$
\begin{align*}
\sum_{i=k+1}^{k+m} \frac{1}{2^l}\sum_{j=0}^{2^l-1}\left\|\pi\left(v^{-j2^{il}}\right)P_{\left(i+1\right)l}f\left(a\right)-P_{il}f\left(a\right)\right\|_X^q&=\sum_{i=k+1}^{k+m} \frac{1}{2^l}\sum_{j=0}^{2^l-1}\left\|P_{\left(i+1\right)l}f\left(a\right)-\pi\left(v^{j2^{il}}\right)P_{il}f\left(a\right)\right\|_X^q\\
&\le \left(2K_q\left(X\right)\right)^q.
\end{align*}
Thus
\begin{align*}
\left|S\right|\left(2K_q\left(X\right)\right)^q&\ge \sum_{i=k+1}^{k+m} \frac{1}{2^l}\sum_{j=0}^{2^l-1}\sum_{a\in S}\left\|\pi\left(v^{-j2^{il}}\right)P_{\left(i+1\right)l}f\left(a\right)-P_{il}f\left(a\right)\right\|_X^q\\
&\ge m\min_{\substack{k+1\le i\le k+m\\ 0\le j\le 2^l-1}}\sum_{a\in S}\left\|\pi\left(v^{-j2^{il}}\right)P_{\left(i+1\right)l}f\left(a\right)-P_{il}f\left(a\right)\right\|_X^q,
\end{align*}
and so there exist $i\in \left[k+1,k+m\right]$, $j\in \left[0,2^l-1\right]$ such that
\[
\max_{a\in S}\left\|\pi\left(v^{-j2^{il}}\right)P_{\left(i+1\right)l}f\left(a\right)-P_{il}f\left(a\right)\right\|_X \lesssim_\Gamma \frac{K_q\left(X\right)}{m^{1/q}}.
\]
As $v^{\lfloor n^\rho\rfloor}=a_1\cdots a_{b}$ for some $a_i\in S$, where $b=O_\Gamma\left(n\right)$, we have by the cocycle identity
\[
f\left(v^{\lfloor n^\rho\rfloor}\right)=\sum_{i=1}^b \pi\left(a_1\cdots a_{i-1}\right)f\left(a_i\right).
\]
Since $P_r$, $r\ge 0$, and $\pi\left(g\right)$, $g\in \Gamma$, commute, and $v$ and $g\in \Gamma$ commute, and $\pi\left(g\right) $ is an isometry for $g\in \Gamma$, we have
\begin{align*}
\left\|\pi\left(v^{-j2^{il}}\right)P_{\left(i+1\right)l}f\left(v^{\lfloor n^\rho\rfloor}\right)-P_{il}f\left(v^{\lfloor n^\rho\rfloor}\right)\right\|_X\le \sum_{i=1}^b \left\|\pi\left(v^{-j2^{il}}\right)P_{\left(i+1\right)l}f\left(a_i\right)-P_{il}f\left(a_i\right)\right\|_X\lesssim_\Gamma \frac{K_q\left(X\right)n}{m^{1/q}},
\end{align*}
as claimed.
\end{proof}
The next lemma generalizes Lemma 4.2 of \cite{austin2013sharp}.
\begin{lemma}\label{L2}
For every $m,n\in \mathbb{N}$,
\[
\left\|P_m f\left(v^{\lfloor n^\rho\rfloor}\right)\right\|_X\lesssim_\Gamma \frac{n^{\left(\rho^2+\rho-1\right)/\left(2\rho-1\right)}}{2^{m\left(\rho-1\right)/\left(2\rho-1\right)}}.
\]
\end{lemma}
\begin{proof}
If $\frac{n^\rho}{2^m}\ge \frac{\rho-1}\rho$ then we have the obvious bound
\[
\left\|P_m f\left(v^{\lfloor n^\rho\rfloor}\right)\right\|_X\le \left\| f\left(v^{\lfloor n^\rho\rfloor}\right)\right\|_X\le d_W\left(e_\Gamma,v^{\lfloor n^\rho\rfloor}\right)\lesssim_\Gamma n\left(\frac{n^\rho}{2^m}\right)^{\left(\rho-1\right)/\left(2\rho-1\right)}.
\]
Thus we assume $\frac{n^\rho}{2^m}< \frac{\rho-1}\rho$.

Because
\[
P_m-\pi\left(v^k\right)P_m= \frac{1}{2^m}\sum_{j=0}^{k-1}\pi\left(v\right)^j-\frac{1}{2^m}\sum_{j=2^m}^{2^m+k-1}\pi\left(v\right)^j,
\]
we have
\[
\left\|P_m-\pi\left(v^k\right)P_m\right\|\le \frac{2k}{2^m}.
\]
By the cocycle identity,
\[
f\left(v^{\lfloor k^\rho\rfloor \lfloor n^\rho\rfloor}\right)=\sum_{j=0}^{\lfloor k^\rho\rfloor -1} \pi\left(v^{j\lfloor n^\rho\rfloor }\right)f\left(v^{\lfloor n^\rho\rfloor}\right)
\]
Because $f$ is $1$-Lipschitz, $d_W\left(e_\Gamma,v^{\lfloor k^\rho\rfloor \lfloor n^\rho\rfloor}\right)\lesssim_\Gamma kn$, and $\left\|P_m\right\|\le 1$,
\begin{align*}
kn&\gtrsim_\Gamma \left\|P_mf \left(v^{\lfloor k^\rho\rfloor \lfloor n^\rho\rfloor}\right)\right\|_X=\left\|\lfloor k^\rho\rfloor P_mf\left(v^{\lfloor n^\rho\rfloor}\right) -\sum_{j=0}^{\lfloor k^\rho\rfloor -1} \left(P_m-\pi\left(v^{j\lfloor n^\rho\rfloor }\right)P_m\right)f\left(v^{\lfloor n^\rho\rfloor}\right)\right\|_X\\
&\ge \lfloor k^\rho\rfloor\left\| P_mf\left(v^{\lfloor n^\rho\rfloor}\right)\right\|_X -\sum_{j=0}^{\lfloor k^\rho\rfloor -1} \left\|P_m-\pi\left(v^{j\lfloor n^\rho\rfloor }\right)P_m\right\|\cdot\left\|f\left(v^{\lfloor n^\rho\rfloor}\right)\right\|_X\\
&\ge \lfloor k^\rho\rfloor\left\| P_mf\left(v^{\lfloor n^\rho\rfloor}\right)\right\|_X -\sum_{j=0}^{\lfloor k^\rho\rfloor -1}\frac{2j\lfloor n^\rho\rfloor}{2^m}d_W\left(e_\Gamma,v^{\lfloor n^\rho\rfloor}\right)
\end{align*}
and rearranging terms,
\[
\left\| P_mf\left(v^{\lfloor n^\rho\rfloor}\right)\right\|_X\lesssim_\Gamma \frac{n}{k^{\rho-1}}+\frac{n^{\rho+1}k^\rho}{2^m}.
\]
Since this is true for all $k$, choosing $k=\left\lceil \left(\frac{\left(\rho-1\right)2^m}{\rho n^\rho}\right)^{1/\left(2\rho-1\right)}\right\rceil$ gives the stated bound. Indeed, by $\frac{n^\rho}{2^m}< \frac{\rho-1}\rho$ we have $k\asymp \left(\frac{\left(\rho-1\right)2^m}{\rho n^\rho}\right)^{1/\left(2\rho-1\right)}$ and plugging in gives the desired bound.
\end{proof}
The following generalizes Lemma 4.3 of \cite{austin2013sharp}.
\begin{lemma}\label{L3}
For every $m,n\in \mathbb{N}$,
\[
\left\| f\left(v^{\lfloor n^\rho\rfloor}\right)-P_m f\left(v^{\lfloor n^\rho\rfloor}\right)\right\|_X\lesssim_\Gamma 2^{m/\left(\rho+1\right)}n^{1/\left(\rho+1\right)}.
\]
\end{lemma}
\begin{proof}
If we define $\tilde{f}:\Gamma\to X$ by
\[
\tilde{f}\left(h\right)\coloneqq f\left(h\right)-P_mf\left(h\right)=\left(I-P_m\right)f\left(h\right),
\]
then $\tilde{f}\in Z^1\left(\pi\right)$. Let $k\ge 1$ be an integer to be determined later. If we set
\[
w\coloneqq -\frac 1k \sum_{j=0}^{k-1}\tilde{f}\left(v^j\right),
\]
then
\begin{equation}\label{eq:w-size}
\left\|w\right\|_X\lesssim_\Gamma \frac 1k\sum_{j=0}^{k-1}j^{1/\rho}\lesssim k^{1/\rho}.
\end{equation}

For every $h\in \Gamma$ we have the following identity:
\begin{align}\label{B-1}
\begin{aligned}
-\pi\left(h\right)w+\tilde{f}\left(h\right)&=\frac 1k \sum_{j=0}^{k-1}\left(\pi\left(h\right)\tilde{f}\left(v^j\right)+\tilde{f}\left(h\right)\right)=\frac 1k \sum_{j=0}^{k-1}\tilde{f}\left(hv^j\right)=\frac 1k \sum_{j=0}^{k-1}\tilde{f}\left(v^jh\right)\\
&=\frac 1k \sum_{j=0}^{k-1}\left(\pi\left(v^j\right)\tilde{f}\left(h\right)+\tilde{f}\left(v^j\right)\right)=\frac 1k \sum_{j=0}^{k-1}\pi\left(v^j\right)\tilde{f}\left(h\right)-w.
\end{aligned}
\end{align}
But
\begin{align*}
\frac 1k \sum_{j=0}^{k-1}\pi\left(v^j\right)\tilde{f}\left(h\right)&=\frac 1k \sum_{j=0}^{k-1}\left(\pi\left(v^j\right)f\left(h\right)-\frac{1}{2^m}\sum_{i=0}^{2^m-1}\pi\left(v^{j+i}\right)f\left(h\right)\right)\\
&=\frac{1}{2^m}\sum_{i=0}^{2^m-1}\left(\frac 1k\sum_{j=0}^{k-1}\pi\left(v^j\right)-\frac 1k\sum_{j=i}^{i+k-1}\pi\left(v^j\right)\right)f\left(h\right),
\end{align*}
so that
\[
\left\|\frac 1k \sum_{j=0}^{k-1}\pi\left(v^j\right)\tilde{f}\left(h\right)\right\|_X \le \frac{d_W\left(h,e_\Gamma\right)}{2^m}\sum_{i=0}^{2^m-1}\frac{2i}{k}\le \frac{2^m}{k}d_W\left(h,e_\Gamma\right).
\]
Because of \eqref{B-1}, $\tilde{f}$ is close to a coboundary in the following sense:
\[
\left\|\tilde{f}\left(h\right)-\left(\pi\left(h\right)w-w\right)\right\|\le \frac{2^m}{k}d_W\left(h,e_\Gamma\right).
\]
Writing $v^{\lfloor n^\rho\rfloor}=a_1\cdots a_{b}$ for some $a_i\in S$, where $b=O_\Gamma\left(n\right)$, we have
\begin{align*}
\left\|\tilde{f}\left(v^{\lfloor n^\rho\rfloor}\right)\right\|_X&=\left\|\sum_{i=1}^{b}\pi\left(a_1\cdots a_{i-1}\right)\tilde{f}\left(a_i\right)\right\|_X\lesssim_\Gamma \left\|\sum_{i=1}^{b}\pi\left(a_1\cdots a_{i-1}\right)\left(\pi\left(a_i\right)w-w\right)\right\|_X+\frac{n2^m}{k}\\
&=\left\|\pi\left(v^{\lfloor n^\rho\rfloor}\right)w-w\right\|_X+\frac{n2^m}{k}\stackrel{\mathclap{\eqref{eq:w-size}}}{\lesssim}_\Gamma k^{1/\rho}+\frac{n2^m}{k}.
\end{align*}
With $k=\lceil n^{\frac{\rho}{\rho+1}}2^{\frac{m\rho}{\rho+1}}\rceil$, we have the stated bound.
\end{proof}

We now prove theorem \ref{thm:amenablesublinear}.
\begin{proof}[Proof of Theorem \ref{thm:amenablesublinear}]
Recall that by \cite[Theorem 9.1]{naor2011lp} we may assume that $f\in Z^1\left(\pi\right)$ for some action $\pi$ of $\Gamma$ on $X$ by linear isometric automorphisms.

Let $C_1<C_2$ be constants depending on $\Gamma$ such that
\[
C_1n\le d_W\left(v^{\lfloor n^\rho\rfloor},e_\Gamma\right)\le C_2 n,\quad n\in \mathbb{N}.
\]
Let $\varepsilon>0$. Take $N_\varepsilon>e^e$ sufficiently large such that for $t\ge N_\varepsilon$, if $m$ is the largest integer such that
\[
m\le \frac{\varepsilon q\rho\left(\rho-1\right)}{\rho^2+2\rho-2} \cdot \frac{\log t}{\log\log t},
\]
then
\begin{equation}\label{eq:def m}
\log\frac{2C_2}{C_1}+\frac{q\left(\rho-1\right)\log 2}{\rho^2+2\rho-2}\frac{\log t}{\log m}+\frac{m}{\rho}\log 2+\frac{m\log m\left(\rho^2+2\rho-2\right)}{q\rho\left(\rho-1\right)}<\varepsilon \log t.
\end{equation}

Given $m$, let $k$ be the smallest integer such that
\begin{equation}\label{eq:def k}
m^{\frac{\rho+1}{\rho q}+\frac{\left(\rho^2+2\rho-2\right)\left(k+1\right)}{q\rho\left(\rho-1\right)}}\ge t.
\end{equation}
By definition, we have
\begin{equation}\label{eq:def k-2}
    m^{\frac{\rho+1}{\rho q}+\frac{\left(\rho^2+2\rho-2\right)k}{q\rho\left(\rho-1\right)}}< t
\end{equation}
and by taking logarithms, we have
\[
\frac{\left(\rho^2+2\rho-2\right)k}{q\rho\left(\rho-1\right)}<\frac{\rho+1}{\rho q}+\frac{\left(\rho^2+2\rho-2\right)k}{q\rho\left(\rho-1\right)}< \log_m t,
\]
giving
\begin{equation}\label{eq:prop k}
    \frac{k}{\rho}<\frac{q\left(\rho-1\right)}{\rho^2+2\rho-2}\log_mt.
\end{equation}

Define
\begin{equation}\label{eq:def l}
\ell\coloneqq \left\lceil\frac{\rho^2+2\rho-2}{q\left(\rho-1\right)}\log_2m\right\rceil.
\end{equation}
Using Lemma \ref{L1}, we may find integers $i\in \left[k+1,k+m\right]$ and $j\in \left[0,2^\ell-1\right]$ such that for all $n\in \mathbb{N}$,
\begin{equation}\label{eq:asymptotic}
\left\|\pi\left(v^{-j2^{i\ell}}\right)P_{\left(i+1\right)\ell}f\left(v^{\lfloor n^\rho\rfloor}\right)-P_{i\ell}f\left(v^{\lfloor n^\rho\rfloor}\right)\right\|_X\lesssim_\Gamma \frac{K_q\left(X\right)n}{m^{1/q}}.
\end{equation}
Finally, define
\begin{equation}\label{eq:def n}
n\coloneqq \left\lceil\frac1{C_1} m^{\frac{\rho+1}{\rho q}}2^{\frac{i\ell}{\rho}}\right\rceil,
\end{equation}

Expanding, we have
\begin{align*}
f\left(v^{\lfloor n^\rho\rfloor}\right)=&\pi\left(v^{-j2^{i\ell}}\right)P_{\left(i+1\right)\ell}f\left(v^{\lfloor n^\rho\rfloor}\right)\\
&+\left(P_{i\ell}f\left(v^{\lfloor n^\rho\rfloor}\right) -\pi\left(v^{-j2^{i\ell}}\right)P_{\left(i+1\right)\ell}f\left(v^{\lfloor n^\rho\rfloor}\right) \right) +\left(f\left(v^{\lfloor n^\rho\rfloor}\right)-P_{i\ell}f\left(v^{\lfloor n^\rho\rfloor}\right)\right).
\end{align*}
Thus by Lemmas \ref{L2} and \ref{L3} and inequality \eqref{eq:asymptotic}, we obtain:
\begin{align}\label{eq:before m}
\begin{aligned}
\omega_f\left(d_W\left(v^{\lfloor n^\rho\rfloor},e_\Gamma\right)\right)\le \left\|f\left(v^{\lfloor n^\rho\rfloor}\right)\right\|_X
&\lesssim_\Gamma \frac{n^{\left(\rho^2+\rho-1\right)/\left(2\rho-1\right)}}{2^{\left(i+1\right)\ell\left(\rho-1\right)/\left(2\rho-1\right)}}+\frac{K_q\left(X\right)n}{m^{1/q}}+2^{i\ell/\left(\rho+1\right)}n^{1/\left(\rho+1\right)}\\
&\stackrel{\mathclap{\eqref{eq:def l}\wedge \eqref{eq:def n}}}{\lesssim}_\Gamma\quad \frac{K_q\left(X\right)n}{m^{1/q}}.
\end{aligned}
\end{align}
We compute
\begin{equation*}\label{eq:lower m}
d_W\left(v^{\lfloor n^\rho\rfloor},e_\Gamma\right)\ge C_1n ~\stackrel{\mathclap{\eqref{eq:def n}}}{\ge}~ m^{\frac{\rho+1}{\rho q}}2^{\frac{i\ell}{\rho}}\ge m^{\frac{\rho+1}{\rho q}}2^{\frac{\left(k+1\right)\ell}{\rho}}
~\stackrel{\mathclap{\eqref{eq:def l}}}{\ge}~ m^{\frac{\rho+1}{\rho q}+\frac{\left(\rho^2+2\rho-2\right)\left(k+1\right)}{q\rho\left(\rho-1\right)}}~\stackrel{\mathclap{\eqref{eq:def k}}}{\ge}~ t
\end{equation*}
and
\begin{align*}
d_W\left(v^{\lfloor n^\rho\rfloor},e_\Gamma\right)&\le C_2 n~\stackrel{\mathclap{\eqref{eq:def n}}}{\le}~ \frac{2C_2}{C_1}m^{\frac{\rho+1}{\rho q}}2^{\frac{i\ell}{\rho}}\le \frac{2C_2}{C_1}m^{\frac{\rho+1}{\rho q}}2^{\frac{\left(k+m\right)\ell}{\rho}}
~\stackrel{\mathclap{\eqref{eq:def l}}}{\le}~ \frac{2C_2}{C_1}2^{\frac{k+m}{\rho}}m^{\frac{\rho+1}{\rho q}+\frac{k\left(\rho^2+2\rho-2\right)}{q\rho\left(\rho-1\right)}} m^{\frac{m\left(\rho^2+2\rho-2\right)}{q\rho\left(\rho-1\right)}}\\
&\stackrel{\mathclap{\eqref{eq:def k-2}}}{<}~\frac{2C_2}{C_1}2^{\frac{k+m}{\rho}}tm^{\frac{m\left(\rho^2+2\rho-2\right)}{q\rho\left(\rho-1\right)}}~\stackrel{\mathclap{\eqref{eq:prop k}}}{\le}\frac{2C_2}{C_1}2^{\frac{q\left(\rho-1\right)}{\rho^2+2\rho-2}\log_m t}2^{\frac{m}{\rho}}m^{\frac{m\left(\rho^2+2\rho-2\right)}{q\rho\left(\rho-1\right)}}t~\stackrel{\mathclap{\eqref{eq:def m}}}{\le} t^{1+\varepsilon}.
\end{align*}
Therefore $t\le d_W\left(v^{\lfloor n^\rho\rfloor},e_\Gamma\right)\le t^{1+\varepsilon}$. The definition of $m$ implies $m\gtrsim_\Gamma \frac{q\varepsilon \log t}{\log\log t}$, and
thus by \eqref{eq:before m}, we have
$$
\frac{\omega_f\left(d_W\left(v^{\lfloor n^\rho\rfloor},e_\Gamma\right)\right)}{n}\lesssim_\Gamma K_q\left(X\right)\left(\frac{\log\log n}{\varepsilon\log
n}\right)^{1/q}.
$$
The proof of Theorem \ref{thm:amenablesublinear} is complete.
\end{proof}

\section{Proof of Dorronsoro's Theorem \ref{lpgenthm} for Carnot groups}\label{sec:dorronsoro}
Let $G$ be a Carnot group. Assume $G$ is nonabelian, as the case when $G$ is abelian has been proved by Dorronsoro \cite{dorronsoro1985characterization}.

We need to prove two directions. In one direction, we assume $f\in S^p_\alpha\left(G\right)$ and prove the $\lesssim$ statement of \eqref{form13}. In the other direction, we assume $f\in L^p\left(G\right)$ with the left-hand side of \eqref{form13} finite, and prove $f\in S^p_\alpha\left(G\right)$ along with the $\gtrsim$ statement of \eqref{form13}. Our proof will be based on Dorronsoro's original proof of the $G=\mathbb{R}^n$ case in \cite{dorronsoro1985characterization} and will borrow modifications necessary for the Carnot group setting, some inspired by the proof of the $G=\mathbb{H}^{2k+1}$ case in \cite{fassler2020dorronsoro}.

We will first prove Theorem \ref{lpgenthm} for $q=1$, dealing with the $\lesssim$ direction in subsection \ref{subsec:leQ=1} and the $\gtrsim$ direction in subsection \ref{subsec:geQ=1}. Since the left-hand side of \eqref{form13} is minimized when $q=1$, this will finish the proof of the $\gtrsim$ direction; for the $\lesssim$ direction, we will see in subsection \ref{subsec:q>1} that the $\lesssim$ inequality for $q=1$ implies the $\lesssim$ inequality for larger $q$.

For the $\lesssim$ statement, we will first see by an approximation argument that it is enough to look at smooth functions $f$ (subsubsection \ref{subsubsec:smooth}), for which it turns out that we may as well approximate at all scales simultaneously by Taylor polynomials (subsubsection \ref{subsubsec:taylor}). Then, for $0<\alpha<1$, it will turn out that the desired inequality has already been proven in \cite[Theorem 5]{coulhon2001sobolev} (subsubsection \ref{subsubsec:alpha<1}), and for $\alpha>1$ nonintegral we may use an induction argument on $\alpha$ (subsubsection \ref{subsubsec:nonintegral>1}). We finish off the case of $\alpha\ge 1$ integral by an interpolation argument (subsubsection \ref{subsubsec:integral}), with the extra terms arising in this process having been taken care of by a homogenization argument (subsubsection \ref{subsubsec:scaling}).

For the $\gtrsim$ statement, we are given that $f\in L^p\left(G\right)$ and the finiteness of the singular integral given as the left-hand side of \eqref{form13}, and we are to derive $f\in S^p_\alpha\left(G\right)$. We first see again that we may approximate by a fixed polynomial for all scales (which we suspect to be the Taylor polynomial in the distributional sense but we can only prove this up to first derivatives) in Proposition \ref{prop:measurableTaylor}. For $0<\alpha\le 1$ we again use characterizations of the fractional Laplacian given by \cite[Theorem 5]{coulhon2001sobolev} for $0<\alpha<1$ and \cite[Theorem 1.4]{de2021mean} for $\alpha=1$ (subsubsection \ref{subsubsec:gtrsimalphale1}). We then prove the case $\alpha>1$ by induction (subsubsection \ref{subsubsec:gtrsimalpha>1}).

For simplicity, we define the $L^1$-beta numbers
\[
\beta_{f,d}\left(B_r\left(x\right)\right)\coloneqq \beta_{f,d,1}\left(B_r\left(x\right)\right).
\]
If we define the function
\[
\mathfrak{G}_\alpha f\left(x\right)\coloneqq \left(\int_0^\infty \left(\frac{\beta_{f,\lfloor\alpha\rfloor}\left(B_r\left(x\right)\right)}{r^\alpha}\right)^2 \frac{dr}r\right)^{1/2}, \quad x\in G,
\]
then Theorem \ref{lpgenthm} states in the case $q=1$ that
\begin{equation}\label{eq:alphaIneq}
\left\|\mathfrak{G}_\alpha f\right\|_{L^p\left(G\right)}\asymp_{G,\alpha,p} \left\|\left(-\Delta_p\right)^{\alpha/2}f\right\|_{L^p\left(G\right)},\quad f\in S^p_\alpha.
\end{equation}

Before we begin the proof, we briefly remark on coordinate and multi-index notation on $G$. Recall the coordinate system $x=\exp\left(\sum_{r=1}^s\sum_{i=1}^{k_r}x_{r,i}X_{r,i}\right)$. A multi-index $\gamma= \left(\left(\gamma_{r,i}\right)_{i=1}^{k_r}\right)_{r=1}^s$, $\gamma_{r,i}\in \mathbb{Z}_{\ge 0}$, is a multi-index on $\sum_{r=1}^s k_r$ entries. We define the weighted and unweighted degrees
\[
\left|\gamma\right|\coloneqq \sum_{r=1}^s\sum_{i=1}^{k_r}r\gamma_{r,i}\in \mathbb{Z}_{\ge 0},\quad \left|\gamma\right|_0\coloneqq \sum_{r=1}^s\sum_{i=1}^{k_r}\gamma_{r,i}\in \mathbb{Z}_{\ge 0},
\]
as well as the multi-index factorial:
\[
\gamma!\coloneqq\prod_{r=1}^s\prod_{i=1}^{k_r} \gamma_{r,i}!\in \mathbb{Z}_{>0}.
\]
We also denote the differential operator and polynomial
\[
\left(\frac{\partial}{\partial x}\right)^\gamma\coloneqq\prod_{r=1}^s\prod_{i=1}^{k_r} \left(\frac{\partial}{\partial x_{r,i}}\right)^{\gamma_{r,i}},\quad x^\gamma \coloneqq\prod_{r=1}^s\prod_{i=1}^{k_r}x_{r,i}^{\gamma_{r,i}}\in \mathbb{R}.
\]
The latter should not be confused with the previous notation $x^t=\exp\left(\sum_{r=1}^s\sum_{i=1}^{k_r}tx_{r,i}X_{r,i}\right)$ for $t\in \mathbb{R}$.

One can see from the form \eqref{eq:grouplaw1}-\eqref{eq:grouplaw2} of the Baker--Campbell-Hausdorff formula for Carnot groups that we have
\[
X_{r,i}=\frac{\partial}{\partial x_{r,i}}+\sum_{r'=r+1}^s\sum_{j=1}^{k_{r'}}\left(\mathrm{homogeneous~polynomial~of~}\left\{x_{r'',i''}\right\}_{r''<r'}\mathrm{~of~weighted~degree~}r'-r\right)\frac{\partial}{\partial x_{r',j}}.
\]
It follows that $X_{r,i}$ acting on a polynomial of homogeneous weighted degree $d$ will produce either a polynomial of homogeneous weighted degree $d-r$ or the zero polynomial. Hence $X_{r,i}$ acting on a polynomial of weighted degree at most $d$ will produce a polynomial of weighted degree at most $d-r$ (the zero polynomial is defined to have degree $-\infty$). (Recall from subsection \ref{subsec:dorronsoro} that the set of polynomials of weighted degree $\le d$ is left-invariant, and since the vector fields $X_{r,i}$ are left-invariant, the choice of origin is irrelevant in this discussion.) 

We now define the weighted degree of the differential operator $X_{r,i}$ to be $r$, and that of the differential operator $X_{r_1,i_1}\cdots X_{r_n,i_n}$ to be $r_1+\cdots+r_n$. If a nonzero real-linear combination of differential operators of the form $X_{r_1,i_1}\cdots X_{r_n,i_n}$ consists of only those of weighted degree $r$, we say that the sum, which is a left-invariant differential operator, to be of homogeneous weighted degree $r$. Again, a left-invariant differential operator of homogeneous weighted degree $r$ acting on a polynomial of homogeneous weighted degree $d$ will produce either a polynomial of homogeneous weighted degree $d-r$ or the zero polynomial, and acting on a polynomial of weighted degree at most $d$ will produce a polynomial of weighted degree at most $d-r$.

\subsection{$L^1$-optimality of the $A^d_{x,r}f$'s and the $L^1\to W^{n,\infty}$ boundedness of $A^d_{x,r}$}
We begin the proof by discussing some basic properties of $A^d_{x,r}f$ and $\beta_{f,d}\left(B_r\left(x\right)\right)$.

It is clear that $A^d_{x,r}f$ is the optimal $L^2\left(B_r\left(x\right)\right)$-approximation of $f$ in ${\mathcal{A}}_{d}$. It turns out that it is also an optimal $L^1\left(B_r\left(x\right)\right)$-approximation up to constants. To see this, we first observe that for all $d\ge 0$, there exists an integral formula for the coefficients of $A^d_{x,r}f$, given by the Gram--Schmidt process. For example, for $d=1$,
\begin{equation}\label{A1}
A^1_{x,r}f\left(xy\right) = \fint_{B_r\left(x\right)}f\left(z\right)d\mu\left(z\right)+ \sum_{j=1}^{k}\frac{\int_{B_r\left(x\right)} f\left(z\right)\left(z_{1,j}-x_{1,j}\right)d\mu\left(z\right)}{\int_{B_r\left(x\right)} \left(z_{1,j}-x_{1,j}\right)^2 d\mu\left(z\right)}y_{1,j}, \quad y \in G.
\end{equation}
More generally, for each $d\ge 0$, there exists a family $\left\{p^d_\gamma\right\}_{\left|\gamma\right|\le d}$ of polynomials such that
\[
A^d_{0,1}f\left(y\right)=\sum_{\left|\gamma\right|\le d}\left(\int_{B_1}fp^d_\gamma d\mu\right)y^\gamma,\quad y\in G,
\]
and by rescaling we have for $x\in G$ and $r>0$ that
\begin{equation}\label{eq:adxr-integral}
A^d_{x,r}f\left(xy\right)=\sum_{\left|\gamma\right|\le d} r^{-n_h-\left|\gamma\right|}\left(\int_{B_r\left(x\right)}f\left(z\right)p^d_\gamma\left(\delta_{1/r}\left(x^{-1}z\right)\right)d\mu\left(z\right)\right)y^\gamma.
\end{equation}
Therefore
\begin{equation}\label{claim}
\left\|A^{d}_{x,r}f\right\|_{L^\infty\left(B_r\left(x\right)\right)}\lesssim_{G,d}\fint_{B_r\left(x\right)}\left|f\left(y\right)\right|d\mu\left(y\right).
\end{equation}
In other words, $A^d_{x,r}$ is a linear projection of $L^1\left(B_r\left(x\right)\right)$ onto $\left.\mathcal{A}_d\right|_{B_r\left(x\right)}$ which is bounded in the $L^1\to L^\infty$ norm. 

We now prove the $L^1$-optimality of the $A^d_{x,r}f$'s.
\begin{lemma}\label{L1-approx}
For $x\in G$, $r>0$, $f\in L^1_{\mathrm{loc}}\left(B_r\left(x\right)\right)$, and $d\in \mathbb{Z}_{\ge 0}$,
\[
\fint_{B_r\left(x\right)}\left|f\left(y\right)-A^d_{x,r}f\left(y\right)\right|d\mu\left(y\right) \lesssim_{G,d}\fint_{B_r\left(x\right)}\left|f\left(y\right)-A\left(y\right)\right|d\mu\left(y\right),\quad A\in \mathcal{A}_d.
\]
\end{lemma}
\begin{proof}
From \eqref{claim} and $A^d_{x,r}A=A$,
\[
\fint_{B_r\left(x\right)}\left|f-A^d_{x,r}f\right|d\mu\le \fint_{B_r\left(x\right)}\left|f-A\right|d\mu+\fint_{B_r\left(x\right)}\left|A^d_{x,r}\left(f-A\right)\right|d\mu\lesssim_{G,d} \fint_{B_r\left(x\right)}\left|f-A\right|d\mu.
\]
\end{proof}
It follows that $\beta$ possesses a weak monotonicity property.

\begin{corollary}\label{monotoneBeta} Let $x,y \in G$, $0<r<s<\infty$, $d\in\mathbb{Z}_{\ge 0}$ be such that $B_r\left(x\right)\subset B_s\left(y\right)$. Let $f\in L^1\left(B_s\left(y\right)\right)$. Then
\[
\beta_{f,d}\left(B_r\left(x\right)\right) \lesssim_{G,d} \left(\frac{s}{r}\right)^{n_h}\beta_{f,d}\left(B_s\left(y\right)\right).
\]
\end{corollary}

By Corollary \ref{monotoneBeta}, we have that
\[
\beta_{f,d}\left(B_r\left(x\right)\right)\lesssim_{G,d} \beta_{f,d}\left(B_s\left(x\right)\right)\lesssim_{G,d} \beta_{f,d}\left(B_{2r}\left(x\right)\right),\quad r\le s\le 2r.
\]
Thus we may approximate $\mathfrak{G}_\alpha f\left(x\right)$ using a series:
\[
\mathfrak{G}_\alpha f\left(x\right)\asymp_{G,\alpha}\left(\sum_{i=-\infty}^\infty \left[2^{-i\alpha}\beta_{f,\lfloor\alpha\rfloor}\left(B_{2^i}\left(x\right)\right)\right]^2\right)^{1/2}.
\]
We also have the following corollary.
\begin{corollary}\label{cor:betaIntConversion}
    Let $x\in G$, $r>0$, $d\in \mathbb{Z}_{\ge 0}$, $\lambda\in \mathbb{R}$, and $m,n\in \mathbb{Z}$ with $m\le n$. Let $f\in L^1\left(B_{2^{n+1}r}\left(x\right)\right)$. Then
    \[
    \sum_{i=m}^n \left(2^ir\right)^\lambda \beta_{f,d}\left(B_{2^ir}\left(x\right)\right)\lesssim_{G,d,\lambda}\int_{2^mr}^{2^{n+1}r}s^\lambda \beta_{f,d}\left(B_s\left(x\right)\right)\frac{ds}{s}.
    \]
\end{corollary}

In a different direction, while \eqref{claim} tells us that $A^d_{x,r}$ is bounded in the $L^1\to L^\infty$ norm, it is also bounded in the $L^1\to W^{n,\infty}$ norms. To prove this, we start with the following.
\begin{lemma}\label{lem:poly-elem}
    Let $p\left(t\right)$ be a real univariate polynomial of $t\in \mathbb{R}$ of degree at most $d$. Then, for $r>0$, we have
    \[
    \left|p'\left(0\right)\right|\lesssim_d r^{-1}\left\|p\right\|_{L^\infty\left[-r,r\right]}.
    \]
\end{lemma}
\begin{proof}
    By possibly replacing $p\left(t\right)$ by $p\left(rt\right)$, we may assume $r=1$, in which case the stated inequality follows from the fact that $L^\infty \left[-1,1\right]$ is a norm on the space of real polynomials of degree at most $d$, the linearity of the map $p\to p'\left(0\right)$, and the finite-dimensionality of the spaces involved.
\end{proof}
Applying this lemma to curves in the Carnot group $G$, we obtain the following.
\begin{lemma}\label{lem:diffpoly}
    Let $x\in G$, $r>0$, $d\in \mathbb{Z}_{\ge 0}$, $s'\in \left\{1,\cdots,s\right\}$, and $i'\in \left\{1,\cdots,k_{s'}\right\}$. Let $P\in \mathcal{A}^d$. Then
    \[
    \left\|X_{s',i'}P\right\|_{L^\infty\left(B_r\left(x\right)\right)}\lesssim_{G,s',d} r^{-s'}\left\|P\right\|_{L^\infty\left(B_r\left(x\right)\right)}.
    \]
\end{lemma}
\begin{proof}
    We may assume $d\ge s'$, for otherwise the left-hand side is zero.
    
    Note that the restriction map $\mathcal{A}^{d-s'}\cap L^\infty\left(B_2\right)\to \mathcal{A}^{d-s'}\cap L^\infty\left(B_1\right)$ is invertible, so there is an extension map $\mathcal{A}^{d-s'}\cap L^\infty\left(B_1\right)\to \mathcal{A}^{d-s'}\cap L^\infty\left(B_2\right)$ which is bounded by finite dimensionality. By translation and scaling, the extension map $\mathcal{A}^{d-s'}\cap L^\infty\left(B_{r/2}\left(x\right)\right)\to \mathcal{A}^{d-s'}\cap L^\infty\left(B_r\left(x\right)\right)$ is bounded with the same constant, so
    \[
    \left\|X_{s',i'}P\right\|_{L^\infty\left(B_r\left(x\right)\right)}\lesssim_{G,d-s'} \left\|X_{s',i'}P\right\|_{L^\infty\left(B_{r/2}\left(x\right)\right)}.
    \]

    Let $y\in B_{r/2}\left(x\right)$. By \eqref{eq:carnot-distance}, there is a constant $c$ depending on $G$ such that $y\cdot \exp\left(tX_{s',i'}\right)\in B_{r/2}\left(y\right)\subset B_r\left(x\right)$ whenever $\left|t\right|\le cr^{s'}$. But since $P$ is a polynomial of weighted degree at most $d$, it follows that $p\left(t\right)= P\left(y\cdot \exp\left(tX_{s',i'}\right)\right)$ is a polynomial of $t$ of degree at most $d/s'$, with $X_{s',i'}P\left(y\right)=p'\left(0\right)$. It follows now from Lemma \ref{lem:poly-elem} that
    \[
    \left|X_{s',i'}P\left(y\right)\right|=\left|p'\left(0\right)\right|\lesssim_{G,s',d} r^{-s'}\left\|p\right\|_{L^\infty\left[-cr^{s'},cr^{s'}\right]}\le r^{-s'}\left\|P\right\|_{L^\infty\left(B_r\left(x\right)\right)}.
    \]
\end{proof}

Finally, we have the following.
\begin{lemma}
    Let $x\in G$, $r>0$, and $d\in \mathbb{Z}_{\ge 0}$. Let $X$ be a left-invariant differential operator of weighted order $n$. Then for any $f\in L^1\left(B_{r}\left(x\right)\right)$,
    \begin{equation}\label{polydiff}
\left\|X A^d_{x,r}f\right\|_{L^\infty\left(B_r\left(x\right)\right)}\lesssim_{G,n, d,X}r^{-n}\fint_{B_r\left(x\right)}\left|f\right|d\mu.
\end{equation}
\end{lemma}
\begin{proof}
    By definition, $X$ is a nonzero real-linear combination of differential operators of the form
    \[
    X_{r_1,i_1}\cdots X_{r_m,i_m} \mathrm{~with~} r_1+\cdots+r_m=n,
    \]
    so it is enough to prove \eqref{polydiff} for $X=X_{r_1,i_1}\cdots X_{r_m,i_m}$ with $r_1+\cdots+r_m=n$. By repeated applications of Lemma \ref{lem:diffpoly},
    \begin{align*}
    \left\|X_{r_1,i_1}\cdots X_{r_m,i_m}A^d_{x,r}f\right\|_{L^\infty\left(B_r\left(x\right)\right)}&\lesssim_{G,n,d}r^{-r_1}\left\|X_{r_2,i_2}\cdots X_{r_m,i_m}A^d_{x,r}f\right\|_{L^\infty\left(B_r\left(x\right)\right)}\\
    &\lesssim_{G,n,d}\cdots\\
    &\lesssim_{G,n,d}r^{-r_1-\cdots-r_m}\left\|A^d_{x,r}f\right\|_{L^\infty\left(B_r\left(x\right)\right)}\\
    &\stackrel{\mathclap{\eqref{claim}}}{\lesssim_{G,d}}r^{-n}\fint_{B_r\left(x\right)}\left|f\right|d\mu.
    \end{align*}
\end{proof}

\subsection{Taylor series in Carnot groups}\label{subsec:taylor}
Recall the coordinate system on $G$ where each $y\in G$ is uniquely expressed as $y=\exp\left(\sum_{r=1}^s\sum_{i=1}^{k_r}y_{r,i}X_{r,i}\right)$, $y_{r,i}\in \mathbb{R}$. Let $f:G\to \mathbb{R}$ be smooth and let $x,y\in G$. The one-dimensional Taylor theorem tells us that for $n\in \mathbb{Z}_{\ge 0}$,
\[
f\left(xy\right)=f\left(x\right)+\sum_{j=1}^n \frac{1}{j!}\left.\frac{d^j}{dt^j}\right|_{t=0}f\left(xy^t\right)+\frac{1}{n!}\int_0^1\left(1-t\right)^n\frac{d^{n+1}}{dt^{n+1}}f\left(xy^t\right)dt,
\]
where $y^t=\exp\left(\sum_{r=1}^s\sum_{i=1}^{k_r}ty_{r,i}X_{r,i}\right)$, $t\in \mathbb{R}$, is the path of the one-parameter subgroup generated by $\sum_{r=1}^s\sum_{i=1}^{k_r}y_{r,i}X_{r,i}$; we thus have
\[
f\left(xy\right)=f\left(x\right)+\sum_{j=1}^n \frac{1}{j!}\left(\sum_{r=1}^s\sum_{i=1}^{k_r}y_{r,i}X_{r,i}\right)^jf\left(x\right)+\frac{1}{n!}\int_0^1\left(1-t\right)^n\left(\sum_{r=1}^s\sum_{i=1}^{k_r}y_{r,i}X_{r,i}\right)^{n+1}f\left(xy^t\right)dt.
\]
For the zero multi-index $0$, define $\operatorname{Sym}\left(X^0\right)$ to be the identity map on smooth functions, while for each non-zero multi-index $\gamma=\left(\gamma_{r,i}\right)_{1\le r\le s,~1\le i\le k_r}$, define the differential operator
\begin{equation}\label{eq:sym}
    \operatorname{Sym}\left(X^\gamma\right)\coloneqq\frac{\gamma!}{\left(\left|\gamma\right|_0\right)!}\sum_{\pi\in S_{\left|\gamma\right|_0}}Y_{\pi\left(1\right)}\cdots Y_{\pi\left(\left|\gamma\right|_0\right)},
\end{equation}
where $Y_1,\cdots,Y_{\left|\gamma\right|_0}$ is a sequence of the left-invariant differential operators $X_{r,i}$ where $X_{r,i}$ appears exactly $\gamma_{r,i}$ times. Then $\operatorname{Sym}\left(X^\gamma\right)$ is the `coefficient operator' of $y^\gamma$ in $\frac{\gamma!}{\left(\left|\gamma\right|_0\right)!}\left(\sum_{r=1}^s\sum_{i=1}^{k_r}y_{r,i}X_{r,i}\right)^{\left|\gamma\right|_0}$, i.e., for $j\in \mathbb{Z}_{>0}$,
\[
\frac{1}{j!}\left(\sum_{r=1}^s\sum_{i=1}^{k_r}y_{r,i}X_{r,i}\right)^j=\sum_{\substack{\gamma\mathrm{~multi-index}\\ \left|\gamma\right|_0=j}}\frac{1}{\gamma !}y^\gamma\operatorname{Sym}\left(X^\gamma\right),\quad y_{r,i}\in \mathbb{R}.
\]
With this, the above expression can be rewritten as
\begin{equation}\label{eq:Taylor_integral}
f\left(xy\right)=\sum_{\substack{\gamma\mathrm{~multi-index}\\ \left|\gamma\right|_0\le n}}\frac{1}{\gamma !}\left[\operatorname{Sym}\left(X^\gamma\right)f\left(x\right)\right]y^\gamma +\left(n+1\right)\int_0^1\left(1-t\right)^n \sum_{\substack{\gamma\mathrm{~multi-index}\\ \left|\gamma\right|_0=n+1}}\frac{1}{\gamma !}\left[\operatorname{Sym}\left(X^\gamma\right)f\left(xy^t\right)\right]y^\gamma dt.
\end{equation}
By isolating the terms with multi-index $\gamma$ such that $\left|\gamma\right|\le n$, while noting that $\left|\gamma\right|_0\le \left|\gamma\right|$, we see that there exist for each multi-index $\gamma$ with $\left|\gamma\right|\ge n+1$ and $\left|\gamma\right|_0\le n+1$ a function $q_{x,\gamma}f\in \mathcal{D}$ such that
\begin{equation}\label{eq:Taylor_local}
f\left(xy\right)=\underbrace{\sum_{\left|\gamma\right|\le n}\frac{1}{\gamma !}\left[\operatorname{Sym}\left(X^\gamma\right)f\left(x\right)\right]y^\gamma}_{\eqqcolon T^n_xf\left(xy\right)} + \underbrace{\sum_{\substack{\left|\gamma\right|\ge n+1 \\ \left|\gamma\right|_0\le n+1}}q_{x,\gamma}f\left(y\right)y^\gamma}_{=O_{G,f,n}\left(d_G\left(y,0\right)^{n+1}\right)\mathrm{~as~}y\to 0},\quad y\in G,
\end{equation}
where we denote the first sum to be the \emph{Taylor polynomial of $f$ of weighted degree $n$ at $x\in G$}, while we observe the second term to be $O_{G,f,n}\left(d_G\left(y,0\right)^{n+1}\right)$ as $y\to 0$ because $d_G\left(y,0\right)\asymp_G \sum_{r=1}^s\sum_{i=1}^{k_r}\left|y_{r,i}\right|^{1/r}$. We collect this observation into a lemma.

\begin{lemma}\label{lem:taylor_thm}
    Let $f:G\to\mathbb{R}$ be a smooth function and $n\in \mathbb{Z}_{\ge 0}$. Then, for each $x\in G$,
    \[
    f\left(xy\right)=T^n_xf\left(xy\right)+O_{G,f,n}\left(d_G\left(y,0\right)^{n+1}\right)\quad \mathrm{as~}y\to 0,~y\in G,
    \]
    where we denote
    \begin{equation}\label{eq:Taylor_def}
    T^n_xf\left(xy\right)=\sum_{\left|\gamma\right|\le n}\frac{1}{\gamma !}\left[\operatorname{Sym}\left(X^\gamma\right)f\left(x\right)\right]y^\gamma,\quad y\in G.
    \end{equation}
\end{lemma}

We observe that, as expected, the Taylor polynomial of a polynomial is itself.
\begin{lemma}\label{lem:taylor_identity}
    Let $P$ be a polynomial on $G$ of weighted degree at most $n$. Then $T^n_xP=P$ for every $x\in G$.
\end{lemma}
\begin{proof}
    This follows from setting $f=P$ in \eqref{eq:Taylor_integral}, where the second term on the right-hand side is zero because $\operatorname{Sym}\left(X^\gamma\right)P=0$ whenever $\left|\gamma\right|_0=n+1$, since $\left|\gamma\right|\ge \left|\gamma\right|_0=n+1$. Also, in the first term on the right-hand side of \eqref{eq:Taylor_integral}, any term with $\left|\gamma\right|\ge n+1$ yet $\left|\gamma\right|_0\le n$ vanishes.
\end{proof}
We also remark that the Taylor polynomial of degree $n$ is the unique polynomial approximant of weighted degree $\le n$ that offers an approximation of $O_G\left(d_G\left(y,0\right)^{n+1}\right)$. This statement is equivalent to the following.
\begin{lemma}
    Let $P$ be a polynomial of weighted degree at most $n$. If $\left|P\left(y\right)\right|=o\left(d_G\left(y,0\right)^{n}\right)$ as $y\to 0$, then $P=0$.
\end{lemma}
\begin{proof}
    Suppose $P$ were not $0$. Then, in the reduced form of $P$, there would be monomials with nonzero coefficients; let $d_0\le n$ be the smallest weighted degree among these monomials. We may write $P=P_0+Q$, where $P_0$ consists of the monomials of $P$ of weighted degree $d_0$, and $Q$ consists of those with higher weighted degree, hence $\left|Q\left(y\right)\right|=O\left(d_G\left(y,0\right)^{d_0+1}\right)$ as $y\to 0$. Since $P_0$ is not the zero polynomial, and hence a nonzero function, there is a point $y_0\in G$ with $P_0\left(y_0\right)\neq 0$. Then, since $P_0$ is of weighted degree $d_0$, we have $P_0\left(\delta_\lambda\left(y_0\right)\right)=\lambda^{d_0}P_0\left(y_0\right)$ for $\lambda\in \mathbb{R}$ while $Q\left(\delta_\lambda\left(y_0\right)\right)=O\left(\lambda^{d_0+1}\right)$ as $\lambda\to 0$, and hence $\left|P\left(\delta_\lambda\left(y_0\right)\right)\right|\asymp_{P} \lambda^{d_0}$ as $\lambda\to 0$, which together with $d_0\le n$ contradicts our assumption that $\left|P\left(\delta_\lambda\left(y_0\right)\right)\right|=o\left(\lambda^n\right)$ as $\lambda\to 0$.
\end{proof}

We now show that the operation of taking derivatives commutes with taking the Taylor polynomial.
\begin{lemma}\label{lem:taylor_unique}
    Let $f:G\to \mathbb{R}$ be a smooth function and let $d,n\in \mathbb{Z}$ with $1\le d\le n$. Let $Y$ be a left-invariant differential operator of degree $d$. Then for each $x\in G$,
    \[
    YT^n_xf =T^{n-d}_x Yf. 
    \]
    In particular, if $1\le r \le s$, $1\le i\le k_r$, and $d\ge r$, then
    \[
    X_{r,i}T^d_xf=T^{d-r}_x X_{r,i}f.
    \]
\end{lemma}
\begin{proof}
    Differentiating the expression \eqref{eq:Taylor_local}, where we view both sides as functions of $y$, and $Y$ acting with respect to the $y$ variable, we have
    \[
    Yf\left(xy\right)=YT^n_xf\left(xy\right)+\sum_{\substack{\left|\gamma\right|\ge n+1 \\ \left|\gamma\right|_0\le n+1}}Y\left(q_{x,\gamma}f\left(y\right)y^\gamma\right).
    \]
    In the second term on the right-hand side, applying Leibniz's rule several times, $Y$ differentiates either $q_{x,\gamma}f\left(y\right)$, which results in a smooth function, or $y^\gamma$, which results in a polynomial of lesser degree. This results in the second term on the right-hand side being $O\left(d_G\left(y,0\right)^{n-d+1}\right)$ as $y\to 0$, i.e.,
    \[
    Yf\left(xy\right)=YT^n_xf\left(xy\right)+O\left(d_G\left(y,0\right)^{n-d+1}\right)\quad\mathrm{as~}y\to 0.
    \]
    By Lemma \ref{lem:taylor_thm} applied to the function $Yf$ up to degree $n-d$, we have
    \[
    Yf\left(xy\right)=T^{n-d}_x Yf\left(xy\right)+O\left(d_G\left(y,0\right)^{n-d+1}\right)\quad\mathrm{as~}y\to 0.
    \]
    and subtracting gives
    \[
    0=\left(YT^n_xf\left(xy\right)-T^{n-d}_x Yf\left(xy\right)\right)+O\left(d_G\left(y,0\right)^{n-d+1}\right)\quad\mathrm{as~}y\to 0.
    \]
    Since $YT^n_xf-T^{n-d}_xYf$ is a polynomial of degree at most $n-d$, it follows from Lemma \ref{lem:taylor_unique} that $YT^n_xf=T^{n-d}_xYf$.
\end{proof}

We now give an example to illustrate the Taylor polynomial in the Heisenberg group.
\begin{example}
        For the 3-dimensional Heisenberg group $\mathbb{H}^3$, recall from Remark \ref{remark:noHori} that the left-invariant vector fields are given by
        \[
X=\frac{\partial}{\partial x}+\frac{y}{2} \frac{\partial}{\partial z},\quad Y=\frac{\partial}{\partial y}-\frac{x}{2} \frac{\partial}{\partial z},\quad Z=\frac{\partial}{\partial z}.
\]
Denote $g=\left(x_0,y_0,z_0\right)\in \mathbb{H}^3$, $h=\left(x_1,y_1,z_1\right)\in\mathbb{H}^3$. Since
\[
\operatorname{Sym}\left(X\right)=X,\quad \operatorname{Sym}\left(Y\right)=Y,
\]
we have, by \eqref{eq:Taylor_def}, for a smooth function $f:\mathbb{H}^3\to \mathbb{R}$,
\[
T^0_gf\left(gh\right)=f\left(g\right),\quad T^1_gf\left(gh\right)=f\left(g\right)+\left(\frac{\partial f}{\partial x}\left(g\right)+\frac{y_0}{2} \frac{\partial f}{\partial z}\left(g\right)\right)x_1+\left(\frac{\partial f}{\partial y}\left(g\right)-\frac{x_0}{2} \frac{\partial f}{\partial z}\left(g\right)\right)y_1.
\]
Setting $f\left(x,y,z\right)=x$ or $f\left(x,y,z\right)=y$ gives $T^1_gf=f$.
Since
\[
\operatorname{Sym}\left(X^2\right)=X^2=\frac{\partial^2}{\partial x^2}+y\frac{\partial^2}{\partial x \partial z}+\frac{y^2}{4}\frac{\partial^2}{\partial z^2},\quad \operatorname{Sym}\left(Y^2\right)=Y^2=\frac{\partial^2}{\partial y^2}-x\frac{\partial^2}{\partial y \partial z}+\frac{x^2}{4}\frac{\partial^2}{\partial z^2},
\]
and
\[
\operatorname{Sym}\left(XY\right)=\frac 12\left(XY+YX\right)=\frac{\partial^2}{\partial x\partial y}-\frac x2
\frac{\partial^2}{\partial x\partial z}+\frac y2\frac{\partial^2}{\partial y \partial z}-\frac{xy}{4}\frac{\partial^2}{\partial z^2},\quad \operatorname{Sym}\left(Z\right)=Z=\frac{\partial}{\partial z},
\]
we have, by \eqref{eq:Taylor_def}, for a smooth function $f:\mathbb{H}^3\to \mathbb{R}$,
\begin{align*}
T^2_gf\left(gh\right)=&f\left(g\right)+\left(\frac{\partial f}{\partial x}\left(g\right)+\frac{y_0}{2} \frac{\partial f}{\partial z}\left(g\right)\right)x_1+\left(\frac{\partial f}{\partial y}\left(g\right)-\frac{x_0}{2} \frac{\partial f}{\partial z}\left(g\right)\right)y_1\\
&+\frac 12 \left(\frac{\partial^2f}{\partial x^2}\left(g\right)+y_0\frac{\partial^2f}{\partial x \partial z}\left(g\right)+\frac{y_0^2}{4}\frac{\partial^2f}{\partial z^2}\left(g\right)\right)x_1^2\\
&+\left(\frac{\partial^2f}{\partial x\partial y}\left(g\right)-\frac {x_0}2
\frac{\partial^2f}{\partial x\partial z}\left(g\right)+\frac {y_0}2\frac{\partial^2f}{\partial y \partial z}\left(g\right)-\frac{x_0y_0}{4}\frac{\partial^2f}{\partial z^2}\left(g\right)\right)x_1y_1\\
&+\frac 12 \left(\frac{\partial^2f}{\partial y^2}\left(g\right)-x_0\frac{\partial^2f}{\partial y \partial z}\left(g\right)+\frac{x_0^2}{4}\frac{\partial^2f}{\partial z^2}\left(g\right)\right)y_1^2\\
&+\frac{\partial f}{\partial z}\left(g\right) z_1.
\end{align*}
One can verify that $T^2_gf=f$ when $f\left(x,y,z\right)=x^2$, $y^2$, $xy$, or $z$. As a demonstration of Lemma \ref{lem:taylor_unique}, we have
\begin{align*}
XT^2_gf\left(gh\right)=&\left(\frac{\partial f}{\partial x}\left(g\right)+\frac{y_0}{2} \frac{\partial f}{\partial z}\left(g\right)\right)\\
&+\left(\frac{\partial^2f}{\partial x^2}\left(g\right)+y_0\frac{\partial^2f}{\partial x \partial z}\left(g\right)+\frac{y_0^2}{4}\frac{\partial^2f}{\partial z^2}\left(g\right)\right)x_1\\
&+\left(\frac{\partial^2f}{\partial x\partial y}\left(g\right)-\frac {x_0}2
\frac{\partial^2f}{\partial x\partial z}\left(g\right)+\frac {y_0}2\frac{\partial^2f}{\partial y \partial z}\left(g\right)-\frac{x_0y_0}{4}\frac{\partial^2f}{\partial z^2}\left(g\right)+\frac{1}{2}\frac{\partial f}{\partial z}\left(g\right)\right)y_1\\
=&T_g^1Xf\left(gh\right).
\end{align*}
    \end{example}
We are now ready to prove Theorem \ref{lpgenthm}, the Dorronsoro theorem for Carnot groups.
\subsection{Proof of the $\lesssim$ direction, $q=1$}\label{subsec:leQ=1}

\subsubsection{Reduction to smooth compactly supported functions}\label{subsubsec:smooth}
It is enough to prove the $\lesssim$ statement of Theorem \ref{lpgenthm} for $f$ in the space $\mathcal{D}$ of smooth compactly supported functions on $G$. Indeed, suppose we had proven the $\lesssim$ statement of Theorem \ref{lpgenthm} for $f\in \mathcal{D}$. By \cite[Theorem 4.5]{folland1975subelliptic}, $\mathcal{D}$ is dense in $S^p_\alpha$ for all $1<p<\infty$ and $\alpha\ge 0$. Given $f\in S^p_\alpha$, choose a sequence $\left\{f_j\right\}_{j=1}^\infty\subset \mathcal{D}$ that converges to $f$ in  $S^p_\alpha$. Then $f_j\to f$ in $L^p\left(G\right)$, so by the contraction property \eqref{claim} we have
\begin{displaymath} \beta_{f_j,\lfloor{\alpha}\rfloor}\left(B_r\left(x\right)\right) \to \beta_{f,\lfloor{\alpha}\rfloor}\left(B_r\left(x\right)\right)\mbox{ as }j\to \infty, \quad x\in G,~ r>0. \end{displaymath}
It follows that
\[
\left\|\mathfrak{G}_\alpha f\right\|_{L^p\left(G\right)} \le \liminf_{j\to\infty} \left\|\mathfrak{G}_\alpha f_j\right\|_{L^p\left(G\right)} \lesssim_{G,\alpha,p} \liminf_{j\to\infty} \left\|\left(-\Delta_p\right)^{\alpha/2}f_j\right\|_{L^p\left(G\right)} = \left\|\left(-\Delta_p\right)^{\alpha/2}f\right\|_{L^p\left(G\right)}.
\]
where we have used Fatou's lemma twice in the first inequality, and our assumption that the $\lesssim$ statement of Theorem \ref{lpgenthm} holds for $\mathcal{D}$ in the second inequality. This completes the proof of our claim.

\subsubsection{Approximation by Taylor polynomials}\label{subsubsec:taylor}
Recall the notation for Taylor series used in subsection \ref{subsec:taylor}.
We may define the $\beta$-numbers and $\mathfrak{G}$-function using the Taylor polynomial instead:
\[
\tilde{\beta}_{f,d}\left(B_r\left(x\right)\right)\coloneqq \fint_{B_r\left(x\right)}\left|f-T^{\lfloor\alpha\rfloor}_{x}f\right|d\mu,
\]
and
\[
\tilde{\mathfrak{G}}_\alpha f\left(x\right)\coloneqq \left(\int_0^\infty \left(\frac{1}{r^\alpha}\tilde{\beta}_{f,\lfloor\alpha\rfloor}\left(B_r\left(x\right)\right)\right)^2 \frac{dr}r\right)^{1/2}, \quad x\in G.
\]
It turns out that when $\alpha$ is nonintegral, the Taylor polynomials $T^d_xf$ work as a proxy for the local best $L^2$-approximants $A^d_{x,r}f$, in the sense that $\mathfrak{G}_\alpha f$ and $\tilde{\mathfrak{G}}_\alpha f$ are pointwise equivalent up to constant multiplicative factors.
\begin{proposition}\label{prop:taylorIsEnough}
If $f\in \mathcal{D}$ and $\alpha$ is nonintegral, then $\mathfrak{G}_\alpha f\left(x\right)\asymp_{G,\alpha} \tilde{\mathfrak{G}}_\alpha f\left(x\right)$ for all $x\in G$.
\end{proposition}
\begin{proof}
That $\mathfrak{G}_\alpha f\left(x\right)\lesssim_{G,\alpha} \tilde{\mathfrak{G}}_\alpha f\left(x\right)$ follows from Lemma \ref{L1-approx}.
Denote $d=\lfloor\alpha\rfloor$ and
\[
A^d_{x,r}f\left(xy\right)=\sum_{\left|\gamma\right|\le d} f_\gamma\left(x,r\right)y^\gamma,
\]
where the coefficients are given by Lemma \ref{lem:taylor_identity}:
\[
f_\gamma \left(x,r\right)=\frac{1}{\gamma!}\operatorname{Sym}\left(X^\gamma\right)A^d_{x,r}f\left(x\right).
\]
Also denote
\[
T^d_{x}f\left(xy\right)=\sum_{\left|\gamma\right|\le d} f_\gamma \left(x\right)y^\gamma,
\]
then, in the case $f\in \mathcal{D}$, $f_\gamma \left(x,r\right)\to f_\gamma \left(x\right)$ as $r\to 0$ by Taylor expansion. Indeed, recalling that $f\left(xy\right)=T^d_{x}f\left(xy\right)+O_f\left(d_G\left(y,0\right)^{d+1}\right)$,
\[
\left\|A^d_{x,r}f-T^d_xf\right\|_{L^\infty\left(B_r\left(x\right)\right)}=\left\|A^d_{x,r}\left[f-T^d_xf\right]\right\|_{L^\infty\left(B_r\left(x\right)\right)}\stackrel{\mathclap{\eqref{claim}}}\lesssim_{G,d} \fint _{B_r\left(x\right)}\left|f-T^d_xf\right|d\mu=O_{G,d,f}\left(r^{d+1}\right) \quad \mathrm{as ~}r\to 0.
\]
By the above expansions for $A^d_{x,r}f$ and $T^d_xf$, we have
\[
\left\|\sum_{\left|\gamma\right|\le d}\left(f_\gamma\left(x,r\right)-f_\gamma\left(x\right)\right)r^{\left|\gamma\right|}y^\gamma\right\|_{L^\infty\left(B_1\right)}=\left\|\sum_{\left|\gamma\right|\le d}\left(f_\gamma\left(x,r\right)-f_\gamma\left(x\right)\right)y^\gamma\right\|_{L^\infty\left(B_r\right)}=O_{G,d,f}\left(r^{d+1}\right).
\]
But because the space $\mathcal{A}^d$ of polynomials of weighted degree $\le d$ normed by $L^\infty\left(B_1\right)$ is finite dimensional, all norms are equivalent, and hence
\[
\max_{\left|\gamma\right|\le d}r^{\left|\gamma\right|}\left| f_\gamma\left(x,r\right)-f_\gamma\left(x\right)\right|=O_{G,d,f}\left( r^{d+1}\right).
\]
This completes the proof of the convergence $f_\gamma\left(x,r\right)\to f_\gamma\left(x\right)$ as $r\to 0$.

With this, we provide a different bound on $\left|f_\gamma\left(x,r\right)-f_\gamma\left(x\right)\right|$ using the $\beta$ numbers. For each $x\in G$,
\begin{align}\label{r-approx}
\begin{aligned}
    \left|f_\gamma \left(x,r\right)-f_\gamma\left(x\right)\right|&\le \sum_{i=-\infty}^0\left|f_\gamma \left(x,2^ir\right)-f_\gamma\left(x,2^{i-1}r\right)\right|\\
    &= \frac{1}{\gamma!} \sum_{i=-\infty}^0\left|\operatorname{Sym}\left(X^\gamma\right)\left(A^d_{x,2^ir}f-A^d_{x,2^{i-1}r}f\right)\left(x\right)\right|\\
    &= \frac{1}{\gamma!} \sum_{i=-\infty}^0\left|\operatorname{Sym}\left(X^\gamma\right)A^d_{x,2^{i-1}r}\left(A^d_{x,2^ir}f-f\right)\left(x\right)\right|\\
    &\stackrel{\mathclap{\eqref{polydiff}}}{\lesssim}_{G,\left|\gamma\right|}  \sum_{i=-\infty}^0 2^{-i\left|\gamma\right|}r^{-\left|\gamma\right|}\fint_{B_{2^{i-1}r}\left(x\right)}\left|f-A^d_{x,2^ir}f\right|d\mu\\
    &\lesssim_{G}  \sum_{i=-\infty}^0 2^{-i\left|\gamma\right|}r^{-\left|\gamma\right|}\beta_{f,d}\left(B_{2^ir}\left(x\right)\right)\\
    &\stackrel{\mathclap{\mathrm{Corollary}~\ref{cor:betaIntConversion}}}{\lesssim}_{G,d} \int_0^{2r} \beta_{f,d}\left(B_u\left(x\right)\right)u^{-\left|\gamma\right|}\frac{du}{u}.
\end{aligned}
\end{align}
(With the Cauchy--Schwarz inequality, we may prove that $\left|f_\gamma \left(x,r\right)-f_\gamma\left(x\right)\right|\lesssim_{G,\alpha} r^{\alpha-\left|\gamma\right|}\mathfrak{G}_\alpha f\left(x\right)$ for $\left|\gamma\right|<\alpha$; see the proof of Proposition \ref{prop:measurableTaylor}(1) for the computation. This is weaker than the $O_{G,d,f}\left(r^{d+1-\left|\gamma\right|}\right)$ bound we get using Taylor approximation. This is because the Taylor approximation argument leverages on the fact that $f$ has $\left(d+1\right)$ and higher derivatives, whereas the $\mathfrak{G}_\alpha f$ (in principle) only measures the ``$\alpha$-th derivative'' of $f$ at $x$; however, the latter approach is more natural since we have to work with $\mathfrak{G}_\alpha f$.)

Now we bound
\begin{align*}
\tilde{\beta}_{f,d}\left(B_r\left(x\right)\right)&=\fint_{B_r\left(x\right)}\left|f-T^d_{x}f\right|\\
&\le \fint_{B_r\left(x\right)}\left|f-A^d_{x,r}f\right|+\fint_{B_r\left(x\right)}\left|T^d_{x}f-A^d_{x,r}f\right|\\
&\le \beta_{f,d}\left(B_r\left(x\right)\right)+C\sum_{\left|\gamma\right|\le d}\left|f_\gamma\left(x,r\right)-f_\gamma \left(x\right)\right|r^{\left|\gamma\right|}\\
&\le \beta_{f,d}\left(B_r\left(x\right)\right)+C\sum_{i=0}^d r^i \int_0^{2r} \beta_{f,d}\left(B_u\left(x\right)\right)u^{-i}\frac{du}{u},
\end{align*}
where $C=C_{G,d}$ is a constant depending on $G$ and $d$.
By plugging into the definition of $\tilde{\mathfrak{G}}$ and using the triangle inequality,
\begin{align*}
    \tilde{\mathfrak{G}}_\alpha f\left(x\right)\le \mathfrak{G}_\alpha f\left(x\right)+C\sum_{i=0}^d \left(\int_0^\infty \left[r^{i-\alpha} \int_0^{2r} \beta_{f,d}\left(B_u\left(x\right)\right)u^{-i}\frac{du}{u}\right]^{2}\frac{dr}{r}\right)^{1/2}.
\end{align*}
But by Hardy's inequality \eqref{Hardy}, we have
\begin{align*}
    \tilde{\mathfrak{G}}_\alpha f\left(x\right)&\le \mathfrak{G}_\alpha f\left(x\right)+C\sum_{i=0}^d \frac{1}{\alpha-i}\left(\int_0^\infty \left[\frac{\beta_{f,d}\left(B_r\left(x\right)\right)}{r^\alpha}\right]^2 \frac{dr}{r}\right)^{1/2}\\
    &\lesssim_{G,\alpha} \mathfrak{G}_\alpha f\left(x\right).
\end{align*}
This is where we use the fact that $\alpha$ is nonintegral.
\end{proof}

\subsubsection{Inhomogeneous estimates to homogeneous estimates}\label{subsubsec:scaling}

We provide one more reduction in the proof of the $\lesssim$ direction of Theorem \ref{lpgenthm}. It is enough to prove the weaker estimate
\[
\left\|\mathfrak{G}_{\alpha}f\right\|_{L^p\left(G\right)} \lesssim_{G,p,\alpha} \left\|f\right\|_{p,\alpha} = \left\|f\right\|_{L^p\left(G\right)} + \left\|\left(-\Delta_p\right)^{\alpha/2}f\right\|_{L^p\left(G\right)},\quad f\in \mathcal{D}.
\]
(In fact, the inequality Dorronsoro proves in \cite[Theorem 2]{dorronsoro1985characterization} is
\[
\left\|f\right\|_{L^p\left(G\right)}+\left\|\mathfrak{G}_{\alpha}f\right\|_{L^p\left(G\right)} \asymp_{\alpha,p} \left\|f\right\|_{p,\alpha}
\]
when $G=\mathbb{R}^n$.) This is because a dimensional analysis reveals that the above inequality is inhomogeneous, so by exploiting scaling properties we may remove the inhomogeneous term $\left\|f\right\|_{L^p\left(G\right)}$.
\begin{lemma}
For $1<p<\infty$ and $\alpha>0$, if we have the estimate,
\begin{equation}\label{inhom} \left\|\mathfrak{G}_\alpha f\right\|_{L^p\left(G\right)} \le C_{G,p,\alpha} \left\|f\right\|_{p,\alpha},\quad f \in \mathcal{D},
\end{equation}
then with the same constant $C_{G,p,\alpha}$,
\begin{equation}\label{hom} \left\|\mathfrak{G}_\alpha f\right\|_{L^p\left(G\right)} \le C_{G,p,\alpha}\left\|\left(-\Delta_p\right)^{\alpha/2} f\right\|_{L^p\left(G\right)},\quad f\in \mathcal{D}.
\end{equation}
\end{lemma}

\begin{proof}
Given $f \in \mathcal{D}$, set $f_s\coloneqq f \circ \delta_s \in \mathcal{D}$, for $s>0$, where $\delta_s$ is the Carnot group dilation. It is easy to see that
\[
\begin{cases}
\left\|f_s\right\|_{L^p\left(G\right)} = s^{-n_h/p} \left\|f\right\|_{L^p\left(G\right)},\\
\left\|\mathfrak{G}_\alpha f_s\right\|_{L^p\left(G\right)} =  s^{\alpha-n_h/p}  \left\|\mathfrak{G}_\alpha f\right\|_{L^p\left(G\right)},
\end{cases}\quad s>0.
\]
Also, one may verify
\[
\left\|\left(-\Delta_p\right)^{\alpha/2} f_s\right\|_{L^p\left(G\right)} = s^{\alpha-n_h/p}  \left\|\left(-\Delta_p\right)^{\alpha/2} f\right\|_{L^p\left(G\right)},\quad s>0,
\]
using the definition of $\left(-\Delta_p\right)^{\alpha/2}$; the proof is the same as that of \cite[Lemma 2.6]{fassler2020dorronsoro}.

We now obtain \eqref{hom} for $f$ from \eqref{inhom} for $f_s$ by taking $s\to\infty$.
\end{proof}

\subsubsection{The case $0<\alpha<1$}\label{subsubsec:alpha<1}
We now begin proving the $\lesssim$ direction of Theorem \ref{lpgenthm}.

Fix $1<p<\infty$ and $0<\alpha<1$. Note that we have
\[
\tilde{\mathfrak{G}}_\alpha f\left(x\right)= \left(\int_0^\infty \left[\frac{1}{r^\alpha}\fint_{B_r}\left|f\left(xy\right)-f\left(x\right)\right|d\mu\left(y\right)\right]^2 \frac{dr}{r}\right)^{1/2},\quad f\in \mathcal{D},~x\in G,
\]
and by Proposition \ref{prop:taylorIsEnough}
\[
\mathfrak{G}_\alpha f\left(x\right)\asymp_{G,\alpha}\tilde{\mathfrak{G}}_{\alpha} f\left(x\right),\quad  x\in G.
\]
But by \cite[Theorem 5]{coulhon2001sobolev},
\[
\left\|\tilde{\mathfrak{G}}_\alpha f\right\|_{L^p\left(G\right)}\asymp_{G,\alpha,p} \left\|\left(-\Delta_p\right)^{\alpha/2}f\right\|_{L^p\left(G\right)}, \quad f\in \mathcal{D}.
\]
Therefore $\lesssim$ of Theorem \ref{lpgenthm} follows in this case. Note that we haven't proven both directions of the inequalities of Theorem \ref{lpgenthm} yet due to the restriction $f\in \mathcal{D}$.

\subsubsection{The case $\alpha$ nonintegral, $\alpha>1$}\label{subsubsec:nonintegral>1}

We will reduce to the case $0<\alpha<1$.

For induction, we need the following result. Although we are currently working with $f\in \mathcal{D}$, we state the proposition below for $f\in S_\alpha^p\left(G\right)$, because we will need it later again in the proof of the $\gtrsim$ direction of Dorronsoro's theorem.
\begin{proposition}[{\cite[Theorem 4.10]{folland1975subelliptic}}]\label{prop:inductOnDelta}
Let $1<p<\infty$ and $\alpha> 1$, and let $f\in L^p\left(G\right)$. Then $f\in S^p_\alpha\left(G\right)$ if and only if $f\in S^p_{\alpha-1}\left(G\right)$ and the distributional derivatives $X_if\in S^p_{\alpha-1}\left(G\right)$ for $i=1,\cdots,k$, in which case
\begin{equation}\label{eq:inductOnDelta}
\sum_{j=1}^{k} \left\|\left(-\Delta_p\right)^{\left(\alpha-1\right)/2}X_{j}f\right\|_{L^p\left(G\right)} \asymp_{G,\alpha,p} \left\|\left(-\Delta_p\right)^{\alpha/2}f\right\|_{L^p\left(G\right)}, \quad f\in S_\alpha^p\left(G\right).
\end{equation}
\end{proposition}
\begin{remark}
To be precise, Theorem 4.10 of \cite{folland1975subelliptic} states the inhomogeneous estimate
\begin{align*}
&\left\|f\right\|_{L^p\left(G\right)}+\left\|\left(-\Delta_p\right)^{\left(\alpha-1\right)/2}f\right\|_{L^p\left(G\right)}+\sum_{j=1}^{k} \left\|X_jf\right\|_{L^p\left(G\right)}+\left\|\left(-\Delta_p\right)^{\left(\alpha-1\right)/2}X_{j}f\right\|_{L^p\left(G\right)} \\
&\asymp_{G,\alpha,p} \left\|f\right\|_{L^p\left(G\right)}+\left\|\left(-\Delta_p\right)^{\alpha/2}f\right\|_{L^p\left(G\right)}, \quad f\in S_\alpha^p\left(G\right).
\end{align*}
But now \eqref{eq:inductOnDelta} easily follows using the homogenization argument of subsubsection \ref{subsubsec:scaling}.
\end{remark}

The following was inspired by and generalizes \cite[Theorem 5]{dorronsoro1985characterization} and \cite[Proposition 4.2]{fassler2020dorronsoro}.

\begin{proposition}\label{prop:g_alpha_induct}
Let $1 < p < \infty$ and $ \alpha>0 $. Then,
\[
\left\|\mathfrak{G}_{\alpha + 1}f\right\|_{L^p\left(G\right)} \lesssim_{G,\alpha} \sum_{j=1}^k \left\|\mathfrak{G}_{\alpha}\left(X_jf\right)\right\|_{L^p\left(G\right)}, \quad f \in \mathcal{D}. 
\]
\end{proposition}
\begin{proof}
Let $d=\lfloor\alpha\rfloor$. Define a function $\tilde{A}^{d+1}_{x,r}f \in {\mathcal{A}}_{d+1}$ for $x\in G$ and $r>0$ as
\[
\tilde{A}^{d+1}_{x,r}f\left(xy\right)\coloneqq T^{d+1}_xf\left(xy\right)-\fint_{B_r\left(x\right)} T^{d+1}_xf\left(z\right)d\mu\left(z\right)+\fint_{B_r\left(x\right)} f\left(z\right)d\mu\left(z\right).
\]
We choose $C\ge 1$ depending on $G$ such that the weak 1-Poincar\'e inequality due to Jerison \cite{jerison1986poincare} holds:
\[
\fint_{B_s\left(y\right)}\left|g-\fint_{B_s\left(y\right)} g\left(z\right)d\mu\left(z\right)\right|\lesssim_G s\fint_{B_{Cs}\left(y\right)} \left|\nabla g\right|,\quad y\in G, ~s>0,~g\in \mathcal{D}.
\]
As $\fint_{B_r\left(x\right)} \left(f\left(z\right)-\tilde{A}^{d+1}_{x,r}f\left(z\right)\right)d\mu\left(z\right)=0$, we have
\begin{align*}
    \beta_{f,d +1}\left(B_r\left(x\right)\right)~&\stackrel{\mathclap{\mathrm{Lemma }~\ref{L1-approx}}}{\lesssim}_{G,d} \fint_{B_r\left(x\right)}\left|f-\tilde{A}^{d+1}_{x,r}f\right|d\mu \quad\stackrel{\mathclap{\mathrm{weak~1-Poincar\acute{e}}}}{\lesssim}_G \quad r\fint_{B_{Cr}\left(0\right)}\left|\nabla f\left(xy\right)-\nabla T^{d+1}_xf\left(xy\right)\right|d\mu\left(y\right)\\
    &\asymp_G r\sum_{j=1}^k \fint_{B_{Cr}\left(0\right)}\left|X_jf\left(xy\right)-X_jT^{d+1}_xf\left(xy\right)\right|d\mu\left(y\right)\\
&\stackrel{\mathclap{\mathrm{Lemma~}\ref{lem:taylor_unique}}}{=}\quad r\sum_{j=1}^k \fint_{B_{Cr}\left(0\right)}\left|X_jf\left(xy\right)-T^d_xX_jf\left(xy\right)\right|d\mu\left(y\right)\\
&=r\sum_{j=1}^k \tilde{\beta}_{X_jf,d}\left(B_{Cr}\left(x\right)\right).
\end{align*}
Now, by definition,
\begin{align*}
\mathfrak{G}_{\alpha+1}f\left(x\right)&=\left(\int_0^\infty \left[\frac{1}{r^{\alpha+1}}\beta_{f,d+1}\left(B_r\left(x\right)\right)\right]^2\frac{dr}{r}\right)^{1/2}\lesssim_{G,d} \sum_{j=1}^k\left(\int_0^\infty \left[\frac{1}{r^\alpha}\tilde{\beta}_{X_jf,d}\left(B_{Cr}\left(x\right)\right)\right]^2\frac{dr}{r}\right)^{1/2}\\
&= C^\alpha \sum_{j=1}^k\tilde{\mathfrak{G}}_\alpha \left(X_jf\right)\left(x\right),\quad x\in G.
\end{align*}
By Proposition \ref{prop:taylorIsEnough}, the proof is complete.
\end{proof}

If we suppose by induction that the $\lesssim$ statement of Theorem \ref{lpgenthm} holds for all nonintegral $\alpha<d$ given a fixed $d\in \mathbb{Z}_{>0}$, then from the above two Propositions, we have that for $d<\alpha<d+1$ and $f\in \mathcal{D}$,
\[
\left\|\mathfrak{G}_\alpha f\right\|_{L^p\left(G\right)} \stackrel{\mathclap{\mathrm{Proposition~}\ref{prop:g_alpha_induct}}}{\lesssim}_{G,\alpha} ~\sum_{j=1}^k \left\|\mathfrak{G}_{\alpha-1} \left(X_jf\right)\right\|_{L^p\left(G\right)}\stackrel{\mathclap{\mathrm{Induction}}}{\lesssim}_{G,\alpha,p}\sum_{j=1}^k \left\|\left(-\Delta_p\right)^{\left(\alpha-1\right)/2}X_jf\right\|_{L^p\left(G\right)} \stackrel{\mathclap{\mathrm{Proposition~}\ref{prop:inductOnDelta}}}{\asymp}_{G,\alpha,p} \left\|\left(-\Delta_p\right)^{\alpha/2}f\right\|_{L^p\left(G\right)}.
\]
This completes the proof of the $\lesssim$ statement of Theorem \ref{lpgenthm} for $q=1$ and $\alpha$ nonintegral.


\subsubsection{Interpolation and the case $\alpha$ integral}\label{subsubsec:integral}

The case $\alpha=d\in \mathbb{Z}$ is proven by complex interpolation. This subsubsection follows Section 5 of \cite{fassler2020dorronsoro} closely.

For $1<p<\infty$ and $0<\alpha_0<\alpha_1<\infty$, the pair $\left(S^p_{\alpha_0},S^p_{\alpha_1}\right)$ of Banach spaces is an interpolation pair in the sense of Calder\'on, as they embed continuously in the space $\mathcal{S}'\left(G\right)$ of tempered distributions on $G$. Thus, we may define the complex interpolation space $\left[S^p_{\alpha_0},S^p_{\alpha_1}\right]_\theta$, $\theta\in \left(0,1\right)$.

The following lemma asserts that we have a continuous embedding
\[
S^p_{\left(1-\theta\right)\alpha_0+\theta\alpha_1}\hookrightarrow \left[S^p_{\alpha_0},S^p_{\alpha_1}\right]_\theta.
\]

\begin{lemma}[{\cite[Lemma 5.1]{fassler2020dorronsoro}}]\label{lem:interpolation}
Let $1<p<\infty$, $0<\alpha_0<\alpha_1<\infty$, and $\theta \in \left(0,1\right)$. Then
\[
\left\|f\right\|_{\left[S^p_{\alpha_0},S^p_{\alpha_1}\right]_\theta}\lesssim_{G,\alpha_0,\alpha_1,p} \left\|f\right\|_{p,\left(1-\theta\right)\alpha_0+\theta\alpha_1}.
\]
\end{lemma}
The proof is the same as that in \cite[Lemma 5.1]{fassler2020dorronsoro}, since they only use that $G$ is a graded Lie group.

Next, we define, for $\alpha>0$, the Banach space $H_\alpha$ of functions $F:B_1\times\left(0,\infty\right)\to\mathbb{R}$ with norm
\[
\left\|F\right\|_{H_\alpha}\coloneqq \left(\int_0^\infty \left[\frac{1}{r^\alpha}\fint_{B_1}\left|F\left(y,r\right)\right|d\mu\left(y\right)\right]^2\frac{dr}{r}\right)^{1/2}<\infty.
\]
Then, for $1<p<\infty$, we may define the function space $L^p\left(G,H_\alpha\right)$. Now, for $0<\alpha_0<\alpha_1<\infty$,
\[
\left(L^p\left(G,H_{\alpha_0}\right),L^p\left(G,H_{\alpha_1}\right)\right)
\]
is a compatible couple, as both embed continuously into $L_\mathrm{loc}^1\left(G\times B_1\times \left(0,\infty\right)\right)$. As in \cite{dorronsoro1985characterization,fassler2020dorronsoro}, we apply the results of \cite[p. 107, 121]{bergh2012interpolation} so that for $\theta\in \left(0,1\right)$,
\[
\left[L^p\left(G,H_{\alpha_0}\right),L^p\left(G,H_{\alpha_1}\right)\right]_\theta = L^p\left(G,\left[H_{\alpha_0},H_{\alpha_1}\right]_\theta\right)=L^p\left(G,H_{\left(1-\theta\right)\alpha_0+\theta\alpha_1}\right).
\]

Now, for $d-1<\alpha_0<d<\alpha_1<d+1$ (recall $d=\alpha$) and $\theta\in \left(0,1\right)$, consider the linear map $T:S^p_{\alpha_0}+S^p_{\alpha_1}\to L^p\left(G,H_{\alpha_0}\right)+L^p\left(G,H_{\alpha_1}\right)$ given by
\[
f\mapsto Tf\left(x,r,y\right)=f\left(x\delta_r\left(y\right)\right)-A^{d}_{x,r}f\left(x\delta_r\left(y\right)\right).
\]
By the $\lesssim$ direction of Theorem \ref{lpgenthm} for $q=1$ and $\alpha\in \left(d-1,d\right)\cup \left(d,d+1\right)$ (which is now proven), we have
\[
\left\|Tf\right\|_{L^p\left(G,H_{\alpha}\right)}\stackrel{\mathclap{\mathrm{Lemma~}\ref{L1-approx}}}{\lesssim}_{G,d} \left\|\mathfrak{G}_{\alpha} f\right\|_{L^p\left(G\right)}\lesssim_{G,\alpha_0,\alpha_1,p} \left\|f\right\|_{p,\alpha},\quad f\in S^p_{\alpha},~\alpha=\alpha_0,\alpha_1,
\]
i.e., $T:S^p_{\alpha_0}\to L^p\left(G,H_{\alpha_0}\right)$ and $T:S^p_{\alpha_1}\to L^p\left(G,H_{\alpha_1}\right)$ are bounded (we remark that the first inequality need not be an equality, as $d$ may not equal $\lfloor\alpha\rfloor$, so that it is necessary to use Lemma \ref{L1-approx}). By complex interpolation,
\[
T:\left[S^p_{\alpha_0},S^p_{\alpha_1}\right]_\theta \to \left[L^p\left(G,H_{\alpha_0}\right),L^p\left(G,H_{\alpha_1}\right)\right]_\theta=L^p\left(G,H_{\left(1-\theta\right)\alpha_0+\theta\alpha_1}\right)
\]
is bounded.

Now set $\alpha_0=d-1/2$, $\alpha_1=d+1/2$, and $\theta=1/2$. Then
\[
\left\|\mathfrak{G}_d f\right\|_{L^p\left(G\right)}=\left\|Tf\right\|_{L^p\left(G,H_d\right)}\lesssim_{G,d} \left\|f\right\|_{\left[S^p_{d-1/2},S^p_{d+1/2}\right]_{1/2}}\stackrel{\mathclap{\mathrm{Lemma~}\ref{lem:interpolation}}}{\lesssim}_{G,d,p}\left\|f\right\|_{p,d},\quad f\in \mathcal{D}.
\]
This completes the proof of the $\lesssim$ direction of Theorem \ref{lpgenthm} for $\alpha=d\in \mathbb{Z}$.

\subsection{Proof of the $\gtrsim$ direction}\label{subsec:geQ=1}
Now, given an $f\in L^p\left(G\right)$ and $\alpha>0$ with $\mathfrak{G}_\alpha f\in L^p\left(G\right)$, we need to prove that $f\in S^p_\alpha\left(G\right)$ and that the $\gtrsim$ direction of \eqref{form13} holds. By Jensen's inequality 
\[
\beta_{f,\lfloor \alpha \rfloor}\left(B_r\left(x\right)\right)\le \beta_{f,\lfloor \alpha \rfloor,q}\left(B_r\left(x\right)\right),
\]
the $\gtrsim$ direction of \eqref{form13} for $q>1$ would follow from the case $q=1$. We will assume $q=1$ and prove the $\gtrsim$ direction of \eqref{form13} in this subsection.

We first prove the following preparatory result.
\begin{proposition}\label{prop:measurableTaylor}
Let $f\in L^p\left(G\right)$ and $\alpha>0$ with $\mathfrak{G}_\alpha f\in L^p\left(G\right)$. Define the scalars $f_\gamma\left(x,r\right)$ by
\[
A^{\lfloor\alpha\rfloor}_{x,r}f\left(xy\right)=\sum_{\left|\gamma\right|\le \lfloor \alpha \rfloor} f_\gamma\left(x,r\right)y^\gamma,
\]
or equivalently, by \eqref{eq:Taylor_def},
\begin{equation}\label{eq:taylor_f_gamma}
    f_\gamma\left(x,r\right)=\frac{1}{\gamma!}\operatorname{Sym}\left(X^\gamma\right)A^{\lfloor\alpha\rfloor}_{x,r}f\left(x\right),\quad \left|\gamma\right|\le \lfloor\alpha\rfloor.
\end{equation}
Then

\begin{enumerate}
    \item For each multi-index $\gamma$ with $\left|\gamma\right|<\alpha$, we have
\[
\left|f_\gamma \left(x,r\right)-f_\gamma\left(x,s\right)\right|\lesssim_{G,\alpha}r^{\alpha-\left|\gamma\right|}\mathfrak{G}_\alpha f\left(x\right),\quad x\in G,~0<s<r,
\]
and
\[
\left|f_\gamma\left(x,r\right)\right|\lesssim_{G,\alpha} \mathfrak{G}_\alpha f\left(x\right)+Mf\left(x\right),\quad x\in G,~0<r<1,
\]
where $M$ denotes the Hardy--Littlewood maximal operator. In particular, $f_\gamma\left(\cdot,r\right)\in L^p\left(G\right)$ with
\[
\left\|f_\gamma\left(\cdot,r\right)\right\|_{L^p\left(G\right)}\lesssim_{G,\alpha} \left\|\mathfrak{G}_\alpha f\right\|_{L^p\left(G\right)}+\left\|f\right\|_{L^p\left(G\right)}.
\]
    \item For each multi-index $\gamma$ with $\left|\gamma\right|<\alpha$, there exists $f_\gamma\in L^p\left(G\right)$ such that $f_\gamma\left(\cdot,r\right)\to f_\gamma\left(\cdot\right)$ as $r\to 0$, satisfying
    \[
    \left|f_\gamma \left(x,r\right)-f_\gamma\left(x\right)\right|\lesssim_{G,\left|\gamma\right|} \int_{0}^{2r} \beta_{f,\lfloor\alpha\rfloor}\left(B_u\left(x\right)\right)u^{-\left|\gamma\right|}\frac{du}{u}\lesssim_{G,\alpha}r^{\alpha-\left|\gamma\right|}\mathfrak{G}_\alpha f\left(x\right), \quad r>0,~\mathrm{a.e.~}x\in G.
    \]
    \item If $\alpha$ is nonintegral, then
    \[
\left(\int_0^\infty \left(\frac{1}{r^\alpha}\fint_{B_r}\left|f\left(xy\right)-\sum_{\left|\gamma\right|\le \lfloor \alpha \rfloor} f_\gamma\left(x\right)y^\gamma\right|d\mu\left(y\right)\right)^2 \frac{dr}r\right)^{1/2}\asymp_{G,\alpha} \mathfrak{G}_\alpha f\left(x\right), \quad \mathrm{a.e.~}x\in G.
\]
    \item If $\alpha$ is integral and $\left|\gamma\right|=\alpha$, then $\left|f_\gamma\left(x,r\right)\right|\lesssim_{G,\alpha} \log \left(2/r\right)\mathfrak{G}_\alpha f\left(x\right)+Mf\left(x\right)$, $0<r<1$. Thus, $f_\gamma\left(\cdot,r\right)\in L^p\left(G\right)$ with $\left\|f_\gamma\left(\cdot,r\right)\right\|_{L^p\left(G\right)}\lesssim_{G,\alpha} \log\left(2/r\right)\left\|\mathfrak{G}_\alpha f\right\|_{L^p\left(G\right)}+\left\|f\right\|_{L^p\left(G\right)}$.
    \item If $0<\alpha<1$, then $f_0\left(x\right)=f\left(x\right)$ for a.e.~$x\in G$.
\end{enumerate}

\end{proposition}
\begin{proof}Denote $d=\lfloor\alpha\rfloor$.
\begin{enumerate}
    \item Repeating the computation \eqref{r-approx} made in the proof of Proposition \ref{prop:taylorIsEnough} and denoting by $n\ge 0$ the largest nonnegative integer such that $2^ns<r$,
\begin{align*}
    \left|f_\gamma \left(x,r\right)-f_\gamma\left(x,s\right)\right|&\le \sum_{i=1}^n\left|f_\gamma \left(x,2^is\right)-f_\gamma\left(x,2^{i-1}s\right)\right|+\left|f_\gamma \left(x,r\right)-f_\gamma\left(x,2^ns\right)\right|\\
    &= \frac{1}{\gamma!} \sum_{i=1}^n\left|\operatorname{Sym}\left(X^\gamma\right)\left(A^d_{x,2^is}f-A^d_{x,2^{i-1}s}f\right)\left(x\right)\right|+\frac{1}{\gamma!} \left|\operatorname{Sym}\left(X^\gamma\right)\left(A^d_{x,r}f-A^d_{x,2^{n}s}f\right)\left(x\right)\right|\\
    &= \frac{1}{\gamma!} \sum_{i=1}^n\left|\operatorname{Sym}\left(X^\gamma\right)A^d_{x,2^{i-1}s}\left(A^d_{x,2^is}f-f\right)\left(x\right)\right|+\frac{1}{\gamma!} \left|\operatorname{Sym}\left(X^\gamma\right)A^d_{x,2^{n}s}\left(A^d_{x,r}f-f\right)\left(x\right)\right|\\
    &\stackrel{\mathclap{\eqref{polydiff}}}{\lesssim}_{G,\left|\gamma\right|}  \sum_{i=1}^n 2^{-i\left|\gamma\right|}s^{-\left|\gamma\right|}\fint_{B_{2^{i-1}s}\left(x\right)}\left|f-A^d_{x,2^is}f\right|d\mu+2^{-n\left|\gamma\right|}s^{-\left|\gamma\right|}\fint_{B_{2^{n}s}\left(x\right)}\left|f-A^d_{x,r}f\right|d\mu\\
    &\lesssim_{G}  \sum_{i=1}^n 2^{-i\left|\gamma\right|}s^{-\left|\gamma\right|}\beta_{f,d}\left(B_{2^is}\left(x\right)\right)+2^{-n\left|\gamma\right|}s^{-\left|\gamma\right|}\beta_{f,d}\left(B_{r}\left(x\right)\right)\\
    &\stackrel{\mathclap{\mathrm{Corollaries}~\ref{monotoneBeta},\ref{cor:betaIntConversion}}}{\lesssim}_{G,d}\quad \int_s^{2r} \beta_{f,d}\left(B_u\left(x\right)\right)u^{-\left|\gamma\right|}\frac{du}{u}.
\end{align*}
i.e.,
\[
\left|f_\gamma \left(x,r\right)-f_\gamma\left(x,s\right)\right|\lesssim_{G,d} \int_{s}^{2r} \beta_{f,d}\left(B_u\left(x\right)\right)u^{-\left|\gamma\right|}\frac{du}{u},\quad x\in G,~0<s<r.
\]
In particular, by Cauchy--Schwarz,
\begin{align*}
\left|f_\gamma \left(x,r\right)-f_\gamma\left(x,s\right)\right|&\lesssim_{G,d} \left(\int_{s}^{2r}\left[\frac{\beta_{f,d}\left(B_u\left(x\right)\right)}{u^\alpha}\right]^2\frac{du}{u}\right)^{1/2}\left(\int_{s}^{2r}\left(u^{\alpha-\left|\gamma\right|}\right)^2\frac{du}{u}\right)^{1/2}\\
&\lesssim_\alpha r^{\alpha-\left|\gamma\right|}\mathfrak{G}_\alpha f\left(x\right),\quad x\in G,~0<s<r.
\end{align*}
Also, for $0<r<1$, $\left|f_\gamma\left(x,r\right)\right|\le \left|f_\gamma\left(x,r\right)-f_\gamma\left(x,1\right)\right|+\left|f_\gamma\left(x,1\right)\right|\lesssim_{G,\alpha} \mathfrak{G}_\alpha f\left(x\right)+Mf\left(x\right)$, where the latter inequality used the fact that $f_\gamma\left(x,1\right)$ is given as the integral of $f$ multiplied by a polynomial determined by $\gamma$ and $\lfloor\alpha\rfloor$, namely \eqref{eq:adxr-integral}. In particular, this implies that $f_\gamma\left(\cdot,r\right)\in L^p$, and that
\[
\left\|f_\gamma\left(\cdot,r\right)\right\|_{L^p\left(G\right)}\lesssim_{G,\alpha} \left\|\mathfrak{G}_\alpha f\right\|_{L^p\left(G\right)}+\left\|f\right\|_{L^p\left(G\right)}.
\]
    \item By (1) and the fact that $\mathfrak{G}_\alpha f\in L^p\left(G\right)$, $\left\{f_\gamma\left(\cdot,r\right)\right\}_{r>0} $ is a Cauchy sequence in $L^p$ as $r\to 0$. Thus, there exists $f_\gamma\left(\cdot\right)\in L^p\left(G\right)$ so that $f_\gamma \left(\cdot,r\right)\to f_\gamma \left(\cdot\right)$ in $L^p\left(G\right)$ as $r\to 0$. In the proof of (1), we obtained
    \[
    \left|f_\gamma \left(x,r\right)-f_\gamma\left(x,s\right)\right|\lesssim_{G,\left|\gamma\right|} \int_{s}^{2r} \beta_{f,d}\left(B_u\left(x\right)\right)u^{-\left|\gamma\right|}\frac{du}{u}\lesssim_{G,\alpha} r^{\alpha-\left|\gamma\right|}\mathfrak{G}_\alpha f\left(x\right),\quad x\in G,~0<s<r.
    \]
    Taking the limit $s\to 0$ along a sequence such that $f_\gamma\left(x,s\right)\to f_\gamma\left(x\right)$ for a.e.~$x\in G$, we have
    \[
    \left|f_\gamma \left(x,r\right)-f_\gamma\left(x\right)\right|\lesssim_{G,\left|\gamma\right|} \int_{0}^{2r} \beta_{f,d}\left(B_u\left(x\right)\right)u^{-\left|\gamma\right|}\frac{du}{u}\lesssim_{G,\alpha} r^{\alpha-\left|\gamma\right|}\mathfrak{G}_\alpha f\left(x\right),\quad \mathrm{a.e.}~x\in G,~r>0.
    \]
    
    \item The $\gtrsim$ statement follows from Lemma \ref{L1-approx}. For the $\lesssim$ statement, we bound
\begin{align*}
\fint_{B_r}\left|f\left(xy\right)-\sum_{\left|\gamma\right|\le \lfloor \alpha \rfloor} f_\gamma\left(x\right)y^\gamma\right|d\mu&\le \fint_{B_r\left(x\right)}\left|f-A^{\lfloor \alpha \rfloor}_{x,r}f\right|d\mu+\fint_{B_r}\left|\sum_{\left|\gamma\right|\le \lfloor \alpha \rfloor} f_\gamma\left(x\right)y^\gamma-A^{\lfloor \alpha \rfloor}_{x,r}f\left(xy\right)\right|d\mu\left(y\right)\\
&\lesssim_{G,\alpha} \beta_{f,\lfloor \alpha \rfloor}\left(B_r\left(x\right)\right)+\sum_{\left|\gamma\right|\le \lfloor \alpha \rfloor}\left|f_\gamma\left(x,r\right)-f_\gamma \left(x\right)\right|r^{\left|\gamma\right|}\\
&\lesssim_{G,\alpha} \beta_{f,\lfloor \alpha \rfloor}\left(B_r\left(x\right)\right)+\sum_{i=0}^{\lfloor \alpha \rfloor} r^i \int_0^{2r} \beta_{f,\lfloor \alpha \rfloor}\left(B_u\left(x\right)\right)u^{-i}\frac{du}{u},
\end{align*}
where we used (2) in the last inequality. By plugging into the definition of $\mathfrak{G}$ and using the triangle inequality,
\begin{align*}
    &\left(\int_0^\infty \left(\frac{1}{r^\alpha}\fint_{B_r}\left|f\left(xy\right)-\sum_{\left|\gamma\right|\le \lfloor \alpha \rfloor} f_\gamma\left(x\right)y^\gamma\right|d\mu\left(y\right)\right)^2 \frac{dr}r\right)^{1/2}\\
    &\lesssim_{G,\alpha} \mathfrak{G}_\alpha f\left(x\right)+\sum_{i=0}^{\lfloor \alpha \rfloor} \left(\int_0^\infty \left[r^{i-\alpha} \int_0^{2r} \beta_{f,\lfloor \alpha \rfloor}\left(B_u\left(x\right)\right)u^{-i}\frac{du}{u}\right]^{2}\frac{dr}{r}\right)^{1/2}\\
    &\lesssim_{G,\alpha} \mathfrak{G}_\alpha f\left(x\right)+\sum_{i=0}^{\lfloor \alpha \rfloor} \frac{1}{\alpha-i}\left(\int_0^\infty \left[\frac{\beta_{f,\lfloor \alpha \rfloor}\left(B_r\left(x\right)\right)}{r^\alpha}\right]^2 \frac{dr}{r}\right)^{1/2}\\
    &\lesssim_\alpha \mathfrak{G}_\alpha f\left(x\right).
\end{align*}
where in the second inequality we used Hardy's inequality \eqref{Hardy} with $\nu=\alpha-i+\frac 12$ and $p=2$ along with the fact that $\alpha$ is nonintegral.

\item By (1),
\[
\left|f_\gamma \left(x,r\right)-f_\gamma\left(x,1\right)\right|\lesssim_{G,\left|\gamma\right|} \int_{r}^{2} \beta_{f,d}\left(B_u\left(x\right)\right)u^{-\left|\gamma\right|}\frac{du}{u},\quad x\in G, ~0<r<1.
\]
By Cauchy--Schwarz, when $\alpha$ is an integer and $\left|\gamma\right|=\alpha$,
\[
\left|f_\gamma \left(x,r\right)-f_\gamma\left(x,1\right)\right|\lesssim_{G,\left|\gamma\right|} \left(\int_{r}^{2}\left[\frac{\beta_{f,d}\left(B_u\left(x\right)\right)}{u^\alpha}\right]^2\frac{du}{u}\right)^{1/2}\left(\int_{r}^{2}\frac{du}{u}\right)^{1/2}\lesssim_\alpha \log\left(2/r\right)\mathfrak{G}_\alpha f\left(x\right),
\]
and therefore $\left|f_\gamma\left(x,r\right)\right|\lesssim_{G,\alpha}\log\left(2/r\right)\mathfrak{G}_\alpha f\left(x\right)+Mf\left(x\right)$. It follows that $f_\gamma\left(\cdot,r\right)\in L^p\left(G\right)$ with
\[
\left\|f_\gamma\left(\cdot,r\right)\right\|_{L^p\left(G\right)}\lesssim_{G,\alpha} \log\left(2/r\right)\left\|\mathfrak{G}_\alpha f\right\|_{L^p\left(G\right)}+\left\|f\right\|_{L^p\left(G\right)}.
\]
\item If $0<\alpha<1$, then $f_0\left(x,r\right)=\fint_{B_r\left(x\right)}f$. Taking $r\to 0$, the claim follows by the Lebesgue differentiation theorem for doubling metric measure spaces \cite[(3.4.10)]{heinonen2015sobolev}.
\end{enumerate}

\end{proof}
\begin{remark}
It seems likely that for $\left|\gamma\right|<\alpha$ we have $f_\gamma(x)=\frac{1}{\gamma!}\operatorname{Sym}\left(X^\gamma\right)f(x)$ for a.e.~$x\in G$, with the latter derivative existing in the distributional sense. This is evident when $0<\alpha<1$ and $\gamma=0$ as statement (5) above, and we will prove this holds when $\alpha>1$ and $\left|\gamma\right|=1$ in Lemma \ref{lem:weakFirstDeriv}. We will not prove this statement in full generality since we do not need it in this paper.
\end{remark}

\subsubsection{The case $0<\alpha\le 1$}\label{subsubsec:gtrsimalphale1}
The case $0<\alpha<1$ is a repetition and strengthening of subsubsection \ref{subsubsec:alpha<1}. Recalling
\[
\tilde{\mathfrak{G}}_\alpha f\left(x\right)= \left(\int_0^\infty \left[\frac{1}{r^\alpha}\fint_{B_r}\left|f\left(xy\right)-f\left(x\right)\right|d\mu\left(y\right)\right]^2 \frac{dr}{r}\right)^{1/2},\quad x\in G,
\]
by Proposition \ref{prop:measurableTaylor} (3) and (5) we have
\[
\mathfrak{G}_\alpha f\left(x\right)\asymp_{G,\alpha}\tilde{\mathfrak{G}}_\alpha f\left(x\right),\quad \mathrm{for~a.e.~} x\in G.
\]
But by \cite[Theorem 5]{coulhon2001sobolev},
\[
\left\|\tilde{\mathfrak{G}}_\alpha f\right\|_{L^p\left(G\right)}\asymp_{G,p} \left\|\left(-\Delta_p\right)^{\alpha/2}f\right\|_{L^p\left(G\right)}.
\]

For $\alpha=1$, it is known from \cite[Theorem 1.4]{de2021mean} that for $0<\alpha<2$, if we define
\[
S_\alpha f\left(x\right)=\left(\int_0^\infty \left[r^{-\alpha}\fint _{B_r}\left|\Delta_y^{2,\mathrm{sym}}f\left(x\right)\right|d\mu\left(y\right)\right]^2\frac{dr}{r}\right)^{1/2},
\]
where we denote the symmetric difference
\[
\Delta_y^{2,\mathrm{sym}}f\left(x\right)=f\left(xy\right)+f\left(xy^{-1}\right)-2f\left(x\right),
\]
then
\[
\left\|f\right\|_{p,\alpha}\asymp_{G,p,\alpha} \left\|f\right\|_{L^p\left(G\right)}+\left\|S_\alpha f\right\|_{L^p\left(G\right)}.
\]
By the scaling argument of subsubsection \ref{subsubsec:scaling}, we obtain
\[
\left\|\left(-\Delta_p\right)^{\alpha/2}f\right\|_{L^p\left(G\right)}\asymp_{G,p,\alpha} \left\|S_\alpha f\right\|_{L^p\left(G\right)}.
\]

Recall that balls are defined as open balls. Given a ball $B_r\left(z_0\right)$ for $z_0\in G$ and $r>0$ and a point $y\in B_r\left(z_0\right)$, we claim that we may find a sequence $\left\{z_n\right\}_{n=1}^\infty$ such that $B_{2^{-n-1}r}\left(z_{n+1}\right)\subset B_{2^{-n}r}\left(z_{n}\right)$ for all $n\in \mathbb{Z}_{\ge 0}$, and $z_n=y$ for all sufficiently large $n$. Indeed, if $d_G\left(y,z_0\right)\le \frac r2$, simply take $z_n=y$ for $n\ge 1$. If $d_G\left(y,z_0\right)>\frac r2$, then take a piecewise smooth path $\gamma:\left[0,\lambda\right]\to G$ of unit speed from $z_0$ to $y$ of length $t\in\left(\frac r2,\frac{3d_G\left(y,z_0\right)+r}{4}\right)$, and take $z_1=\gamma\left(\frac r2\right)$. Then $d_G\left(z_1,z_0\right)\le \frac r2$ so that $B_{r/2}\left(z_1\right)\subset B_r\left(z_0\right)$, while $d_G\left(y,z_1\right)< \frac{3d_G\left(y,z_0\right)-r}{4}<\frac r2$ so that $y\in B_{r/2}\left(z_1\right)$ and
\[
\frac 32 \frac{r-d_G\left(y,z_0\right)}{r}< \frac{r/2-d_G\left(y,z_1\right)}{r/2}.
\]
We may inductively continue this construction: if $d_G\left(y,z_1\right)\le \frac r4$, take $z_n=y$ for $n\ge 2$; if $d_G\left(y,z_1\right)\in \left(\frac r4,\frac r2\right)$, we may choose by the above argument some $z_2\in B_{r/2}\left(z_1\right)$ such that $B_{r/4}\left(z_2\right)\subset B_{r/2}\left(z_1\right)$, $y\in B_{r/4}\left(z_2\right)$, and
\[
\frac 32 \frac{r/2-d_G\left(y,z_1\right)}{r/2}< \frac{r/4-d_G\left(y,z_2\right)}{r/4},
\]
and so on for $z_3,z_4,\cdots$. This process must terminate, i.e., we must have $d_G\left(y,z_m\right)\le 2^{-m-1}r$ for some $m$, so that $z_n=y$ for $n\ge m$. This is because otherwise, we would have $d_G\left(y,z_m\right)>2^{-m-1}r$ for all $m$, which would imply
\[
\left(\frac 32\right)^m \frac{r-d_G\left(y,z_0\right)}{r}\le \frac{2^{-m}r-d_G\left(y,z_m\right)}{2^{-m}r}<\frac 12\quad \forall m\ge 0,
\]
contradicting the strict positivity of $\frac{r-d_G\left(y,z_0\right)}{r}$. This completes the proof of the claim.

Let $f$ be a locally integrable function, $d\in \mathbb{N}$, $x=z_0\in G$ and $r>0$. For any $y\in B_r\left(z_0\right)$, we have, by the above claim, a nested sequence of balls $\left\{B_{2^{-n}r}\left(z_n\right)\right\}_{n=0}^\infty$ eventually centered around $y$, i.e., $B_{2^{-n-1}r}\left(z_{n+1}\right)\subset B_{2^{-n}r}\left(z_{n}\right)$ for all $n\in \mathbb{Z}_{\ge 0}$ and $z_n=y$ for all sufficiently large $n$. We have
\begin{align*}
\fint_{B_{2^{-k}r}\left(z_k\right)}\left|f-A_{z_0,r}^df\right|d\mu&\le \fint_{B_{2^{-k}r}\left(z_k\right)}\left|f-A_{z_k,2^{-k}r}^d f\right|d\mu+\sum_{i=0}^{k-1}\fint_{B_{2^{-k}r}\left(z_k\right)}\left|A_{z_i,2^{-i}r}^df-A_{z_{i+1},2^{-i-1}r}^df\right|d\mu\\
&\le \fint_{B_{2^{-k}r}\left(z_k\right)}\left|f-A_{z_k,2^{-k}r}^d f\right|d\mu+\sum_{i=0}^{k-1}\left\|A_{z_{i+1},2^{-i-1}r}^d\left(f-A_{z_i,2^{-i}r}^df\right)\right\|_{L^\infty\left(B_{2^{-i-1}r}\left(z_{i+1}\right)\right)}\\
&\stackrel{\mathclap{\eqref{claim}}}{\lesssim}_{G,d}\fint_{B_{2^{-k}r}\left(z_k\right)}\left|f-A_{z_k,2^{-k}r}^d f\right|d\mu+\sum_{i=0}^{k-1}\fint_{B_{2^{-i-1}r}\left(z_{i+1}\right)}\left|f-A_{z_i,2^{-i}r}^df\right|d\mu\\
&\lesssim_G \sum_{i=0}^k \fint_{B_{2^{-i}r}\left(z_i\right)}\left|f-A^d_{z_i,2^{-i}r}f\right|d\mu=\sum_{i=0}^k \beta_{f,d}\left(B_{2^{-i}r}\left(z_i\right)\right).
\end{align*}
But since $y\in B_{2^{-i}r}\left(z_i\right)$, we have $B_{2^{-i}r}\left(z_i\right)\subset B_{2^{-i+1}r}\left(y\right)$, so
\[
\sum_{i=0}^k \beta_{f,d}\left(B_{2^{-i}r}\left(z_i\right)\right)\quad\stackrel{\mathclap{\mathrm{Corollary~}\ref{monotoneBeta}}}{\lesssim}_{G,d}~\sum_{i=0}^k \beta_{f,d}\left(B_{2^{-i+1}r}\left(y\right)\right)\quad\stackrel{\mathclap{\mathrm{Corollary~}\ref{cor:betaIntConversion}}}{\lesssim}_{G,d} \int_0^{4r}\beta_{f,d}\left(B_s\left(y\right)\right)\frac{ds}{s}
\]
so, for a.e.~$y\in B_r\left(x\right)$,
\begin{equation}\label{eq:localize-beta}
\left|f\left(y\right)-A_{x,r}^df\left(y\right)\right|=\lim_{k\to\infty}\fint_{B_{2^{-k}r}\left(z_k\right)}\left|f-A_{z_0,r}^df\right|d\mu\lesssim_{G,d} \int_0^{4r}\beta_{f,d}\left(B_s\left(y\right)\right)\frac{ds}{s},
\end{equation}
where the first equality uses Lebesgue's differentiation theorem for doubling metric measure spaces and $x=z_0$.

As the operator $\Delta_h^{2,\mathrm{sym}}$ annihilates polynomials of weighted degree $1$, we obtain for a.e.~$x\in G$
\begin{align*}
&\fint_{B_r}\left|\Delta_h^{2,\mathrm{sym}}f\left(x\right)\right|d\mu\left(h\right)=\fint_{B_r}\left|\Delta_h^{2,\mathrm{sym}}\left(f-A_{x,2r}^1f\right)\left(x\right)\right|d\mu\left(h\right)\\
&=\fint_{B_r}\left|\left(f-A_{x,2r}^1f\right)\left(xh\right)+\left(f-A_{x,2r}^1f\right)\left(xh^{-1}\right)-2\left(f-A_{x,2r}^1f\right)\left(x\right)\right|d\mu\left(h\right)\\
&\stackrel{\mathclap{\eqref{eq:localize-beta}}}{\lesssim}_{G,d}\fint_{B_r}\left(\int_0^{8r}\beta_{f,1}\left(B_s\left(xh\right)\right)\frac{ds}{s}+\int_0^{8r}\beta_{f,1}\left(B_s\left(xh^{-1}\right)\right)\frac{ds}{s}+2\int_0^{8r}\beta_{f,1}\left(B_s\left(x\right)\right)\frac{ds}{s}\right)d\mu\left(h\right)\\
&\lesssim \int_0^{8r}M\beta_{f,1}\left(B_s\left(\cdot\right)\right)\left(x\right)\frac{ds}{s},
\end{align*}
where $M\beta_{f,1}\left(B_s\left(\cdot\right)\right)$ denotes the maximal function of the function $y\mapsto M\beta_{f,1}\left(B_s\left(y\right)\right)$.
We obtain
\begin{align*}
S_1 f\left(x\right)&=\left(\int_0^\infty \left[r^{-1}\fint _{B_r}\left|\Delta_h^{2,\mathrm{sym}}f\left(x\right)\right|d\mu\left(h\right)\right]^2\frac{dr}{r}\right)^{1/2}\\
&=\left(\int_0^\infty \left[r^{-1}\int_0^{8r}M\beta_{f,1}\left(B_s\left(\cdot\right)\right)\left(x\right)\frac{ds}{s}\right]^2\frac{dr}{r}\right)^{1/2}\\
&\le \left(\int_0^\infty \left[\frac{M\beta_{f,1}\left(B_r\left(\cdot\right)\right)\left(x\right)}r\right]^2\frac{dr}{r}\right)^{1/2}
\end{align*}
where in the inequality, we used Hardy's inequality \eqref{Hardy} with $\nu=\frac 32$ and $p=2$.
Next, we have
\begin{align*}
&\left(\int_G\left(\int_0^\infty \left[\frac{M\beta_{f,1}\left(B_r\left(\cdot\right)\right)\left(x\right)}r\right]^2\frac{dr}{r}\right)^{p/2}d\mu\left(x\right)\right)^{1/p}\\
&\lesssim_G \max\left\{p,\frac{1}{p-1}\right\} \left(\int_G\left(\int_0^\infty \left[\frac{\beta_{f,1}\left(B_r\left(\cdot\right)\right)\left(x\right)}r\right]^2\frac{dr}{r}\right)^{p/2}d\mu\left(x\right)\right)^{1/p}.
\end{align*}
Indeed, by Corollaries \ref{monotoneBeta} and \ref{cor:betaIntConversion}, this is an instance of the Fefferman--Stein vector-valued maximal function inequality
\begin{equation}\label{FS}
\left( \int_G \left[\sum_{j \in \mathbb{Z}} M\left(g_j
\right)^u \right]^{v/u}  d\mu \right)^{1/v} \lesssim_G \max\left\{1,\frac{1}{u-1}\right\}\max\left\{v,\frac{1}{v-1}\right\}\left( \int_G \left[\sum_{j \in \mathbb{Z}} g_j^u \right]^{v/u}  d\mu \right)^{1/v}
\end{equation}
on doubling metric measure spaces for measurable functions $\left(g_j:G\to \mathbb{R}\right)_{j \in \mathbb{Z}}$  and $u,v \in \left(1,\infty\right)$, due to \cite[Theorem 1.2]{grafakos2009vector}. Indeed, we apply \eqref{FS} to the functions
\[
g_j\left(x\right) = 2^j\beta_{f,1}\left(B_{2^{-j}}\left(x\right)\right)
\]
and exponents $u=2 > 1$ and $v=p>1$, noting by Corollary \ref{monotoneBeta} that
\[
M\left(r^{-1}\beta_{f,1}\left(B_r\left(\cdot\right)\right)\right) \lesssim_G M\left(g_j\right), \quad r \in \left[2^{-j-1},2^{-j}\right],
\]
and
\[
g_j\lesssim_G r^{-1}\beta_{f,1}\left(B_r\left(\cdot\right)\right), \quad r \in\left[2^{-j},2^{-j+1}\right],
\]

Thus, we have $\left\|S_1f\right\|_{L^p\left(G\right)}\lesssim_{G,p} \left\|\mathfrak{G}_1f\right\|_{L^p\left(G\right)}$, and thus $\left\|\left(-\Delta_p\right)^{1/2}f\right\|_{L^p\left(G\right)}\lesssim_{G,p} \left\|\mathfrak{G}_1f\right\|_{L^p\left(G\right)}$, as desired.

\subsubsection{The case $\alpha>1$}\label{subsubsec:gtrsimalpha>1}
We use induction on $\alpha$.

The following lemma generalizes \cite[Lemma 1]{dorronsoro1985characterization}.
\begin{lemma}\label{lem:inductOnG}
Let $f\in L^p\left(G\right)$ such that $\mathfrak{G}_\alpha f\in L^p\left(G\right)$, and let $\gamma$ be a multi-index with $\left|\gamma\right|<\alpha$. With $f_\gamma$ as in Proposition \ref{prop:measurableTaylor}, $\left\|\mathfrak{G}_{\alpha-\left|\gamma\right|}f_\gamma\right\|_{L^p\left(G\right)}\lesssim_{G,\alpha} \left\|\mathfrak{G}_\alpha f\right\|_{L^p\left(G\right)}$.
\end{lemma}
\begin{proof}

Set $d=\lfloor \alpha\rfloor$. Fix $x\in G$ and $r>0$. For $y\in B_r\left(x\right)$,
\begin{align*}
\left|f_\gamma\left(y\right)-\frac{1}{\gamma!}\operatorname{Sym}\left(X^\gamma\right) A^d_{x,2r}f\left(y\right)\right|&\le \left|f_\gamma\left(y\right)-f_\gamma \left(y,r\right)\right|+\frac{1}{\gamma!}\left|\operatorname{Sym}\left(X^\gamma\right) \left(A^d_{y,r}f-A^d_{x,2r}f\right)\left(y\right)\right| \\
&\stackrel{\mathclap{\mathrm{Proposition~}\ref{prop:measurableTaylor}}}{\lesssim}_{G,\alpha} \int_0^{2r} \beta_{f,d}\left(B_s\left(y\right)\right)s^{-\left|\gamma\right|-1}ds+\left|\operatorname{Sym}\left(X^\gamma\right) A^d_{y,r}\left(f-A^d_{x,2r}f\right)\left(y\right)\right|\\
&\stackrel{\mathclap{\eqref{polydiff}}}{\lesssim}_{G,\alpha} \int_0^{2r} \beta_{f,d}\left(B_s\left(y\right)\right)s^{-\left|\gamma\right|-1}ds+r^{-\left|\gamma\right|}\fint_{B_y\left(r\right)}\left|f-A^d_{x,2r}f\right|d\mu\\
&\lesssim_{G,\alpha} \int_0^{2r} \beta_{f,d}\left(B_s\left(y\right)\right)s^{-\left|\gamma\right|-1}ds+r^{-\left|\gamma\right|}\beta_{f,d}\left(B_{2r}\left(x\right)\right).
\end{align*}
Thus, by Lemma \ref{L1-approx},
\begin{align*}
\beta_{f_\gamma,d-\left|\gamma\right|}\left(B_r\left(x\right)\right)&\lesssim_{G,\alpha} \fint_{B_{r}\left(x\right)}\left(\int_0^{2r} \beta_{f,d}\left(B_s\left(y\right)\right)s^{-\left|\gamma\right|-1}ds+r^{-\left|\gamma\right|}\beta_{f,d}\left(B_{2r}\left(x\right)\right)\right)d\mu\left(y\right)\\
&\le \int_0^{2r} M\left(\beta_{f,d}\left(B_s\left(\cdot\right)\right)\right)\left(x\right)s^{-\left|\gamma\right|-1}ds+r^{-\left|\gamma\right|}\beta_{f,d}\left(B_{2r}\left(x\right)\right).
\end{align*}
By Hardy's inequality \eqref{Hardy} with $\nu=\alpha-\left|\gamma\right|+\frac 12$ and $p=2$,
\[
\mathfrak{G}_{\alpha-\left|\gamma\right|}f_\gamma\left(x\right)\lesssim_{G,\alpha} \mathfrak{G}_\alpha f\left(x\right)+\left(\int_0^\infty \left[r^{-\alpha}M\beta_{f,d}\left(B_r\left(\cdot\right)\right)\left(x\right)\right]^2\frac{dr}{r}\right)^{1/2},
\]
and the lemma follows by Corollary \ref{cor:betaIntConversion} and the aforementioned Fefferman--Stein vector-valued maximal operator inequality \eqref{FS} with $u=v=2>1$ and $g_j\left(x\right)=2^{\alpha j}\beta_{f,d}\left(B_{2^{-j+1}}\left(x\right)\right)$.
\end{proof}

The following lemma is inspired by \cite[Lemma 2]{dorronsoro1985characterization}.
\begin{lemma}\label{lem:weakFirstDeriv}
Let $f\in L^p\left(G\right)$ with $\mathfrak{G}_\alpha f\in L^p\left(G\right)$, $\alpha>1$. Then the weak partials $X_{i}f$ exist and coincide with $f_{\left(1,i\right)}$ a.e., where $\left(1,i\right)$ stands for the multi-index $\gamma$ with $\gamma_{r,j}=\delta_{\left(r,j\right),\left(1,i\right)}$.
\end{lemma}
\begin{proof}
Denote $e_{r,i}=\exp\left(x_{r,i}\right)$ for $r=1,\cdots,s$ and $i=1,\cdots,k_r$.
It is enough to show that
\[
\left(f\left(x e_{1,i}^r\right)-f\left(x\right)\right)/r\to f_{\left(1,i\right)} \mathrm{~in~}L^p\left(G\right) \mathrm{~as~}r\to 0.
\]
First,
\begin{align*}
    &\left|\frac{f\left(x e_{1,i}^r\right)-f\left(x\right)}{r}-f_{\left(1,i\right)}\left(x\right)\right|\\
    &\le \frac{\left|f\left(x e_{1,i}^r\right)-f_0\left(xe_{1,i}^r,2r\right)\right|}{r}+\frac{\left|f\left(x\right)-f_0\left(x,2r\right)\right|}{r}+\left|\frac{f_0\left(xe_{1,i}^r,2r\right)-f_0\left(x,2r\right)}{r}-f_{\left(1,i\right)}\left(x\right)\right|\\
    &\stackrel{\mathclap{\mathrm{Proposition}~\ref{prop:measurableTaylor}(2)}}{\lesssim}_{G,\alpha} r^{\alpha-1}\left(\mathfrak{G}_\alpha f\left(xe_{1,i}^r\right)+\mathfrak{G}_\alpha f\left(x\right)\right)+\left|\frac{f_0\left(xe_{1,i}^r,2r\right)-f_0\left(x,2r\right)}{r}-f_{\left(1,i\right)}\left(x\right)\right|.
\end{align*}

Next, since $f_0\left(y,r\right)=A_{y,r}^{\lfloor\alpha\rfloor}f\left(y\right)$,
\begin{align*}
    f_0\left(xe_{1,i}^r,2r\right)-f_0\left(x,2r\right)&=A^{\lfloor\alpha\rfloor}_{xe_{1,i}^r,2r}f\left(xe_{1,i}^r\right)-A^{\lfloor\alpha\rfloor}_{x,2r}f\left(xe_{1,i}^r\right)+A^{\lfloor\alpha\rfloor}_{x,2r}f\left(xe_{1,i}^r\right)-A^{\lfloor\alpha\rfloor}_{x,2r}f\left(x\right)\\
    &=A^{\lfloor\alpha\rfloor}_{xe_{1,i}^r,2r}f\left(xe_{1,i}^r\right)-A^{\lfloor\alpha\rfloor}_{x,2r}f\left(xe_{1,i}^r\right)+\sum_{j=1}^{\lfloor\alpha\rfloor}f_{j\left(1,i\right)}\left(x,2r\right)r^j/j!.
\end{align*}
Thus
\begin{align*}
    &\left|\frac{f_0\left(xe_{1,i}^r,2r\right)-f_0\left(x,2r\right)}{r}-f_{\left(1,i\right)}\left(x\right)\right|\\
    &\le \left|\frac{f_0\left(xe_{1,i}^r,2r\right)-f_0\left(x,2r\right)}{r}-f_{\left(1,i\right)}\left(x,2r\right)\right|+\left|f_{\left(1,i\right)}\left(x,2r\right)-f_{\left(1,i\right)}\left(x\right)\right|\\
    & \stackrel{\mathclap{\mathrm{Proposition}~\ref{prop:measurableTaylor}(2)}}{\lesssim}_{G,\alpha}  \quad\frac{\left|A^{\lfloor\alpha\rfloor}_{xe_{1,i}^r,2r}f\left(xe_{1,i}^r\right)-A^{\lfloor\alpha\rfloor}_{x,2r}f\left(xe_{1,i}^r\right)\right|}{r}+\sum_{j=2}^{\lfloor\alpha\rfloor}\left|f_{j\left(1,i\right)}\left(x,2r\right)\right|r^{j-1}+r^{\alpha-1}\mathfrak{G}_\alpha f\left(x\right)\\
    &\lesssim_{G,\alpha} r^{-1}\fint_{B_{r}\left(xe_{1,i}^r\right)}\left|A^{\lfloor\alpha\rfloor}_{xe_{1,i}^r,2r}f\left(y\right)-A^{\lfloor\alpha\rfloor}_{x,2r}f\left(y\right)\right|d\mu\left(y\right)+\sum_{j=2}^{\lfloor\alpha\rfloor}\left|f_{j\left(1,i\right)}\left(x,2r\right)\right|r^{j-1}+r^{\alpha-1}\mathfrak{G}_\alpha f\left(x\right)\\
    &\le r^{-1}\fint_{B_{r}\left(xe_{1,i}^r\right)}\left|f-A^{\lfloor\alpha\rfloor}_{xe_{1,i}^r,2r}f\left(y\right)\right|d\mu\left(y\right)+r^{-1}\fint_{B_{r}\left(xe_{1,i}^r\right)}\left|f-A^{\lfloor\alpha\rfloor}_{x,2r}f\left(y\right)\right|d\mu\left(y\right)\\
    &\quad+\sum_{j=2}^{\lfloor\alpha\rfloor}\left|f_{j\left(1,i\right)}\left(x,2r\right)\right|r^{j-1}+r^{\alpha-1}\mathfrak{G}_\alpha f\left(x\right)\\
    &\lesssim_G r^{-1}\beta_{f,\lfloor\alpha\rfloor}\left(B_{2r}\left(xe_{1,i}^r\right)\right)+r^{-1}\beta_{f,\lfloor\alpha\rfloor}\left(B_{2r}\left(x\right)\right)+\sum_{j=2}^{\lfloor\alpha\rfloor}\left|f_{j\left(1,i\right)}\left(x,2r\right)\right|r^{j-1}+r^{\alpha-1}\mathfrak{G}_\alpha f\left(x\right),
\end{align*}
where in the third inequality we used the fact that the $L^1\left(B_1\right)$ and $L^\infty\left(B_1\right)$ norms on the space of polynomials of weighted degree $\le \lfloor\alpha\rfloor$ are equivalent. Taking $L^p$ norms,
\begin{align*}
\left\|\frac{f\left(\cdot e_{1,i}^r\right)-f\left(\cdot\right)}{r}-f_{\left(1,i\right)}\left(\cdot\right)\right\|_{L^p\left(G\right)} &\lesssim_{G,\alpha} r^{-1}\left\|\beta_{f,\lfloor\alpha\rfloor}\left(B_{2r}\left(\cdot\right)\right)\right\|_{L^p\left(G\right)}+\sum_{j=2}^{\lfloor\alpha\rfloor}r^{j-1}\left\|f_{j\left(1,i\right)}\left(\cdot,2r\right)\right\|_{L^p\left(G\right)}\\
&\qquad+r^{\alpha-1}\left\|\mathfrak{G}_\alpha f\right\|_{L^p\left(G\right)}\\
&\stackrel{\mathclap{\mathrm{Corollary~}\ref{cor:betaIntConversion}}}{\lesssim}_{G,\alpha} r^{\alpha-1}\left\|\mathfrak{G}_\alpha f\right\|_{L^p\left(G\right)}+\sum_{j=2}^{\lfloor\alpha\rfloor}r^{j-1}\left\|f_{j\left(1,i\right)}\left(\cdot,2r\right)\right\|_{L^p\left(G\right)}.
\end{align*}
The first term converges to $0$ as $r\to 0$ since $\alpha>1$. The second term also converges to $0$ as $r\to 0$: indeed, if $j<\lfloor \alpha \rfloor$, then by Proposition \ref{prop:measurableTaylor}(1), 
\[
\left\|f_{j\left(1,i\right)}\left(\cdot,2r\right)\right\|_{L^p\left(G\right)}\lesssim_{G,\alpha} \left\|\mathfrak{G}_\alpha f\right\|_{L^p\left(G\right)}+\left\|f\right\|_{L^p\left(G\right)},
\]
and if $j=\lfloor\alpha\rfloor=\alpha$, then by Proposition \ref{prop:measurableTaylor}(4), 
\[
\left\|f_{j\left(1,i\right)}\left(\cdot,2r\right)\right\|_{L^p\left(G\right)}\lesssim_{G,\alpha} \left(\log r^{-1}\right)\left\|\mathfrak{G}_\alpha f\right\|_{L^p\left(G\right)}+\left\|f\right\|_{L^p\left(G\right)}.
\]
Thus $\left\|\frac{f\left(\cdot e_{1,i}^r\right)-f\left(\cdot\right)}{r}-f_{\left(1,i\right)}\left(\cdot\right)\right\|_{L^p\left(G\right)}\to 0$ as $r\to 0$.
\end{proof}

The following lemma is an analog of \cite[Lemma 3]{dorronsoro1985characterization}.
\begin{lemma}\label{lem:lemma3}
Let $f\in L^p\left(G\right)$ with $\mathfrak{G}_\alpha f\in L^p\left(G\right)$, $\alpha>1$. Then $\mathfrak{G}_{\alpha-1}f\in L^p\left(G\right)$ with
\[
\left\|\mathfrak{G}_{\alpha-1}f\right\|_{L^p\left(G\right)}\lesssim_{G,\alpha} \left\|\mathfrak{G}_\alpha f\right\|_{L^p\left(G\right)}+\left\|f\right\|_{L^p\left(G\right)}.
\]
\end{lemma}
\begin{proof}
Write
\[
A^{\lfloor\alpha\rfloor}_{x,r}f\left(xy\right)=\underbrace{\sum_{\left|\gamma\right|<\lfloor\alpha\rfloor}f_\gamma\left(x,r\right)y^\gamma}_{\eqqcolon A^{<\lfloor\alpha\rfloor}_{x,r}f\left(xy\right)}+\underbrace{\sum_{\left|\gamma\right|=\lfloor\alpha\rfloor}f_\gamma\left(x,r\right)y^\gamma}_{\eqqcolon A^{=\lfloor\alpha\rfloor}_{x,r}f\left(xy\right)}.
\]
Then by Lemma \ref{L1-approx},
\begin{equation}\label{eq:betaOneStepDown}
\beta_{f,\lfloor\alpha\rfloor-1}\left(B_r\left(x\right)\right)\lesssim_{G,\alpha}\fint_{B_r\left(x\right)} \left|f-A^{<\lfloor\alpha\rfloor}_{x,r}f\right|d\mu \le \beta_{f,\lfloor\alpha\rfloor}\left(B_r\left(x\right)\right)+\left\|A^{=\lfloor\alpha\rfloor}_{x,r}f\right\|_{L^\infty\left(B_r\left(x\right)\right)}.
\end{equation}
For $i\in \mathbb{Z}_{\ge 0}$ and $\left|\gamma\right|=\lfloor\alpha\rfloor$,
\[
\left|f_\gamma\left(x,2^ir\right)\right|~\stackrel{\mathclap{\eqref{eq:taylor_f_gamma}}}{\le}~ \left|\operatorname{Sym}\left(X^\gamma\right)A^{\lfloor\alpha\rfloor}_{x,2^ir}f\right|\stackrel{\mathclap{\eqref{polydiff}}}{\lesssim}_{G,\alpha} \left(2^ir\right)^{-\lfloor\alpha\rfloor}\fint_{B_{2^ir}\left(x\right)}\left|f\right|d\mu \lesssim_G \left(2^ir\right)^{-\lfloor\alpha\rfloor-n_h/p}\left\|f\right\|_{L^p\left(G\right)}\to 0\quad \mathrm{as~} i\to\infty,
\]
and so we may make a telescoping estimate as follows:
\begin{align*}
\left|f_\gamma\left(x,r\right)\right|&\le \sum_{i=0}^\infty \left|f_\gamma\left(x,2^ir\right)-f_\gamma\left(x,2^{i+1}r\right)\right|~\stackrel{\mathclap{\eqref{eq:taylor_f_gamma}}}{\le}~\sum_{i=0}^\infty \left|\operatorname{Sym}\left(X^\gamma\right)\left(A^{\lfloor\alpha\rfloor}_{x,2^ir}f-A^{\lfloor\alpha\rfloor}_{x,2^{i+1}r}f\right)\right| \\
&=\sum_{i=0}^\infty \left|\operatorname{Sym}\left(X^\gamma\right)\left(A^{\lfloor\alpha\rfloor}_{x,2^ir}\left(f-A^{\lfloor\alpha\rfloor}_{x,2^{i+1}r}f\right)\right)\right|\\
&\stackrel{\mathclap{\eqref{polydiff}}}{\lesssim}_{G,\alpha}\sum_{i=0}^\infty  \left(2^ir\right)^{-\lfloor\alpha\rfloor}\fint_{B_{2^ir}\left(x\right)}\left|f-A^{\lfloor\alpha\rfloor}_{x,2^{i+1}r}f\right|d\mu \lesssim_{G,\alpha} \sum_{i=0}^\infty  \left(2^ir\right)^{-\lfloor\alpha\rfloor} \beta_{f,\lfloor\alpha\rfloor}\left(B_{2^{i+1}r}\left(x\right)\right).\\
&\stackrel{\mathclap{\mathrm{Corollary~}\ref{cor:betaIntConversion}}}{\lesssim}_{G,\alpha}  \int_r^\infty \beta_{f,\lfloor\alpha\rfloor}\left(B_\rho\left(x\right)\right)\rho^{-\lfloor\alpha\rfloor-1}d\rho.
\end{align*}
We obtain that for $y\in B_r\left(x\right)$, 
\[
\left|A^{=\lfloor\alpha\rfloor}_{x,r}f\left(xy\right)\right|\le \sum_{\left|\gamma\right|=\lfloor\alpha\rfloor}\left|f_\gamma\left(x,r\right)y^\gamma\right|\lesssim_{G,\alpha} r^{\lfloor\alpha\rfloor}\int_r^\infty \beta_{f,\lfloor\alpha\rfloor}\left(B_\rho\left(x\right)\right)\rho^{-\lfloor\alpha\rfloor-1}d\rho.
\]
Substituting this into \eqref{eq:betaOneStepDown}, we have
\[
\beta_{f,\lfloor\alpha\rfloor-1}\left(B_r\left(x\right)\right)\lesssim_{G,\alpha}\beta_{f,\lfloor\alpha\rfloor}\left(B_r\left(x\right)\right)+r^{\lfloor\alpha\rfloor}\int_r^\infty \beta_{f,\lfloor\alpha\rfloor}\left(B_\rho\left(x\right)\right)\rho^{-\lfloor\alpha\rfloor-1}d\rho,
\]
and
\begin{align*}
\mathfrak{G}_{\alpha-1}f\left(x\right)&=\left(\int_0^\infty\left[\frac{\beta_{f,\lfloor\alpha\rfloor-1}\left(B_r\left(x\right)\right)}{r^{\alpha-1}}\right]^2 \frac{dr}{r}\right)^{1/2}\\
&\lesssim_{G,\alpha} \left(\int_0^\infty\left[\frac{\beta_{f,\lfloor\alpha\rfloor}\left(B_r\left(x\right)\right)}{r^{\alpha-1}}\right]^2 \frac{dr}{r}\right)^{1/2}+\left(\int_0^\infty \left[r^{\lfloor\alpha\rfloor-\alpha+1/2}\int_r^\infty \beta_{f,\lfloor\alpha\rfloor}\left(B_\rho\left(x\right)\right)\rho^{-\lfloor\alpha\rfloor-1}d\rho\right]^2 dr\right)^{1/2}\\
&\lesssim \left(\int_0^\infty\left[\frac{\beta_{f,\lfloor\alpha\rfloor}\left(B_r\left(x\right)\right)}{r^{\alpha-1}}\right]^2 \frac{dr}{r}\right)^{1/2}\\
&\le \left(\int_0^1\left[\frac{\beta_{f,\lfloor\alpha\rfloor}\left(B_r\left(x\right)\right)}{r^{\alpha-1}}\right]^2 \frac{dr}{r}\right)^{1/2}+\left(\int_1^\infty\left[\frac{\beta_{f,\lfloor\alpha\rfloor}\left(B_r\left(x\right)\right)}{r^{\alpha-1}}\right]^2 \frac{dr}{r}\right)^{1/2}\\
&\stackrel{\mathclap{\eqref{claim}}}{\lesssim}_{G,\alpha} \left(\int_0^1\left[\frac{\beta_{f,\lfloor\alpha\rfloor}\left(B_r\left(x\right)\right)}{r^\alpha}\right]^2 \frac{dr}{r}\right)^{1/2}+\left(\int_1^\infty\left[\frac{Mf\left(x\right)}{r^{\alpha-1}}\right]^2 \frac{dr}{r}\right)^{1/2}\\
&\lesssim_\alpha \mathfrak{G}_\alpha f\left(x\right)+Mf\left(x\right),
\end{align*}
where the second inequality uses the second form of Hardy's inequality \eqref{secondHardy} with $p=2$ and $\nu=\lfloor\alpha\rfloor-\alpha+1/2$. This completes the proof.
\end{proof}

To complete the induction, suppose the $\gtrsim$ direction of Theorem \ref{lpgenthm} holds for $\alpha-1$. If $f\in L^p\left(G\right)$ with $\mathfrak{G}_\alpha f\in L^p\left(G\right)$, by Lemma \ref{lem:lemma3}, $\mathfrak{G}_{\alpha-1}f\in L^p\left(G\right)$, so that $f\in S^p_{\alpha-1}$ by induction. On the other hand, by Proposition \ref{prop:measurableTaylor}(2) and Lemmas \ref{lem:inductOnG} and \ref{lem:weakFirstDeriv}, for $i=1,\cdots,k$, $X_if\in L^p\left(G\right)$ with $\mathfrak{G}_{\alpha-1}\left(X_if\right)\in L^p\left(G\right)$, so again by induction $X_if\in S^p_{\alpha-1}\left(G\right)$. Now, by Proposition \ref{prop:inductOnDelta}, $f\in S^p_\alpha\left(G\right)$. Tracking the norms involved in this argument, we have that
\begin{align*}
&\left\|\mathfrak{G}_\alpha f\right\|_{L^p\left(G\right)}\stackrel{\mathclap{\mathrm{Lemma~}\ref{lem:inductOnG}}}{\gtrsim}_{G,\alpha}\quad\sum_{i=1}^k \left\|\mathfrak{G}_{\alpha-1}f_{\left(1,i\right)}\right\|_{L^p\left(G\right)}\quad\stackrel{\mathclap{\mathrm{Lemma~}\ref{lem:weakFirstDeriv}}}{=}\qquad\sum_{i=1}^k \left\|\mathfrak{G}_{\alpha-1}\left(X_if\right)\right\|_{L^p\left(G\right)}\\
&\stackrel{\mathclap{\mathrm{induction}}}{\gtrsim}_{G.\alpha,p}\sum_{i=1}^k \left\|\left(-\Delta_p\right)^{\left(\alpha-1\right)/2}X_if\right\|_{L^p\left(G\right)}\stackrel{\mathclap{\mathrm{Proposition~}\ref{prop:inductOnDelta}}}{\asymp}_{G,\alpha,p} \quad\left\|\left(-\Delta_p\right)^{\alpha/2}f\right\|_{L^p\left(G\right)}.
\end{align*}
This completes the $\gtrsim$ direction of Theorem \ref{lpgenthm}.

\subsection{The case of $L^q$ $\beta$-numbers for $q>1$}\label{subsec:q>1}
It remains to prove the $\lesssim$ direction of \eqref{form13} for $q>1$.

Since the $\lesssim$ direction of \eqref{form13} has been proven for $q=1$, it is enough to show that the left-hand side of \eqref{form13} for $q>1$ satisfying \eqref{eq:dorronsoro-condition} is upper bounded up to constants by the left-hand side of \eqref{form13} when $q=1$, i.e., it is enough to prove for $f\in L^1_{\mathrm{loc}}\left(G\right)$
\begin{align}\label{eq:lq-l1}
\begin{aligned}
    &\left(\int_G \left(\int_0^\infty \left[\frac 1{r^\alpha} \beta_{f,\lfloor \alpha \rfloor,q}\left(B_r\left(x\right)\right) \right]^2 \frac{dr}{r} \right)^{p/2} d\mu\left(x\right) \right)^{1/p} \\
    &\lesssim_{G,\alpha,p,q}\left(\int_G \left(\int_0^\infty \left[\frac 1{r^\alpha} \beta_{f,\lfloor \alpha \rfloor}\left(B_r\left(x\right)\right) \right]^2 \frac{dr}{r} \right)^{p/2} d\mu\left(x\right) \right)^{1/p}
\end{aligned}
\end{align}
for $q$ in the range \eqref{eq:dorronsoro-condition}.

We remark that the proof of \eqref{eq:lq-l1} in the case of $G=\mathbb{R}^n$ and $\mathbb{H}^3$ can be respectively found in \cite[Section 5]{dorronsoro1985characterization} and \cite[Section 6]{fassler2020dorronsoro}, from which this subsection was inspired.

Denote $d=\lfloor\alpha\rfloor$. Recalling \eqref{eq:localize-beta}, we have that for $x\in G$ and $r>0$,
\[
\left|f\left(y\right) - A^d_{x,r}f\left(y\right)\right| \lesssim_{G,d} \int_0^{4r} \beta_{f,d}\left(B_s\left(y\right)\right)  \frac{ds}{s}\quad \mathrm{for~a.e.~}y\in B_r\left(x\right).
\]
Since, for $s>0$ and $y\in G$, we have
\[
\beta_{f,d}\left(B_s\left(y\right)\right)\quad\stackrel{\mathclap{\mathrm{Corollary~}\ref{monotoneBeta}}}\lesssim_{G,d}~ \beta_{f,d}\left(B_{2s}\left(z\right)\right),\quad z\in B_s\left(y\right),
\]
we obtain
\begin{equation}\label{form9}
\left|f\left(y\right) - A^d_{x,r}f\left(y\right)\right| \lesssim_{G,d} \int_0^{4r} \fint_{B_s\left(y\right)} \beta_{f,d}\left(B_{2s}\left(z\right)\right)  d\mu\left(z\right)  \frac{ds}{s},\quad \mbox{a.e. }y\in B_r\left(x\right).
\end{equation}

We now separate \eqref{eq:dorronsoro-condition} into two cases:
\[
\begin{cases}
1\le q<\frac{\min\left\{p,2\right\}n_h}{n_h-\alpha\min\left\{p,2\right\}},& \mathrm{if~}\alpha\le \frac{n_h}{\min\left\{p,2\right\}},\\
    q\le \infty,& \mathrm{if~}\alpha> \frac{n_h}{\min\left\{p,2\right\}}.
\end{cases}
\]
\begin{itemize}
    \item[Case 1.]Let
    \[
p>1, \quad 0<\alpha\le \frac{n_h}{\min\left\{p,2\right\}}, \quad\mathrm{and}~ 1 \le q < \frac{\min\left\{p,2\right\}n_h}{n_h - \alpha\min\left\{p,2\right\}}.
\]
By \eqref{eq:dorronsoro-condition}, we have
\[
\frac 1q>\frac{1}{\min\left\{p,2\right\}}-\frac{\alpha}{n_h}.
\]
By Jensen's inequality, $\beta_{f,d,q_1}\left(B_r\left(x\right)\right)\le \beta_{f,d,q_2}\left(B_r\left(x\right)\right)$ for $q_1<q_2$, so we may restrict to a range where $q$ is large:
\[
1-\frac{\alpha}{n_h}>\frac 1q>\frac{1}{\min\left\{p,2\right\}}-\frac{\alpha}{n_h}.
\]
Choose $0<\eta<\alpha$ sufficiently close to $\alpha$ such that
\[
1-\frac{\eta}{n_h}>1-\frac{\alpha}{n_h}>\frac 1q>\frac{1}{\min\left\{p,2\right\}}-\frac{\eta}{n_h}>\frac{1}{\min\left\{p,2\right\}}-\frac{\alpha}{n_h},
\]
which shows us that there is a choice of $1<w<\min\left\{p,2\right\}$ such that
\[
\frac 1q=\frac 1w-\frac \eta {n_h}.
\]
Note that $w<\min\left\{p,2\right\}\le \frac {n_h}\alpha$.

Summarising, we may choose some $1 < w < \min\left\{p,2\right\}$ and $0 < \eta < \min\left\{n_h/w,\alpha\right\}$ such that
\begin{equation}\label{form10}
q = \frac{w{n_h}}{{n_h} - \eta w}.
\end{equation}
Given this choice of $q,w$, and $\eta$, \cite[Theorem 4.1]{heikkinen2013fractional} tells us that in ${n_h}$-regular metric measure spaces, the following fractional Hardy--Littlewood maximal function $M_\eta$ is a bounded operator $L^{w}\left(G\right) \to L^{q}\left(G\right)$:
\[
M_{\eta}g\left(y\right) := \sup_{s > 0} \left\{ s^{\eta} \fint_{B_s\left(y\right)} \left|g\left(z\right)\right|  d\mu\left(z\right) \right\}.
\]
It follows that
\begin{align*}
r^{n_h/q}\beta_{f,d,q}\left(B_r\left(x\right)\right) &\stackrel{\mathclap{\eqref{form9}}}{\lesssim}_{G,\alpha} \left(\int_{B_r\left(x\right)}\left[\int_0^{4r} \fint_{B_s\left(y\right)} \beta_{f,d}\left(B_{2s}\left(z\right)\right)  d\mu\left(z\right)  \frac{ds}{s}\right]^q d\mu\left(y\right)\right)^{1/q}\\
&\le \int_0^{4r}\left(\int_{B_r\left(x\right)}\left[ \fint_{B_s\left(y\right)} \beta_{f,d}\left(B_{2s}\left(z\right)\right)  d\mu\left(z\right)  \right]^q d\mu\left(y\right)\right)^{1/q}\frac{ds}{s}\\
& \le \int_0^{4r} s^{-\eta} \left( \int_{B_r\left(x\right)} M_\eta\left(\beta_{f,d}\left(B_{2s}\left(\cdot\right)\right)\chi_{B_{5r}\left(x\right)}\left(\cdot\right)\right)\left(y\right)^q  d\mu\left(y\right) \right)^{1/q}  \frac{ds}{s}\\
& \lesssim_{G,w,q} \int_0^{4r} s^{-\eta} \left( \int_{B_{5r}\left(x\right)} \beta_{f,d}\left(B_{2s}\left(z\right)\right)^w  d\mu\left(z\right) \right)^{1/w}  \frac{ds}{s}\\
& \lesssim_G \int_0^{4r} s^{-\eta} r^{n_h/w} \left[M\left(\beta_{f,d}\left(B_{2s}\left(\cdot\right)\right)^w\right)\left(x\right)\right]^{1/w}  \frac{ds}{s},
\end{align*}
where in the second inequality we used Minkowski's inequality and in the fourth inequality we used the boundedness of $M_\eta:L^w\left(G\right) \to L^q\left(G\right)$. By \eqref{form10}, we have
\[
\beta_{f,d,q}\left(B_r\left(x\right)\right) \lesssim_{G,\alpha,p,q} r^\eta \int_0^{4r} s^{-\eta - 1} \left[M\left(\beta_{f,d}\left(B_{2s}\left(\cdot\right)\right)^w\right)\left(x\right)\right]^{1/w}  ds.
\]
Next, noting that $\alpha>\eta$ and using Hardy's inequality \eqref{Hardy} with $\nu=\alpha-\eta+1/2$, we obtain
\begin{align*}
&\left( \int_0^{\infty} \left[\frac{1}{r^\alpha}\beta_{f,d,q}\left(B_r\left(x\right)\right)\right]^2 \frac{dr}{r} \right)^{1/2} \\
& \lesssim_{G,\alpha,p,q} \left(\int_0^{\infty} r^{2\eta - 2\alpha-1} \left[ \int_{0}^{4r} s^{-\eta-1} \left[M\left(\beta_{f,d}\left(B_{2s}\left(\cdot\right)\right)^w\right)\left(x\right)\right]^{1/w}  ds \right]^2  dr \right)^{1/2} \\
& \lesssim_{\alpha,q} \left(\int_0^\infty \left[M\left(r^{-w\alpha}\beta_{f,d}\left(B_{2r}\left(\cdot\right)\right)^w\right)\left(x\right)\right]^{2/w} \frac{dr}{r} \right)^{1/2}.
\end{align*}
We conclude that \eqref{eq:lq-l1} holds:
\begin{align*}
&\left( \int_G \left( \int_0^\infty \left[\frac 1{r^\alpha} \beta_{f,d,q}\left(B_r\left(x\right)\right)\right]^2  \frac{dr}{r} \right)^{p/2}  d\mu\left(x\right) \right)^{1/p} \\
&\lesssim_{G,\alpha,p,q} \left(\int_G\left(\int_0^\infty \left[M\left(r^{-w\alpha}\beta_{f,d}\left(B_{2r}\left(\cdot\right)\right)^w\right)\left(x\right)\right]^{2/w} \frac{dr}{r} \right)^{p/2}d\mu\left(x\right)\right)^{1/p}\\
&\lesssim_{G,\alpha,p,q} \left(\int_G \left( \int_0^\infty \left[\frac{1}{r^\alpha} \beta_{f,d}\left(B_r\left(x\right)\right) \right]^{2}  \frac{dr}{r} \right)^{p/2}  d\mu\left(x\right) \right)^{1/p},
\end{align*}
where in the second inequality we applied the Fefferman--Stein vector-valued maximal function inequality \eqref{FS} with functions
\[
g_{j}\left(x\right) = 2^{-\alpha w j}\beta_{f,d}\left(B_{2^j}\left(x\right)\right)^w
\]
and exponents $u = 2/w > 1$ and $v = p/w > 1$, along with Corollary \ref{cor:betaIntConversion}.

\item[Case 2.]Let
\[
p>1, \quad \alpha>\frac{n_h}{\min\left\{p,2\right\}}, \quad\mathrm{and}~ q \le \infty.
\]
By Jensen's inequality, we may suppose $q=\infty$. Fix $0<\eta<\alpha$ and $1<q'< \min\left\{p,2\right\}$, with $\eta$ close to $\alpha$ and $q'$ close to $\min\left\{p,2\right\}$, such that
\[
\alpha>\eta>\frac{n_h}{q'}>\frac{n_h}{\min\left\{p,2\right\}}.
\]
By \eqref{form9}, for $x\in G$, $r>0$, and a.e.~$y\in B_r\left(x\right)$,
\begin{align*}
&\left|f\left(y\right) - A^d_{x,r}f\left(y\right)\right|\\ &\stackrel{\mathclap{\eqref{form9}}}\lesssim_{G,d} \int_0^{4r} \fint_{B_s\left(y\right)} \beta_{f,d}\left(B_{2s}\left(z\right)\right)  d\mu\left(z\right)  \frac{ds}{s}\le \int_0^{4r} \left(\fint_{B_s\left(y\right)} \left(\beta_{f,d}\left(B_{2s}\left(z\right)\right)\right)^{q'}d\mu\left(z\right)\right)^{1/q'}   \frac{ds}{s}\\
&\asymp_G\int_0^{4r} s^{\eta-n_h/q'}s^{-\eta}\left(\int_{B_s\left(y\right)} \left(\beta_{f,d}\left(B_{2s}\left(z\right)\right)\right)^{q'}d\mu\left(z\right)\right)^{1/q'}   \frac{ds}{s}\\
&\le r^{\eta-n_h/q'}\int_0^{4r} s^{-\eta}\left(\int_{B_{5r}\left(x\right)} \left(\beta_{f,d}\left(B_{2s}\left(z\right)\right)\right)^{q'}d\mu\left(z\right)\right)^{1/q'}   \frac{ds}{s}\\
&\asymp_G r^{\eta}\int_0^{4r}s^{-\eta}\left(\fint_{B_{5r}\left(x\right)} \left(\beta_{f,d}\left(B_{2s}\left(z\right)\right)\right)^{q'}d\mu\left(z\right)\right)^{1/q'}   \frac{ds}{s}\\
&\le r^{\eta}\int_0^{4r}s^{-\eta}\left(M\left(\beta_{f,d}\left(B_{2s}\left(\cdot\right)\right)^{q'}\right)\left(x\right)\right)^{1/q'}   \frac{ds}{s},
\end{align*}
where in the second inequality we used Jensen's inequality and in the inequality on the fourth line we used that $B_s\left(y\right)\subset B_{5r}\left(x\right)$. We deduce for $x\in G$ and $r>0$
\[
\beta_{f,d,\infty}\left(B_r\left(x\right)\right)\lesssim_{G,\alpha}r^{\eta}\int_0^{4r}s^{-\eta}\left(M\left(\beta_{f,d}\left(B_{2s}\left(\cdot\right)\right)^{q'}\right)\left(x\right)\right)^{1/q'}   \frac{ds}{s}.
\]

We proceed as in Case 1: noting that $\alpha>\eta$ and using Hardy's inequality \eqref{Hardy} with $\nu=\alpha-\eta+1/2$, we obtain
\begin{align*}
&\left( \int_0^{\infty} \left[\frac{1}{r^\alpha}\beta_{f,d,\infty}\left(B_r\left(x\right)\right)\right]^2 \frac{dr}{r} \right)^{1/2} \\
& \lesssim_{G,\alpha} \left(\int_0^{\infty} r^{2\eta - 2\alpha-1} \left[ \int_{0}^{4r} s^{-\eta-1} \left[M\left(\beta_{f,d}\left(B_{2s}\left(\cdot\right)\right)^{q'}\right)\left(x\right)\right]^{1/q'}  ds \right]^2  dr \right)^{1/2} \\
& \lesssim_{G,\alpha,p} \left(\int_0^\infty \left[M\left(r^{-q'\alpha}\beta_{f,d}\left(B_{2r}\left(\cdot\right)\right)^{q'}\right)\left(x\right)\right]^{2/q'} \frac{dr}{r} \right)^{1/2}.
\end{align*}
We conclude that \eqref{eq:lq-l1} holds:
\begin{align*}
&\left( \int_G \left( \int_0^\infty \left[\frac 1{r^\alpha} \beta_{f,d,\infty}\left(B_r\left(x\right)\right)\right]^2  \frac{dr}{r} \right)^{p/2}  d\mu\left(x\right) \right)^{1/p} \\&\lesssim_{G,\alpha,p} \left(\int_G\left(\int_0^\infty \left[M\left(r^{-q'\alpha}\beta_{f,d}\left(B_{2r}\left(\cdot\right)\right)^{q'}\right)\left(x\right)\right]^{2/q'} \frac{dr}{r} \right)^{p/2}d\mu\left(x\right)\right)^{1/p}\\
&\lesssim_{G,\alpha,p} \left(\int_G \left( \int_0^\infty \left[\frac{1}{r^\alpha} \beta_{f,d}\left(B_r\left(x\right)\right) \right]^{2}  \frac{dr}{r} \right)^{p/2}  d\mu\left(x\right) \right)^{1/p},
\end{align*}
where in the second inequality we applied the Fefferman--Stein vector-valued maximal function inequality \eqref{FS} with functions
\[
g_{j}\left(x\right) = 2^{-q'\alpha  j}\beta_{f,d}\left(B_{2^j}\left(x\right)\right)^{q'}
\]
and exponents $u = 2/q' > 1$ and $v = p/q' > 1$, along with Corollary \ref{cor:betaIntConversion}.
\end{itemize}
This completes the proof of Theorem \ref{lpgenthm}.

\section{Derivation of Theorem \ref{VvsH} from the Dorronsoro Theorem \ref{lpgenthm}}\label{sec:VvsH}
In this section, we prove Theorem \ref{VvsH}, which proves Theorem \ref{thm:lp} since it is a special case. See Section 7 of \cite{fassler2020dorronsoro} for a simple first-order version of the inequality of Theorem \ref{VvsH} in the special case of $G=\mathbb{H}^3$.

We may find a collection $\mathcal{B}_{r}$ of balls $B$ of radius $r$ whose union covers $G$, and such that the concentric balls $\hat{B}$ of radius $\left(n+1\right)r$ have bounded overlap: take their centers from a maximally $r$-separated subset of $G$,\footnote{I thank the anonymous referee for simplifying this argument.} i.e., among the subsets $N\subset G$ such that $d_G\left(x,y\right)\ge r$ for all distinct $x,y\in N$, take the one that is maximal with respect to inclusion. Then
\begin{align*}
&\int_G \left(\frac{1}{r^\alpha}\left|\sum_{j=0}^n \left(-1\right)^j \binom{n}{j}f\left(x\left(\delta_{r}\left(v\right)\right)^j\right)\right| \right)^p d\mu\left(x\right) \\
& \le \sum_{B \in {\mathcal{B}}_{r}}  \int_B \left(\frac{1}{r^\alpha}\left|\sum_{j=0}^n \left(-1\right)^j \binom{n}{j}f\left(x\left(\delta_{r}\left(v\right)\right)^j\right)\right| \right)^p d\mu\left(x\right)\\
& \lesssim_{n,p} \sum_{j=0}^n  \binom{n}{j}^p\sum_{B \in {\mathcal{B}}_{r}}  \int_B \left(\frac{1}{r^\alpha}\left|f\left(x\left(\delta_{r}\left(v\right)\right)^j\right)-A^{\lfloor\alpha\rfloor}_{\hat{B}}f\left(x\left(\delta_{r}\left(v\right)\right)^j\right)\right| \right)^p d\mu\left(x\right)\\
&\qquad\qquad +\sum_{B\in \mathcal{B}_{r}}\int_B \left(\frac{1}{r^\alpha}\left|\sum_{j=0}^n \left(-1\right)^j \binom{n}{j}A^{\lfloor\alpha\rfloor}_{\hat{B}}f\left(x\left(\delta_{r}\left(v\right)\right)^j\right)\right| \right)^p d\mu\left(x\right)\\
&\lesssim_{n,p} \sum_{B \in {\mathcal{B}}_{r}} \int_{\hat{B}} \left(\frac{1}{r^\alpha}\left|f\left(x\right)-A^{\lfloor\alpha\rfloor}_{\hat{B}}f\left(x\right)\right| \right)^{p} d\mu\left(x\right)+0\\
&= \sum_{B \in {\mathcal{B}}_{r}} \left[ \frac{1}{r^\alpha} \beta_{f,\lfloor\alpha\rfloor ,p}\left(\hat{B}\right)\right]^{p}\left|\hat{B}\right| \lesssim_{G,\alpha} \int_{G}  \left[\frac{1}{r^\alpha} \cdot \beta_{f,\lfloor\alpha\rfloor,p}\left(B_{2\left(n+1\right)r}\left(x\right)\right)\right]^{p} d\mu\left(x\right).
\end{align*}
We have used in the penultimate inequality the fact that since $A^{\lfloor\alpha\rfloor}_{\hat{B}}f$ is a polynomial of weighted degree at most $\lfloor\alpha\rfloor$, and since $\delta_r\left(v\right)\in \exp\left(V_{\lfloor\alpha/n\rfloor+1}\oplus\cdots\oplus V_s\right)$, $A^{\lfloor\alpha\rfloor}_{\hat{B}}f\left(x\left(\delta_{r}\left(v\right)\right)^j\right)$ is a polynomial in $j$ of degree at most $n-1$, and hence 
\[
\Delta_{\delta_r\left(v\right)}^n A^{\lfloor\alpha\rfloor}_{\hat{B}}f\left(x\right)=\sum_{j=0}^n \left(-1\right)^{n-j} \binom{n}{j}A^{\lfloor\alpha\rfloor}_{\hat{B}}f\left(x\left(\delta_{r}\left(v\right)\right)^j\right)=0.
\]
In the last inequality we used Corollary \ref{monotoneBeta}. Now, we may apply Minkowski's integral inequality to obtain
\begin{align*}
& \left(\int_0^\infty \left[ \int_G \left(\frac{1}{r^\alpha}\left|\sum_{j=0}^n \left(-1\right)^j \binom{n}{j}f\left(x\left(\delta_{r}\left(v\right)\right)^j\right)\right| \right)^p d\mu\left(x\right) \right]^{2/p}  \frac{dr}{r} \right)^{1/2}\\
& \lesssim_{G,\alpha,n,p} \left(\int_0^\infty \left[ \int_G \left[\frac{1}{r^\alpha}\beta_{f,\lfloor\alpha\rfloor, p}\left(B_{2\left(n+1\right)r}\left(x\right)\right)\right]^p  d\mu\left(x\right) \right]^{2/p}  \frac{dr}{r} \right)^{1/2}\\
& \le \left( \int_G \left[ \int_0^\infty \left[\frac{1}{r^\alpha} \beta_{f,\lfloor\alpha\rfloor,p}\left(B_{2\left(n+1\right)r}\left(x\right)\right)\right]^2  \frac{dr}{r} \right]^{p/2}  d\mu\left(x\right) \right)^{1/p}.
\end{align*}
Finally, Dorronsoro's Theorem \ref{lpgenthm} is applicable for $1<p\le 2$, $\alpha>0$ and $q=p$, as \eqref{eq:dorronsoro-condition} is satisfied. Thus
\[
\left( \int_G \left[ \int_0^\infty \left[\frac{1}{r^\alpha} \beta_{f,\lfloor\alpha\rfloor,p}\left(B_{2\left(n+1\right)r}\left(x\right)\right)\right]^2  \frac{dr}{r} \right]^{p/2}  d\mu\left(x\right) \right)^{1/p}\lesssim_{G,\alpha,n,p} \left\|\left(-\Delta_p\right)^{\alpha/2} f\right\|_{L^p\left(G\right)}.
\]
Inequality \eqref{eq:VvsH} follows, and the proof of Theorem \ref{VvsH} is complete.

\bibliographystyle{myalpha}
\bibliography{main}

\appendix

\section{Compactly supported Lipschitz functions belong to $Ch_0^{1,p}\left(G;X\right)$}\label{app:lipnabla}
In this section, we will prove the claims made in the paragraph containing the equations \eqref{ineq:lip-nabla} and \eqref{eq:cheeger-lipschitz}.

We start by proving \eqref{ineq:lip-nabla}: if $f:G\to X$ is continuously differentiable, then
\[\tag{\ref{ineq:lip-nabla}}
    \frac{1}{\sqrt{k}}\left\|\nabla f(x)\right\|_{\ell_2^k\left(X\right)}\le \operatorname{lip}_x\left(f\right)\le \left\|\nabla f(x)\right\|_{\ell_2^k\left(X\right)},\quad x\in G.
\]

We begin by showing that, when $G$ is a sub-Riemannian Lie group and $X$ is a Banach space, for a continuously differentiable map $f:G\to X$ we have
\begin{equation}\label{eq:lip-nabla}
    \operatorname{lip}_x\left(f\right)=\left\|\nabla f(x)\right\|_{B\left(\mathbb{R}^k; X\right)},\quad x\in G,
\end{equation}
where $B\left(\mathbb{R}^k; X\right)$ denotes the space of real-linear operators from $\mathbb{R}^k$ to $X$ with the operator norm, i.e., for $\left(v_1,\cdots,v_k\right)\in X^k$, the norm $\left\|\left(v_1,\cdots,v_k\right)\right\|_{B\left(\mathbb{R}^k; X\right)}$ denotes the operator norm of the linear mapping $\mathbb{R}^k\to X$, $\left(a_1,\cdots,a_k\right)\mapsto \sum_{i=1}^k a_iv_i$, or equivalently
\[
\left\|\left(v_1,\cdots,v_k\right)\right\|_{B\left(\mathbb{R}^k; X\right)}=\sup_{\left(a_1,\cdots,a_k\right)\in \mathbb{S}^{k-1}}\left\|\sum_{i=1}^k a_iv_i\right\|_X,\quad v_1,\cdots,v_k\in X.
\]
\begin{proof}[Proof of \eqref{eq:lip-nabla}]
    Given $g_0,g_1\in B_r(x)$ where $r>0$, for $\varepsilon\in (0,r)$ we may find a piecewise smooth path $\gamma:\left[0,1\right]\to G$ that has endpoints $\gamma\left(0\right)=g_0$ and $\gamma\left(1\right)=g_1$, is horizontal, i.e., $\gamma'\left(t\right)\in \operatorname{span}\left\{\left(X_1\right)_{\gamma\left(t\right)},\cdots,\left(X_k\right)_{\gamma\left(t\right)}\right\}$ at the points $t\in \left(0,1\right)$ of differentiability, and has total length $\int_0^1 \left|\gamma'\left(t\right)\right|dt<d_G\left(g_0,g_1\right)+\varepsilon$. As $\gamma$ has length less than $2r+\varepsilon<3r$, the image of $\gamma$ is contained in $B_{3r}\left(x\right)$. Then writing
\[
\gamma'\left(t\right)=\sum_{i=1}^k \gamma'_i(t)\left(X_i\right)_{\gamma\left(t\right)},\quad \left|\gamma'\left(t\right)\right|=\sqrt{\sum_{i=1}^k \left(\gamma'_i\left(t\right)\right)^2},
\]
we have
\begin{align*}
\left\|f\left(g_1\right)-f\left(g_0\right)\right\|_X&\le \int_0^1 \left\|\frac{df\left(\gamma\left(t\right)\right)}{dt}\right\|_Xdt=\int_0^1 \left\|\sum_{i=1}^k \gamma_i'\left(t\right)\left(X_if\left(\gamma\left(t\right)\right)\right)\right\|_Xdt\\
&\le \int_0^1 \left\|\nabla f\left(\gamma\left(t\right)\right)\right\|_{B\left(\mathbb{R}^k; X\right)}\left|\gamma'\left(t\right)\right|dt\le \left\|\nabla f\right\|_{L^\infty\left(B_{3r}\left(x\right);B\left(\mathbb{R}^k; X\right)\right)} \int_0^1\left|\gamma'\left(t\right)\right|dt\\
&\le \left\|\nabla f\right\|_{L^\infty\left(B_{3r}\left(x\right);B\left(\mathbb{R}^k; X\right)\right)} \left(d_G\left(g_0,g_1\right)+\varepsilon\right),
\end{align*}
where in the first inequality we used the triangle inequality, and in the second inequality we used the definition of the operator norm on $B\left(\mathbb{R}^k; X\right)$. As $\varepsilon$ was arbitrary, we have
\[
\left\|f\left(g_1\right)-f\left(g_0\right)\right\|_X\le \left\|\nabla f\right\|_{L^\infty\left(B_{3r}\left(x\right);B\left(\mathbb{R}^k; X\right)\right)} d_G\left(g_0,g_1\right),\quad g_0,g_1\in B_r\left(x\right).
\]
Thus
\[
\operatorname{lip}_x\left(f\right)=\limsup_{r\to 0}\sup_{g_0,g_1\in B_r\left(x\right),g_0\neq g_1}\frac{\left\|f\left(g_1\right)-f\left(g_0\right)\right\|_X}{d_G\left(g_0,g_1\right)}\le \limsup_{r\to 0}\left\|\nabla f\right\|_{L^\infty\left(B_{3r}\left(x\right);B\left(\mathbb{R}^k; X\right)\right)}\le \left\|\nabla f(x)\right\|_{B\left(\mathbb{R}^k; X\right)},
\]
where in the second inequality we used the continuity of derivatives of $f$.

In the other direction, for each $x\in G$, choose $\left(a_1,\cdots,a_k\right)\in \mathbb{S}^{k-1}$ so that
\[
\left\|\nabla f(x)\right\|_{B\left(\mathbb{R}^k; X\right)}=\left\|\sum_{i=1}^ka_iX_if\left(x\right)\right\|_X.
\]
For $t>0$, we have by definition of differentiability that
\[
f\left(x\exp\left(t\sum_{i=1}^k a_iX_i\right)\right)-f\left(x\right)=t\sum_{i=1}^k a_iX_if\left(x\right)+o\left(t\right)\quad \mathrm{as~}t\to 0,
\]
while
\begin{align*}
d_G\left(x\exp\left(t\sum_{i=1}^k a_iX_i\right),x\right)&=d_G\left(\exp\left(t\sum_{i=1}^ka_iX_i\right),e_G\right)\le \int_0^t \left|\frac{d}{ds}\exp\left(s\sum_{i=1}^k a_iX_i\right)\right|ds\\
&=\int_0^t \left|\left(\sum_{i=1}^k a_iX_i\right)_{\exp\left(s\sum_{i=1}^k a_iX_i\right)}\right|ds=t,
\end{align*}
so that
\begin{align*}
\operatorname{lip}_x\left(f\right)&\ge \lim_{t\to 0}\frac{\left\|f\left(x\exp\left(t\sum_{i=1}^k a_iX_i\right)\right)-f\left(x\right)\right\|_X}{d_G\left(x\exp\left(t\sum_{i=1}^k a_iX_i\right),x\right)}\ge \frac{\left\|t\sum_{i=1}^k a_iX_if\left(x\right)+o\left(t\right)\right\|_X}{t}\\
&=\left\|\sum_{i=1}^k a_iX_if\left(x\right)\right\|_X= \left\|\nabla f\left(x\right)\right\|_{B\left(\mathbb{R}^k; X\right)}.
\end{align*}
This completes the proof of \eqref{eq:lip-nabla}.
\end{proof}

The norm $\left\|\left(v_1,\cdots,v_k\right)\right\|_{B\left(\mathbb{R}^k; X\right)}$ satisfies
\begin{equation}\label{eq:elltwo-operator}
    \frac{1}{\sqrt{k}}\left\|\left(v_1,\cdots,v_k\right)\right\|_{\ell_2^k\left(X\right)}\le \left\|\left(v_1,\cdots,v_k\right)\right\|_{B\left(\mathbb{R}^k; X\right)}\le \left\|\left(v_1,\cdots,v_k\right)\right\|_{\ell_2^k\left(X\right)},\quad v_1,\cdots,v_k\in X,
\end{equation}
where
\[
\left\|\left(v_1,\cdots,v_k\right)\right\|_{\ell_2^k\left(X\right)}\coloneqq \sqrt{\sum_{i=1}^k \left\|v_i\right\|_X^2},\quad v_1,\cdots,v_k\in X.
\]

\begin{proof}[Proof of \eqref{eq:elltwo-operator}]

The upper bound follows from the triangle inequality and the Cauchy--Schwarz inequality: for all $\left(a_1,\cdots,a_k\right)\in \mathbb{S}^{k-1}$,
\[
\left\|\sum_{i=1}^k a_iv_i\right\|_X\le \sum_{i=1}^k \left|a_i\right|\left\|v_i\right\|_X\le \left\|\left(v_1,\cdots,v_k\right)\right\|_{\ell_2^k\left(X\right)}.
\]
For the lower bound, choose $i=\operatorname{argmax}_{j=1,\cdots,k}\left\|v_j\right\|_X$. We have
\[
\left\|\left(v_1,\cdots,v_k\right)\right\|_{B\left(\mathbb{R}^k; X\right)}\ge \left\|v_j\right\|_X\ge \frac{1}{\sqrt{k}}\left\|\left(v_1,\cdots,v_k\right)\right\|_{\ell_2^k\left(X\right)}.
\]
\end{proof}
The inequality \eqref{ineq:lip-nabla} follows from \eqref{eq:lip-nabla} and \eqref{eq:elltwo-operator}.

We now prove that compactly supported Lipschitz functions belong to $Ch_0^{1,p}\left(G;X\right)$. Let $f:G\to X$ be a compactly supported Lipschitz function. For each $n\in \mathbb{Z}_{>0}$, let $\eta_n:G\to \left[0,\infty\right)$ be a nonnegative smooth function supported in $B_{1/n}$ with unit mass:
\begin{equation}\label{eq:etan-unitmass}
    \int_G \eta_n\left(h\right)d\mu\left(h\right)=1.
\end{equation}
Consider the $G$-convolutions $f_n\coloneqq f*_G\eta_n:G\to X$, i.e.,
\begin{equation}\label{eq:fn-convolution}
\left(f*_G\eta_n\right)\left(g\right)\coloneqq\int_Gf\left(gh^{-1}\right)\eta_n\left(h\right)d\mu\left(h\right)=\int_Gf\left(h\right)\eta_n\left(h^{-1}g\right)d\mu\left(h\right).
\end{equation}
It is clear that $f_n\in C_c^\infty\left(G;X\right)$. From the density of compactly supported continuous functions in $L^{p}\left(G;X\right)$ and Young's convolution inequality, we have $f_n\to f$ in $L^p\left(G\right)$ as $n\to \infty$.

However, for each $h\in G$,
\begin{align}\label{eq:approx-lip}
\begin{aligned}
\frac{1}{\sqrt{k}}\left\|\nabla f_n\left(h\right)\right\|_{\ell_2^k(X)}&\stackrel{\mathclap{\eqref{eq:elltwo-operator}}}{\le}~ \left\|\nabla f_n\left(h\right)\right\|_{B\left(\mathbb{R}^k;X\right)}\\
&\stackrel{\mathclap{\eqref{eq:lip-nabla}}}{=}~\, ~\operatorname{lip}_h\left(f_n\right)\le \sup_{x,y\in B_{1/n}\left(h\right),~x\neq y}\frac{\left\|f_n\left(x\right)-f_n\left(y\right)\right\|_X}{d_G\left(x,y\right)}\\
&\stackrel{\mathclap{\eqref{eq:fn-convolution}}}{\le} \sup_{x,y\in B_{1/n}\left(h\right),~x\neq y}\frac{\int_{B_{1/n}}\left\|f\left(x\tilde{h}^{-1}\right)-f\left(y\tilde{h}^{-1}\right)\right\|_X\eta_n\left(\tilde{h}\right)d\mu\left(\tilde{h}\right)}{d_G\left(x,y\right)}\\
&\le \operatorname{Lip}\left(f\right)\sup_{x,y\in B_{1/n}\left(h\right),~x\neq y}\int_{B_{1/n}}\eta_n\left(\tilde{h}\right)d\mu\left(\tilde{h}\right)\\
&=\operatorname{Lip}\left(f\right).
\end{aligned}
\end{align}
so that
\[
\left\|\nabla f_n\right\|_{L^p\left(G;\ell_2^k\left(X\right)\right)}\le \sqrt{k}\operatorname{Lip}\left(f\right)\mu\left(\operatorname{supp}f_n\right)^{1/p}\le \sqrt{k}\operatorname{Lip}\left(f\right)\mu\left(\operatorname{Nbhd}_{1/n}\left(\operatorname{supp}f\right)\right)^{1/p},
\]
where $\operatorname{Nbhd}_\varepsilon\left(\cdot\right)$ denotes the $\epsilon$-neighborhood of a set. Since $\operatorname{supp}f$ is compact, we have
\[
\left\|f\right\|_{\dot{Ch}^{1,p}_0\left(G;X\right)}\le\liminf_{n\to\infty} \left\|\nabla f_n\right\|_{L^p\left(G;\ell_2^k\left(X\right)\right)}\le \sqrt{k}
 \operatorname{Lip}\left(f\right)\mu\left(\operatorname{supp}\left(f\right)\right)^{1/p}<\infty,
\]
proving \eqref{eq:cheeger-lipschitz} and showing that $f\in Ch_0^{1,p}\left(G;X\right)$.
\section{Distance formula on nilpotent Lie groups}
\label{app:doubling}
Let $G$ be a simply connected nilpotent Lie group. We will describe a distance formula up to constants.

For the local geometry, if we denote by $V$ the subspace of $\mathfrak{g}$ that generates $\mathfrak{g}$ by brackets and through which we measure distances, then, for $2\le r\le s$, we choose $X^{\mathrm{loc}}_{r,1},\cdots, X^{\mathrm{loc}}_{r,k^{\mathrm{loc}}_r}\in \underbrace{\left[V,\left[V,\cdots,V\right]\right]}_{r~\mathrm{times}}$ so that
\[
\left\{X_{r',i}^{\mathrm{loc}}\right\}_{1\le r'\le r,~1\le i\le k_{r'}^{\mathrm{loc}}} \mathrm{~is~a~basis~of~}V+\left[V,V\right]+\underbrace{\left[V,\left[V,\cdots,V\right]\right]}_{r~\mathrm{times}}
\]
(it is possible that $k_r^{\mathrm{loc}}=0$ for large enough $r\le s$). Extend these vectors to left-invariant vector fields over $G$. This defines polynomial functions $\left\{x^{\mathrm{loc}}_{r,i}\right\}_{1\le r\le s,1\le i\le k_r^{\mathrm{loc}}}$.

One can parametrize $G$ by $\mathbb{R}^n$, by first identifying $G$ with $\mathfrak{g}$ via the exponential map $\exp$, and then identifying $\mathfrak{g}$ with $\mathbb{R}^n$ via the basis $\left\{X^{\mathrm{loc}}_{r,i}\right\}_{1\le r\le s,1\le i\le k_r^{\mathrm{loc}}}$. Then, by \cite[Lemma 2.1]{jean2014control}, we have
\[
d_G\left(\exp\left(\sum_{r=1}^s\sum_{i=1}^{k^{\mathrm{loc}}_r}x^{\mathrm{loc}}_{r,i} X^{\mathrm{loc}}_{r,i}\right),e_G\right)\asymp_G \sum_{r=1}^s\sum_{i=1}^{k^{\mathrm{loc}}_r}\left|x^{\mathrm{loc}}_{r,i}\right|^{1/r},\quad \mathrm{if~}\forall r,i~\left|x^{\mathrm{loc}}_{r,i}\right|\le 1.
\]

For the global geometry, for $1\le r\le s$, we choose $X^{\mathrm{glob}}_{r,1},\cdots, X^{\mathrm{glob}}_{r,k^{\mathrm{glob}}_r}\in \underbrace{\left[\mathfrak{g},\left[\mathfrak{g},\cdots,\mathfrak{g}\right]\right]}_{r~\mathrm{times}}$ so that
\[
\left\{X_{r',i}^{\mathrm{glob}}\right\}_{r\le r'\le s,~1\le i\le k_{r'}^{\mathrm{glob}}}\mathrm{~is~a~basis~of~}\underbrace{\left[\mathfrak{g},\left[\mathfrak{g},\cdots,\mathfrak{g}\right]\right]}_{r~\mathrm{times}}.
\]
(This time, we cannot have $k_r^{\mathrm{glob}}=0$ for any $r\le s$.) Extend these vectors to left-invariant vector fields over $G$.

One can parametrize $G$ by $\mathbb{R}^n$, by first identifying $G$ with $\mathfrak{g}$ via the exponential map $\exp$, and then identifying $\mathfrak{g}$ with $\mathbb{R}^n$ via the basis $\left\{X^{\mathrm{glob}}_{r,i}\right\}_{1\le r\le s,1\le i\le k_r^{\mathrm{glob}}}$. Through this basis, we may define polynomial functions $\left\{x^{\mathrm{glob}}_{r,i}\right\}_{1\le r\le s,1\le i\le k_r^{\mathrm{glob}}}$. By \cite[Proposition 2.13]{breuillard2012nilpotent} and \cite{pansu1983croissance}, we have the distance formula
\[
d_G\left(\exp\left(\sum_{r=1}^s\sum_{i=1}^{k^{\mathrm{glob}}_r}x^{\mathrm{glob}}_{r,i} X^{\mathrm{glob}}_{r,i}\right),e_G\right)\asymp_G \sum_{r=1}^s\sum_{i=1}^{k^{\mathrm{glob}}_r}\left|x^{\mathrm{glob}}_{r,i}\right|^{1/r},\quad \mathrm{if~}\exists r,i~\left|x^{\mathrm{glob}}_{r,i}\right|> 1.
\]

In particular, denoting the bi-invariant Haar measure by $\mu$, which is the push-forward of the Lebesgue measure on $\mathfrak{g}$ by $\exp$, we have
\[
\mu\left(B_r\right)\asymp_G
\begin{cases}
r^{n^{\mathrm{loc}}_h},\quad r\le 1,\\
r^{n^{\mathrm{glob}}_h},\quad r>1,
\end{cases}
\]
where
\[
n^{\mathrm{loc}}_h=\sum_{r=1}^s rk_r^{\mathrm{loc}}=\dim\mathfrak{g}+\sum_{r=1}^{s-1}\left(\dim\mathfrak{g}-\dim\left(V+\left[V,V\right]+\underbrace{\left[V,\left[V,\cdots,V\right]\right]}_{r~\mathrm{times}}\right)\right)
\]
and
\[
n^{\mathrm{glob}}_h=\sum_{r=1}^s rk_r^{\mathrm{glob}}=\sum_{r=1}^s\dim\underbrace{\left[\mathfrak{g},\left[\mathfrak{g},\cdots,\mathfrak{g}\right]\right]}_{r~\mathrm{times}}.
\]
(As
\[
\mathfrak{g}=\left(V+\left[V,V\right]+\underbrace{\left[V,\left[V,\cdots,V\right]\right]}_{r~\mathrm{times}}\right)+\underbrace{\left[\mathfrak{g},\left[\mathfrak{g},\cdots,\mathfrak{g}\right]\right]}_{r+1~\mathrm{times}},\quad r=1,\cdots,s-1,
\]
we have $n_h^{\mathrm{loc}}\le n_h^{\mathrm{glob}}$.) We can see from the above formula that $\mu$ is a doubling measure, or, more strongly, that $\mu\left(B_{2r}\right)\asymp_G \mu\left(B_r\right)$ for all $r \ge 0$.

We remark that it is known \cite[Theorem 1.1]{breuillard2014geometry} that the limit $\lim_{r\to\infty}\mu\left(B_r\right)/r^{n_h^{\mathrm{glob}}}$ exists and is a positive real number.

\section{An $L^p$-valued Assouad embedding theorem}\label{app:assouad}
In this subsection, we prove the following theorem.\footnote{The previous version of this paper stated without proof an Assouad embedding theorem with $L^p$-targets, namely that, for a $K$-doubling metric space $\left(\mathcal{M},d\right)$ and $1\le p<\infty$, one has $c_p\left(\mathcal{M},d^{1-\varepsilon}\right)\lesssim_K \varepsilon^{-1/\max\{p,2\}}$ for $\varepsilon\in \left(0,\frac 12\right)$, since the author believed it to be common knowledge. However, the author could not find a definitive source for it. I thank the anonymous referee for pointing this out, for suggesting to prove an Assouad embedding theorem with $L^p$-targets, and for suggesting to apply the stochastic decompositions of \cite{lee2005extending}. The results of Theorem \ref{thm:lp-assouad} might be well-known among experts.}
\begin{theorem}\label{thm:lp-assouad}
    Let $\left(\mathcal{M},d\right)$ be a $K\left(\ge 2\right)$-doubling metric space, let $\varepsilon\in \left(0,1\right)$ and let $p\in \left[1,\infty\right)$. The following statements hold.
    \begin{enumerate}
    \item The snowflaked space $\left(\mathcal{M},d^{1-\varepsilon}\right)$ embeds with $O\left(1+\frac{(1-\varepsilon)^{1/\min\left\{p,2\right\}}}{\varepsilon^{1/\max\left\{p,2\right\}}}\max\left\{\frac{\log K}{p},1\right\}\right)$-bilipschitz distortion into $L^p[0,1]$.
    \item For each $R>0$, we have a map $\psi:\mathcal{M}\to L^p[0,1]$ that is
    \begin{itemize}
    \item $O\left(R^{-\varepsilon}\varepsilon^{-1/\max\left\{p,2\right\}}\max\left\{\frac{\log K}{p},1\right\}\right)$-Lipschitz as a map $(\mathcal{M},d)\to L^p[0,1]$,
    \item $O\left(\left(1-\varepsilon\right)^{-1/\max\left\{p,2\right\}}+\varepsilon^{-1/\max\left\{p,2\right\}}\max\left\{\frac{\log K}{p},1\right\}\right)$-Lipschitz as a map $(\mathcal{M},d^{1-\varepsilon})\to L^p[0,1]$,
    \item and satisfies
\begin{equation}\label{eq:large-scale-holder-lower-bound}
\left\|\psi\left(x\right)-\psi\left(y\right)\right\|_{L^p[0,1]}\gtrsim d\left(x,y\right)^{1-\varepsilon} \quad \mathrm{whenever}~d\left(x,y\right)>R.
\end{equation}
\end{itemize}
    \end{enumerate}
\end{theorem}

\begin{remark}\,
\begin{enumerate}
    \item We did not treat the case $p=\infty$ since there is an isometric embedding $\left(\mathcal{M},d^{1-\varepsilon}\right)\to \ell_\infty$ by Fr\'echet \cite{frechet1906quelques} and Kuratowski \cite{kuratowski1935quelques}. We remark that in the limit as $p\to\infty$ in Theorem \ref{thm:lp-assouad}(1), the distortion approaches a constant factor $O\left(1\right)$ independent of $\varepsilon$ and $K$, so Theorem \ref{thm:lp-assouad} is asymptotically $O(1)$-optimal in the large $p$ regime.
    \item One corollary of Theorem \ref{thm:lp-assouad} is that the $(1-\varepsilon)$-snowflake of a $K$-doubling metric space embeds $O(1)$-bilipschitzly into $L^p[0,1]$ whenever the quantity
    \[
    \frac{(1-\varepsilon)^{1/\min\left\{p,2\right\}}}{\varepsilon^{1/\max\left\{p,2\right\}}}\max\left\{\frac{\log K}{p},1\right\}
    \]
    is $O(1)$. This is achieved, for example, when $1/\varepsilon=O(1)$ and $\frac{\log K}{p}=O(1)$, or when $p=O(1)$ and $1-\varepsilon\le 1/4\log^2K$.
\end{enumerate}
        
\end{remark}

Before we begin the proof of Theorem \ref{thm:lp-assouad}, we recall the following definition.
\begin{definition}[{\cite[Definition 3.2, 3.3]{lee2005extending}}]
    Let $\Delta,\beta>0$, $\gamma\in (0,1]$ and let $\mathcal{N}$ be a metric space. A $\left(\beta,\gamma\right)$-padded $\Delta$-bounded finitely supported stochastic decomposition of $\mathcal{M}$ with respect to itself is a 4-tuple $\left(\Omega,\mu,I,\left\{\Gamma^i\left(\cdot\right)\right\}_{i\in I}\right)$, consisting of a probability space $\left(\Omega,\mu\right)$, a countable index set $I$,  and a partition $\left\{\Gamma^i\left(\omega\right)\right\}_{i\in I}$ of $\mathcal{M}$ into Borel subsets, for each element $\omega\in \Omega$, such that
\begin{itemize}
    \item for every $\omega\in \Omega$ and $i\in I$, $\operatorname{diam}\left(\Gamma^i\left(\omega\right)\right)\le \Delta$,
    \item for every $x\in \mathcal{M}$,
    \begin{equation}\label{eq:padded}
    \mu\left(\omega:d\left(x,\mathcal{M}\setminus\Gamma^x\left(\omega\right)\right)\ge \beta\Delta\right)\ge \gamma,
    \end{equation}
    where $\Gamma^x\left(\omega\right)$ denotes the partition element of $\left\{\Gamma^i(\omega)\right\}_{i\in I}$ that contains $x$.
\end{itemize}
\end{definition}

By \cite[Lemma 2.2]{naor2012assouad}, if $\mathcal{M}$ is $K$-doubling, then for each $\Delta>0$ and $\beta\in \left(0,\frac{1}{64}\right]$, $\mathcal{M}$ admits a $\left(\beta,K^{-64\beta}\right)$-padded $\Delta$-bounded finitely supported stochastic decomposition of $\mathcal{M}$ with respect to itself.

\begin{proof}[Proof of Theorem \ref{thm:lp-assouad}]
Let $\mathcal{M}$ be a nonempty $K$-doubling metric space. Since $\mathcal{M}$ is separable, since a continuous image of a separable space is separable, and since any separable subspace of an $L^p\left(\sigma\right)$ space is isometric to a subspace of $L^p\left[0,1\right]$ \cite[Fact 1.20]{ostrovskii2013metric}, it is enough to construct mappings into some $L^p(\sigma)$ space with the desired properties.

\begin{enumerate}
    \item 
    We begin by observing that a separable Hilbert space embeds isometrically into $L^p[0,1]$ for $1\le p<\infty$ \cite[Proposition 6.4.13]{albiac2006topics}, and that
\[
\min\left\{\left(\frac{1-\varepsilon}{\varepsilon}\right)^{1/p},\left(\frac{1-\varepsilon}{\varepsilon}\right)^{1/2}\right\}\asymp \frac{(1-\varepsilon)^{1/\min\left\{p,2\right\}}}{\varepsilon^{1/\max\left\{p,2\right\}}}.
\]
So, it is enough to construct an embedding with distortion $1+\frac{\left(1-\varepsilon\right)^{1/p}}{\varepsilon^{1/p}}\max\left\{\frac{\log K}{p},1\right\}$.

Now we modify the construction of the proof of \cite[Theorem 5.1]{lee2005metric}. Fix $\varepsilon\in \left(0,1\right)$, $p\ge 1$, and $\beta\in \left(0,\frac{1}{64}\right]$. For each $n\in \mathbb{Z}$, setting $\Delta=2^n\log K$, we have a $\left(\beta,K^{-64\beta}\right)$-padded $2^n\log K$-bounded finitely supported stochastic decomposition $\left(\Omega_n,\mu_n,I_n,\left\{\Gamma_n^i\left(\cdot\right)\right\}_{i\in I_n}\right)$ of $\mathcal{M}$ with respect to itself. By the Kolmogorov extension theorem, there is a probability space $\left(\Omega_n',\mu_n'\right)$ such that $\Omega_n'=\left\{-1,1\right\}^{I_n}$ and for $\omega'\in \Omega_n'$, the variables $\left\{\omega'\left(i\right)\right\}_{i\in I_n}$ are independent, $\{-1,1\}$-valued, mean-zero Bernoulli random variables. Define $\phi_n:\mathcal{M}\to L^p\left(\Omega_n\times \Omega_n'\right)$
 as\[
\phi_n\left(x\right)\left(\omega,\omega'\right)\coloneqq \omega'\left(\Gamma^x_{n}\left(\omega\right)\right)\min\left\{d\left(x,\mathcal{M}\setminus \Gamma^x_n\left(\omega\right)\right),\beta 2^n\log K\right\},\quad x\in \mathcal{M},
\]
where $\omega'\left(\Gamma^x_n\left(\omega\right)\right)$ means $\omega'\left(i\right)$ for $i\in I_n$ such that $\Gamma_n^i\left(\omega\right)=\Gamma^x_n\left(\omega\right)$.

For $x\in \mathcal{M}$, $\omega\in \Omega_n$, $\omega'\in \Omega_n'$, we have
\[
\left|\omega'\left(\Gamma^x_n\left(\omega\right)\right)\min\left\{d\left(x,\mathcal{M}\setminus \Gamma^x_n\left(\omega\right)\right),\beta 2^n\log K\right\}\right|\le \beta 2^n\log K,
\]
so 
\begin{equation}\label{eq:assouad-linfty}
\left\|\phi_n\left(x\right)\right\|_{L^p\left(\Omega_n\times\Omega_n'\right)}\le \beta 2^n\log K.
\end{equation}
For $x,y\in \mathcal{M}$, $\omega\in \Omega_n$, $\omega'\in \Omega_n'$, if $\Gamma^x_n\left(\omega\right)=\Gamma^y_n\left(\omega\right)$, then
\[
\left|\phi_n\left(x\right)\left(\omega,\omega'\right)-\phi_n\left(y\right)\left(\omega,\omega'\right)\right|\le \left|d\left(x,\mathcal{M}\setminus \Gamma^x_n\left(\omega\right)\right)-d\left(y,\mathcal{M}\setminus \Gamma^x_n\left(\omega\right)\right)\right|\le d\left(x,y\right),
\]
while if $\Gamma^x_n\left(\omega\right)\neq \Gamma^y_n\left(\omega\right)$, then
\[
\left|\phi_n\left(x\right)\left(\omega,\omega'\right)-\phi_n\left(y\right)\left(\omega,\omega'\right)\right|\le d\left(x,\mathcal{M}\setminus \Gamma^x_n\left(\omega\right)\right)+d\left(y,\mathcal{M}\setminus \Gamma^y_n\left(\omega\right)\right)\le 2d\left(x,y\right),
\]
so that
\begin{equation}\label{eq:assouad-lipschitz}
\left\|\phi_n\left(x\right)-\phi_n\left(y\right)\right\|_{L^p\left(\Omega_n\times \Omega_n'\right)}\le 2d\left(x,y\right).
\end{equation}
For $x,y\in \mathcal{M}$ such that $d\left(x,y\right)>2^n\log K$, we must have $\Gamma^x_n\left(\omega\right)\neq \Gamma^y_n\left(\omega\right)$ for $\omega\in \Omega_n$, so
\begin{align*}
&\left\|\phi_n\left(x\right)\left(\omega,\cdot\right)-\phi_n\left(y\right)\left(\omega,\cdot\right)\right\|^p_{L^p\left(\Omega_n'\right)}\\
&=\frac 12\left|\min\left\{d\left(x,\mathcal{M}\setminus \Gamma^x_n\left(\omega\right)\right),\beta 2^n\log K\right\}-\min\left\{d\left(y,\mathcal{M}\setminus \Gamma^y_n\left(\omega\right)\right),\beta 2^n\log K\right\}\right|^p\\
&\quad +\frac 12\left|\min\left\{d\left(x,\mathcal{M}\setminus \Gamma^x_n\left(\omega\right)\right),\beta 2^n\log K\right\}+\min\left\{d\left(y,\mathcal{M}\setminus \Gamma^y_n\left(\omega\right)\right),\beta 2^n\log K\right\}\right|^p\\
&\stackrel{\mathclap{\eqref{eq:unifcvxlp}}}{\ge} \min\left\{d\left(x,\mathcal{M}\setminus \Gamma^x_n\left(\omega\right)\right),\beta 2^n\log K\right\}^p+\min\left\{d\left(y,\mathcal{M}\setminus \Gamma^y_n\left(\omega\right)\right),\beta 2^n\log K\right\}^p.
\end{align*}
and
\begin{align}\label{eq:assouad-lowerbound}
\begin{aligned}
    &\left\|\phi_n\left(x\right)-\phi_n\left(y\right)\right\|_{L^p\left(\Omega_n\times \Omega_n'\right)}^p\\
    &=\mathbb{E}_{\omega\sim\mu_n}\left\|\phi_n\left(x\right)\left(\omega,\cdot\right)-\phi_n\left(y\right)\left(\omega,\cdot\right)\right\|^p_{L^p\left(\Omega_n'\right)}\\
    &=\mathbb{E}_{\omega\sim\mu_n}\left(\min\left\{d\left(x,\mathcal{M}\setminus \Gamma^x_n\left(\omega\right)\right),\beta 2^n\log K\right\}^p+\min\left\{d\left(y,\mathcal{M}\setminus \Gamma^y_n\left(\omega\right)\right),\beta 2^n\log K\right\}^p\right)\\
    &\ge \left(\beta 2^n\log K\right)^p\left(\mu_n\left(\omega:d\left(x,\mathcal{M}\setminus \Gamma^x_n\left(\omega\right)\right)\ge \beta 2^n\log K\right)+\mu_n\left(\omega:d\left(y,\mathcal{M}\setminus \Gamma^y_n\left(\omega\right)\right)\ge\beta 2^n\log K\right)\right)\\
    &\stackrel{\mathclap{\eqref{eq:padded}}}{\ge}~  2\left(\beta 2^n\log K\right)^p K^{-64\beta}.
\end{aligned}
\end{align}

We now compute for distinct $x,y\in \mathcal{M}$ that, choosing $m\in \mathbb{Z}$ with $2^m\log K <d\left(x,y\right)\le 2^{m+1}\log K$, then
\begin{align}
\begin{aligned}\label{eq:weierstrass-sum}
    &\sum_{n\in \mathbb{Z}}2^{-pn\varepsilon}\left\|\phi_n\left(x\right)-\phi_n\left(y\right)\right\|_{L^p\left(\Omega_n\times \Omega_n'\right)}^p\\
    &\le \sum_{n<m}2^{-pn\varepsilon}\left\|\phi_n\left(x\right)-\phi_n\left(y\right)\right\|_{L^p\left(\Omega_n\times \Omega_n'\right)}^p+\sum_{n\ge m}2^{-pn\varepsilon}\left\|\phi_n\left(x\right)-\phi_n\left(y\right)\right\|_{L^p\left(\Omega_n\times \Omega_n'\right)}^p\\
    &\stackrel{\mathclap{\eqref{eq:assouad-linfty},\eqref{eq:assouad-lipschitz}}}{\le}\quad \sum_{n<m}2^{-pn\varepsilon}\left(\beta 2^{n+1}\log K\right)^p+\sum_{n\ge m}2^{-pn\varepsilon}\cdot \left(2d\left(x,y\right)\right)^p\\
    &=\frac{2^{p\varepsilon}\beta^p \left(\log K\right)^p2^{m p\left(1-\varepsilon\right)}}{1-2^{-p\left(1-\varepsilon\right)}}+\frac{2^{p} d\left(x,y\right)^p 2^{-pm\varepsilon}}{1-2^{-p\varepsilon}}\\
    &< \frac{2^{p\varepsilon}\beta^p \left(\log K\right)^{p\varepsilon}d\left(x,y\right)^{p\left(1-\varepsilon\right)}}{1-2^{-p\left(1-\varepsilon\right)}}+\frac{2^{p\left(1+\varepsilon\right)}\left(\log K\right)^{p\varepsilon} d\left(x,y\right)^{p\left(1-\varepsilon\right)}}{1-2^{-p\varepsilon}}\\
    &\lesssim 2^{p\left(1+\varepsilon\right)}\left(\log K\right)^{p\varepsilon} \left(\beta^p\max\left\{\frac{1}{p\left(1-\varepsilon\right)},1\right\}+\max\left\{\frac{1}{p\varepsilon},1\right\}\right)d\left(x,y\right)^{p\left(1-\varepsilon\right)}.
    \end{aligned}
\end{align}

Now, fixing a point $x_0\in \mathcal{M}$, define
\[
\Phi:\mathcal{M}\to \bigoplus_{n\in \mathbb{Z}} L^p\left(\Omega_n\times \Omega_n'\right)=L^p\left(\bigsqcup_{n\in \mathbb{Z}}\Omega_n\times \Omega_n'\right),
\]
where $\bigoplus_{n\in \mathbb{Z}}$ denotes an $L^p$-sum, as
\[
\Phi\coloneqq\bigoplus_{n\in \mathbb{Z}}2^{-n\varepsilon}\left(\phi_n-\phi_n\left(x_0\right)\right).
\]
Setting $y=x_0$ in \eqref{eq:weierstrass-sum} tells us that $\Phi$ is well-defined, i.e., $\left\|\Phi\left(x\right)\right\|_{L^p\left(\bigsqcup_{n\in \mathbb{Z}}\Omega_n\times \Omega_n'\right)}<\infty$. Also, it follows from \eqref{eq:weierstrass-sum} that
\[
\left\|\Phi\left(x\right)-\Phi\left(y\right)\right\|_{L^p\left(\bigsqcup_{n\in \mathbb{Z}}\Omega_n\times \Omega_n'\right)}\lesssim  \left(\log K\right)^{\varepsilon} \left(\frac{\beta}{\left(1-\varepsilon\right)^{1/p}}+\frac{1}{\varepsilon^{1/p}}\right)d\left(x,y\right)^{1-\varepsilon}.
\]
For distinct $x,y\in \mathcal{M}$ and $ m\in \mathbb{Z}$ with $2^{m}\log K<d\left(x,y\right)\le 2^{\left(m+1\right)}\log K$, we have
\begin{align*}
    \left\|\Phi\left(x\right)-\Phi\left(y\right)\right\|^p_{L^p\left(\bigsqcup_{n\in \mathbb{Z}}\Omega_n\times \Omega_n'\right)}&\ge \sum_{n\le m}2^{-pn\varepsilon}\left\|\phi_n\left(x\right)-\phi_n\left(y\right)\right\|^p_{L^p\left(\Omega_n\times\Omega_n'\right)}\\
    &\stackrel{\mathclap{\eqref{eq:assouad-lowerbound}}}{\ge}~\sum_{n\le m}2^{pn\left(1-\varepsilon\right)}\beta^p\left(\log K\right)^p K^{-64\beta}\\
    &=\frac{2^{pm\left(1-\varepsilon\right)}}{1-2^{-p\left(1-\varepsilon\right)}}\beta^p\left(\log K\right)^p K^{-64\beta}\\
    &\gtrsim 2^{-p\left(1-\varepsilon\right)}\max\left\{\frac{1}{p\left(1-\varepsilon\right)},1\right\}\beta^p\left(\log K\right)^{p\varepsilon}K^{-64\beta}d\left(x,y\right)^{p\left(1-\varepsilon\right)}
\end{align*}
and so
\[
\left\|\Phi\left(x\right)-\Phi\left(y\right)\right\|_{L^p\left(\bigsqcup_{n\in \mathbb{Z}}\Omega_n\times \Omega_n'\right)}\gtrsim \frac{\beta}{\left(1-\varepsilon\right)^{1/p}} \left(\log K\right)^{\varepsilon}K^{-64\beta/p}d\left(x,y\right)^{1-\varepsilon}.
\]

All in all, we have for $x,y\in \mathcal{M}$ distinct,
\begin{align*}
\frac{\beta}{\left(1-\varepsilon\right)^{1/p}} \left(\log K\right)^{\varepsilon}K^{-64\beta/p}&\lesssim\frac{\left\|\Phi\left(x\right)-\Phi\left(y\right)\right\|_{L^p\left(\bigsqcup_{n\in \mathbb{Z}}\Omega_n\times \Omega_n'\right)}}{d\left(x,y\right)^{1-\varepsilon}}\\
&\lesssim  \left(\log K\right)^{\varepsilon} \left(\frac{\beta}{\left(1-\varepsilon\right)^{1/p}}+\frac{1}{\varepsilon^{1/p}}\right),
\end{align*}
and the $\left(1-\varepsilon\right)$-bi-H\"older distortion of the map $\Phi$ is
\[
\frac{\frac{\beta}{\left(1-\varepsilon\right)^{1/p}}+\frac{1}{\varepsilon^{1/p}}}{\frac{\beta}{\left(1-\varepsilon\right)^{1/p}} }K^{64\beta/p}\asymp \left(1+\beta^{-1}\frac{\left(1-\varepsilon\right)^{1/p}}{\varepsilon^{1/p}}\right)K^{64\beta/p}.
\]
With the choice of $\beta=\min\left\{\frac{p}{64\log K},\frac{1}{64}\right\}$, we have
\[
\left(1+\beta^{-1}\frac{\left(1-\varepsilon\right)^{1/p}}{\varepsilon^{1/p}}\right)K^{64\beta/p}\asymp 1+\frac{\left(1-\varepsilon\right)^{1/p}}{\varepsilon^{1/p}}\max\left\{\frac{\log K}{p},1\right\}.
\]
This gives statement (1).

\item
We begin by observing that a separable Hilbert space embeds isometrically into $L^p[0,1]$ for $1\le p<\infty$ \cite[Proposition 6.4.13]{albiac2006topics}.
So, it is enough to prove the case $p\in \left[2,\infty\right)$ of statement (2).

We now define 
\[
\Phi_{\ge 0}:\mathcal{M}\to \bigoplus_{n_0\le n\in \mathbb{Z}} L^p\left(\Omega_n\times \Omega_n'\right)=L^p\left(\bigsqcup_{n_0\le n\in \mathbb{Z}}(\Omega_n\times \Omega_n')\right),
\]
as
\[
\Phi_{\ge n_0}\coloneqq\bigoplus_{n_0\le n\in \mathbb{Z}}2^{-n\varepsilon}\left(\phi_n-\phi_n\left(x_0\right)\right),
\]
Setting $y=x_0$ in \eqref{eq:weierstrass-sum} tells us that $\Phi_{\ge 0}$ is well-defined, i.e., $\left\|\Phi_{\ge 0}\left(x\right)\right\|_{L^p\left(\bigsqcup_{0\le n\in \mathbb{Z}}\Omega_n\times \Omega_n'\right)}<\infty$. We have, for $x,y\in \mathcal{M}$,
\begin{align*}
\left\|\Phi_{\ge 0}\left(x\right)-\Phi_{\ge 0}\left(y\right)\right\|_{L^p\left(\bigsqcup_{0\le n\in \mathbb{Z}}\Omega_n\times \Omega_n'\right)}&\le \left\|\Phi\left(x\right)-\Phi\left(y\right)\right\|_{L^p\left(\bigsqcup_{0\le n\in \mathbb{Z}}\Omega_n\times \Omega_n'\right)}\\
&\stackrel{\mathclap{\eqref{eq:weierstrass-sum}}}{\lesssim}  \left(\log K\right)^{\varepsilon} \left(\beta\left(1-\varepsilon\right)^{-1/p}+\varepsilon^{-1/p}\right) d\left(x,y\right)^{1-\varepsilon}.
\end{align*}

Repeating the computation of \eqref{eq:weierstrass-sum}, we have
\begin{align*}
    \left\|\Phi_{\ge 0}\left(x\right)-\Phi_{\ge 0}\left(y\right)\right\|^p_{L^p\left(\bigsqcup_{0\le n\in \mathbb{Z}}\Omega_n\times \Omega_n'\right)}&\le \sum_{0\le n\in \mathbb{Z}}2^{-pn\varepsilon}\left\|\phi_n\left(x\right)-\phi_n\left(y\right)\right\|_{L^p\left(\Omega_n\times \Omega_n'\right)}^p\\
    &\stackrel{\mathclap{\eqref{eq:assouad-lipschitz}}}{\le}~ \sum_{n\ge 0}2^{-pn\varepsilon}\cdot \left(2d\left(x,y\right)\right)^p\\
    &=\frac{2^p d\left(x,y\right)^p }{1-2^{-p\varepsilon}}\\
    &\lesssim 2^{p} d\left(x,y\right)^{p}\max\left\{\frac{1}{p\varepsilon},1\right\},
\end{align*}
and
\[
\left\|\Phi_{\ge 0}\left(x\right)-\Phi_{\ge 0}\left(y\right)\right\|_{L^p\left(\bigsqcup_{0\le n\in \mathbb{Z}}\Omega_n\times \Omega_n'\right)}\lesssim  \varepsilon^{-1/p}d\left(x,y\right)
\]
i.e., $\Phi_{\ge 0}$ is $O\left(\varepsilon^{-1/p}\right)$-Lipschitz.

For distinct $x,y\in \mathcal{M}$ with $d\left(x,y\right)>\log K$, taking $ m\in \mathbb{Z}_{\ge 0}$ with $2^{m}\log K<d\left(x,y\right)\le 2^{\left(m+1\right)}\log K$, we have
\begin{align*}
    \left\|\Phi_{\ge 0}\left(x\right)-\Phi_{\ge 0}\left(y\right)\right\|_{L^p\left(\bigsqcup_{0\le n\in \mathbb{Z}}\Omega_n\times \Omega_n'\right)}&\ge 2^{-m\varepsilon}\left\|\phi_m\left(x\right)-\phi_m\left(y\right)\right\|_{L^p\left(\Omega_m\times\Omega_m'\right)}\\
    &\stackrel{\mathclap{\eqref{eq:assouad-lowerbound}}}{\ge}~\beta 2^{m\left(1-\varepsilon\right)}K^{-64\beta/p}\log K\\
    &\gtrsim \beta\left(\log K\right)^{\varepsilon}K^{-64\beta/p}d\left(x,y\right)^{1-\varepsilon}.
\end{align*}

All in all, for any $K$-doubling metric space $(\mathcal{M},d)$, we have produced a map $\Phi_{\ge 0}:\mathcal{M}\to L^p[0,1]$ that is $O\left(\varepsilon^{-1/p}\right)$-Lipschitz as a map $(\mathcal{M},d)\to L^p[0,1]$, $O\left(\left(\log K\right)^{\varepsilon} \left(\beta\left(1-\varepsilon\right)^{-1/p}+\varepsilon^{-1/p}\right) \right)$-Lipschitz as a map $(\mathcal{M},d^{1-\varepsilon})\to L^p[0,1]$, and such that
\[
\left\|\Phi_{\ge 0}\left(x\right)-\Phi_{\ge 0}\left(y\right)\right\|_{L^p}\gtrsim \beta\left(\log K\right)^{\varepsilon}K^{-64\beta/p}d\left(x,y\right)^{1-\varepsilon} \quad \mathrm{whenever}~d\left(x,y\right)> \log K.
\]
For $R>0$, replacing $(\mathcal{M},d)$ by $(\mathcal{M},\left(\log K\right)R^{-1}d)$ and setting $\psi=\frac{R^{1-\varepsilon}}{\beta \log K}\Phi_{\ge 0}$, we have a map $\psi:\mathcal{M}\to L^p[0,1]$ that is $O\left(R^{-\varepsilon}\beta^{-1}\varepsilon^{-1/p}\right)$-Lipschitz as a map $(\mathcal{M},d)\to L^p[0,1]$, $O\left(\left(1-\varepsilon\right)^{-1/p}+\beta^{-1}\varepsilon^{-1/p}\right)$-Lipschitz as a map $(\mathcal{M},d^{1-\varepsilon})\to L^p[0,1]$, and satisfies
\[
\left\|\psi\left(x\right)-\psi\left(y\right)\right\|_{L^p}\gtrsim K^{-64\beta/p}d\left(x,y\right)^{1-\varepsilon} \quad \mathrm{whenever}~d\left(x,y\right)>R.
\]
With the choice of $\beta=\min\left\{\frac{p}{64\log K},\frac{1}{64}\right\}$, the proof of statement (2) is complete.
\end{enumerate}
\end{proof}

The reason for formulating the statement of Theorem \ref{thm:lp-assouad}(2) is that sometimes we know better local embeddability properties of the space (i.e., embeddings of small balls), besides the embeddability properties that follow from being doubling. This situation occurs in the proof of Corollary \ref{cor:nilpotentlp} and Theorems \ref{thm:precise-dist-nilp} and \ref{thm:gendistortion}, when we are computing an upper bound for the $L^p$-distortion of balls of large radius in a Riemannian Lie group. Suppose we are able to ``paste'' these local embeddings into a mapping defined on the entire space, that satisfies the lower Lipschitz or lower H\"older bounds at small scales. In this case, taking an $L^p$-direct sum of the mapping guaranteed by Theorem \ref{thm:lp-assouad}(2), which satisfies the H\"older lower bound \eqref{eq:large-scale-holder-lower-bound} for large scales, and an embedding that satisfies a better Lipschitz or H\"older lower bound for small scales, can give an embedding with stronger properties than guaranteed by Theorem \ref{thm:lp-assouad}(1). (Of course, the mappings under consideration should all satisfy upper Lipschitz or H\"older bounds.)

The question remains how one might paste local embeddings, i.e., maps from a small ball, to create global mappings with local embeddability properties. One such result is presented below, which is used in the proof of Theorem \ref{thm:gendistortion} in conjunction with Theorem \ref{thm:lp-assouad}(2).
\begin{theorem}
    Let $\left(\mathcal{M},d\right)$ be a separable metric space, let $\left(X,\|\cdot\|_X\right)$ be a Banach space, and let $p\in \left[1,\infty\right)$. Suppose for some $R>0$, $\Delta>2R$, and $\gamma\in \left(0,1\right]$, that $\mathcal{M}$ admits a $\left(\frac{2R}{\Delta},\gamma\right)$-padded $\Delta$-bounded finitely supported stochastic decomposition of $\mathcal{M}$ with respect to itself.
Suppose there exists $D\ge 1$ such that
    \[
    c_X\left(B_{\Delta}\left(x\right)\right)\le D,\quad \forall x\in \mathcal{M}.
    \]
    Then, there is a $4D \frac{\Delta}{R}$-Lipschitz map $\psi:\mathcal{M}\to L^p\left(\Omega,\mu;X\right)$, for some probability space $\left(\Omega,\mu\right)$, such that
    \[
\left\|\psi\left(x\right)-\psi\left(y\right)\right\|_{L^p\left(\Omega,\mu;X\right)}\ge \gamma^{1/p} d\left(x,y\right) \quad \mathrm{whenever}~d\left(x,y\right)\le R.
\]
\end{theorem}
\begin{proof}

Let $\left(\Omega,\mu,I,\left\{\Gamma^i\left(\cdot\right)\right\}_{i\in I}\right)$ be a $\left(\frac{2R}{\Delta},\gamma\right)$-padded $\Delta$-bounded finitely supported stochastic decomposition of $\mathcal{M}$ with respect to itself, i.e., a probability space $\left(\Omega,\mu\right)$, a countable index set $I$,  and a partition $\left\{\Gamma^i\left(\omega\right)\right\}_{i\in I}$ of $\mathcal{M}$ into Borel subsets, for each element $\omega\in \Omega$, such that
\begin{itemize}
    \item for every $\omega\in \Omega$ and $i\in I$, $\operatorname{diam}\left(\Gamma^i\left(\omega\right)\right)\le \Delta$,
    \item for every $x\in \mathcal{M}$,
    \begin{equation}\label{eq:padded-spec}
    \mu\left(\omega:d\left(x,\mathcal{M}\setminus\Gamma^x\left(\omega\right)\right)\ge 2R\right)\ge \gamma,
    \end{equation}
    where $\Gamma^x\left(\omega\right)$ denotes the partition element of $\left\{\Gamma^i(\omega)\right\}_{i\in I}$ that contains $x$.
\end{itemize}
For each $\omega\in \Omega$ and $i\in I$, since $\operatorname{diam}\left(\Gamma^i\left(\omega\right)\right)\le \Delta$, we have $\Gamma^i\left(\omega\right)\subset B_\Delta\left(z\right)$ for some $z\in \mathcal{M}$, and since $c_X\left(B_{\Delta}\left(z\right)\right)\le D$, there is a mapping $f^{\omega,i}:\Gamma^i\left(\omega\right)\to X$ such that
\[
d\left(y_1,y_2\right)\le \left\|f^{\omega,i}\left(y_1\right)-f^{\omega,i}\left(y_2\right)\right\|_X \le 2D d\left(y_1,y_2\right),\quad y_1,y_2\in \Gamma^i\left(\omega\right).
\]
Possibly replacing $f^{\omega,i}$ by $f^{\omega,i}-f^{\omega,i}\left(z\right)$, we may assume $f^{\omega,i}\left(z\right)=0$. Since $f^{\omega,i}$ is $2D$-Lipschitz, it is clear that $\left\|f^{\omega,i}\left(x\right)\right\|_X\le 2D\Delta$ for every $x\in \Gamma^i\left(x\right)$.

Now define the map $\psi:\mathcal{M}\to L^p\left(\Omega,\mu;X\right)$ by setting
\[
\psi\left(x\right)\left(\omega\right)\coloneqq \min\left\{R^{-1}d\left(x,\mathcal{M}\setminus \Gamma^i\left(\omega\right)\right),1\right\}f^{\omega,i}\left(x\right),\quad x\in \mathcal{M},~\omega\in \Omega,
\]
where $i$ is the element of $I$ such that $x\in \Gamma^i\left(\omega\right)$. To see well-definedness, note that $\left\|\psi\left(x\right)\left(\omega\right)\right\|_X\le \left\|f^{\omega,i}\left(x\right)\right\|_X\le 2D\Delta$, hence  $\left\|\psi\left(x\right)\right\|_{L^p\left(\Omega,\mu;X\right)}\le 2D\Delta$. For $x,y\in \mathcal{M}$ and $\omega\in \Omega$, let $i,j\in I$ be such that $x\in \Gamma^i\left(\omega\right)$ and $y\in \Gamma^j\left(\omega\right)$. If $i\neq j$, then $y\notin \Gamma^i\left(\omega\right)$, so that
\[
\left\|\psi\left(x\right)\left(\omega\right)\right\|_X\le R^{-1}d\left(x,\mathcal{M}\setminus \Gamma^i\left(\omega\right)\right) \left\|f^{\omega,i}\left(x\right)\right\|_X\le d\left(x,y\right)\cdot 2D\frac{\Delta}{R},
\]
and likewise $\left\|\psi\left(y\right)\left(\omega\right)\right\|_X\le 2D\frac{\Delta}{R} d\left(x,y\right)$, so that $\left\|\psi\left(x\right)\left(\omega\right)-\psi\left(y\right)\left(\omega\right)\right\|_X\le 4D\frac{\Delta}{R} d\left(x,y\right)$. On the other hand, if $i=j$, then
\begin{align*}
\left|\min\left\{R^{-1}d\left(x,\mathcal{M}\setminus \Gamma^i\left(\omega\right)\right),1\right\}-\min\left\{R^{-1}d\left(y,\mathcal{M}\setminus \Gamma^i\left(\omega\right)\right),1\right\}\right|&\le R^{-1}\left|d\left(x,\mathcal{M}\setminus \Gamma^i\left(\omega\right)\right)-d\left(y,\mathcal{M}\setminus \Gamma^i\left(\omega\right)\right)\right|\\
&\le R^{-1}d\left(x,y\right),
\end{align*}
so that
\begin{align*}
\left\|\psi\left(x\right)\left(\omega\right)-\psi\left(y\right)\left(\omega\right)\right\|_X\le &\left\|\left(\min\left\{R^{-1}d\left(x,\mathcal{M}\setminus \Gamma^i\left(\omega\right)\right),1\right\}-\min\left\{R^{-1}d\left(y,\mathcal{M}\setminus \Gamma^i\left(\omega\right)\right),1\right\}\right)f^{\omega,i}\left(x\right)\right\|_X\\
&+\left\|\min\left\{R^{-1}d\left(x,\mathcal{M}\setminus \Gamma^i\left(\omega\right)\right),1\right\}\left(f^{\omega,i}\left(x\right)-f^{\omega,i}\left(y\right)\right)\right\|\\
&\le 2D\frac{\Delta} d\left(x,y\right)+D d\left(x,y\right)\le 2.5D\frac{\Delta}{R} d\left(x,y\right),
\end{align*}
where in the last inequality we used $\Delta>2R$. Thus we have
\begin{align*}
\left\|\psi\left(x\right)-\psi\left(y\right)\right\|_{L^p\left(\Omega,\mu;X\right)}^p&=\mathbb{E}_\omega \left(\left\|\psi\left(x\right)\left(\omega\right)-\psi\left(y\right)\left(\omega\right)\right\|_X^p\mathbbm{1}_{i=j}\right)+\mathbb{E}_\omega \left(\left\|\psi\left(x\right)\left(\omega\right)-\psi\left(y\right)\left(\omega\right)\right\|_X^p\mathbbm{1}_{i\neq j}\right)\\
&\le \left(4D\frac{\Delta}{R} d\left(x,y\right)\right)^p\mu\left\{\omega:i=j\right\}+\left(2.5D\frac{\Delta}{R} d\left(x,y\right)\right)^p\mu\left\{\omega:i\neq j\right\}\\
&\le \left(4D\frac{\Delta}{R} d\left(x,y\right)\right)^p,
\end{align*}
i.e.,
\[
\left\|\psi\left(x\right)-\psi\left(y\right)\right\|_{L^p\left(\Omega,\mu;X\right)}\le 4D\frac{\Delta}{R} d\left(x,y\right).
\]

On the other hand, if $d\left(x,y\right)\le R$, then on the event that $d\left(x,\mathcal{M}\setminus\Gamma^i\left(\omega\right)\right)\ge 2R$, which by \eqref{eq:padded-spec} has probability
\[
\mu\left(\omega:d\left(x,\mathcal{M}\setminus\Gamma^i\left(\omega\right)\right)\ge 2R\right)\ge \gamma.
\]
we have $i=j$ and
\[
d\left(y,\mathcal{M}\setminus\Gamma^i\left(\omega\right)\right)\ge d\left(x,\mathcal{M}\setminus\Gamma^i\left(\omega\right)\right)-d\left(x,y\right)\ge 2R-R=R,
\]
so that
\begin{align*}
\left\|\psi\left(x\right)\left(\omega\right)-\psi\left(y\right)\left(\omega\right)\right\|_X&\ge \left\|\min\left\{R^{-1}d\left(x,\mathcal{M}\setminus \Gamma^i\left(\omega\right)\right),1\right\}f^{\omega,i}\left(x\right)-\min\left\{R^{-1}d\left(y,\mathcal{M}\setminus \Gamma^i\left(\omega\right)\right),1\right\}f^{\omega,i}\left(y\right)\right\|\\
&=\left\|f^{\omega,i}\left(x\right)-f^{\omega,i}\left(y\right)\right\|\ge  d\left(x,y\right).
\end{align*}
Thus
\begin{align*}
\left\|\psi\left(x\right)-\psi\left(y\right)\right\|_{L^p\left(\Omega,\mu;X\right)}^p&\ge\mathbb{E}_\omega \left(\left\|\psi\left(x\right)\left(\omega\right)-\psi\left(y\right)\left(\omega\right)\right\|_X^p\mathbbm{1}_{d\left(x,\mathcal{M}\setminus\Gamma^i\left(\omega\right)\right)\ge 2R}\right)\\
&\ge d\left(x,y\right)^p\mu\left(\omega:d\left(x,\mathcal{M}\setminus\Gamma^i\left(\omega\right)\right)\ge 2R\right)\\
&\ge d\left(x,y\right)^p\gamma.
\end{align*}
This completes the proof.
\end{proof}

Since a $K(\ge 2)$-doubling metric space $\mathcal{M}$ admits a $\left(\beta,K^{-64\beta}\right)$-padded $\Delta$-bounded finitely supported stochastic decomposition for each $\Delta>0$ and $\beta\in \left(0,\frac{1}{64}\right]$ by \cite[Lemma 2.2]{naor2012assouad}, with the choice of $\beta=\frac 1{64}$, $\Delta=128R$, and $\gamma=\frac 1K$, we obtain the following corollary.
\begin{corollary}\label{cor:doubling-loc-ext}
    Let $\left(\mathcal{M},d\right)$ be a $K(\ge 2)$-doubling metric space, let $\left(X,\|\cdot\|_X\right)$ be a Banach space, and let $p\in \left[1,\infty\right)$.
Suppose there exists $D\ge 1$ such that
    \[
    c_X\left(B_{128R}\left(x\right)\right)\le D,\quad \forall x\in \mathcal{M}.
    \]
    Then, there is an $O(D)$-Lipschitz map $\psi:\mathcal{M}\to L^p\left(\Omega,\mu;X\right)$, for some probability space $\left(\Omega,\mu\right)$, such that
    \[
\left\|\psi\left(x\right)-\psi\left(y\right)\right\|_{L^p\left(\Omega,\mu;X\right)}\ge K^{-1/p} d\left(x,y\right) \quad \mathrm{whenever}~d\left(x,y\right)\le R.
\]
\end{corollary}
\begin{corollary}\label{cor:doubling-loc-emb}
There is a universal constant $C>0$ with the following property. Given a $K(\ge 2)$-doubling metric space $\left(\mathcal{M},d\right)$, an exponent $p\in \left[1,\infty\right)$, and a radius $R>0$, the following is true.
\begin{enumerate}
\item For each $r\ge 2R$, there is a $C\log\left(r/R\right)^{1/\max\left\{p,2\right\}}\max\left\{\frac{\log K}{p},1\right\}+CK^{1/p}\sup_{ z\in \mathcal{M}}c_p\left(B_{R}\left(z\right)\right)$-Lipschitz map $\psi:\mathcal{M}\to L^p[0,1]$ such that
\[
\left\|\psi\left(x\right)-\psi\left(y\right)\right\|_{L^p[0,1]}\ge d\left(x,y\right) \quad \mathrm{whenever}~d\left(x,y\right)\le 2r.
\]
\item We have
    \[
\sup_{ z\in \mathcal{M}}c_p\left(B_{r}\left(z\right)\right)\le    C\log\left(r/R\right)^{1/\max\left\{p,2\right\}}\max\left\{\frac{\log K}{p},1\right\}+CK^{1/p}\sup_{ z\in \mathcal{M}}c_p\left(B_{R}\left(z\right)\right),\quad r\ge 2R.
    \]
\end{enumerate}
\end{corollary}
\begin{proof}\,
\begin{enumerate}
\item
Taking $\varepsilon=\frac{1}{2\log\left(r/R\right)}$, and $\frac{R}{128}$ instead of $R$, in Theorem \ref{thm:lp-assouad}(2), we have a map $\psi_1:\mathcal{M}\to L^p[0,1]$ that is
    \begin{itemize}
    \item $O\left(R^{-\varepsilon}\log\left(r/R\right)^{1/\max\left\{p,2\right\}}\max\left\{\frac{\log K}{p},1\right\}\right)$-Lipschitz as a map $(\mathcal{M},d)\to L^p[0,1]$,
    \item and satisfies
\begin{equation}
\left\|\psi_1\left(x\right)-\psi_1\left(y\right)\right\|_{L^p[0,1]}\gtrsim R^{-\varepsilon}d\left(x,y\right) \quad \mathrm{whenever}~\frac{R}{128}<d\left(x,y\right)\le 2r.
\end{equation}
\end{itemize}
By Corollary \ref{cor:doubling-loc-ext}, we have a map $\psi_2:\mathcal{M}\to L^p[0,1]$ that is
    \begin{itemize}
    \item $O\left(\sup_{ y\in \mathcal{M}}c_p\left(B_{R}\left(y\right)\right)\right)$-Lipschitz as a map $(\mathcal{M},d)\to L^p[0,1]$,
    \item and satisfies
\begin{equation}
\left\|\psi_2\left(x\right)-\psi_2\left(y\right)\right\|_{L^p[0,1]}\gtrsim K^{-1/p}d\left(x,y\right) \quad \mathrm{whenever}~d\left(x,y\right)\le \frac{R}{128}.
\end{equation}
\end{itemize}
Take $\psi=R^\varepsilon \psi_1\oplus K^{1/p}\psi_2$.
\item For each $z\in \mathcal{M}$, the map $\left.\psi\right|_{B_r\left(z\right)}:B_r\left(z\right)\to L^p[0,1]$ gives the desired embedding.
\end{enumerate}
\end{proof}

\end{document}